\UseRawInputEncoding
\documentclass[12pt]{article}
\usepackage[T1]{fontenc}
\usepackage{lmodern}
\usepackage{longtable}
\usepackage[margin=1in]{geometry}
\usepackage{amsmath,amssymb,amsthm,mathtools}
\usepackage{microtype,indentfirst}
\usepackage{tikz}
\usepackage[section]{placeins}
\usepackage{capt-of}
\definecolor{annOrange}{RGB}{230,159,0}
\definecolor{annBlue}{RGB}{0,114,178}
\definecolor{annGreen}{RGB}{0,158,115}
\definecolor{annPurple}{RGB}{170,79,135}
\definecolor{annGold}{RGB}{240,208,48}
\usepackage[colorlinks=true,linkcolor=blue,citecolor=blue,urlcolor=blue]{hyperref}
\numberwithin{equation}{section}
\newtheorem{theorem}{Theorem}[section]
\newtheorem{lemma}[theorem]{Lemma}
\newtheorem{proposition}[theorem]{Proposition}
\newtheorem{corollary}[theorem]{Corollary}
\theoremstyle{remark}
\newtheorem{remark}[theorem]{Remark}
\newcommand{\R}{\mathbb R}

\newcommand{\dd}{\,\mathrm d}
\title{P\'olya's conjecture for Neumann eigenvalues of annuli\thanks{Research supported by NNSF of China (No. 12371110).}}
\author{Guowei Dai\thanks{School of Mathematical Sciences, Dalian
University of Technology, Dalian 116024, P.R. China.
Email: daiguowei@dlut.edu.cn.}
\and Yingxin Sun\thanks{School of Mathematical Sciences, Dalian
University of Technology, Dalian 116024, P.R. China.
Email: sunyingxin2023@mail.dlut.edu.cn.}}
\date{}

\usepackage{xcolor}
\hypersetup{colorlinks=false,pdfborder={0 0 0}}
\begin{document}
\maketitle
\begin{abstract}
We prove strict P\'olya inequalities for every positive Neumann eigenvalue of
every Euclidean annulus. The proof preserves the exact Weyl
area term through a radial action identity and combines variational
bounds with uniform Bessel-phase and angular-sum estimates. Explicit
polynomial positivity arguments and angular-layer comparisons
cover the remaining finite spectral ranges. We also obtain a positive
counting surplus uniform in the radius ratio, an explicit subtraction
of order $\sqrt{k}/(R^2-r^2)$ for all $k\ge1604$, and a positive
square-root correction over the entire spectrum with a universal,
non-explicit coefficient. The square-root order is sharp, and the
first positive pair and the low-energy count admit optimal constants
expressed by the first positive zero of $J_1'$. Neumann inequalities for annular products
and congruent annular sectors follow by established transfer principles.
\end{abstract}


\section{Introduction}\label{sec:introduction}

\indent Let $\Omega\subset\mathbb R^d$, $d\ge2$, be a bounded connected domain
with piecewise smooth boundary. We consider the eigenvalues of the
positive Laplacian $-\Delta$, with Dirichlet boundary condition
$u=0$ or Neumann boundary condition $\partial_nu=0$, where
$\partial_n$ is differentiation in the outward normal direction.
Write
\[
 0<\lambda_1\left(\Omega\right)\le\lambda_2\left(\Omega\right)\le\cdots,
 \,\,\,
 0=\mu_0\left(\Omega\right)<\mu_1\left(\Omega\right)
 \le\mu_2\left(\Omega\right)\le\cdots,
\]
with multiplicities included. If $\left|\Omega\right|$ is the volume of $\Omega$
and $\omega_d$ is the volume of the unit ball in $\mathbb R^d$, the
leading term in Weyl's law is
\[
 W_k\left(\Omega\right)
 =4\pi^2\left(\frac{k}{\omega_d\left|\Omega\right|}\right)^{2/d}.
\]
Thus both $\lambda_k/W_k$ and $\mu_k/W_k$ tend to one as $k\to+\infty$. Beginning with Weyl's work \cite{Weyl}, this asymptotic law
has expressed a basic link between volume and the distribution of
eigenvalues. P\'olya's conjecture asks for considerably more:
\[
 \mu_k\left(\Omega\right)\le W_k\left(\Omega\right)
 \le\lambda_k\left(\Omega\right),\,\,\, k\ge1.
\]
It places every eigenvalue on the prescribed side of the asymptotic
value, including the low and intermediate parts of the spectrum.

P\'olya formulated the conjecture in 1954 \cite{Polya1954} and proved
it for planar tiling domains in his 1961 paper \cite{Polya1961}.
A tiling domain is one whose congruent copies cover Euclidean space,
apart from a set of measure zero, without overlapping interiors.
His proof combines such a geometric replication with spectral
bracketing, which compares spectra by imposing or releasing
conditions at internal interfaces, and Weyl's law on a large comparison region. This
converts an asymptotic statement into an inequality at every
spectral index, and remains a foundational contribution to the
problem. P\'olya's Neumann argument required an additional
regularity assumption on the tiling; Kellner removed this
restriction \cite{Kellner}. Urakawa extended this Dirichlet approach
in arbitrary dimension by incorporating the lattice packing density,
the largest volume fraction occupied by non-overlapping lattice
translates of the domain. The exact P\'olya coefficient is recovered
for lattice tiles \cite{Urakawa}.

For general domains, sharp averaged estimates remain central.
Berezin and Li--Yau established Dirichlet trace and eigenvalue-sum
bounds \cite{Berezin,LiYau}; Kr\"oger obtained Neumann sum and
individual bounds \cite{Kroger}. Their individual coefficients are
\[
 \lambda_k\left(\Omega\right)\ge\frac{d}{d+2}W_k\left(\Omega\right),\,\,\,
 \mu_k\left(\Omega\right)\le\left(\frac{d+2}{2}\right)^{2/d}W_k\left(\Omega\right).
\]
Melas added a geometric correction to the Dirichlet sum inequality
\cite{Melas}. Frank and Larson extended sharp Riesz-mean bounds
to certain exponents below one on convex domains \cite{FrankLarson},
and Gan, Jiang, and Lin improved the Berezin--Li--Yau and Kr\"oger
bounds \cite{GanJiangLin}. A Riesz mean is an average
$\sum_j\left(E-\lambda_j\right)_+^\gamma$, where $E$ is an energy,
$\gamma>0$, and $x_+=\max\left\{x,0\right\}$; the Neumann version is analogous.
At low indices, the Szeg\H{o}--Weinberger and Bucur--Henrot
inequalities settle the first and second positive Neumann
eigenvalues \cite{Szego,Weinberger,BucurHenrot}.

Geometric extensions include Laptev's product principle
\cite{Laptev}, the non-tiling families of Freitas and Salavessa
\cite{FreitasSalavessa}, and the thin-product results of He and
Wang \cite{HW}. Jiang and Lin give quantified Dirichlet bounds
with arbitrarily small coefficient loss above an explicit threshold
\cite{JiangLin}; Filonov improves the universal Neumann coefficient
on planar convex domains \cite{FilonovConvex}. These developments
give substantial control of averages, geometric families, and
spectral ranges. The exact individual inequalities for arbitrary
bounded Euclidean domains remain open.

Rotationally symmetric domains have provided decisive recent progress.
Filonov, Levitin, Polterovich, and Sher proved both P\'olya inequalities
for the planar disk and circular sectors of arbitrary aperture,
and the Dirichlet inequality for balls in all dimensions
\cite[Theorem~1.2, Corollary~1.6, and Theorem~1.12]{FLPSballs}.
Their uniform comparisons between spectral counts and lattice counts
made the leading Weyl term effective throughout the spectrum.
They subsequently established the Dirichlet inequality for every
annulus, combining variational comparisons, Bessel-phase
estimates, refined lattice counting, and a rigorous finite
verification \cite[Theorem~1.1]{FLPS}.
Guo, Miao, Wang, and Zhan refined the disk bounds and obtained further
results for balls and cylinders \cite{GMWZ}.
The remaining Neumann ball cases in dimensions at least three
have been settled in recent preprints of Filonov, Levitin,
Polterovich, and Sher \cite[Theorem~1.1]{FLPSneumann} and,
independently, Li \cite[Theorem~1.1]{LiNeumann}.
The present paper concerns the Neumann problem on annuli.

The annular problem has two boundary circles and two singular
geometric limits: a vanishing hole and a vanishing relative width.
There is no general Neumann domain-monotonicity principle that
would reduce all radii to the disk. The two-term Weyl
law is a high-frequency tool rather than an all-frequency count;
the general boundary correction is explained by Ivrii's theorem
under its non-periodicity hypothesis \cite{Ivrii}, meaning that
periodic billiard trajectories occupy a set of measure zero in
phase space. For annuli the following more specific results
make the gap to the present theorem precise.

For $0<r<R$, set
\[
 A_{r,R}=\left\{x\in\mathbb R^2:r<\left|x\right|<R\right\},\,\,\,
 h=R-r,\,\,\, D=R^2-r^2,\,\,\, \delta=\frac{h}{r}.
\]
Thus $h$ is the width, $\delta$ is the relative width, and
$\left|A_{r,R}\right|=\pi D$. Denote the Neumann eigenvalues by
$0=\nu_0<\nu_1\le\nu_2\le\cdots$, with multiplicities, and define
the counting function at energy $\Lambda$ by
\[
 N_A\left(\Lambda\right)
 =\#\left\{j\ge0:\nu_j\left(A_{r,R}\right)\le\Lambda\right\}.
\]
The inclusive count contains the zero eigenvalue for every $\Lambda\ge0$.
We also use the strict counting function
\[
 \mathcal M_A(\Lambda)
 =\#\{j\ge0:\nu_j(A_{r,R})<\Lambda\}.
\]
Thus $\mathcal M_A(0)=0$, whereas its positive-energy count includes
the constant mode. The notation $\mathcal M_\Omega$ has the same
meaning for any other Neumann domain $\Omega$.

For this family, Song and Zou adapted Filonov's variational packing
method \cite{FilonovConvex} to obtain explicit Neumann bounds.
In the proof of \cite[Theorem~1.1, equation~(3.3)]{SZ}, they establish
\[
 N_A\left(\Lambda\right)\ge
 c_{\mathrm{SZ}}\left(\delta\right)\left|A\right|\Lambda,
 \,\,\, \delta>0,\quad \Lambda\ge0,
\]
where
\[
 c_{\mathrm{SZ}}\left(\delta\right)
 =\frac{11}{j_{0,1}^{\,2} B}\,
   \frac{s_\delta-1}{s_\delta+1},
 \,\,\,
 B=22\sqrt3+\pi,
 \,\,\,
 s_\delta=\sqrt{1+\frac{B}{\pi}\delta\left(\delta+2\right)}.
\]
Here $j_{0,1}$ is the first positive zero of the Bessel function
$J_0$ of the first kind and order zero.
Their coefficient improves Kr\"oger's $1/\left(8\pi\right)$ for sufficiently
large relative width and approaches $0.0461144$ as
$\delta\to+\infty$, whereas P\'olya's coefficient is
$1/\left(4\pi\right)\approx0.0795775$.
They also obtain annular-sector estimates and observe that exact
P\'olya bounds follow for sufficiently narrow congruent sectors
of a fixed annulus \cite[Theorems~1.4--1.5 and Remark~1.6]{SZ}.
Our aim is the exact Weyl coefficient for every relative width
and the entire Neumann spectrum, with a sector consequence for
every aperture $2\pi/q$, $q\ge2$.

The numerical comparison is made in Table~\ref{tab:coefficients}.
For a fixed relative width the best of the two quoted bounds has
coefficient $\max\left\{1/\left(8\pi\right),c_{\mathrm{SZ}}\left(\delta\right)\right\}$.
The limiting row denotes $\delta\to+\infty$, rather than a finite
annulus. All decimal entries are rounded illustrations of the
displayed exact formulas and are used only for illustration.
\begin{table}[htbp]
\centering\small
\begin{tabular}{c|c|c|c}
$\delta$ & $c_{\mathrm{SZ}}\left(\delta\right)$ & Best quoted coefficient & Weyl coefficient\\\hline
1/4 & 0.022440 & 0.039789 & 0.079577\\
1 & 0.033575 & 0.039789 & 0.079577\\
4 & 0.041203 & 0.041203 & 0.079577\\
$\to+\infty$ & 0.046114 & 0.046114 & 0.079577
\end{tabular}
\caption{Coefficients multiplying $\left|A\right|\Lambda$ in the Neumann
counting lower bounds of Kr\"oger and Song--Zou, compared with
the exact Weyl coefficient proved here.}
\label{tab:coefficients}
\end{table}

For each fixed annulus, the Neumann Weyl remainder estimates of
Guo, Jiang, Wang, and Yang \cite[Theorem~1.1]{GJWY} give
\[
 N_A\left(\Lambda\right)
 =\frac{\left|A\right|}{4\pi}\Lambda
 +\frac{\left|\partial A\right|}{4\pi}\sqrt\Lambda
 +O_{r,R}\left(\Lambda^{1/3}\right),
 \,\,\, \Lambda\to+\infty,
\]
where $\left|\partial A\right|=2\pi\left(r+R\right)$ is the boundary length.
The positive boundary term therefore yields P\'olya's inequality
at sufficiently high energy for each fixed annulus.
This asymptotic result supplies neither an explicit threshold
uniform in the radii nor a bound at the intervening energies.
Here the remainder notation means that its absolute value is
bounded by $C_{r,R}\Lambda^{1/3}$ for sufficiently large energy,
with a constant depending on the two radii. Our objective is
the complete inequality at every energy and every radius ratio.

\begin{theorem}\label{thm:allannuli}
For every $0<r<R$ and every $\Lambda\ge0$,
\begin{equation}\label{eq:allannulicount}
 N_A\left(\Lambda\right)>\frac{R^2-r^2}{4}\Lambda.
\end{equation}
The corresponding eigenvalue inequalities are strict:
\begin{equation}\label{eq:allannulieigenvalue}
 \nu_k\left(A_{r,R}\right)<\frac{4k}{R^2-r^2},
 \,\,\, k\ge1.
\end{equation}
\end{theorem}

The strict inequality in \eqref{eq:allannulieigenvalue} is obtained
from strictly subthreshold trial spaces and phase counts; it is not
merely an inversion of the inclusive inequality
\eqref{eq:allannulicount}. These comparisons also provide quantitative
improvements. Section~\ref{sec:positivecorrections} proves, with
$D=R^2-r^2$ and $U=R\sqrt\Lambda$, that
\[
 N_A(\Lambda)>\frac{D\Lambda}{4}+\frac U{20}
 \qquad(U\ge80),
\]
and hence
\[
 \nu_k(A_{r,R})\le\frac4D
       \left(\sqrt{k+\frac1{400}}-\frac1{20}\right)^2
 \qquad(k\ge1604).
\]
The stronger surplus before simplification and its eigenvalue
consequence are given in Theorem~\ref{thm:quantitativeallannuli}
and Corollary~\ref{cor:quantitativeallannulieigen}.
There is also a dimensionless universal constant $\gamma>0$ such that
\[
 \begin{gathered}
 \nu_k(A_{r,R})\le\frac{4k-\gamma\sqrt k}{D}\quad(k\ge1),\\
 N_A(\Lambda)>\frac{D\Lambda}{4}
              +\frac\gamma8+\frac\gamma{16}\sqrt{D\Lambda}\quad(\Lambda\ge0).
 \end{gathered}
\]
Theorem~\ref{thm:fullspectrumsqrt} establishes this all-spectrum
assertion by combining the explicit high-index estimate with
strictness, continuity, and the two geometric endpoint estimates.
The coefficient $\gamma$ is non-explicit. The square-root order is
sharp by the two-term Weyl law, whereas the best universal coefficient
is not determined here. Section~\ref{sec:positivecorrections} gives
stronger high-frequency coefficient choices and counterexamples to
larger powers and to excessive coefficients. It also proves the sharp
first-pair bound and the optimal low-energy surplus
\[
 \begin{gathered}
 D\nu_1=D\nu_2<(j'_{1,1})^2,\\
 N_A(\Lambda)>\frac{D\Lambda}{4}+1-\frac{(j'_{1,1})^2}{4}
 \quad(0\le D\Lambda\le32).
 \end{gathered}
\]
Here $j'_{1,1}$ is the first positive zero of $J_1'$.

Figure~\ref{fig:annularcount} illustrates the two functions in
\eqref{eq:allannulicount} as functions of the energy $\Lambda$.
The jump convention agrees with the inclusive count used here.

\begin{figure}[htbp]
\centering
\begin{tikzpicture}[x=.64cm,y=.24cm,font=\small]
  \draw[->] (0,0) -- (17.15,0) node[right] {$\Lambda$};
  \draw[->] (0,0) -- (0,23.3) node[above] {count};
  \foreach \xx in {4,8,12,16}{\draw (\xx,0) -- (\xx,-.35) node[below] {$\xx$};}
  \foreach \yy in {5,10,15,20}{\draw (0,\yy) -- (-.13,\yy) node[left] {$\yy$};}
  \node[below left] at (0,0) {$0$};
  \draw[black,thick] (0,0) -- (16,12);
  \node[black,below right] at (12.1,8.8) {$3\Lambda/4$};
  \draw[annBlue,thick] (.9,21.2) -- (2,21.2);
  \node[right] at (2.2,21.2) {$N_{A_{1,2}}\left(\Lambda\right)$};
  \draw[black,thick] (8.7,21.2) -- (9.8,21.2);
  \node[right] at (10,21.2) {Weyl term};
  \fill[annBlue] (0,1) circle (1.5pt);
  \draw[annBlue,thick] (0.00000000,1) -- (0.45878406,1);
  \draw[annBlue,densely dotted] (0.45878406,1) -- (0.45878406,3);
  \draw[annBlue,fill=white] (0.45878406,1) circle (1.5pt);
  \fill[annBlue] (0.45878406,3) circle (1.5pt);
  \draw[annBlue,thick] (0.45878406,3) -- (1.79721411,3);
  \draw[annBlue,densely dotted] (1.79721411,3) -- (1.79721411,5);
  \draw[annBlue,fill=white] (1.79721411,3) circle (1.5pt);
  \fill[annBlue] (1.79721411,5) circle (1.5pt);
  \draw[annBlue,thick] (1.79721411,5) -- (3.91595455,5);
  \draw[annBlue,densely dotted] (3.91595455,5) -- (3.91595455,7);
  \draw[annBlue,fill=white] (3.91595455,5) circle (1.5pt);
  \fill[annBlue] (3.91595455,7) circle (1.5pt);
  \draw[annBlue,thick] (3.91595455,7) -- (6.69574554,7);
  \draw[annBlue,densely dotted] (6.69574554,7) -- (6.69574554,9);
  \draw[annBlue,fill=white] (6.69574554,7) circle (1.5pt);
  \fill[annBlue] (6.69574554,9) circle (1.5pt);
  \draw[annBlue,thick] (6.69574554,9) -- (10.04537236,9);
  \draw[annBlue,densely dotted] (10.04537236,9) -- (10.04537236,11);
  \draw[annBlue,fill=white] (10.04537236,9) circle (1.5pt);
  \fill[annBlue] (10.04537236,11) circle (1.5pt);
  \draw[annBlue,thick] (10.04537236,11) -- (10.21811334,11);
  \draw[annBlue,densely dotted] (10.21811334,11) -- (10.21811334,12);
  \draw[annBlue,fill=white] (10.21811334,11) circle (1.5pt);
  \fill[annBlue] (10.21811334,12) circle (1.5pt);
  \draw[annBlue,thick] (10.21811334,12) -- (10.77461712,12);
  \draw[annBlue,densely dotted] (10.77461712,12) -- (10.77461712,14);
  \draw[annBlue,fill=white] (10.77461712,12) circle (1.5pt);
  \fill[annBlue] (10.77461712,14) circle (1.5pt);
  \draw[annBlue,thick] (10.77461712,14) -- (12.47001477,14);
  \draw[annBlue,densely dotted] (12.47001477,14) -- (12.47001477,16);
  \draw[annBlue,fill=white] (12.47001477,14) circle (1.5pt);
  \fill[annBlue] (12.47001477,16) circle (1.5pt);
  \draw[annBlue,thick] (12.47001477,16) -- (13.92094994,16);
  \draw[annBlue,densely dotted] (13.92094994,16) -- (13.92094994,18);
  \draw[annBlue,fill=white] (13.92094994,16) circle (1.5pt);
  \fill[annBlue] (13.92094994,18) circle (1.5pt);
  \draw[annBlue,thick] (13.92094994,18) -- (15.36682767,18);
  \draw[annBlue,densely dotted] (15.36682767,18) -- (15.36682767,20);
  \draw[annBlue,fill=white] (15.36682767,18) circle (1.5pt);
  \fill[annBlue] (15.36682767,20) circle (1.5pt);
  \draw[annBlue,thick] (15.36682767,20) -- (16,20);
\end{tikzpicture}
\caption{
Neumann count for $A_{1,2}$, $0\le\Lambda\le16$ (blue, including zero), and Weyl term $3\Lambda/4$ (black). Filled upper and open lower dots denote the count and its left limit; dotted verticals join them. Positive angular orders give double jumps. The presentation follows \cite[Figure~1]{FLPSballs}; the plotted positive eigenvalues are floating-point approximations from the radial Neumann equation in Section~\ref{sec:radialeigenvalues}. This numerical illustration is not used in the proof.}
\label{fig:annularcount}
\end{figure}
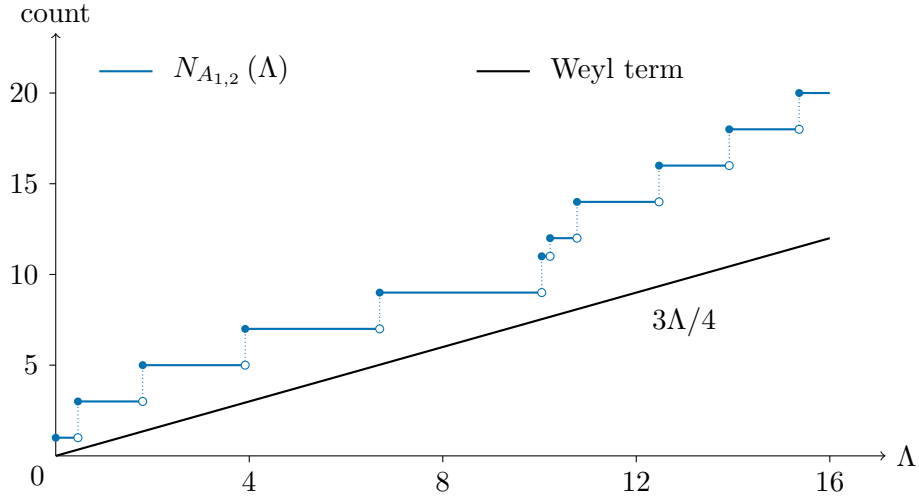

Two geometric consequences of Theorem~\ref{thm:allannuli} follow
from established transfer principles. We state them here to
describe the additional domains covered by the annular theorem.

\begin{corollary}\label{cor:annularproducts}
\leavevmode\normalcolor
Let $\Omega\subset\mathbb R^d$, $d\ge1$, be a bounded domain with
piecewise smooth boundary. The Neumann eigenvalues of
$A_{r,R}\times\Omega$, enumerated from $\mu_0=0$ and with
multiplicities, satisfy
\[
 \mu_k\left(A_{r,R}\times\Omega\right)<
 4\pi^2\left(
 \frac{k}{\omega_{d+2}\left|A_{r,R}\right|\,\left|\Omega\right|}
 \right)^{2/\left(d+2\right)},\,\,\, k\ge1.
\]
In particular, every annular cylinder satisfies this inequality.
\end{corollary}

\begin{corollary}\label{cor:annularsectors}
\leavevmode\normalcolor
For an integer $q\ge2$, let
\[
 S_q=\left\{
 \left(s\cos\theta,s\sin\theta\right):
 r<s<R,\ 0<\theta<2\pi/q
 \right\}.
\]
With Neumann conditions on the circular and straight parts of its
boundary, its counting function and eigenvalues satisfy
\[
 N_{S_q}\left(\Lambda\right)>\frac{\left|S_q\right|}{4\pi}\Lambda
 \,\,\,\left(\Lambda\ge0\right),\,\,\,
 \mu_k\left(S_q\right)<\frac{4qk}{R^2-r^2}
 \,\,\,\left(k\ge1\right).
\]
\end{corollary}

The proof of Theorem~\ref{thm:allannuli} uses variational bounds,
uniform phase estimates, and finite lists of explicit
inequalities. The low-frequency comparisons are polynomial identities
with explicit positivity arguments. The remaining angular layers follow
from a fixed-area endpoint principle and rational action minorants.
The thin-annulus theorem remains an independent analytic proof.
The product and sector consequences use the transfer
arguments of Laptev and Neumann bracketing.

The Dirichlet annulus proof of Filonov, Levitin, Polterovich,
and Sher already combines variational comparisons, differences
of two Bessel phases, lattice counting, and finite verification
\cite[Sections~1.3, 4, and 5]{FLPS}. We use this framework and
their uniform phase enclosures \cite{Phase}. The Neumann
problem requires an additional analysis of the boundary conditions
and the direction of the counting error. For Dirichlet conditions
the ordinary Bessel phase difference increases with frequency
\cite[Lemma~4.1]{FLPS}. For Neumann conditions the derivatives of
the Bessel functions occur instead. In a positive angular order,
their continuous phase decreases below its turning argument and
increases above it; its difference at the two boundary arguments
can be negative at low frequency. An exact count therefore
requires the positive-crossing argument in
Lemma~\ref{lem:exactphasecount}. The initial count is also different:
the constant Neumann mode has zero energy and occurs once,
whereas the positive angular orders have no zero mode.

The boundary-condition obstruction is visible directly in the
radial equation. The transformation $w\left(s\right)=\sqrt{s}\,u\left(s\right)$ gives
Robin endpoint conditions, meaning that the derivative is
proportional to the endpoint value, and introduces two endpoint
terms of opposite signs in the radial quadratic form; the exact
identity is given before Lemma~\ref{lem:radial}. Thus the
Dirichlet comparison with a flat cylinder needs a separate
Neumann replacement. Lemma~\ref{lem:radial} treats the positive
zero-order spectrum by differentiating the radial equation,
with its zero mode supplied separately.

The main quantitative difficulty is to keep a lower count above
the exact area term after summing all angular orders. A fixed
loss in each radial comparison accumulates to a loss proportional
to the frequency and can exhaust the boundary surplus. After accounting for angular multiplicities, the action integral
yields exactly $D\Lambda/4$. In the inner
angular range we retain the periodic fractional-part error and
control its sum by curvature and bounded variation. In the outer
range we retain the positive term in the convex floor-sum theorem
of \cite[Theorem~6.1]{FLPSballs}. Near a turning argument, the
derivative-phase error is singular; the clipped
phase envelope supplies a finite bound, while farther away we keep
the smaller varying error. The turning-point split and the
phase-enclosure method already occur in the FLPS analysis
\cite{FLPS,Phase}. The Neumann error balance, its explicit
variational overlap, and the continuous-parameter comparisons
below are the additional estimates established here.

Put $\rho=r/R$ and $X=R\sqrt\Lambda$, the radius ratio and the
outer frequency. Scaling the outer radius to one gives energy
$X^2$: the change of variables $x=Ry$ yields
\[
 N_{A_{r,R}}\left(\Lambda\right)=N_{A_{\rho,1}}\left(X^2\right),
 \,\,\, \frac{R^2-r^2}{4}\Lambda=\frac{1-\rho^2}{4}X^2.
\]
Write $Q=(1-\rho^2)X^2$. The analytic low-frequency proposition
covers every radius ratio at $Q\le100$. The small-hole argument covers
$\rho\le1/10$ at every frequency, and the thin-annulus theorem covers
$\rho\ge4/5$. For the intervening radii the strengthened phase estimate
applies at $X\ge40$; shorter thresholds apply on the larger radius
intervals. The angular-layer comparisons cover the remaining finite
range at fixed $Q$. Consecutive rows meet at their exact area endpoints,
and the final rows overlap the analytic high-frequency estimates.

Figure~\ref{fig:coverage} and Table~\ref{tab:coverage} summarize
these ranges. The boundary $Q=100$ is the curve
$X=10/\sqrt{1-\rho^2}$. Boundaries between these major parameter
ranges are included in both adjacent arguments.

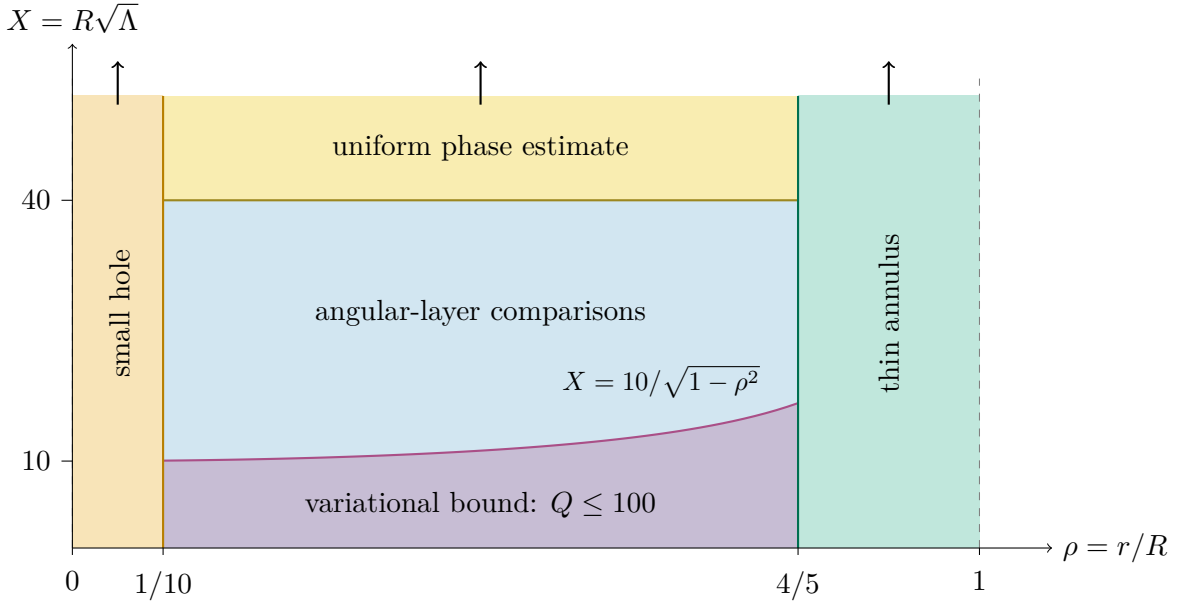
\begin{figure}[htbp]
\centering
\begin{tikzpicture}[x=12cm,y=.115cm,font=\small]
 \fill[annBlue,opacity=.18] (.1,0) rectangle (.8,40);
 \fill[annGold,opacity=.38] (.1,40) rectangle (.8,52);
 \fill[annPurple,opacity=.27]
   (.1,0) -- (.1,{10/sqrt(1-.1*.1)})
   -- plot[domain=.1:.8,samples=60] (\x,{10/sqrt(1-\x*\x)})
   -- (.8,0) -- cycle;
 \fill[annOrange,opacity=.28] (0,0) rectangle (.1,52);
 \fill[annGreen,opacity=.25] (.8,0) rectangle (1,52);
 \draw[thick,annPurple] plot[domain=.1:.8,samples=60]
       (\x,{10/sqrt(1-\x*\x)});
 \draw[thick,annGold!65!black] (.1,40)--(.8,40);
 \draw[thick,annOrange!80!black] (.1,0)--(.1,52);
 \draw[thick,annGreen!70!black] (.8,0)--(.8,52);
 \draw[dashed,gray] (0,0)--(0,54) (1,0)--(1,54);
 \draw[->] (0,0)--(1.08,0) node[right] {$\rho=r/R$};
 \draw[->] (0,0)--(0,58) node[above] {$X=R\sqrt\Lambda$};
 \foreach \v/\lab in {0/0,.1/{1/10},.8/{4/5},1/1}
   \draw (\v,0)--(\v,-.8) node[below=2pt] {$\lab$};
 \foreach \v in {10,40}
   \draw (0,\v)--(-.012,\v) node[left] {$\v$};
 \node[rotate=90] at (.052,27) {small hole};
 \node[rotate=90] at (.9,27) {thin annulus};
 \node at (.45,46) {uniform phase estimate};
 \node at (.45,27) {angular-layer comparisons};
 \node at (.45,5) {variational bound: $Q\le100$};
 \node[above left,font=\footnotesize] at (.77,16)
       {$X=10/\sqrt{1-\rho^2}$};
 \foreach \v in {.05,.45,.9}
   \draw[->,thick] (\v,51)--(\v,56);
\end{tikzpicture}
\caption{The parameter ranges in the analytic proof. The middle region
between $Q=100$ and $X=40$ is covered by the angular-layer comparisons in
Propositions~\ref{prop:analyticleftbridge},
\ref{prop:analyticmiddlebridge}, \ref{prop:analytichalfbridge}, and
\ref{prop:analyticwidebridge}. The arrows indicate that the end regions
and the high-frequency region continue to all larger frequencies.}
\label{fig:coverage}
\end{figure}

\begin{table}[!htbp]
\centering\small\setlength{\tabcolsep}{4pt}
\begin{tabular}{p{1.35in}p{1.25in}p{3.45in}}
Radius ratio & Spectral range & Counting estimate\\\hline
$0<\rho\le1/10$ & $X\ge0$ & Proposition~\ref{prop:analytic10}\\[2pt]
$4/5\le\rho<1$ & $X\ge0$ & Theorem~\ref{thm:thinall}\\[2pt]
$0<\rho<1$ & $Q\le100$ & Proposition~\ref{prop:analyticlowfrequency}\\[2pt]
$1/10\le\rho\le4/5$ & $X\ge40$ & Lemma~\ref{lem:analytic40}\\[2pt]
$1/10\le\rho\le4/5$ & $Q\ge100$, $X\le40$ &
Propositions~\ref{prop:analyticleftbridge},
\ref{prop:analyticmiddlebridge}, \ref{prop:analytichalfbridge}, and
\ref{prop:analyticwidebridge}
\end{tabular}
\caption{A cover of all radius ratios and frequencies, with
$Q=(1-\rho^2)X^2$. Shorter high-frequency thresholds within the
middle radius range are proved in the indicated bridge arguments.}
\label{tab:coverage}
\end{table}

Section~\ref{sec:radialtools} develops the radial, angular-sum and phase
tools. Section~\ref{sec:thinproof} proves the independent thin-annulus
result. Section~\ref{sec:uniformphase} establishes uniform phase estimates
and high-frequency bounds. Section~\ref{sec:allproof} assembles the
remaining parameter ranges and proves Theorem~\ref{thm:allannuli} and
its geometric corollaries. Section~\ref{sec:positivecorrections} gives
the quantitative corrections and discusses their sharpness.
The appendices contain the low-eigenvalue and disk estimates and the
complete angular-layer tables.

\section{Preliminaries}\label{sec:radialtools}

\indent We first compare separated radial eigenvalues with interval
eigenvalues. Two summation estimates then control low frequencies,
while the radial action and Bessel-phase enclosures give the
high-frequency comparison. All boundary counts retain their endpoint
values.

\subsection{Radial eigenvalues}\label{sec:radialeigenvalues}
In polar coordinates, the angular functions $1$, $\cos\left(m\theta\right)$ and
$\sin\left(m\theta\right)$, $m\ge1$, separate the Neumann Laplacian. Thus angular
order zero occurs once and every positive integer order occurs twice.
This multiplicity is represented by the multiset
\[
 \mathcal V=\left\{0,1,1,2,2,\ldots\right\}.
\]
If $n_m\left(\Lambda\right)$ is the radial counting function in angular
order $m$, including the zero mode for $m=0$, then
\[
 N_A\left(\Lambda\right)=n_0\left(\Lambda\right)+2\sum_{m\ge1}n_m\left(\Lambda\right)
             =\sum_{m\in\mathcal V}n_m\left(\Lambda\right).
\]
The sum is finite because the angular potential is at least $m^2/R^2$.
When the inner radius is normalized to one, $R=1+\delta$ and
$t=\sqrt\Lambda$ denote the outer radius and the frequency. The
normalization multiplies energies by $r^2$.

Write $\lambda_{m,n}$ for the $n$th radial Neumann eigenvalue in
angular order $m\ge0$, with $n=0,1,\ldots$. Its Rayleigh form in
$L^2\left(\left(r,R\right),s\dd s\right)$, whose squared norm is
$\int_r^R s\left|u\left(s\right)\right|^2\dd s$, is
\[
 \int_r^R\left(s\left|u'\right|^2+\frac{m^2}{s}\left|u\right|^2\right)\dd s.
\]
The form domain is $H^1\left(r,R\right)$, independent of $m$.
We use the min--max principle to compare Rayleigh quotients on
finite-dimensional trial spaces \cite{Davies}.

For comparison with the Dirichlet radial problem, put
$w\left(s\right)=\sqrt{s}\,u\left(s\right)$. Direct differentiation gives
\[
 -w''+\frac{m^2-1/4}{s^2}w=\lambda w,\,\,\,
 w'\left(r\right)=\frac{w\left(r\right)}{2r},\,\,\, w'\left(R\right)=\frac{w\left(R\right)}{2R}.
\]
Here $\lambda$ is a radial eigenvalue, measured as an energy.
These are Robin conditions. Expanding the square and integrating
the cross term yields the identity
\[
\begin{aligned}
 \int_r^R\left(s\left|u'\right|^2+\frac{m^2}{s}\left|u\right|^2\right)\dd s
 &=\int_r^R\left(\left|w'-w/\left(2s\right)\right|^2+\frac{m^2}{s^2}\left|w\right|^2\right)\dd s\\
 &=\int_r^R\left(\left|w'\right|^2+\frac{m^2-1/4}{s^2}\left|w\right|^2\right)\dd s
   +\frac{\left|w\left(r\right)\right|^2}{2r}-\frac{\left|w\left(R\right)\right|^2}{2R}.
\end{aligned}
\]
The endpoint terms vanish for Dirichlet trial functions, but
are present for Neumann trial functions. In order zero the
ground state is $u=1$, or $w=\sqrt{s}$, with energy zero.
The following comparison handles these boundary terms through
the equation for $u'$.

\begin{lemma}\label{lem:radial}
For every integer $m\ge0$ and integer $n\ge0$,
\begin{equation*}
 \lambda_{m,n}\le\frac{\pi^2n^2}{h^2}
             +\frac{m^2+3/4}{r^2}.
\end{equation*}
For $n=0$ the term $3/\left(4r^2\right)$ can be omitted.
The resulting comparison is strict whenever $(m,n)\ne(0,0)$;
the remaining value is $\lambda_{0,0}=0$.
\end{lemma}
\begin{proof}
For every nonzero trial function,
\[
 \frac{\int_r^R m^2s^{-1}\left|u\right|^2\dd s}
 {\int_r^R s\left|u\right|^2\dd s}\le\frac{m^2}{r^2}.
\]
Adding this bound on the same trial subspace and applying the
min--max principle gives
$\lambda_{m,n}\le\lambda_{0,n}+m^2/r^2$.
Here $\lambda_{0,0}=0$. If $\lambda>0$ and
$-u''-s^{-1}u'=\lambda u$ with $u'\left(r\right)=u'\left(R\right)=0$, then $w=u'$
satisfies
\[
 -w''-s^{-1}w'+s^{-2}w=\lambda w,
 \,\,\, w\left(r\right)=w\left(R\right)=0.
\]
Conversely, for a nonzero solution of this Dirichlet problem, set
$u=-\left(w'+s^{-1}w\right)/\lambda$. The differential equation gives $u'=w$,
so $u$ solves the original Neumann problem. The maps are inverse.
Thus, including multiplicities, the positive radial Neumann spectrum
in angular order zero coincides with the radial Dirichlet spectrum
in angular order one. The unitary
change $v=\sqrt{s}\,w$ turns the latter into
\[
 -v''+\frac{3}{4s^2}v=\lambda v,\,\,\, v\left(r\right)=v\left(R\right)=0.
\]
The Dirichlet Wirtinger inequality on this interval is
\begin{equation}\label{eq:dirichlet-wirtinger}
 \int_r^R|v'|^2\,\mathrm ds
 \ge\frac{\pi^2}{h^2}\int_r^R|v|^2\,\mathrm ds,
 \qquad v\in H_0^1(r,R),\quad h=R-r.
\end{equation}
Here $w=u'$ vanishes at both endpoints, so $v=\sqrt{s}\,u'$ also
belongs to $H_0^1(r,R)$. Since the transformed potential
$3/(4s^2)$ is positive, its first Dirichlet eigenvalue satisfies
\begin{equation}\label{eq:radial-first-lower}
 \lambda_{0,1}>\frac{\pi^2}{h^2}.
\end{equation}
Comparison of its potential with $3/\left(4r^2\right)$ yields
$\lambda_{0,n}\le\left(\pi n/h\right)^2+3/\left(4r^2\right)$ for $n\ge1$.
For $n\ge1$, the potential deficit
$3/(4r^2)-3/(4s^2)$ has strictly positive integral against every
nonzero element of the finite-dimensional sine trial space. Its
minimum on the unit sphere of that space is positive, making the
Dirichlet comparison strict. Adding the angular potential on a
finite eigenfunction trial space preserves strictness. For $n=0$
and $m>0$, the constant trial has quotient
$2m^2\log(R/r)/(R^2-r^2)<m^2/r^2$. The case $(m,n)=(0,0)$
is the constant eigenfunction and gives equality only at zero.
\end{proof}

\subsection{Two concave-sum bounds}

The radial comparison in Lemma~\ref{lem:radial} leads to sums of
the profile $\sqrt{1-x^2}$ over integer radial rows. Its integral
supplies the leading area term, but that term alone does not
compensate for the integer-rounding losses in the angular count.
We therefore need an explicit positive surplus over the integral.
The first estimate holds at every scale; the second gives a larger
surplus once at least two radial rows are available.

\begin{lemma}\label{lem:sum}
Let $\left\lfloor a\right\rfloor$ denote the greatest integer not exceeding $a$.
For every $a>0$,
\begin{equation}\label{eq:sum}
 S\left(a\right):=\sum_{n=0}^{\left\lfloor a\right\rfloor}
       \sqrt{1-\left(n/a\right)^2}
 \ge\frac\pi4 a+1-\frac\pi4.
\end{equation}
The constant $1-\pi/4$ is optimal in this inequality valid for all $a$.
\end{lemma}

\begin{proof}
Write $f\left(x\right)=\sqrt{1-x^2}$ on $\left[0,1\right]$ and
$I=\int_0^1f\left(x\right)\dd x=\pi/4$. For an integer $N\ge2$, let $U_1,\ldots,U_{N-1}$ be independent
uniform random variables on $\left[0,1\right]$, and let
$U_{\left(1\right)}\le\cdots\le U_{\left(N-1\right)}$ denote their
increasing rearrangement (the order statistics). For $1\le j\le N-1$, the density of
$U_{\left(j\right)}$ is
\[
 p_j\left(x\right)=
 \frac{\left(N-1\right)!}{\left(j-1\right)!\left(N-1-j\right)!}
 x^{j-1}\left(1-x\right)^{N-1-j},\,\,\, 0<x<1.
\]
Indeed, for small $\varepsilon>0$, one sample point falls in
$\left[x,x+\varepsilon\right]$, while exactly $j-1$ of the others
lie below $x$. Choosing these points gives the displayed coefficient
after dividing the probability by $\varepsilon$ and letting
$\varepsilon\to0+$. Events with two points in the short interval
have probability $O\left(\varepsilon^2\right)$ and do not affect the limit.
The integral identity
\[
 \int_0^1x^p\left(1-x\right)^q\,\dd x
 =\frac{p!q!}{\left(p+q+1\right)!},
 \,\,\, p,q\in\left\{0,1,\ldots\right\},
\]
obtained by repeated integration by parts, consequently gives
\[
 \mathbb E U_{\left(j\right)}
 =\int_0^1x p_j\left(x\right)\,\dd x=\frac jN.
\]
Here $\mathbb E$ denotes expectation.

For the concave function $f\left(x\right)=\sqrt{1-x^2}$, the usual
concavity inequality gives
\[
 \mathbb E f\left(U_{\left(j\right)}\right)\le f\left(\mathbb E U_{\left(j\right)}\right)=f\left(j/N\right).
\]
In this case the inequality follows directly from the supporting
tangent at $z=j/N$: for every $u\in\left[0,1\right]$,
\[
 f\left(u\right)\le f\left(z\right)+f'\left(z\right)\left(u-z\right).
\]
Taking expectations cancels the last term because
$\mathbb E U_{\left(j\right)}=z$. Sorting leaves the sum of the values
unchanged for every outcome, so
\[
 \begin{aligned}
 \sum_{j=1}^{N-1}f\left(j/N\right)
 &\ge\sum_{j=1}^{N-1}\mathbb E f\left(U_{\left(j\right)}\right)
   =\mathbb E\left(\sum_{i=1}^{N-1}f\left(U_i\right)\right)\\
 &=\left(N-1\right)\int_0^1f\left(x\right)\dd x=\left(N-1\right)I.
 \end{aligned}
\]
Adding $f\left(0\right)=1$ and $f\left(1\right)=0$ gives
$S\left(N\right)-NI\ge1-I$. The case $N=1$ is equality.

For $a\in\left[N,N+1\right]$, with $N\ge1$, the function $S\left(a\right)-Ia$ is
continuous and concave on that closed interval. Indeed, for
$j\ge1$ and $a>j$,
\[
 \frac{\dd^2}{\dd a^2}\sqrt{1-\frac{j^2}{a^2}}
 =-\frac{j^2\left(3a^2-2j^2\right)}{a^3\left(a^2-j^2\right)^{3/2}}<0,
\]
and the term appearing at the right endpoint has value zero there.
A concave function is no smaller than the minimum of its endpoint
values. The integer estimate therefore proves \eqref{eq:sum}.
If $0<a\le1$, then $S\left(a\right)=1$ and the result is immediate.
Equality at $a=1$ proves optimality of the constant.
\end{proof}

\begin{lemma}\label{lem:largesum}
For the sum in Lemma~\ref{lem:sum},
\[
 S\left(a\right)\ge\frac\pi4a+\frac{29}{100},\,\,\, a\ge2.
\]
\end{lemma}
\begin{proof}
For an integer $N\ge2$, put $f_N\left(x\right)=\sqrt{1-\left(x/N\right)^2}$.
The composite trapezoidal rule is the sum of the areas of the
endpoint trapezoids on the $N$ unit intervals:
\[
 T_N:=\sum_{j=0}^{N-1}\frac{f_N\left(j\right)+f_N\left(j+1\right)}2
     =\frac12f_N\left(0\right)+\sum_{j=1}^{N-1}f_N\left(j\right)+\frac12f_N\left(N\right)
     =S\left(N\right)-\frac12.
\]
Here $f_N\left(0\right)=1$ and $f_N\left(N\right)=0$. Since $f_N$ is concave, each
endpoint chord lies below its graph. Therefore
\[
 E_N:=\int_0^N f_N\left(x\right)\dd x-T_N\ge0,\,\,\,
 S\left(N\right)-\frac{\pi N}{4}=\frac12-E_N.
\]
To estimate the error on a unit interval, define
\[
 K_j\left(x\right)=\frac{\left(x-j\right)\left(j+1-x\right)}2
       =\frac18-\frac12\left(x-j-\frac12\right)^2,
       \,\,\, j\le x\le j+1.
\]
The completed-square expression shows directly that
$0\le K_j\left(x\right)\le1/8$, with its maximum at $x=j+1/2$.
Thus $K_j$ vanishes at the two endpoints,
$K_j'\left(j\right)=1/2$, $K_j'\left(j+1\right)=-1/2$, and $K_j''=-1$.
For a continuously differentiable function $F$ whose derivative
is absolutely continuous on $\left[j,j+1\right]$, two integrations
by parts give
\begin{equation}\label{eq:unittrapezoid}
 \begin{aligned}
 \int_j^{j+1}K_j\left(x\right)F''\left(x\right)\dd x
 &=\left[K_jF'\right]_j^{j+1}-\left[K_j'F\right]_j^{j+1}
       +\int_j^{j+1}K_j''\left(x\right)F\left(x\right)\dd x\\
 &=\frac{F\left(j\right)+F\left(j+1\right)}2
       -\int_j^{j+1}F\left(x\right)\dd x.
 \end{aligned}
\end{equation}
For $j=0,\ldots,N-2$, the function $f_N$ is smooth on this interval.
Consequently the concave trapezoidal error on this interval is
\[
 e_j:=\int_j^{j+1}f_N\left(x\right)\dd x
           -\frac{f_N\left(j\right)+f_N\left(j+1\right)}2
      =-\int_j^{j+1}K_j\left(x\right)f_N''\left(x\right)\dd x.
\]
Since $0\le K_j\le1/8$ and $f_N''\le0$, summing the first
$N-1$ intervals gives
\[
 0\le\sum_{j=0}^{N-2}e_j
 \le-\frac18\int_0^{N-1}f_N''\left(x\right)\dd x
 =-\frac{f_N'\left(N-1\right)}8
 =\frac{N-1}{8N\sqrt{2N-1}}.
\]
Indeed $f_N'\left(x\right)=-x/\left(N\sqrt{N^2-x^2}\right)$ and
$f_N''\left(x\right)=-N/\left(N^2-x^2\right)^{3/2}$ for $x<N$.
Since $f_N'\left(x\right)\to-\infty$ as $x\uparrow N$, applying
the derivative-increment bound through the last endpoint would
give no finite estimate. We estimate the last interval directly. With $y=N-x$,
\[
 f_N\left(N-y\right)=\sqrt{\frac{2y}{N}-\frac{y^2}{N^2}}
           \le\sqrt{\frac{2y}{N}},\,\,\, 0\le y\le1.
\]
As $f_N\left(N-1\right)=\sqrt{2N-1}/N$, this gives
\[
 e_{N-1}
 =\int_0^1f_N\left(N-y\right)\dd y-\frac{f_N\left(N-1\right)}2
 \le\frac{2\sqrt2}{3\sqrt N}-\frac{\sqrt{2N-1}}{2N}.
\]
Adding the two bounds proves
\[
 E_N\le\frac{N-1}{8N\sqrt{2N-1}}
          +\frac{2\sqrt2}{3\sqrt N}
          -\frac{\sqrt{2N-1}}{2N}.
\]
For $N\ge2$, use $\left(N-1\right)/N\le1$,
$\sqrt{2N-1}\ge\sqrt N$, and
$\sqrt{2N-1}/N\ge\sqrt{3/2}/\sqrt N$ to obtain
\[
 E_N\le\frac{1/8+2\sqrt2/3-\sqrt{3/2}/2}{\sqrt N}.
\]
The rational comparisons $\sqrt2<99/70$ and
$\sqrt{3/2}>61/50$ bound the numerator by
$641/1400<23/50$. For $N\ge5$, the bound
$\sqrt N\ge\sqrt5>11/5$ therefore gives
\[
 E_N<\frac{23}{110}<\frac{21}{100},\,\,\,
 S\left(N\right)-\frac{\pi N}{4}>\frac{16}{55}>\frac{29}{100}.
\]

It remains to check $N=2,3,4$. Their sums are
\[
 S\left(2\right)=1+\frac{\sqrt3}{2},\,\,\,
 S\left(3\right)=1+\frac{2\sqrt2+\sqrt5}{3},\,\,\,
 S\left(4\right)=1+\frac{\sqrt{15}+2\sqrt3+\sqrt7}{4}.
\]
Using $\pi<22/7$, $\sqrt3>173/100$, $\sqrt2>7/5$,
$\sqrt5>11/5$, $\sqrt{15}>19/5$ and $\sqrt7>13/5$ gives
\[
\begin{aligned}
 S\left(2\right)-\frac\pi2
 &>1+\frac{173}{200}-\frac{11}{7}=\frac{411}{1400},\\
 S\left(3\right)-\frac{3\pi}{4}
 &>1+\frac{2\left(7/5\right)+11/5}{3}-\frac{33}{14}=\frac{13}{42},\\
 S\left(4\right)-\pi
 &>1+\frac{19/5+2\left(173/100\right)+13/5}{4}-\frac{22}{7}
   =\frac{451}{1400}.
\end{aligned}
\]
Each rational lower bound exceeds $29/100$. Finally, the concavity
of $S\left(a\right)-\pi a/4$ on every $\left[N,N+1\right]$, established in
Lemma~\ref{lem:sum}, extends these integer bounds to every $a\ge2$.
\end{proof}

The number $29/100$ is a convenient rational surplus. At $N=2$ the surplus is
$1+\sqrt3/2-\pi/2\approx0.2958$; for $N\ge5$ the proof gives
a surplus greater than $1/2-23/110=16/55\approx0.2909$.
Thus $29/100$ leaves a positive margin in both estimates, while
the separately checked cases $N=3,4$ bridge them. The matching argument below uses this larger explicit surplus
in place of the universal constant $1-\pi/4$ from
Lemma~\ref{lem:sum}.

\subsection{Variational counting bounds}

We now turn the two summation bounds into a spectral count.
The radially constant trial functions give a sharper estimate
for the first radial row, while Lemma~\ref{lem:radial} controls
the higher rows. Proposition~\ref{prop:mixedlow} is valid at every
frequency in its stated range; its use at low frequencies in
Section~\ref{sec:thinproof} comes from the sign of the resulting
Weyl surplus. The negative quadratic term eventually outweighs
the positive linear term, so the proof then needs the phase estimate.
The improvement in Lemma~\ref{lem:largesum} extends the range over
which the variational surplus remains positive.

\begin{proposition}\label{prop:mixedlow}
Let $0<\delta\le1/4$, $R=1+\delta$, and
\[
 a_\delta=\sqrt{\frac{R^2-1}{2\log R}},\quad
 B_0=2a_\delta-\frac\pi2-\frac\delta\pi,\quad
 B_2=2a_\delta-\frac{71}{50}-\frac\delta\pi.
\]
For $t>\sqrt{3}/2$,
\begin{equation}\label{eq:mixedlow}
 N_A\left(t^2\right)-\frac{R^2-1}{4}t^2
 \ge-\frac{\delta^2}{4}t^2+B_0t-1-\frac{3\delta}{8}.
\end{equation}
If $\delta\sqrt{t^2-3/4}/\pi\ge2$, one may replace $B_0$ by $B_2$.
Moreover
\begin{equation}\label{eq:mixedB}
 B_0>\frac37+\frac35\delta,\,\,\,
 B_2>\frac{29}{50}+\frac{78}{125}\delta.
\end{equation}
\end{proposition}
\begin{proof}
A radially constant trial function in the integer angular order $m$
has Rayleigh quotient
\[
 m^2\frac{\int_1^R s^{-1}\dd s}{\int_1^R s\dd s}
 =\frac{2m^2\log R}{R^2-1}=\frac{m^2}{a_\delta^2},
\]
so $\lambda_{m,0}\le m^2/a_\delta^2$. For $m>0$ this inequality is
strict: equality in the ground-state variational principle would make
the constant trial function an eigenfunction, which is impossible
for the nonconstant potential $m^2/s^2$. When $m=0$, the constant
mode has eigenvalue zero and is strictly below every positive
energy. For higher radial rows keep Lemma~\ref{lem:radial}.
Writing $z=\sqrt{t^2-3/4}$ and $A=\delta z/\pi$, the order count
$\#\left\{m\in\mathbb Z:\left|m\right|\le b\right\}\ge2b-1$ gives directly
\[
 N_A\left(t^2\right)\ge2a_\delta t-1+
 \sum_{n=1}^{\left\lfloor A\right\rfloor}
       \left(2z\sqrt{1-\left(n/A\right)^2}-1\right).
\]
The sum over $n\ge1$ counts only higher radial rows; the zeroth
row is supplied by the separate constant trial function. In
particular, negative lower bounds $2b-1$ for individual short
rows are harmless, since their actual counts are nonnegative. Apply
Lemma~\ref{lem:sum} to obtain
\[
 N_A\left(t^2\right)\ge\frac\delta2\left(t^2-3/4\right)+2a_\delta t
                  -\left(\frac\pi2+\frac\delta\pi\right)z-1.
\]
Replacing $z$ by $t$ in the negative term proves
\eqref{eq:mixedlow}. If $A\ge2$, Lemma~\ref{lem:largesum} replaces
$\pi/2$ in that display by $71/50$, proving the stronger assertion.

Writing $L=\log R$, one has
$a_\delta^2=R\sinh\left(L\right)/L\ge R$. On $0\le\delta\le1/4$,
\[
 \sqrt{1+\delta}\ge1+\frac{33}{70}\delta
                         \ge1+\frac7{15}\delta,
\]
as follows by squaring; indeed $33/35+1089/19600<1$.
Together with $3<\pi<22/7$ this proves the first bound in
\eqref{eq:mixedB}. For the second use $\pi>157/50$ and
\[
 \frac{33}{35}-\frac{50}{157}=\frac{3431}{5495}>\frac{78}{125}.
\]
\end{proof}

\subsection{The action and its angular quadrature}
Normalize $r=1$ and let $R=1+\delta>1$. We use the same unit-disk
profile as \cite[Section~2 and (2.4)]{FLPSballs}, where it is
denoted by $h$, and the disk action of \cite[(5.1)]{FLPS}.
This established profile is an input to our Neumann estimates.
Define
\[
 g\left(x\right)=\frac{\sqrt{1-x^2}-x\arccos x}{\pi}\quad\left(0\le x\le1\right),
 \,\,\,
 \mathcal G_\delta\left(x\right)=
 \begin{cases}
 Rg\left(x/R\right)-g\left(x\right),&0\le x\le1,\\
 Rg\left(x/R\right),&1\le x\le R,\\
 0,&x\ge R,
 \end{cases}
\]
and set $f_t\left(v\right)=t\mathcal G_\delta\left(v/t\right)$ for $t>0$. The variable $v$ is
continuous in the action profile, even when the angular orders
sampled in the spectrum are integers. Equivalently,
\[
 f_t\left(v\right)=\frac1\pi\int_1^R
 \sqrt{\left(t^2-v^2/s^2\right)_+}\dd s,
\]
where $y_+=\max\left\{y,0\right\}$. This is the radial phase-space length
divided by $\pi$. Its change from a concave profile on $\left[0,t\right]$
to a convex profile on $\left[t,Rt\right]$ reflects the turning point
$s=v/t$, where the expression inside the square root vanishes.

For $X>0$, define the zero-extended disk action by
\begin{equation}\label{eq:diskaction}
 G_X\left(v\right)=
 \begin{cases}
 Xg\left(v/X\right),&0\le v\le X,\\
 0,&v\ge X.
 \end{cases}
\end{equation}
This is the disk action in \cite[(2.4)]{FLPSballs}.
The identity $f_t\left(v\right)=G_{Rt}\left(v\right)-G_t\left(v\right)$ relates our notation to
\cite[(5.1) and (5.4)]{FLPS}. The cancellation of the two disk
areas is the action identity used in Section~7.1 of that work.
Figure~\ref{fig:radialaction} shows the two curvature regions and
the area that produces the exact Weyl term.

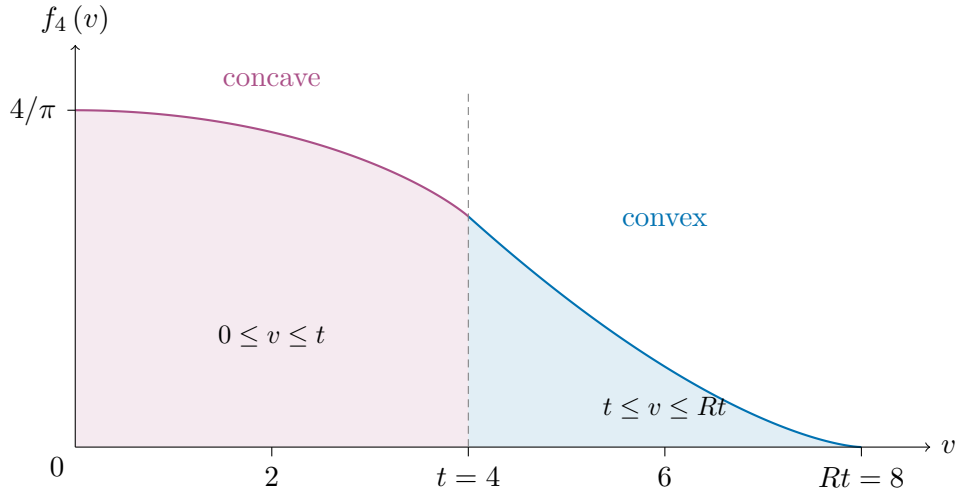
\begin{figure}[htbp]
\centering
\begin{tikzpicture}[x=1.3cm,y=3.5cm,font=\small]
  \fill[annPurple,opacity=.12] (0.0,0) -- (0.0000000,1.2732395) -- (0.0400000,1.2732077) -- (0.0800000,1.2731122) -- (0.1200000,1.2729530) -- (0.1600000,1.2727301) -- (0.2000000,1.2724435) -- (0.2400000,1.2720930) -- (0.2800000,1.2716787) -- (0.3200000,1.2712005) -- (0.3600000,1.2706582) -- (0.4000000,1.2700518) -- (0.4400000,1.2693812) -- (0.4800000,1.2686462) -- (0.5200000,1.2678468) -- (0.5600000,1.2669827) -- (0.6000000,1.2660539) -- (0.6400000,1.2650601) -- (0.6800000,1.2640012) -- (0.7200000,1.2628770) -- (0.7600000,1.2616873) -- (0.8000000,1.2604319) -- (0.8400000,1.2591105) -- (0.8800000,1.2577228) -- (0.9200000,1.2562687) -- (0.9600000,1.2547479) -- (1.0000000,1.2531600) -- (1.0400000,1.2515048) -- (1.0800000,1.2497819) -- (1.1200000,1.2479910) -- (1.1600000,1.2461318) -- (1.2000000,1.2442039) -- (1.2400000,1.2422069) -- (1.2800000,1.2401405) -- (1.3200000,1.2380041) -- (1.3600000,1.2357974) -- (1.4000000,1.2335199) -- (1.4400000,1.2311712) -- (1.4800000,1.2287507) -- (1.5200000,1.2262579) -- (1.5600000,1.2236923) -- (1.6000000,1.2210534) -- (1.6400000,1.2183405) -- (1.6800000,1.2155530) -- (1.7200000,1.2126904) -- (1.7600000,1.2097519) -- (1.8000000,1.2067369) -- (1.8400000,1.2036446) -- (1.8800000,1.2004744) -- (1.9200000,1.1972254) -- (1.9600000,1.1938969) -- (2.0000000,1.1904879) -- (2.0400000,1.1869977) -- (2.0800000,1.1834253) -- (2.1200000,1.1797697) -- (2.1600000,1.1760300) -- (2.2000000,1.1722052) -- (2.2400000,1.1682941) -- (2.2800000,1.1642957) -- (2.3200000,1.1602086) -- (2.3600000,1.1560319) -- (2.4000000,1.1517640) -- (2.4400000,1.1474037) -- (2.4800000,1.1429496) -- (2.5200000,1.1384001) -- (2.5600000,1.1337538) -- (2.6000000,1.1290090) -- (2.6400000,1.1241639) -- (2.6800000,1.1192167) -- (2.7200000,1.1141657) -- (2.7600000,1.1090086) -- (2.8000000,1.1037435) -- (2.8400000,1.0983680) -- (2.8800000,1.0928798) -- (2.9200000,1.0872763) -- (2.9600000,1.0815549) -- (3.0000000,1.0757127) -- (3.0400000,1.0697467) -- (3.0800000,1.0636536) -- (3.1200000,1.0574299) -- (3.1600000,1.0510718) -- (3.2000000,1.0445754) -- (3.2400000,1.0379361) -- (3.2800000,1.0311494) -- (3.3200000,1.0242099) -- (3.3600000,1.0171121) -- (3.4000000,1.0098497) -- (3.4400000,1.0024159) -- (3.4800000,0.9948030) -- (3.5200000,0.9870025) -- (3.5600000,0.9790047) -- (3.6000000,0.9707986) -- (3.6400000,0.9623715) -- (3.6800000,0.9537086) -- (3.7200000,0.9447922) -- (3.7600000,0.9356009) -- (3.8000000,0.9261079) -- (3.8400000,0.9162784) -- (3.8800000,0.9060650) -- (3.9200000,0.8953969) -- (3.9600000,0.8841513) -- (4.0000000,0.8719822) -- (4.0,0) -- cycle;
  \draw[annPurple,thick] (0.0000000,1.2732395) -- (0.0400000,1.2732077) -- (0.0800000,1.2731122) -- (0.1200000,1.2729530) -- (0.1600000,1.2727301) -- (0.2000000,1.2724435) -- (0.2400000,1.2720930) -- (0.2800000,1.2716787) -- (0.3200000,1.2712005) -- (0.3600000,1.2706582) -- (0.4000000,1.2700518) -- (0.4400000,1.2693812) -- (0.4800000,1.2686462) -- (0.5200000,1.2678468) -- (0.5600000,1.2669827) -- (0.6000000,1.2660539) -- (0.6400000,1.2650601) -- (0.6800000,1.2640012) -- (0.7200000,1.2628770) -- (0.7600000,1.2616873) -- (0.8000000,1.2604319) -- (0.8400000,1.2591105) -- (0.8800000,1.2577228) -- (0.9200000,1.2562687) -- (0.9600000,1.2547479) -- (1.0000000,1.2531600) -- (1.0400000,1.2515048) -- (1.0800000,1.2497819) -- (1.1200000,1.2479910) -- (1.1600000,1.2461318) -- (1.2000000,1.2442039) -- (1.2400000,1.2422069) -- (1.2800000,1.2401405) -- (1.3200000,1.2380041) -- (1.3600000,1.2357974) -- (1.4000000,1.2335199) -- (1.4400000,1.2311712) -- (1.4800000,1.2287507) -- (1.5200000,1.2262579) -- (1.5600000,1.2236923) -- (1.6000000,1.2210534) -- (1.6400000,1.2183405) -- (1.6800000,1.2155530) -- (1.7200000,1.2126904) -- (1.7600000,1.2097519) -- (1.8000000,1.2067369) -- (1.8400000,1.2036446) -- (1.8800000,1.2004744) -- (1.9200000,1.1972254) -- (1.9600000,1.1938969) -- (2.0000000,1.1904879) -- (2.0400000,1.1869977) -- (2.0800000,1.1834253) -- (2.1200000,1.1797697) -- (2.1600000,1.1760300) -- (2.2000000,1.1722052) -- (2.2400000,1.1682941) -- (2.2800000,1.1642957) -- (2.3200000,1.1602086) -- (2.3600000,1.1560319) -- (2.4000000,1.1517640) -- (2.4400000,1.1474037) -- (2.4800000,1.1429496) -- (2.5200000,1.1384001) -- (2.5600000,1.1337538) -- (2.6000000,1.1290090) -- (2.6400000,1.1241639) -- (2.6800000,1.1192167) -- (2.7200000,1.1141657) -- (2.7600000,1.1090086) -- (2.8000000,1.1037435) -- (2.8400000,1.0983680) -- (2.8800000,1.0928798) -- (2.9200000,1.0872763) -- (2.9600000,1.0815549) -- (3.0000000,1.0757127) -- (3.0400000,1.0697467) -- (3.0800000,1.0636536) -- (3.1200000,1.0574299) -- (3.1600000,1.0510718) -- (3.2000000,1.0445754) -- (3.2400000,1.0379361) -- (3.2800000,1.0311494) -- (3.3200000,1.0242099) -- (3.3600000,1.0171121) -- (3.4000000,1.0098497) -- (3.4400000,1.0024159) -- (3.4800000,0.9948030) -- (3.5200000,0.9870025) -- (3.5600000,0.9790047) -- (3.6000000,0.9707986) -- (3.6400000,0.9623715) -- (3.6800000,0.9537086) -- (3.7200000,0.9447922) -- (3.7600000,0.9356009) -- (3.8000000,0.9261079) -- (3.8400000,0.9162784) -- (3.8800000,0.9060650) -- (3.9200000,0.8953969) -- (3.9600000,0.8841513) -- (4.0000000,0.8719822);
  \fill[annBlue,opacity=.12] (4.0,0) -- (4.0000000,0.8719822) -- (4.0400000,0.8586857) -- (4.0800000,0.8454629) -- (4.1200000,0.8323142) -- (4.1600000,0.8192397) -- (4.2000000,0.8062397) -- (4.2400000,0.7933145) -- (4.2800000,0.7804644) -- (4.3200000,0.7676897) -- (4.3600000,0.7549906) -- (4.4000000,0.7423674) -- (4.4400000,0.7298205) -- (4.4800000,0.7173501) -- (4.5200000,0.7049565) -- (4.5600000,0.6926401) -- (4.6000000,0.6804011) -- (4.6400000,0.6682400) -- (4.6800000,0.6561571) -- (4.7200000,0.6441526) -- (4.7600000,0.6322270) -- (4.8000000,0.6203805) -- (4.8400000,0.6086137) -- (4.8800000,0.5969268) -- (4.9200000,0.5853203) -- (4.9600000,0.5737945) -- (5.0000000,0.5623498) -- (5.0400000,0.5509867) -- (5.0800000,0.5397056) -- (5.1200000,0.5285068) -- (5.1600000,0.5173910) -- (5.2000000,0.5063584) -- (5.2400000,0.4954096) -- (5.2800000,0.4845451) -- (5.3200000,0.4737653) -- (5.3600000,0.4630707) -- (5.4000000,0.4524619) -- (5.4400000,0.4419394) -- (5.4800000,0.4315037) -- (5.5200000,0.4211554) -- (5.5600000,0.4108951) -- (5.6000000,0.4007233) -- (5.6400000,0.3906406) -- (5.6800000,0.3806477) -- (5.7200000,0.3707452) -- (5.7600000,0.3609338) -- (5.8000000,0.3512141) -- (5.8400000,0.3415869) -- (5.8800000,0.3320528) -- (5.9200000,0.3226126) -- (5.9600000,0.3132670) -- (6.0000000,0.3040169) -- (6.0400000,0.2948630) -- (6.0800000,0.2858062) -- (6.1200000,0.2768474) -- (6.1600000,0.2679874) -- (6.2000000,0.2592272) -- (6.2400000,0.2505677) -- (6.2800000,0.2420100) -- (6.3200000,0.2335550) -- (6.3600000,0.2252039) -- (6.4000000,0.2169578) -- (6.4400000,0.2088177) -- (6.4800000,0.2007849) -- (6.5200000,0.1928608) -- (6.5600000,0.1850464) -- (6.6000000,0.1773434) -- (6.6400000,0.1697529) -- (6.6800000,0.1622766) -- (6.7200000,0.1549161) -- (6.7600000,0.1476728) -- (6.8000000,0.1405486) -- (6.8400000,0.1335453) -- (6.8800000,0.1266647) -- (6.9200000,0.1199089) -- (6.9600000,0.1132800) -- (7.0000000,0.1067802) -- (7.0400000,0.1004119) -- (7.0800000,0.0941776) -- (7.1200000,0.0880802) -- (7.1600000,0.0821223) -- (7.2000000,0.0763072) -- (7.2400000,0.0706382) -- (7.2800000,0.0651188) -- (7.3200000,0.0597530) -- (7.3600000,0.0545450) -- (7.4000000,0.0494995) -- (7.4400000,0.0446216) -- (7.4800000,0.0399170) -- (7.5200000,0.0353919) -- (7.5600000,0.0310535) -- (7.6000000,0.0269099) -- (7.6400000,0.0229703) -- (7.6800000,0.0192454) -- (7.7200000,0.0157482) -- (7.7600000,0.0124940) -- (7.8000000,0.0095021) -- (7.8400000,0.0067974) -- (7.8800000,0.0044139) -- (7.9200000,0.0024020) -- (7.9600000,0.0008490) -- (8.0000000,0.0000000) -- (8.0,0) -- cycle;
  \draw[annBlue,thick] (4.0000000,0.8719822) -- (4.0400000,0.8586857) -- (4.0800000,0.8454629) -- (4.1200000,0.8323142) -- (4.1600000,0.8192397) -- (4.2000000,0.8062397) -- (4.2400000,0.7933145) -- (4.2800000,0.7804644) -- (4.3200000,0.7676897) -- (4.3600000,0.7549906) -- (4.4000000,0.7423674) -- (4.4400000,0.7298205) -- (4.4800000,0.7173501) -- (4.5200000,0.7049565) -- (4.5600000,0.6926401) -- (4.6000000,0.6804011) -- (4.6400000,0.6682400) -- (4.6800000,0.6561571) -- (4.7200000,0.6441526) -- (4.7600000,0.6322270) -- (4.8000000,0.6203805) -- (4.8400000,0.6086137) -- (4.8800000,0.5969268) -- (4.9200000,0.5853203) -- (4.9600000,0.5737945) -- (5.0000000,0.5623498) -- (5.0400000,0.5509867) -- (5.0800000,0.5397056) -- (5.1200000,0.5285068) -- (5.1600000,0.5173910) -- (5.2000000,0.5063584) -- (5.2400000,0.4954096) -- (5.2800000,0.4845451) -- (5.3200000,0.4737653) -- (5.3600000,0.4630707) -- (5.4000000,0.4524619) -- (5.4400000,0.4419394) -- (5.4800000,0.4315037) -- (5.5200000,0.4211554) -- (5.5600000,0.4108951) -- (5.6000000,0.4007233) -- (5.6400000,0.3906406) -- (5.6800000,0.3806477) -- (5.7200000,0.3707452) -- (5.7600000,0.3609338) -- (5.8000000,0.3512141) -- (5.8400000,0.3415869) -- (5.8800000,0.3320528) -- (5.9200000,0.3226126) -- (5.9600000,0.3132670) -- (6.0000000,0.3040169) -- (6.0400000,0.2948630) -- (6.0800000,0.2858062) -- (6.1200000,0.2768474) -- (6.1600000,0.2679874) -- (6.2000000,0.2592272) -- (6.2400000,0.2505677) -- (6.2800000,0.2420100) -- (6.3200000,0.2335550) -- (6.3600000,0.2252039) -- (6.4000000,0.2169578) -- (6.4400000,0.2088177) -- (6.4800000,0.2007849) -- (6.5200000,0.1928608) -- (6.5600000,0.1850464) -- (6.6000000,0.1773434) -- (6.6400000,0.1697529) -- (6.6800000,0.1622766) -- (6.7200000,0.1549161) -- (6.7600000,0.1476728) -- (6.8000000,0.1405486) -- (6.8400000,0.1335453) -- (6.8800000,0.1266647) -- (6.9200000,0.1199089) -- (6.9600000,0.1132800) -- (7.0000000,0.1067802) -- (7.0400000,0.1004119) -- (7.0800000,0.0941776) -- (7.1200000,0.0880802) -- (7.1600000,0.0821223) -- (7.2000000,0.0763072) -- (7.2400000,0.0706382) -- (7.2800000,0.0651188) -- (7.3200000,0.0597530) -- (7.3600000,0.0545450) -- (7.4000000,0.0494995) -- (7.4400000,0.0446216) -- (7.4800000,0.0399170) -- (7.5200000,0.0353919) -- (7.5600000,0.0310535) -- (7.6000000,0.0269099) -- (7.6400000,0.0229703) -- (7.6800000,0.0192454) -- (7.7200000,0.0157482) -- (7.7600000,0.0124940) -- (7.8000000,0.0095021) -- (7.8400000,0.0067974) -- (7.8800000,0.0044139) -- (7.9200000,0.0024020) -- (7.9600000,0.0008490) -- (8.0000000,0.0000000);
  \draw[->] (0,0) -- (8.7,0) node[right] {$v$};
  \draw[->] (0,0) -- (0,1.52) node[above] {$f_4\left(v\right)$};
  \node[below left] at (0,0) {$0$};
  \draw (0,{4/pi}) -- (-.08,{4/pi}) node[left] {$4/\pi$};
  \foreach \xx/\lab in {2/2,4/{t=4},6/6,8/{Rt=8}}{
    \draw (\xx,0) -- (\xx,-.035) node[below] {$\lab$};
  }
  \draw[gray,densely dashed] (4,0) -- (4,1.35);
  \node[annPurple] at (2,1.39) {concave};
  \node[annBlue] at (6,.86) {convex};
  \node[font=\footnotesize] at (2,.42) {$0\le v\le t$};
  \node[font=\footnotesize] at (6,.15) {$t\le v\le Rt$};
\end{tikzpicture}
\caption{
Radial action for $r=1$, $R=2$, $t=4$: $f_4\left(v\right)=G_8\left(v\right)-G_4\left(v\right)$. It is concave on $\left[0,4\right]$ and convex on $\left[4,8\right]$, with common slope $f_4'\left(4\right)=-1/3$. The shaded area is $\int_0^8f_4\left(v\right)\dd v=6$; twice this area is the Weyl term $12$ at energy $16$. Disk-profile differences follow \cite[(5.1), (5.4), and Section~7.1]{FLPS}; the figure is drawn from these explicit formulas.}
\label{fig:radialaction}
\end{figure}

For a function $w$, its total variation $\operatorname{Var}_I\left(w\right)$
on an interval $I$ is the supremum of the sums of absolute increments
over finite ordered partitions of $I$. It is finite precisely when
$w$ has bounded variation.
The angular count samples the action at integer orders. The next lemma gives the curvature and quadrature bounds that allow this sampled action to be compared with the exact Weyl term.
\begin{lemma}\label{lem:actionquadrature}
The action has $f_t'\left(0\right)=0$ and
\[
 2\int_0^{Rt}f_t\left(v\right)\dd v=\frac{R^2-1}{4}t^2,
 \,\,\,
 -f_t''\left(v\right)\ge\frac{\delta}{\pi Rt}\quad\left(0\le v<t\right).
\]
For the integer angular multiset $\mathcal V$,
\[
 \left|\sum_{m\in\mathbb Z}f_t\left(\left|m\right|\right)
          -\frac{R^2-1}{4}t^2\right|
 \le\tau_\delta:=\frac{\arccos\left(1/R\right)}{2\pi}<\frac14.
\]
\end{lemma}
\begin{proof}
Integration by parts gives
\[
 \int_0^1\sqrt{1-x^2}\dd x=\frac\pi4,\,\,\,
 \int_0^1x\arccos x\dd x=\frac\pi8,
\]
and hence $\int_0^1g\left(x\right)\dd x=1/8$. Scaling the two terms in
$\mathcal G_\delta$ now gives the displayed integral identity.
On the core interval,
\[
 -f_t''\left(v\right)=\frac1\pi\left(
 \frac1{\sqrt{t^2-v^2}}-\frac1{\sqrt{R^2t^2-v^2}}\right).
\]
This is increasing in $v$ and takes the value $\delta/\left(\pi Rt\right)$
at zero. The even extension $F_t\left(x\right)=f_t\left(\left|x\right|\right)$, extended by zero
outside $\left[-Rt,Rt\right]$, is $C^1$ and has an integrable second derivative.
The derivative of $f_t$ decreases from zero to
$-\arccos\left(1/R\right)/\pi$ on $\left[0,t\right]$ and increases to zero on $\left[t,Rt\right]$.
Consequently
\[
 \operatorname{Var}_{\mathbb R}\left(F_t'\right)=
 \frac{4\arccos\left(1/R\right)}\pi.
\]
On each unit interval of the integer grid, the trapezoidal
remainder is bounded by $1/8$ of the integral of $\left|F_t''\right|$.
Summing these bounds gives the stated error, since the full
trapezoidal sum equals $\sum_mF_t\left(m\right)$.
\end{proof}

\subsection{Derivative phases and radial counting}
$J_m$ and $Y_m$ denote the Bessel functions of the first and second
kind, that is, the standard independent solutions of
$x^2w''+xw'+\left(x^2-m^2\right)w=0$ \cite[Chapter~10]{NIST}.
We use the continuous derivative phase $\varphi_m$ defined by
\[
 J_m'\left(x\right)+iY_m'\left(x\right)=\mathcal M_m\left(x\right)e^{i\varphi_m\left(x\right)},\,\,\,
 \mathcal M_m\left(x\right)>0,\quad \varphi_m\left(0+\right)=\pi/2.
\]
For $m>0$ it decreases on $\left(0,m\right)$, increases on $\left(m,+\infty\right)$,
and $0<\varphi_m\left(m\right)<\pi/2$. For $m=0$ it is increasing.
Here $\mathcal M_m\left(x\right)=\sqrt{J_m'\left(x\right)^2+Y_m'\left(x\right)^2}$
is the modulus of the complex derivative vector. The word continuous specifies a real-valued lift of the argument:
we follow the angle without resetting it after a full turn. The
principal argument, chosen in $\left(-\pi,\pi\right]$, alone would lose multiples of $2\pi$ and cannot
be used in the floor count below. The derivative vector is never
zero: if $J_m'\left(x\right)$ and $Y_m'\left(x\right)$ both vanished, their Wronskian
$J_mY_m'-J_m'Y_m=2/\left(\pi x\right)$ would vanish as well. Thus the
normalization at $0+$ selects a unique continuous phase.

For completeness, differentiating an argument gives
\[
 \varphi_m'\left(x\right)
 =\frac{J_m'\left(x\right)Y_m''\left(x\right)-Y_m'\left(x\right)J_m''\left(x\right)}
        {J_m'\left(x\right)^2+Y_m'\left(x\right)^2}.
\]
Inserting the Bessel equations into the numerator cancels the
terms containing $J_m'Y_m'/x$ and leaves
\[
 \left(1-\frac{m^2}{x^2}\right)
 \left(J_m\left(x\right)Y_m'\left(x\right)-J_m'\left(x\right)Y_m\left(x\right)\right)
 =\frac{2\left(x^2-m^2\right)}{\pi x^3}.
\]
Consequently
\begin{equation}\label{eq:phasevelocity}
 \varphi_m'\left(x\right)=
 \frac{2\left(x^2-m^2\right)}
 {\pi x^3\left(J_m'\left(x\right)^2+Y_m'\left(x\right)^2\right)}.
\end{equation}
The denominator is positive. This proves the stated decrease for
$x<m$, increase for $x>m$, and monotonicity for $m=0$.
Figure~\ref{fig:derivativephase} illustrates both the continuous
branch and the jumps of the principal argument.

\begin{figure}[htbp]
\centering
\begin{tikzpicture}[x=1.3cm,y=1.3cm,font=\small]
  \draw[->] (0,-1.12) -- (0,3.23) node[above] {$\varphi/\pi$};
  \draw[->] (0,0) -- (8.42,0) node[right] {$x$};
  \foreach \yy/\lab in {-1/{-1},.5/{1/2},1/1,2/2,3/3}{
    \draw (0,\yy) -- (-.06,\yy) node[left] {$\lab$};
  }
  \foreach \xx in {2,4,6,8}{\draw (\xx,0) -- (\xx,-.06) node[below] {$\xx$};}
  \draw[gray,dotted] (0,.5) -- (8,.5);
  \draw[annBlue,dotted] (2,-.92) -- (2,1.10);
  \draw[annGreen,thick] (0.02500,0.50016) -- (0.05837,0.50085) -- (0.09174,0.50208) -- (0.12510,0.50382) -- (0.15847,0.50607) -- (0.19184,0.50879) -- (0.22521,0.51196) -- (0.25858,0.51555) -- (0.29195,0.51954) -- (0.32531,0.52390) -- (0.35868,0.52861) -- (0.39205,0.53365) -- (0.42542,0.53899) -- (0.45879,0.54462) -- (0.49215,0.55052) -- (0.52552,0.55666) -- (0.55889,0.56304) -- (0.59226,0.56964) -- (0.62563,0.57643) -- (0.65900,0.58342) -- (0.69236,0.59059) -- (0.72573,0.59792) -- (0.75910,0.60541) -- (0.79247,0.61304) -- (0.82584,0.62081) -- (0.85921,0.62871) -- (0.89257,0.63673) -- (0.92594,0.64486) -- (0.95931,0.65310) -- (0.99268,0.66145) -- (1.02605,0.66988) -- (1.05941,0.67841) -- (1.09278,0.68702) -- (1.12615,0.69571) -- (1.15952,0.70448) -- (1.19289,0.71331) -- (1.22626,0.72222) -- (1.25962,0.73119) -- (1.29299,0.74022) -- (1.32636,0.74931) -- (1.35973,0.75845) -- (1.39310,0.76764) -- (1.42646,0.77689) -- (1.45983,0.78618) -- (1.49320,0.79551) -- (1.52657,0.80489) -- (1.55994,0.81431) -- (1.59331,0.82377) -- (1.62667,0.83327) -- (1.66004,0.84280) -- (1.69341,0.85236) -- (1.72678,0.86196) -- (1.76015,0.87159) -- (1.79351,0.88124) -- (1.82688,0.89093) -- (1.86025,0.90064) -- (1.89362,0.91038) -- (1.92699,0.92014) -- (1.96036,0.92993) -- (1.99372,0.93974) -- (2.02709,0.94958) -- (2.06046,0.95943) -- (2.09383,0.96930) -- (2.12720,0.97920) -- (2.16056,0.98911) -- (2.19393,0.99904) -- (2.22730,1.00899) -- (2.26067,1.01896) -- (2.29404,1.02894) -- (2.32741,1.03894) -- (2.36077,1.04895) -- (2.39414,1.05898) -- (2.42751,1.06902) -- (2.46088,1.07908) -- (2.49425,1.08915) -- (2.52762,1.09923) -- (2.56098,1.10932) -- (2.59435,1.11943) -- (2.62772,1.12955) -- (2.66109,1.13967) -- (2.69446,1.14981) -- (2.72782,1.15996) -- (2.76119,1.17012) -- (2.79456,1.18029) -- (2.82793,1.19047) -- (2.86130,1.20066) -- (2.89467,1.21086) -- (2.92803,1.22106) -- (2.96140,1.23128) -- (2.99477,1.24150) -- (3.02814,1.25173) -- (3.06151,1.26197) -- (3.09487,1.27221) -- (3.12824,1.28247) -- (3.16161,1.29273) -- (3.19498,1.30299) -- (3.22835,1.31327) -- (3.26172,1.32355) -- (3.29508,1.33383) -- (3.32845,1.34413) -- (3.36182,1.35442) -- (3.39519,1.36473) -- (3.42856,1.37504) -- (3.46192,1.38535) -- (3.49529,1.39567) -- (3.52866,1.40600) -- (3.56203,1.41633) -- (3.59540,1.42667) -- (3.62877,1.43701) -- (3.66213,1.44736) -- (3.69550,1.45771) -- (3.72887,1.46806) -- (3.76224,1.47842) -- (3.79561,1.48878) -- (3.82897,1.49915) -- (3.86234,1.50952) -- (3.89571,1.51990) -- (3.92908,1.53028) -- (3.96245,1.54066) -- (3.99582,1.55105) -- (4.02918,1.56144) -- (4.06255,1.57184) -- (4.09592,1.58223) -- (4.12929,1.59263) -- (4.16266,1.60304) -- (4.19603,1.61345) -- (4.22939,1.62386) -- (4.26276,1.63427) -- (4.29613,1.64469) -- (4.32950,1.65511) -- (4.36287,1.66553) -- (4.39623,1.67596) -- (4.42960,1.68639) -- (4.46297,1.69682) -- (4.49634,1.70725) -- (4.52971,1.71769) -- (4.56308,1.72813) -- (4.59644,1.73857) -- (4.62981,1.74902) -- (4.66318,1.75946) -- (4.69655,1.76991) -- (4.72992,1.78036) -- (4.76328,1.79082) -- (4.79665,1.80127) -- (4.83002,1.81173) -- (4.86339,1.82219) -- (4.89676,1.83265) -- (4.93013,1.84312) -- (4.96349,1.85358) -- (4.99686,1.86405) -- (5.03023,1.87452) -- (5.06360,1.88499) -- (5.09697,1.89547) -- (5.13033,1.90594) -- (5.16370,1.91642) -- (5.19707,1.92690) -- (5.23044,1.93738) -- (5.26381,1.94787) -- (5.29718,1.95835) -- (5.33054,1.96884) -- (5.36391,1.97932) -- (5.39728,1.98981) -- (5.43065,2.00030) -- (5.46402,2.01080) -- (5.49738,2.02129) -- (5.53075,2.03179) -- (5.56412,2.04228) -- (5.59749,2.05278) -- (5.63086,2.06328) -- (5.66423,2.07378) -- (5.69759,2.08429) -- (5.73096,2.09479) -- (5.76433,2.10529) -- (5.79770,2.11580) -- (5.83107,2.12631) -- (5.86444,2.13682) -- (5.89780,2.14733) -- (5.93117,2.15784) -- (5.96454,2.16835) -- (5.99791,2.17887) -- (6.03128,2.18938) -- (6.06464,2.19990) -- (6.09801,2.21041) -- (6.13138,2.22093) -- (6.16475,2.23145) -- (6.19812,2.24197) -- (6.23149,2.25249) -- (6.26485,2.26302) -- (6.29822,2.27354) -- (6.33159,2.28406) -- (6.36496,2.29459) -- (6.39833,2.30512) -- (6.43169,2.31564) -- (6.46506,2.32617) -- (6.49843,2.33670) -- (6.53180,2.34723) -- (6.56517,2.35776) -- (6.59854,2.36830) -- (6.63190,2.37883) -- (6.66527,2.38936) -- (6.69864,2.39990) -- (6.73201,2.41043) -- (6.76538,2.42097) -- (6.79874,2.43151) -- (6.83211,2.44204) -- (6.86548,2.45258) -- (6.89885,2.46312) -- (6.93222,2.47366) -- (6.96559,2.48420) -- (6.99895,2.49475) -- (7.03232,2.50529) -- (7.06569,2.51583) -- (7.09906,2.52637) -- (7.13243,2.53692) -- (7.16579,2.54746) -- (7.19916,2.55801) -- (7.23253,2.56856) -- (7.26590,2.57911) -- (7.29927,2.58965) -- (7.33264,2.60020) -- (7.36600,2.61075) -- (7.39937,2.62130) -- (7.43274,2.63185) -- (7.46611,2.64240) -- (7.49948,2.65295) -- (7.53285,2.66351) -- (7.56621,2.67406) -- (7.59958,2.68461) -- (7.63295,2.69517) -- (7.66632,2.70572) -- (7.69969,2.71628) -- (7.73305,2.72683) -- (7.76642,2.73739) -- (7.79979,2.74795) -- (7.83316,2.75851) -- (7.86653,2.76906) -- (7.89990,2.77962) -- (7.93326,2.79018) -- (7.96663,2.80074) -- (8.00000,2.81130);
  \draw[annBlue,thick] (0.02500,0.50000) -- (0.05837,0.50000) -- (0.09174,0.50000) -- (0.12510,0.49999) -- (0.15847,0.49998) -- (0.19184,0.49996) -- (0.22521,0.49992) -- (0.25858,0.49986) -- (0.29195,0.49978) -- (0.32531,0.49966) -- (0.35868,0.49949) -- (0.39205,0.49928) -- (0.42542,0.49900) -- (0.45879,0.49866) -- (0.49215,0.49823) -- (0.52552,0.49771) -- (0.55889,0.49709) -- (0.59226,0.49635) -- (0.62563,0.49549) -- (0.65900,0.49448) -- (0.69236,0.49332) -- (0.72573,0.49200) -- (0.75910,0.49050) -- (0.79247,0.48881) -- (0.82584,0.48692) -- (0.85921,0.48482) -- (0.89257,0.48251) -- (0.92594,0.47997) -- (0.95931,0.47720) -- (0.99268,0.47420) -- (1.02605,0.47096) -- (1.05941,0.46749) -- (1.09278,0.46379) -- (1.12615,0.45988) -- (1.15952,0.45576) -- (1.19289,0.45144) -- (1.22626,0.44696) -- (1.25962,0.44233) -- (1.29299,0.43758) -- (1.32636,0.43274) -- (1.35973,0.42784) -- (1.39310,0.42292) -- (1.42646,0.41802) -- (1.45983,0.41318) -- (1.49320,0.40843) -- (1.52657,0.40381) -- (1.55994,0.39935) -- (1.59331,0.39511) -- (1.62667,0.39110) -- (1.66004,0.38735) -- (1.69341,0.38390) -- (1.72678,0.38076) -- (1.76015,0.37795) -- (1.79351,0.37549) -- (1.82688,0.37338) -- (1.86025,0.37165) -- (1.89362,0.37028) -- (1.92699,0.36929) -- (1.96036,0.36867) -- (1.99372,0.36842) -- (2.02709,0.36853) -- (2.06046,0.36900) -- (2.09383,0.36982) -- (2.12720,0.37098) -- (2.16056,0.37247) -- (2.19393,0.37429) -- (2.22730,0.37641) -- (2.26067,0.37884) -- (2.29404,0.38155) -- (2.32741,0.38454) -- (2.36077,0.38780) -- (2.39414,0.39131) -- (2.42751,0.39507) -- (2.46088,0.39906) -- (2.49425,0.40328) -- (2.52762,0.40772) -- (2.56098,0.41236) -- (2.59435,0.41719) -- (2.62772,0.42222) -- (2.66109,0.42743) -- (2.69446,0.43281) -- (2.72782,0.43836) -- (2.76119,0.44406) -- (2.79456,0.44992) -- (2.82793,0.45592) -- (2.86130,0.46206) -- (2.89467,0.46833) -- (2.92803,0.47473) -- (2.96140,0.48125) -- (2.99477,0.48789) -- (3.02814,0.49464) -- (3.06151,0.50150) -- (3.09487,0.50847) -- (3.12824,0.51553) -- (3.16161,0.52268) -- (3.19498,0.52993) -- (3.22835,0.53727) -- (3.26172,0.54469) -- (3.29508,0.55219) -- (3.32845,0.55977) -- (3.36182,0.56743) -- (3.39519,0.57516) -- (3.42856,0.58296) -- (3.46192,0.59083) -- (3.49529,0.59876) -- (3.52866,0.60675) -- (3.56203,0.61481) -- (3.59540,0.62292) -- (3.62877,0.63109) -- (3.66213,0.63931) -- (3.69550,0.64759) -- (3.72887,0.65592) -- (3.76224,0.66430) -- (3.79561,0.67272) -- (3.82897,0.68119) -- (3.86234,0.68971) -- (3.89571,0.69827) -- (3.92908,0.70687) -- (3.96245,0.71551) -- (3.99582,0.72419) -- (4.02918,0.73291) -- (4.06255,0.74167) -- (4.09592,0.75046) -- (4.12929,0.75929) -- (4.16266,0.76815) -- (4.19603,0.77704) -- (4.22939,0.78597) -- (4.26276,0.79493) -- (4.29613,0.80391) -- (4.32950,0.81293) -- (4.36287,0.82197) -- (4.39623,0.83105) -- (4.42960,0.84015) -- (4.46297,0.84927) -- (4.49634,0.85843) -- (4.52971,0.86760) -- (4.56308,0.87680) -- (4.59644,0.88603) -- (4.62981,0.89528) -- (4.66318,0.90455) -- (4.69655,0.91384) -- (4.72992,0.92315) -- (4.76328,0.93249) -- (4.79665,0.94184) -- (4.83002,0.95122) -- (4.86339,0.96061) -- (4.89676,0.97003) -- (4.93013,0.97946) -- (4.96349,0.98891) -- (4.99686,0.99838) -- (5.03023,1.00786) -- (5.06360,1.01736) -- (5.09697,1.02688) -- (5.13033,1.03641) -- (5.16370,1.04596) -- (5.19707,1.05553) -- (5.23044,1.06511) -- (5.26381,1.07470) -- (5.29718,1.08431) -- (5.33054,1.09394) -- (5.36391,1.10357) -- (5.39728,1.11322) -- (5.43065,1.12289) -- (5.46402,1.13256) -- (5.49738,1.14225) -- (5.53075,1.15195) -- (5.56412,1.16167) -- (5.59749,1.17139) -- (5.63086,1.18113) -- (5.66423,1.19088) -- (5.69759,1.20064) -- (5.73096,1.21041) -- (5.76433,1.22019) -- (5.79770,1.22999) -- (5.83107,1.23979) -- (5.86444,1.24960) -- (5.89780,1.25943) -- (5.93117,1.26926) -- (5.96454,1.27910) -- (5.99791,1.28895) -- (6.03128,1.29882) -- (6.06464,1.30869) -- (6.09801,1.31857) -- (6.13138,1.32845) -- (6.16475,1.33835) -- (6.19812,1.34825) -- (6.23149,1.35817) -- (6.26485,1.36809) -- (6.29822,1.37802) -- (6.33159,1.38796) -- (6.36496,1.39790) -- (6.39833,1.40785) -- (6.43169,1.41781) -- (6.46506,1.42778) -- (6.49843,1.43776) -- (6.53180,1.44774) -- (6.56517,1.45773) -- (6.59854,1.46772) -- (6.63190,1.47773) -- (6.66527,1.48773) -- (6.69864,1.49775) -- (6.73201,1.50777) -- (6.76538,1.51780) -- (6.79874,1.52783) -- (6.83211,1.53788) -- (6.86548,1.54792) -- (6.89885,1.55798) -- (6.93222,1.56803) -- (6.96559,1.57810) -- (6.99895,1.58817) -- (7.03232,1.59824) -- (7.06569,1.60832) -- (7.09906,1.61841) -- (7.13243,1.62850) -- (7.16579,1.63860) -- (7.19916,1.64870) -- (7.23253,1.65881) -- (7.26590,1.66892) -- (7.29927,1.67904) -- (7.33264,1.68916) -- (7.36600,1.69929) -- (7.39937,1.70942) -- (7.43274,1.71955) -- (7.46611,1.72969) -- (7.49948,1.73984) -- (7.53285,1.74999) -- (7.56621,1.76014) -- (7.59958,1.77030) -- (7.63295,1.78046) -- (7.66632,1.79063) -- (7.69969,1.80080) -- (7.73305,1.81098) -- (7.76642,1.82116) -- (7.79979,1.83134) -- (7.83316,1.84153) -- (7.86653,1.85172) -- (7.89990,1.86191) -- (7.93326,1.87211) -- (7.96663,1.88231) -- (8.00000,1.89252);
  \draw[gray,densely dashed,thick] (0.02500,0.50000) -- (0.05837,0.50000) -- (0.09174,0.50000) -- (0.12510,0.49999) -- (0.15847,0.49998) -- (0.19184,0.49996) -- (0.22521,0.49992) -- (0.25858,0.49986) -- (0.29195,0.49978) -- (0.32531,0.49966) -- (0.35868,0.49949) -- (0.39205,0.49928) -- (0.42542,0.49900) -- (0.45879,0.49866) -- (0.49215,0.49823) -- (0.52552,0.49771) -- (0.55889,0.49709) -- (0.59226,0.49635) -- (0.62563,0.49549) -- (0.65900,0.49448) -- (0.69236,0.49332) -- (0.72573,0.49200) -- (0.75910,0.49050) -- (0.79247,0.48881) -- (0.82584,0.48692) -- (0.85921,0.48482) -- (0.89257,0.48251) -- (0.92594,0.47997) -- (0.95931,0.47720) -- (0.99268,0.47420) -- (1.02605,0.47096) -- (1.05941,0.46749) -- (1.09278,0.46379) -- (1.12615,0.45988) -- (1.15952,0.45576) -- (1.19289,0.45144) -- (1.22626,0.44696) -- (1.25962,0.44233) -- (1.29299,0.43758) -- (1.32636,0.43274) -- (1.35973,0.42784) -- (1.39310,0.42292) -- (1.42646,0.41802) -- (1.45983,0.41318) -- (1.49320,0.40843) -- (1.52657,0.40381) -- (1.55994,0.39935) -- (1.59331,0.39511) -- (1.62667,0.39110) -- (1.66004,0.38735) -- (1.69341,0.38390) -- (1.72678,0.38076) -- (1.76015,0.37795) -- (1.79351,0.37549) -- (1.82688,0.37338) -- (1.86025,0.37165) -- (1.89362,0.37028) -- (1.92699,0.36929) -- (1.96036,0.36867) -- (1.99372,0.36842) -- (2.02709,0.36853) -- (2.06046,0.36900) -- (2.09383,0.36982) -- (2.12720,0.37098) -- (2.16056,0.37247) -- (2.19393,0.37429) -- (2.22730,0.37641) -- (2.26067,0.37884) -- (2.29404,0.38155) -- (2.32741,0.38454) -- (2.36077,0.38780) -- (2.39414,0.39131) -- (2.42751,0.39507) -- (2.46088,0.39906) -- (2.49425,0.40328) -- (2.52762,0.40772) -- (2.56098,0.41236) -- (2.59435,0.41719) -- (2.62772,0.42222) -- (2.66109,0.42743) -- (2.69446,0.43281) -- (2.72782,0.43836) -- (2.76119,0.44406) -- (2.79456,0.44992) -- (2.82793,0.45592) -- (2.86130,0.46206) -- (2.89467,0.46833) -- (2.92803,0.47473) -- (2.96140,0.48125) -- (2.99477,0.48789) -- (3.02814,0.49464) -- (3.06151,0.50150) -- (3.09487,0.50847) -- (3.12824,0.51553) -- (3.16161,0.52268) -- (3.19498,0.52993) -- (3.22835,0.53727) -- (3.26172,0.54469) -- (3.29508,0.55219) -- (3.32845,0.55977) -- (3.36182,0.56743) -- (3.39519,0.57516) -- (3.42856,0.58296) -- (3.46192,0.59083) -- (3.49529,0.59876) -- (3.52866,0.60675) -- (3.56203,0.61481) -- (3.59540,0.62292) -- (3.62877,0.63109) -- (3.66213,0.63931) -- (3.69550,0.64759) -- (3.72887,0.65592) -- (3.76224,0.66430) -- (3.79561,0.67272) -- (3.82897,0.68119) -- (3.86234,0.68971) -- (3.89571,0.69827) -- (3.92908,0.70687) -- (3.96245,0.71551) -- (3.99582,0.72419) -- (4.02918,0.73291) -- (4.06255,0.74167) -- (4.09592,0.75046) -- (4.12929,0.75929) -- (4.16266,0.76815) -- (4.19603,0.77704) -- (4.22939,0.78597) -- (4.26276,0.79493) -- (4.29613,0.80391) -- (4.32950,0.81293) -- (4.36287,0.82197) -- (4.39623,0.83105) -- (4.42960,0.84015) -- (4.46297,0.84927) -- (4.49634,0.85843) -- (4.52971,0.86760) -- (4.56308,0.87680) -- (4.59644,0.88603) -- (4.62981,0.89528) -- (4.66318,0.90455) -- (4.69655,0.91384) -- (4.72992,0.92315) -- (4.76328,0.93249) -- (4.79665,0.94184) -- (4.83002,0.95122) -- (4.86339,0.96061) -- (4.89676,0.97003) -- (4.93013,0.97946) -- (4.96349,0.98891) -- (4.99686,0.99838);
  \draw[gray,densely dashed,thick] (5.03023,-0.99214) -- (5.06360,-0.98264) -- (5.09697,-0.97312) -- (5.13033,-0.96359) -- (5.16370,-0.95404) -- (5.19707,-0.94447) -- (5.23044,-0.93489) -- (5.26381,-0.92530) -- (5.29718,-0.91569) -- (5.33054,-0.90606) -- (5.36391,-0.89643) -- (5.39728,-0.88678) -- (5.43065,-0.87711) -- (5.46402,-0.86744) -- (5.49738,-0.85775) -- (5.53075,-0.84805) -- (5.56412,-0.83833) -- (5.59749,-0.82861) -- (5.63086,-0.81887) -- (5.66423,-0.80912) -- (5.69759,-0.79936) -- (5.73096,-0.78959) -- (5.76433,-0.77981) -- (5.79770,-0.77001) -- (5.83107,-0.76021) -- (5.86444,-0.75040) -- (5.89780,-0.74057) -- (5.93117,-0.73074) -- (5.96454,-0.72090) -- (5.99791,-0.71105) -- (6.03128,-0.70118) -- (6.06464,-0.69131) -- (6.09801,-0.68143) -- (6.13138,-0.67155) -- (6.16475,-0.66165) -- (6.19812,-0.65175) -- (6.23149,-0.64183) -- (6.26485,-0.63191) -- (6.29822,-0.62198) -- (6.33159,-0.61204) -- (6.36496,-0.60210) -- (6.39833,-0.59215) -- (6.43169,-0.58219) -- (6.46506,-0.57222) -- (6.49843,-0.56224) -- (6.53180,-0.55226) -- (6.56517,-0.54227) -- (6.59854,-0.53228) -- (6.63190,-0.52227) -- (6.66527,-0.51227) -- (6.69864,-0.50225) -- (6.73201,-0.49223) -- (6.76538,-0.48220) -- (6.79874,-0.47217) -- (6.83211,-0.46212) -- (6.86548,-0.45208) -- (6.89885,-0.44202) -- (6.93222,-0.43197) -- (6.96559,-0.42190) -- (6.99895,-0.41183) -- (7.03232,-0.40176) -- (7.06569,-0.39168) -- (7.09906,-0.38159) -- (7.13243,-0.37150) -- (7.16579,-0.36140) -- (7.19916,-0.35130) -- (7.23253,-0.34119) -- (7.26590,-0.33108) -- (7.29927,-0.32096) -- (7.33264,-0.31084) -- (7.36600,-0.30071) -- (7.39937,-0.29058) -- (7.43274,-0.28045) -- (7.46611,-0.27031) -- (7.49948,-0.26016) -- (7.53285,-0.25001) -- (7.56621,-0.23986) -- (7.59958,-0.22970) -- (7.63295,-0.21954) -- (7.66632,-0.20937) -- (7.69969,-0.19920) -- (7.73305,-0.18902) -- (7.76642,-0.17884) -- (7.79979,-0.16866) -- (7.83316,-0.15847) -- (7.86653,-0.14828) -- (7.89990,-0.13809) -- (7.93326,-0.12789) -- (7.96663,-0.11769) -- (8.00000,-0.10748);
  \draw[gray,densely dashed,->] (5.013546,.98) -- (5.013546,-.98);
  \node[annGreen,above] at (7.65,2.70) {$\varphi_0$};
  \node[annBlue,above] at (7.65,1.67) {$\varphi_2$};
  \node[gray,below] at (7.13,-.63) {$\operatorname{Arg}\left(J_2'+iY_2'\right)/\pi$};
  \node[annBlue,below,font=\footnotesize] at (2,-.93) {minimum at $x=2$};
\end{tikzpicture}
\caption{Continuous derivative phases for orders zero and two.
Continuous derivative phases $\varphi_0,\varphi_2$ (solid) start at $\pi/2$ as $x\downarrow0$; $\varphi_0$ increases and $\varphi_2$ has its unique minimum at $x=2$. The principal argument for order two (dashed) drops by $2\pi$ at the branch cut (arrow). These curves illustrate the phase normalization.}
\label{fig:derivativephase}
\end{figure}
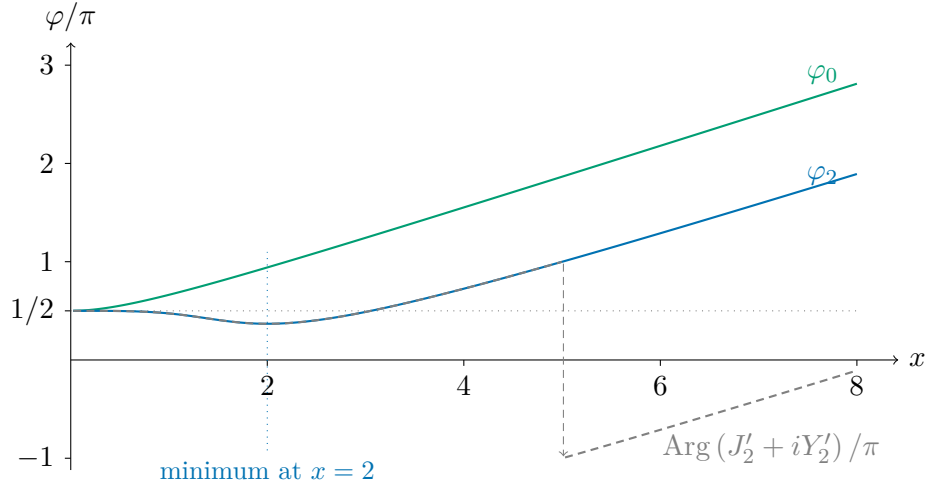

The phase normalization above is that of \cite[Section~1.2]{Phase}.
Numbered references to \cite{Phase} use the author version arXiv:2402.06956v3.
We first express the radial counting function in terms of the phase
difference at the two boundary circles.
\begin{lemma}\label{lem:exactphasecount}
Let $n_m\left(t^2\right)$ be the radial Neumann counting function in integer angular
order $m\ge0$, including its zero eigenvalue when $m=0$. For every $t>0$,
\begin{equation}\label{eq:exactphasecount}
 n_m\left(t^2\right)=\left\lfloor
 \frac{\varphi_m\left(Rt\right)-\varphi_m\left(t\right)}\pi\right\rfloor+1.
\end{equation}
For the strictly subthreshold radial count at $t>0$ one has
\begin{equation}\label{eq:strictphasecount}
 n_m^{<}(t^2)=\left\lceil
 \frac{\varphi_m(Rt)-\varphi_m(t)}\pi\right\rceil.
\end{equation}
\end{lemma}
\begin{proof}
Write $d_m\left(t\right)=\varphi_m\left(Rt\right)-\varphi_m\left(t\right)$ and consider
\[
 u_t\left(s\right)=Y_m'\left(t\right)J_m\left(ts\right)-J_m'\left(t\right)Y_m\left(ts\right).
\]
The Bessel Wronskian identity
$J_m\left(x\right)Y_m'\left(x\right)-J_m'\left(x\right)Y_m\left(x\right)=2/\left(\pi x\right)$
\cite[(10.5.2)]{NIST} gives $u_t\left(1\right)=2/\left(\pi t\right)>0$ and
$\partial_su_t\left(1\right)=0$, while
\[
 \partial_su_t\left(R\right)=-t\mathcal M_m\left(t\right)\mathcal M_m\left(Rt\right)\sin d_m\left(t\right).
\]
The right Neumann condition is therefore equivalent to
$\sin d_m\left(t\right)=0$, or $d_m\left(t\right)=j\pi$ for an integer $j$.
Geometrically, the derivative vectors at $t$ and $Rt$ are then
parallel or antiparallel. Their angle must be followed on the
continuous branch in Figure~\ref{fig:derivativephase}; retaining
the number of turns is what turns this boundary condition into
a spectral count. Thus positive radial eigenvalues occur exactly
at the integer levels of $d_m/\pi$. The single minimum property of $\varphi_m$ implies
$d_m\left(t\right)>-\pi/2$. For $m>0$, $d_m\left(t\right)<0$ at small positive $t$;
for $m=0$, $d_0\left(t\right)>0$ and tends to zero as $t\downarrow0$.

To turn this zero condition into an exact count, it remains to
check the direction of each crossing. Every crossing of an
integer level is positive. To see this,
suppose $d_m\left(t\right)=j\pi$. The Wronskian identity gives
\[
 u_t\left(R\right)=\left(-1\right)^j\frac{\mathcal M_m\left(t\right)}{\mathcal M_m\left(Rt\right)}
                       \frac2{\pi Rt}.
\]
Normalize $u_t\left(1\right)=1$ by a positive scalar and write
$z=\partial_\lambda u_t$, where $\lambda=t^2$. The radial
equation is $-\left(su_t'\right)'+m^2s^{-1}u_t=\lambda su_t$.
Multiply its differentiated equation by $u_t$, subtract the
original equation multiplied by $z$, and integrate. Because
$z\left(1\right)=z'\left(1\right)=u_t'\left(R\right)=0$ at an eigenvalue, the boundary terms give
\[
 -Ru_t\left(R\right)\,\partial_\lambda\partial_su_t\left(R\right)
       =\int_1^R s\,u_t\left(s\right)^2\dd s>0.
\]
At the root the derivatives of the prefactor multiplying
$\sin d_m$ vanish from the expression, because $\sin d_m=0$.
Since $\partial_\lambda d_m=d_m'\left(t\right)/\left(2t\right)$, differentiating the
right derivative gives a positive prefactor times
$-\left(-1\right)^j d_m'\left(t\right)$. The value $u_t\left(R\right)$ has sign $\left(-1\right)^j$,
so the positive left-hand side of the preceding identity has
the sign of $d_m'\left(t\right)$. Therefore $d_m'\left(t\right)>0$.
Positive radial eigenvalues are simple for this regular
one-dimensional boundary problem \cite{Davies}.
Starting with $-1/2<d_m\left(t\right)/\pi<0$ for $m>0$, or with
$0<d_0\left(t\right)/\pi\ll1$ for $m=0$, each encountered integer level
therefore adds exactly one eigenvalue. Continuity proves
\eqref{eq:exactphasecount}, including equality
at an eigenvalue. The extra initial count for $m=0$ is precisely its
zero eigenvalue. No global monotonicity of $d_m$ is needed.

At an eigenfrequency the inclusive count includes the corresponding
level, whereas the strictly subthreshold count excludes it. Away
from an eigenfrequency the two counts coincide. Replacing
$\lfloor d_m/\pi\rfloor+1$ by $\lceil d_m/\pi\rceil$ therefore
proves \eqref{eq:strictphasecount}, including the order-zero constant
at every positive energy.
\end{proof}

We now estimate the phases in \eqref{eq:exactphasecount}.
For $x>m\ge0$, put
\[
 \begin{aligned}
 A_m\left(x\right)&=\sqrt{x^2-m^2}-m\arccos\left(m/x\right),\\
 E_m\left(x\right)&=\frac{9x^2-2m^2}{24\left(x^2-m^2\right)^{3/2}},\,\,\,
 B_m\left(x\right)=A_m\left(x\right)+\frac\pi4+E_m\left(x\right).
 \end{aligned}
\]
The functions $A_m+\pi/4$ and $B_m$ are precisely the lower and
upper auxiliary functions in \cite[Section~1.6, (1.12) and the
following display]{Phase}, denoted there by
$\underset{\sim}{\varphi_m}$ and $\widetilde{\varphi}_m$.
By \cite[Lemma~1.8]{Phase}, $B_m$ has a unique minimum at
$x_m^*>m$. Write $z_m^*=B_m\left(x_m^*\right)$ and
\[
 C_m\left(x\right)=
 \begin{cases}
 z_m^*,&m<x<x_m^*,\\
 B_m\left(x\right),&x\ge x_m^*.
 \end{cases}
\]
This is the clipped upper envelope
$\widetilde{\widetilde{\varphi}}_m$ of
\cite[Section~1.7]{Phase}. Their Theorem~1.10, (1.16), states
$A_m+\pi/4<\varphi_m<C_m$ for every $x>m\ge0$, including order zero.
Since $C_m\le B_m$, it gives
\begin{equation}\label{eq:phaseinput}
 A_m\left(x\right)+\pi/4<\varphi_m\left(x\right)
 <A_m\left(x\right)+\pi/4+E_m\left(x\right),\,\,\, x>m\ge0.
\end{equation}

The varying error $E_m$ is used away from the turning argument.
The following consequence of the clipped envelope supplies a uniform
error bound near it.
\begin{lemma}\label{lem:globalphaseerror}
For integers $m\ge1$ and $t\ge m$, and also for $m=0$
and $t\ge2$,
\[
 0\le e_m\left(t\right):=\varphi_m\left(t\right)-A_m\left(t\right)-\pi/4<\sqrt{3/2}<3/2<\pi/2.
\]
At $t=m$ the expression is understood by continuity.
\end{lemma}
\begin{proof}
By \cite[Lemma~1.8]{Phase}, the minimum of the raw upper
envelope decreases with the order. Therefore
\[
 z_m^*\le z_0^*=\frac\pi4+\sqrt{\frac32}.
\]
For $m<x<x_m^*$, the constant part of the enclosure in
\cite[Theorem~1.10]{Phase} and $A_m\left(x\right)\ge0$ give
\[
 e_m\left(x\right)
 <z_m^*-A_m\left(x\right)-\frac\pi4
 \le\sqrt{\frac32}.
\]
For $x\ge x_m^*$, use the raw upper bound and the monotonicity
\[
 E_m'\left(x\right)=
 -\frac{x\left(3x^2+4m^2\right)}
 {8\left(x^2-m^2\right)^{5/2}}<0
\]
to obtain
\[
 e_m\left(x\right)<E_m\left(x\right)
 \le E_m\left(x_m^*\right)
 =z_m^*-A_m\left(x_m^*\right)-\frac\pi4
 \le\sqrt{\frac32}.
\]
For $m\ge1$, strict increase of the actual phase on
$\left(m,x_m^*\right)$ gives
$\varphi_m\left(m\right)<\varphi_m\left(x_m^*\right)<z_m^*$,
so the same strict upper bound holds at $x=m$.
The nonnegative lower bound follows from the lower phase
enclosure and continuity. Thus
$0\le e_m\left(x\right)<\sqrt{3/2}<3/2<\pi/2$
on the stated range, since $\pi>3$.
\end{proof}

We now translate these phase comparisons into lower bounds for
individual radial counts. Write
$d_v\left(t\right)=\varphi_v\left(Rt\right)-\varphi_v\left(t\right)$,
and recall the exact counting identity
$n_v\left(t^2\right)=\left\lfloor d_v\left(t\right)/\pi\right\rfloor+1$
from Lemma~\ref{lem:exactphasecount}.

First suppose that an integer angular order $v$ lies in the outer
region $t<v<Rt$. The inner argument $t$ is to the left of the
minimum of $\varphi_v$, so
$\varphi_v\left(t\right)\le\pi/2$.
The outer argument $Rt$ is to its right, where
\eqref{eq:phaseinput} gives
$\varphi_v\left(Rt\right)>A_v\left(Rt\right)+\pi/4$.
In this region the action has the expression
$f_t\left(v\right)=A_v\left(Rt\right)/\pi$.
Consequently
\[
 \frac{d_v\left(t\right)}\pi
 >\frac{A_v\left(Rt\right)+\pi/4-\pi/2}{\pi}
 =f_t\left(v\right)-\frac14.
\]
Taking the integer part, using its monotonicity and the identity
$\left\lfloor y\right\rfloor+1=\left\lfloor y+1\right\rfloor$, yields
\[
 \begin{aligned}
 n_v\left(t^2\right)
 &\ge\left\lfloor f_t\left(v\right)-\frac14\right\rfloor+1\\
 &=\left\lfloor f_t\left(v\right)+\frac34\right\rfloor
 \ge f_t\left(v\right)-\frac14.
 \end{aligned}
\]
Thus the shift $3/4$ comes from the phase difference
$\pi/4-\pi/2=-\pi/4$, together with the additional one in the
exact count. At $v=Rt$ the final bound follows directly from
$f_t\left(Rt\right)=0$ and $n_v\left(t^2\right)\ge0$.
If $v>Rt$, then the angular potential satisfies
$v^2/s^2\ge v^2/R^2>t^2$ throughout the radial interval
$1\le s\le R$. Its Rayleigh quotient is therefore above $t^2$,
and this order has zero count.

In the core $0\le v\le t$, the action instead equals
\[
 \pi f_t\left(v\right)=A_v\left(Rt\right)-A_v\left(t\right).
\]
The phase enclosure at $Rt>v$ and the definition of the inner
phase error give
\[
 \begin{aligned}
 d_v\left(t\right)
 &>A_v\left(Rt\right)+\frac\pi4
   -\left(A_v\left(t\right)+\frac\pi4+e_v\left(t\right)\right)\\
 &=\pi f_t\left(v\right)-e_v\left(t\right).
 \end{aligned}
\]
The equality defining $e_v\left(t\right)$ is interpreted by
continuity when $t=v$. Since
$\left\lfloor y\right\rfloor+1>y$, we conclude that
\[
 n_v\left(t^2\right)
 >f_t\left(v\right)-\frac{e_v\left(t\right)}\pi.
\]
For $t\ge2$, Lemma~\ref{lem:globalphaseerror} applies to every
integer core order: it covers $v=0$ through its separate
order-zero assertion, and $v\ge1$ through $t\ge v$.
Hence the last bound implies
$n_v\left(t^2\right)>f_t\left(v\right)-1/2$.
The outer comparison and the displayed core phase comparison
hold for every $R>1$ on their stated ranges.

\subsection{Sampling a decreasing phase}

Write $\left\{y\right\}=y-\left\lfloor y\right\rfloor$ for the fractional part of $y$, and
$\psi\left(y\right)=\left\{y\right\}-1/2$. Thus $\psi\left(n\right)=-1/2$ at integers.
This convention retains eigenvalues exactly at the counting
threshold. We use the total variation defined above to control
the jumps in this periodic function. Figure~\ref{fig:sawtooth}
shows the endpoint convention and the direction of the jumps after
composition with a decreasing phase.

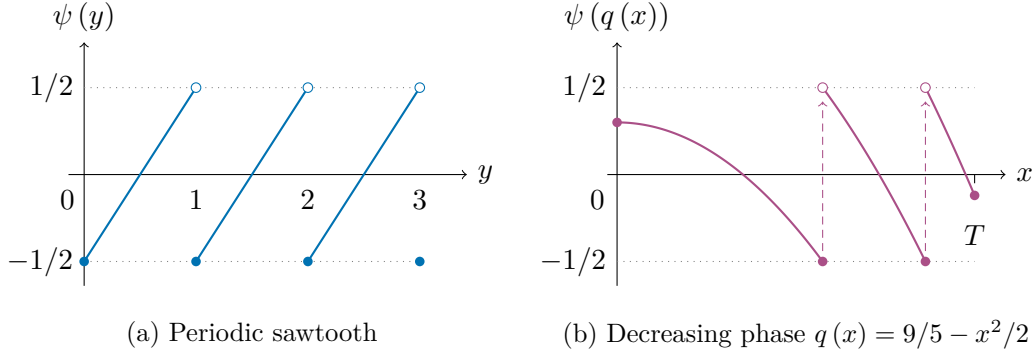
\begin{figure}[htbp]
\centering
\begin{tikzpicture}[font=\small]
  \begin{scope}[x=1.48cm,y=2.3cm]
    \draw[->] (-.14,0) -- (3.42,0) node[right] {$y$};
    \draw[->] (0,-.64) -- (0,.76) node[above] {$\psi\left(y\right)$};
    \draw[dotted,gray] (0,.5) -- (3.06,.5) (0,-.5) -- (3.06,-.5);
    \node[left] at (0,.5) {$1/2$};
    \node[left] at (0,-.5) {$-1/2$};
    \foreach \jj in {0,1,2}{
      \draw[annBlue,thick] (\jj,-.5) -- ({\jj+1},.5);
      \draw[annBlue,fill=white] ({\jj+1},.5) circle (1.8pt);
    }
    \foreach \jj in {0,1,2,3}{
  \fill[annBlue] (\jj,-.5) circle (1.8pt);
}
\foreach \jj in {1,2,3}{
  \node[below=2pt] at (\jj,0) {$\jj$};
}
\node[below=2pt] at (-0.15,0) {$0$};
    \node[font=\footnotesize] at (1.5,-.92) {(a) Periodic sawtooth};
  \end{scope}
  \begin{scope}[shift={(7.05cm,0)},x=2.15cm,y=2.3cm]
    \draw[->] (-.06,0) -- (2.39,0) node[right] {$x$};
    \draw[->] (0,-.64) -- (0,.76) node[above] {$\psi\left(q\left(x\right)\right)$};
    \draw[dotted,gray] (0,.5) -- (2.2,.5) (0,-.5) -- (2.2,-.5);
    \node[left] at (0,.5) {$1/2$};
    \node[left] at (0,-.5) {$-1/2$};
    \draw[annPurple,thick] (0.00000,0.30000) -- (0.01976,0.29980) -- (0.03953,0.29922) -- (0.05929,0.29824) -- (0.07906,0.29688) -- (0.09882,0.29512) -- (0.11859,0.29297) -- (0.13835,0.29043) -- (0.15811,0.28750) -- (0.17788,0.28418) -- (0.19764,0.28047) -- (0.21741,0.27637) -- (0.23717,0.27187) -- (0.25694,0.26699) -- (0.27670,0.26172) -- (0.29646,0.25605) -- (0.31623,0.25000) -- (0.33599,0.24355) -- (0.35576,0.23672) -- (0.37552,0.22949) -- (0.39528,0.22187) -- (0.41505,0.21387) -- (0.43481,0.20547) -- (0.45458,0.19668) -- (0.47434,0.18750) -- (0.49411,0.17793) -- (0.51387,0.16797) -- (0.53363,0.15762) -- (0.55340,0.14687) -- (0.57316,0.13574) -- (0.59293,0.12422) -- (0.61269,0.11230) -- (0.63246,0.10000) -- (0.65222,0.08730) -- (0.67198,0.07422) -- (0.69175,0.06074) -- (0.71151,0.04687) -- (0.73128,0.03262) -- (0.75104,0.01797) -- (0.77081,0.00293) -- (0.79057,-0.01250) -- (0.81033,-0.02832) -- (0.83010,-0.04453) -- (0.84986,-0.06113) -- (0.86963,-0.07813) -- (0.88939,-0.09551) -- (0.90915,-0.11328) -- (0.92892,-0.13145) -- (0.94868,-0.15000) -- (0.96845,-0.16895) -- (0.98821,-0.18828) -- (1.00798,-0.20801) -- (1.02774,-0.22812) -- (1.04750,-0.24863) -- (1.06727,-0.26953) -- (1.08703,-0.29082) -- (1.10680,-0.31250) -- (1.12656,-0.33457) -- (1.14633,-0.35703) -- (1.16609,-0.37988) -- (1.18585,-0.40313) -- (1.20562,-0.42676) -- (1.22538,-0.45078) -- (1.24515,-0.47520) -- (1.26491,-0.50000);
    \draw[annPurple,thick] (1.26491,0.50000) -- (1.27479,0.48745) -- (1.28468,0.47480) -- (1.29456,0.46206) -- (1.30444,0.44922) -- (1.31432,0.43628) -- (1.32420,0.42324) -- (1.33409,0.41011) -- (1.34397,0.39687) -- (1.35385,0.38354) -- (1.36373,0.37012) -- (1.37361,0.35659) -- (1.38350,0.34297) -- (1.39338,0.32925) -- (1.40326,0.31543) -- (1.41314,0.30151) -- (1.42302,0.28750) -- (1.43291,0.27339) -- (1.44279,0.25918) -- (1.45267,0.24487) -- (1.46255,0.23047) -- (1.47244,0.21597) -- (1.48232,0.20137) -- (1.49220,0.18667) -- (1.50208,0.17188) -- (1.51196,0.15698) -- (1.52185,0.14199) -- (1.53173,0.12690) -- (1.54161,0.11172) -- (1.55149,0.09644) -- (1.56137,0.08105) -- (1.57126,0.06558) -- (1.58114,0.05000) -- (1.59102,0.03433) -- (1.60090,0.01855) -- (1.61079,0.00269) -- (1.62067,-0.01328) -- (1.63055,-0.02935) -- (1.64043,-0.04551) -- (1.65031,-0.06177) -- (1.66020,-0.07813) -- (1.67008,-0.09458) -- (1.67996,-0.11113) -- (1.68984,-0.12778) -- (1.69972,-0.14453) -- (1.70961,-0.16138) -- (1.71949,-0.17832) -- (1.72937,-0.19536) -- (1.73925,-0.21250) -- (1.74913,-0.22974) -- (1.75902,-0.24707) -- (1.76890,-0.26450) -- (1.77878,-0.28203) -- (1.78866,-0.29966) -- (1.79855,-0.31738) -- (1.80843,-0.33521) -- (1.81831,-0.35312) -- (1.82819,-0.37114) -- (1.83807,-0.38926) -- (1.84796,-0.40747) -- (1.85784,-0.42578) -- (1.86772,-0.44419) -- (1.87760,-0.46270) -- (1.88748,-0.48130) -- (1.89737,-0.50000);
    \draw[annPurple,thick] (1.89737,0.50000) -- (1.90210,0.49102) -- (1.90682,0.48201) -- (1.91155,0.47298) -- (1.91628,0.46393) -- (1.92101,0.45486) -- (1.92574,0.44577) -- (1.93047,0.43665) -- (1.93520,0.42751) -- (1.93992,0.41835) -- (1.94465,0.40916) -- (1.94938,0.39996) -- (1.95411,0.39073) -- (1.95884,0.38147) -- (1.96357,0.37220) -- (1.96830,0.36290) -- (1.97302,0.35359) -- (1.97775,0.34425) -- (1.98248,0.33488) -- (1.98721,0.32550) -- (1.99194,0.31609) -- (1.99667,0.30666) -- (2.00140,0.29721) -- (2.00613,0.28773) -- (2.01085,0.27823) -- (2.01558,0.26871) -- (2.02031,0.25917) -- (2.02504,0.24961) -- (2.02977,0.24002) -- (2.03450,0.23041) -- (2.03923,0.22078) -- (2.04395,0.21112) -- (2.04868,0.20145) -- (2.05341,0.19175) -- (2.05814,0.18203) -- (2.06287,0.17229) -- (2.06760,0.16252) -- (2.07233,0.15273) -- (2.07706,0.14292) -- (2.08178,0.13309) -- (2.08651,0.12323) -- (2.09124,0.11336) -- (2.09597,0.10346) -- (2.10070,0.09353) -- (2.10543,0.08359) -- (2.11016,0.07362) -- (2.11488,0.06363) -- (2.11961,0.05362) -- (2.12434,0.04359) -- (2.12907,0.03353) -- (2.13380,0.02345) -- (2.13853,0.01335) -- (2.14326,0.00323) -- (2.14798,-0.00692) -- (2.15271,-0.01709) -- (2.15744,-0.02728) -- (2.16217,-0.03749) -- (2.16690,-0.04773) -- (2.17163,-0.05798) -- (2.17636,-0.06826) -- (2.18109,-0.07857) -- (2.18581,-0.08889) -- (2.19054,-0.09924) -- (2.19527,-0.10961) -- (2.20000,-0.12000);
    \foreach \xx in {1.26491106,1.89736660}{
      \fill[annPurple] (\xx,-.5) circle (1.8pt);
      \draw[annPurple,fill=white] (\xx,.5) circle (1.8pt);
      \draw[annPurple,densely dashed,->] (\xx,-.44) -- (\xx,.42);
    }
    \fill[annPurple] (0,.3) circle (1.8pt);
    \fill[annPurple] (2.2,-.12) circle (1.8pt);
    \node[below=2pt] at (-.12,0) {$0$};
    \draw (2.2,0) -- (2.2,-.05);
    \node[below=2pt] at (2.2,-.20) {$T$};
    \node[font=\footnotesize] at (1.1,-.92) {(b) Decreasing phase $q\left(x\right)=9/5-x^2/2$};
  \end{scope}
\end{tikzpicture}
\caption{
The sawtooth convention and its composition with a decreasing phase. Filled and open dots denote values and one-sided limits. For decreasing $q$, $\psi\left(q\left(x\right)\right)$ is left-continuous: each downward integer crossing has value $-1/2$, followed by an upward unit jump. Continuous variation is $q\left(0\right)-q\left(T\right)$, and each jump adds one. Panel (b) is illustrative, with $T=11/5$.}
\label{fig:sawtooth}
\end{figure}

To control the fractional-part error in the angular count, we first estimate its continuous integral. The curvature of the phase gives cancellation beyond the direct bound by the interval length.
\begin{lemma}\label{lem:bvsample}
Let $f\in C^1\left(\left[0,T\right]\right)$ have absolutely continuous derivative, with
$f'\left(0\right)=0$ and $f''\le-\kappa<0$ almost everywhere. For every
$b\in\R$,
\[
 \left|\int_0^T\psi\left(f\left(x\right)+b\right)\dd x\right|
 \le\frac1{2\sqrt\kappa}.
\]
\end{lemma}
\begin{proof}
The continuous periodic function
\[
 \mathcal P\left(y\right)=\frac12\left(\left\{y\right\}-1/2\right)^2-\frac1{16},
 \,\,\, \left|\mathcal P\left(y\right)\right|\le\frac1{16},
\]
is a primitive of $\psi$ almost everywhere. For $0<s<T$, integrate
by parts with $\mathcal P\left(f+b\right)/f'$. Since $f'<0$ and $1/f'$ is
increasing, its endpoint and variation contributions give
\[
 \left|\int_s^T\psi\left(f\left(x\right)+b\right)\dd x\right|
 \le\frac1{8\left|f'\left(s\right)\right|}\le\frac1{8\kappa s}.
\]
Indeed, writing $H=1/f'$, the bound is
$\left(\left|H\left(s\right)\right|+\left|H\left(T\right)\right|+\operatorname{Var}_{\left[s,T\right]}H\right)/16
=\left|H\left(s\right)\right|/8$, because $H<0$ is increasing. The initial integral
has absolute value at most $s/2$, because $\left|\psi\right|\le1/2$.
Thus, whenever $0<s<T$,
\[
 \left|\int_0^T\psi\left(f\left(x\right)+b\right)\dd x\right|
 \le\frac{s}{2}+\frac{1}{8\kappa s}.
\]
The right-hand side is minimized at $s=1/\left(2\sqrt\kappa\right)$,
where its two terms are both $1/\left(4\sqrt\kappa\right)$. Their sum
is $1/\left(2\sqrt\kappa\right)$. If this choice is at least $T$, the
direct bound $T/2\le1/\left(4\sqrt\kappa\right)$ suffices instead. An integrable
singularity of $f''$ at $T$ is handled by taking a limit from below.
\end{proof}

The continuous integral must now be compared with the discrete angular sum, including its multiplicities. The next discrepancy estimate makes this comparison while retaining the actual values at integer crossings.
\begin{lemma}\label{lem:combinedgrid}
Let $v$ be a left-continuous function of bounded variation on $\left[0,T\right]$,
with its actual value at zero retained. For the multiset
$\mathcal V=\left\{\left|m\right|:m\in\mathbb Z\right\}$,
\[
 \left|\sum_{x\in\mathcal V\cap\left[0,T\right]}v\left(x\right)
                 -2\int_0^T v\left(x\right)\dd x\right|
 \le\operatorname{Var}_{\left[0,T\right]}v+\left|v\left(T\right)\right|.
\]
Consequently, if $q$ satisfies the regularity and curvature
assumptions of Lemma~\ref{lem:bvsample} with curvature $\kappa$,
\begin{equation}\label{eq:combinedsaw}
 \left|\sum_{x\in\mathcal V\cap\left[0,T\right]}\psi\left(q\left(x\right)\right)\right|
 \le\kappa^{-1/2}+2\left(q\left(0\right)-q\left(T\right)\right)+\frac32.
\end{equation}
\end{lemma}
\begin{proof}
The integer-grid discrepancy
\[
 \mathcal D\left(x\right)=\#\left\{m\in\mathbb Z:\left|m\right|\le x\right\}-2x
              =2\left\lfloor x\right\rfloor+1-2x,\,\,\, x\ge0.
\]
There are $\left\lfloor x\right\rfloor$ positive orders, the same number of
negative orders, and one zero order. Hence
$\mathcal D\left(x\right)=1-2\left\{x\right\}$, which lies in $\left(-1,1\right]$ and has
absolute value at most one. Let $\sigma$ be the signed
Lebesgue--Stieltjes measure encoding the increments of the
left-continuous $v$: $\sigma\left(\left[x,y\right)\right)=v\left(y\right)-v\left(x\right)$ for $0\le x<y\le T$.
This includes a possible right jump at zero. Indeed
$v\left(x\right)=v\left(T\right)-\sigma\left(\left[x,T\right)\right)$. Integrate this identity against the
signed measure consisting of the angular point masses minus
$2\dd x$, and exchange the finite measure integrals. The inner
measure of $\left[0,y\right]$ is exactly $\mathcal D\left(y\right)$, giving the exact error
\[
 v\left(T\right)\mathcal D\left(T\right)-\int_{\left[0,T\right)}\mathcal D\left(x\right)\dd\sigma\left(x\right).
\]
This proves the first assertion with the stated endpoint convention.
We justify the variation identity with the endpoint convention
used in the preceding Stieltjes calculation. For $T=0$ the
identity is immediate, so suppose $T>0$. The assumptions on $q$
imply $q'\left(x\right)\le-\kappa x<0$ for $x>0$, hence $q$ is strictly
decreasing. It is also continuously differentiable and therefore
absolutely continuous on $\left[0,T\right]$. Put $\Delta=q\left(0\right)-q\left(T\right)$ and
define the finite crossing set
\[
 \mathcal C=\left\{x\in\left[0,T\right):q\left(x\right)\in\mathbb Z\right\}.
\]
Each integer in $\left(q\left(T\right),q\left(0\right)\right]$ has exactly one preimage in
$\mathcal C$, by continuity and strict decrease. Thus
\[
 J=\#\mathcal C
  =\left\lfloor q\left(0\right)\right\rfloor-\left\lfloor q\left(T\right)\right\rfloor.
\]
With $h\left(x\right)=\psi\left(q\left(x\right)\right)=q\left(x\right)-\left\lfloor q\left(x\right)\right\rfloor-1/2$, an
integer crossing has value $h\left(x\right)=-1/2$ and right limit
$h\left(x+\right)=1/2$. Its left limit agrees with its value. Between
crossings, $h'=q'$. More explicitly, for $0\le x<y\le T$,
\[
 h\left(y\right)-h\left(x\right)
  =\int_x^y q'\left(s\right)\,\dd s
      +\#\left(\mathcal C\cap\left[x,y\right)\right).
\]
An integer at the initial endpoint is included in this formula,
whereas an integer at $y$ has not yet produced its right jump.
The signed Stieltjes measure associated with these increments
is consequently
\[
 \dd\sigma_h=q'\left(x\right)\,\dd x+\sum_{c\in\mathcal C}\delta_c
 \quad\hbox{on }\left[0,T\right),
\]
where $\delta_c$ denotes the unit point mass at $c$. Its absolutely
continuous part is nonpositive and its atomic part is positive;
the two parts are mutually singular. Its total variation is
therefore the sum of their masses:
\[
 \operatorname{Var}_{\left[0,T\right]}h
 =-\int_0^T q'\left(x\right)\,\dd x+J
 =q\left(0\right)-q\left(T\right)+J=\Delta+J.
\]
A crossing at zero contributes one right jump; a crossing at $T$
produces no jump outside the interval. Finally,
\[
 J=\Delta+\left\{q\left(T\right)\right\}-\left\{q\left(0\right)\right\}<\Delta+1,
 \,\,\, \left|h\left(T\right)\right|\le\frac12,
\]
so the non-strict estimate $\operatorname{Var}_{\left[0,T\right]}h\le2\Delta+1$
used in the lemma is valid.

The first assertion of the lemma and the continuous integral
estimate now give the complete chain
\[
 \left|\sum_{x\in\mathcal V\cap\left[0,T\right]}h\left(x\right)\right|
 \le2\left|\int_0^T h\left(x\right)\dd x\right|
       +\operatorname{Var}_{\left[0,T\right]}h+\left|h\left(T\right)\right|
 \le\kappa^{-1/2}+2\Delta+\frac32.
\]
This is \eqref{eq:combinedsaw}. The density two accounts for the
factor before the integral. The zero order is counted once and
each positive integer order twice, as encoded in $\mathcal D$.
\end{proof}

\section{Thin annuli: an analytic proof}\label{sec:thinproof}

\indent We first state the thin-annulus theorem, whose analytic
estimate will also cover the thin part of the all-annulus theorem.
\begin{theorem}\label{thm:thinall}
If $0<h/r\le1/4$, then
\[
 N_A\left(\Lambda\right)>\frac{D}{4}\Lambda
 \,\,\,\left(\Lambda\ge0\right),
 \,\,\,
 \nu_k\left(A_{r,r+h}\right)<
 \frac{4k}{\left(r+h\right)^2-r^2}
 \,\,\,\left(k\ge1\right).
\]
\end{theorem}
The relative-width bound $1/4$ ensures that the low-frequency
variational bound and the high-frequency counting bound overlap
with a positive margin. Unlike an exact product with a
circle, a Euclidean annulus has a radial weight and a radial
dependence in the angular energy. Theorem~\ref{thm:thinall} treats
these terms directly and therefore does not follow from the thin
product theorem \cite{HW}.

Throughout this section $r=1$, $R=1+\delta$, and
$0<\delta\le1/4$. We join the low-frequency estimate in
Proposition~\ref{prop:mixedlow} to a phase estimate whose leading
gain is linear in the frequency. The overlap of their validity
ranges is the final step of the proof.

\begin{proposition}\label{prop:explicitcount}
\leavevmode\normalcolor
Let $0<\delta\le1/4$, $R=1+\delta$ and $t\ge2$. Put
\[
 d_\delta=\delta-\sqrt{R^2-1}+\arccos\left(1/R\right),\quad
 L_\delta=1-\frac\delta2-\frac{2d_\delta}{\pi},\quad
 \tau_\delta=\frac{\arccos\left(1/R\right)}{2\pi},
\]
and, with $T=t-1$,
\[
 \eta_t\left(x\right)=\frac{9t^2-2x^2}{24\pi\left(t^2-x^2\right)^{3/2}},\,\,\,
 I_t=\int_0^T\eta_t\left(x\right)\dd x
 =\frac{7T/\sqrt{t^2-T^2}+2\arcsin\left(T/t\right)}{24\pi}.
\]
Then, for the ordinary Neumann count,
\begin{equation}\label{eq:explicitcount}
 N_A\left(t^2\right)-\frac{R^2-1}{4}t^2
 \ge L_\delta t-\sqrt{\frac{\pi Rt}{\delta}}
          -2I_t-4\eta_t\left(T\right)-\frac{9}{2}-\tau_\delta.
\end{equation}
\end{proposition}

\begin{proof}
The cut $T=t-1$ keeps the explicit phase error away from its
singularity at $v=t$. For $t\ge2$, the strip $\left(t-1,t\right]$
contains exactly one positive integer, with angular multiplicity
two, so it contributes only a fixed counting loss.
The action $f_t$ has $f_t'\left(0\right)=0$, $-f_t''\ge\kappa=\delta/\left(\pi Rt\right)$
on $\left[0,t\right)$, and the exact variation
\[
 f_t\left(0\right)-f_t\left(t\right)=\frac{d_\delta}{\pi}t.
\]
On $\left[0,T\right]$ put $q=f_t-\eta_t$. The function $\eta_t$ is increasing
and convex, with derivative zero at zero: after setting $s=x/t$,
its numerator can be written as
$7\left(1-s^2\right)^{-3/2}+2\left(1-s^2\right)^{-1/2}$, divided by $24\pi t$.
Hence $q'\left(0\right)=0$ and $q''\le-\kappa$.
The exact phase identity and \eqref{eq:phaseinput} give
\[
 n_v\left(t^2\right)\ge\left\lfloor f_t\left(v\right)-\eta_t\left(v\right)\right\rfloor+1
 \quad\left(0\le v\le T\right).
\]
For the remaining core orders $T<v\le t$,
Lemma~\ref{lem:globalphaseerror} gives $e_v\left(t\right)/\pi<1/2$,
and hence $n_v\left(t^2\right)\ge f_t\left(v\right)-1/2$.
For $t<v\le Rt$, the one-sided outer comparison gives
$n_v\left(t^2\right)\ge f_t\left(v\right)-1/4$.

Let $N_c,N_o,N_b$ count angular orders, with multiplicity, in
$\left[0,t\right]$, $\left(t,Rt\right]$, and $\left(T,t\right]$, respectively. If
$C\left(b\right)=\#\left\{m\in\mathbb Z:\left|m\right|\le b\right\}$, then $\left|C\left(b\right)-2b\right|\le1$.
Consequently
$N_c/2-N_o/4=3C\left(t\right)/4-C\left(Rt\right)/4
\ge\left(1-\delta/2\right)t-1$. The half-open interval
$\left(t-1,t\right]$ contains exactly the positive integer $\left\lfloor t\right\rfloor$;
its angular multiplicity is two. Hence
\[
 \frac12N_c-\frac14N_o\ge\left(1-\delta/2\right)t-1,
 \,\,\, N_b=2.
\]
The total action sum differs from its integral $\left(R^2-1\right)t^2/4$
by at most $\tau_\delta$, by Lemma~\ref{lem:actionquadrature}.
To apply Lemma~\ref{lem:combinedgrid}, keep the full angular
multiset: the sum below means
$\psi\left(q\left(0\right)\right)+2\sum_{j=1}^{\left\lfloor T\right\rfloor}\psi\left(q\left(j\right)\right)$.
The action decreases and $T<t$, while
$\eta_t\left(0\right)=3/\left(8\pi t\right)>0$. Thus
\[
 \begin{aligned}
 \Delta=q\left(0\right)-q\left(T\right)
 &=f_t\left(0\right)-f_t\left(T\right)+\eta_t\left(T\right)-\eta_t\left(0\right)\\
 &\le f_t\left(0\right)-f_t\left(t\right)+\eta_t\left(T\right)-\eta_t\left(0\right)\\
 &\le\frac{d_\delta}{\pi}t+\eta_t\left(T\right).
 \end{aligned}
\]
Substituting this bound and $\kappa=\delta/\left(\pi Rt\right)$ into
\eqref{eq:combinedsaw} gives
\[
 \begin{aligned}
 \left|\sum_{v\le T}\psi\left(q\left(v\right)\right)\right|
 &\le\kappa^{-1/2}+2\Delta+\frac32\\
 &\le\sqrt{\frac{\pi Rt}{\delta}}
       +\frac{2d_\delta}{\pi}t+2\eta_t\left(T\right)+\frac32.
 \end{aligned}
\]
The constant $3/2$ is the sum of the jump-count allowance $1$
and the endpoint allowance $1/2$.
The same discrepancy formula applied to the smooth increasing
$\eta_t$ gives
\[
 \sum_{v\le T}\eta_t\left(v\right)\le2I_t+2\eta_t\left(T\right).
\]
In the far core use
$\left\lfloor q\right\rfloor+1=f_t-\eta_t+1/2-\psi\left(q\right)$.
In the band the preceding lower bound loses at most one relative
to the nonoscillatory contribution $f_t+1/2$. Combining these
estimates, with $1+2+3/2=9/2$, proves \eqref{eq:explicitcount}.
All sums retain the angular multiplicities.
\end{proof}

Propositions~\ref{prop:mixedlow} and~\ref{prop:explicitcount} provide
complementary lower bounds. The radial-row bound has a negative
term $-\delta^2t^2/4$ and a positive term of order $t$, so its
natural range ends at a multiple of $\delta^{-2}$. The curvature
error in the high-frequency bound is $O\left(\sqrt{t/\delta}\right)$,
which is smaller than its linear gain on this same scale. The
affine numerator $2+\left(10/3\right)\delta$ ensures an overlap throughout
$0<\delta\le1/4$. We therefore join the estimates at the frequency
$t_0$ specified below and cover the remaining low energies using
the zero mode and the first radial row.
\begin{proposition}\label{prop:explicitjoin}
\leavevmode\normalcolor
For every $0<\delta\le1/4$,
\[
 N_A\left(\Lambda\right)>\frac{\delta\left(2+\delta\right)}4\Lambda
 \,\,\,\left(\Lambda\ge0\right).
\]
The comparison estimates meet at
\[
 t_0=\frac{2+\left(10/3\right)\delta}{\delta^2}\ge\frac{136}{3}.
\]
For $t\ge t_0$, their high-frequency excess is greater than $49t/1000$.
\end{proposition}
\begin{proof}
First consider high frequencies. The function $d_\delta$ increases,
since $d_\delta'=1-\sqrt{R^2-1}/R>0$; thus $L_\delta$ decreases.
At $\delta=1/4$, set $\theta=\arccos\left(4/5\right)$. The identity
$\theta=\pi/4-\arctan\left(1/7\right)$ and the alternating arctangent bound give
\[
 \theta<\frac{11}{14}-\left(\frac17-\frac1{1029}\right)
 =\frac{1325}{2058}<\frac{161}{250}.
\]
Consequently, using $\pi>157/50$,
\[
 L_\delta\ge\frac78-\frac{2\theta-1}{\pi}
 >\frac{4919}{6280}>\frac{783}{1000}.
\]
The function $\delta\left(1+\delta\right)/\left(2+\left(10/3\right)\delta\right)$ increases.
For $t\ge t_0$ we therefore have
\[
 \sqrt{\frac{\pi R}{\delta t}}
 \le\sqrt{\frac{15\pi}{136}}<\frac{589}{1000};
\]
the last comparison follows from
$15\pi/136<165/476<\left(589/1000\right)^2$.
Since $t\ge136/3$ and $\sqrt{269/3}>47/5$, the exact formula for
$I_t$ gives
\[
 \frac{2I_t}{t}\le\frac7{12\pi\sqrt{2t-1}}+\frac1{12t}
 <\frac{20}{1000}+\frac2{1000}.
\]
Also $\left(7t+4\right)/\left(2t-1\right)^{3/2}$ decreases, so
\[
 \frac{4\eta_t\left(t-1\right)}{t}
 \le\frac{7t+4}{6\pi\left(2t-1\right)^{3/2}}
 <\frac{120500}{5954853}<\frac{21}{1000}.
\]
Finally $\tau_\delta<161/1570$, and
\[
 \frac{9/2+\tau_\delta}{t}
 <\frac{10839}{106760}<\frac{102}{1000}.
\]
The total normalized loss in \eqref{eq:explicitcount} is less than
$\left(589+20+2+21+102\right)/1000=734/1000$. The normalized excess is
therefore greater than $49/1000$.

For low frequencies put $t_1=7/\delta<t_0$. The number seven
is chosen so that the higher-row parameter exceeds two throughout
$\left[t_1,+\infty\right)$: it is a simple rational number larger than
$2\pi$ with room for the correction $3\delta^2/4$.
On $\left[3,t_1\right]$, the first bound in \eqref{eq:mixedB} makes the
right-hand side of \eqref{eq:mixedlow} a concave quadratic.
Its value at $3$ is bounded below by
\[
 \frac27+\frac{57}{40}\delta-\frac94\delta^2>\frac27,
\]
and at $t_1$ it is at least
\[
 \frac3\delta-\frac{181}{20}-\frac{3\delta}{8}
 \ge\frac{457}{160}>0.
\]
For $t\ge t_1$, the hypothesis for the stronger row bound holds:
$\delta\sqrt{t^2-3/4}\ge\sqrt{49-3\delta^2/4}>2\pi$.
Writing $c=2+\left(10/3\right)\delta$, the second bound in
\eqref{eq:mixedB} gives
\[
 B_2-\frac c4>\frac2{25}-\frac{157}{750}\delta
 \ge\frac{83}{3000}.
\]
At $t_0=c/\delta^2$ the stronger quadratic excess is consequently
at least
\[
 \frac{83}{3000}\frac{136}{3}-\frac{35}{32}
 =\frac{5777}{36000}>0.
\]
The stronger quadratic is larger than the weaker one because
$B_2>B_0$, as $\pi/2>71/50$. It is therefore positive at $t_1$
as well. A concave function on an interval is bounded below by
the smaller of its endpoint values, which covers $\left[t_1,t_0\right]$.

For $t\le1$, the area term is at most $9/64<1$, and the zero mode supplies a
count for every $t\ge0$. For $1\le t\le3$ the
zeroth radial row gives $N_A\left(t^2\right)\ge2a_\delta t-1\ge t$,
whereas $\left(R^2-1\right)t^2/4\le27t/64<t$. This covers the remaining
energies and completes the proof.
\end{proof}

\begin{proof}[Proof of Theorem~\ref{thm:thinall}]
Proposition~\ref{prop:explicitjoin} gives the inclusive counting estimate
when the inner radius is one. Its radial-row comparisons are strict
by Lemma~\ref{lem:radial}. Its phase comparisons strictly exceed their
minorants, so \eqref{eq:strictphasecount} gives the same integer lower
counts for strictly subthreshold modes. Thus its proof also establishes
$\mathcal M_A(\Lambda)>\delta(2+\delta)\Lambda/4$ at positive energy.
Scaling lengths by $r$ gives
\[
 N_A(\Lambda)>\frac{R^2-r^2}{4}\Lambda\quad(\Lambda\ge0),
 \qquad
 \mathcal M_A(\Lambda)>\frac{R^2-r^2}{4}\Lambda\quad(\Lambda>0).
\]
At $\Lambda=4k/(R^2-r^2)$ the strict count exceeds $k$, proving
$\nu_k<4k/(R^2-r^2)$ for $k\ge1$.
\end{proof}

\section{Uniform phase estimates}\label{sec:uniformphase}

We now extend the phase estimates beyond the thin-annulus range.
An outer-frequency bound valid for every radius ratio first reduces
the remaining problem to bounded frequencies. More precise control
of angular sampling and phase error then lowers this threshold on
the middle radius range.

\subsection{Uniform high frequencies}
\label{sec:allratiohigh}

\begin{lemma}\label{lem:outerfloors}
Use the disk action $G_U$ defined in \eqref{eq:diskaction}. Suppose $U\ge2$, $0\le t<U$, and a unit
grid has first point $a>t$ and $n$ points in $\left[a,U\right]$. Then
\[
 \sum_{\substack{j\ge0\\a+j\le U}}
       \left\lfloor G_U\left(a+j\right)+3/4\right\rfloor
 \ge\sum_{\substack{j\ge0\\a+j\le U}}G_U\left(a+j\right)
       +\frac n4-\frac12G_U\left(t\right)-\frac38U^{1/3}-\frac5{16}.
\]
The assertion also holds for an empty grid.
\end{lemma}
\begin{proof}
The quarter-height inverse is taken on $\left[0,U\right]$, where $G_U$ is
continuous and strictly decreasing. It exists because
$G_U\left(0\right)=U/\pi>1/4$ for $U\ge2$, whereas $G_U\left(U\right)=0$. Write
$x_*=G_U^{-1}\left(1/4\right)$ and $w_0=U-x_*$. The level $1/4$ is
dictated by the shift $3/4$: a height $z\ge0$ contributes
zero to $\left\lfloor z+3/4\right\rfloor$ precisely when $z<1/4$.
The lower tail estimate
$w_0>1$ is the established bound of
\cite[Lemma~4.8]{FLPSballs}; we spell out its calculation to verify
the hypothesis needed for the floor-sum theorem.
Set
\[
 \theta_U=\arccos\left(1-\frac1U\right),\,\,\,
 H\left(U\right)=G_U\left(U-1\right)
     =\frac{\sqrt{2U-1}-\left(U-1\right)\theta_U}{\pi}.
\]
Differentiation gives
\[
 H'\left(U\right)=\frac{\sin\theta_U-\theta_U}{\pi}<0,
 \,\,\, \sin\theta_U=\frac{\sqrt{2U-1}}U.
\]
The last inequality follows from $\sin\theta<\theta$ for
$\theta>0$. Furthermore,
\[
 H\left(2\right)=\frac{\sqrt3}{\pi}-\frac13<\frac14;
\]
for example, $\pi>3$ and $\sqrt3<7/4$ suffice for this comparison.
Thus $G_U\left(U-1\right)<1/4=G_U\left(x_*\right)$. Strict decrease gives
$x_*<U-1$, or $w_0>1$.

The upper tail bound follows independently from the action
integral and $\arccos\left(1-z\right)\ge\sqrt{2z}$ for $0\le z\le1$:
\[
 \frac14=G_U\left(U-w_0\right)
 \ge\frac1\pi\int_{U-w_0}^{U}
       \sqrt{\frac{2\left(U-s\right)}U}\,\dd s
 =\frac{2\sqrt2}{3\pi\sqrt U}w_0^{3/2}.
\]
Solving for $w_0$ gives
\[
 1<w_0\le\left(\frac{9\pi^2}{128}\right)^{1/3}U^{1/3}
       <U^{1/3}.
\]
The last coefficient is less than one, since
$\pi<22/7$ implies $9\pi^2<128$. The threshold $U\ge2$ suffices for these estimates.

If the grid is empty, then $n=0$ and both sums vanish. The
remaining terms on the right-hand side of the lemma are
nonpositive, so the assertion holds. If the grid is nonempty
but $a=U$, then $b=U-a=0$, $n=1$, and the only summand in each
sum is zero. The right-hand side is
\[
 -\frac12G_U\left(t\right)-\frac38U^{1/3}-\frac1{16}<0,
\]
which settles this degenerate interval without using a theorem
whose hypothesis requires $b>0$.

We may therefore suppose $a<U$, so that $b=U-a>0$. The profile
$g\left(y\right)=G_U\left(a+y\right)$ on $\left[0,b\right]$ satisfies
\[
 g'\left(y\right)=-\frac1\pi\arccos\left(\frac{a+y}{U}\right),\,\,\,
 g''\left(y\right)=\frac{1}{\pi\sqrt{U^2-\left(a+y\right)^2}}>0
 \quad\left(0\le y<b\right).
\]
It is nonnegative, strictly decreasing, convex, vanishes at $b$,
and has Lipschitz constant at most $1/2$, since $a>t\ge0$.
If $g\left(0\right)\ge1/4$, its quarter-height abscissa is
\[
 g^{-1}\left(1/4\right)=x_*-a=b-w_0\ge0.
\]
The integer index in \cite[Theorem~6.1]{FLPSballs} is therefore
\[
 M_0=1+\left\lfloor b-w_0\right\rfloor.
\]
Indeed $g\left(j\right)\ge1/4$ is equivalent to $j\le b-w_0$, including
equality at the quarter height. The bounds for the floor function
and $w_0>1$ give
\[
 b-w_0<M_0\le b-w_0+1<b.
\]
Thus $M_0\le b$, as required by the cited theorem, including the
case $g\left(0\right)=1/4$, where $b=w_0>1$ and $M_0=1$. That theorem gives
\[
 \begin{aligned}
 \sum_{j=0}^{\left\lfloor b\right\rfloor}\left\lfloor g\left(j\right)+\frac34\right\rfloor
 &\ge\int_0^b g\left(y\right)\,\dd y-\frac{b-3M_0}{8}\\
 &=\int_0^b g\left(y\right)\,\dd y+\frac b4-\frac38\left(b-M_0\right)\\
 &\ge\int_0^b g\left(y\right)\,\dd y+\frac b4-\frac{3w_0}{8}.
 \end{aligned}
\]
If instead $g\left(0\right)<1/4$, then $b<w_0$. Convexity puts the graph
below the chord from $\left(0,g\left(0\right)\right)$ to $\left(b,0\right)$, and hence
\[
 \int_0^b g\left(y\right)\,\dd y\le\frac{b\,g\left(0\right)}2<\frac b8.
\]
The proposed lower bound is consequently less than
$3\left(b-w_0\right)/8<0$, while every floor summand is nonnegative.
It therefore holds in this case as well.

We next compare the integral with the action sum. The derivative
formulas above give $g\left(b\right)=g'\left(b\right)=0$ and $-g'\left(0\right)\le1/2$.
Put $M=\left\lceil b\right\rceil$ and extend $g$ by zero from $b$ to $M$.
Its derivative remains continuous at $b$; the displayed square-root
singularity of $g''$ is integrable, so $g'$ remains absolutely
continuous and there is no additional jump term at $b$.

We use the unit-interval error identity derived in the proof of
Lemma~\ref{lem:largesum}, now with a convex profile. The
trapezoidal error has the opposite sign: on $\left[j,j+1\right]$,
\[
 E_j:=\frac{g\left(j\right)+g\left(j+1\right)}2-\int_j^{j+1}g\left(y\right)\dd y
 =\frac12\int_j^{j+1}\left(y-j\right)\left(j+1-y\right)g''\left(y\right)\dd y.
\]
The endpoint chord lies above the convex profile. By
\eqref{eq:unittrapezoid}, its kernel is the same $K_j$, bounded by
$1/8$. Figure~\ref{fig:trapezoid}
shows the chord error and this kernel. Summing the interval errors gives
\[
 0\le\sum_{j=0}^{M-1}E_j
 \le\frac18\int_0^M g''\left(y\right)\dd y
 =\frac{g'\left(M\right)-g'\left(0\right)}8=-\frac{g'\left(0\right)}8\le\frac1{16}.
\]
The trapezoidal sum weights the first endpoint by $1/2$,
and every endpoint at or beyond $b$ has value zero. Hence, whether
$b$ is an integer or not,
\[
 \sum_{j=0}^{M-1}\frac{g\left(j\right)+g\left(j+1\right)}2
 =\sum_{j=0}^{\left\lfloor b\right\rfloor}g\left(j\right)-\frac12g\left(0\right),
 \,\,\, \int_0^M g\left(y\right)\dd y=\int_0^b g\left(y\right)\dd y.
\]
Combining the last two displays yields
\[
 \sum_{j=0}^{\left\lfloor b\right\rfloor}g\left(j\right)
 \le\int_0^b g\left(y\right)\dd y+\frac12g\left(0\right)+\frac1{16}.
\]
Thus $g\left(0\right)/2$ is the endpoint half-weight correction, and $1/16$
is the kernel bound $1/8$ times the total derivative increase,
which is at most $1/2$.
\begin{figure}[htbp]
\centering
\begin{tikzpicture}[font=\small]
  \begin{scope}[x=4.6cm,y=13cm]
    \fill[annBlue,opacity=.13] (0,0) -- (0.00000,0.25000) -- (0.01667,0.24174) -- (0.03333,0.23361) -- (0.05000,0.22562) -- (0.06667,0.21778) -- (0.08333,0.21007) -- (0.10000,0.20250) -- (0.11667,0.19507) -- (0.13333,0.18778) -- (0.15000,0.18062) -- (0.16667,0.17361) -- (0.18333,0.16674) -- (0.20000,0.16000) -- (0.21667,0.15340) -- (0.23333,0.14694) -- (0.25000,0.14062) -- (0.26667,0.13444) -- (0.28333,0.12840) -- (0.30000,0.12250) -- (0.31667,0.11674) -- (0.33333,0.11111) -- (0.35000,0.10563) -- (0.36667,0.10028) -- (0.38333,0.09507) -- (0.40000,0.09000) -- (0.41667,0.08507) -- (0.43333,0.08028) -- (0.45000,0.07563) -- (0.46667,0.07111) -- (0.48333,0.06674) -- (0.50000,0.06250) -- (0.51667,0.05840) -- (0.53333,0.05444) -- (0.55000,0.05062) -- (0.56667,0.04694) -- (0.58333,0.04340) -- (0.60000,0.04000) -- (0.61667,0.03674) -- (0.63333,0.03361) -- (0.65000,0.03062) -- (0.66667,0.02778) -- (0.68333,0.02507) -- (0.70000,0.02250) -- (0.71667,0.02007) -- (0.73333,0.01778) -- (0.75000,0.01562) -- (0.76667,0.01361) -- (0.78333,0.01174) -- (0.80000,0.01000) -- (0.81667,0.00840) -- (0.83333,0.00694) -- (0.85000,0.00563) -- (0.86667,0.00444) -- (0.88333,0.00340) -- (0.90000,0.00250) -- (0.91667,0.00174) -- (0.93333,0.00111) -- (0.95000,0.00063) -- (0.96667,0.00028) -- (0.98333,0.00007) -- (1.00000,0.00000) -- (1,0) -- cycle;
    \fill[annOrange,opacity=.55] (0.00000,0.25000) -- (0.01667,0.24174) -- (0.03333,0.23361) -- (0.05000,0.22562) -- (0.06667,0.21778) -- (0.08333,0.21007) -- (0.10000,0.20250) -- (0.11667,0.19507) -- (0.13333,0.18778) -- (0.15000,0.18062) -- (0.16667,0.17361) -- (0.18333,0.16674) -- (0.20000,0.16000) -- (0.21667,0.15340) -- (0.23333,0.14694) -- (0.25000,0.14062) -- (0.26667,0.13444) -- (0.28333,0.12840) -- (0.30000,0.12250) -- (0.31667,0.11674) -- (0.33333,0.11111) -- (0.35000,0.10563) -- (0.36667,0.10028) -- (0.38333,0.09507) -- (0.40000,0.09000) -- (0.41667,0.08507) -- (0.43333,0.08028) -- (0.45000,0.07563) -- (0.46667,0.07111) -- (0.48333,0.06674) -- (0.50000,0.06250) -- (0.51667,0.05840) -- (0.53333,0.05444) -- (0.55000,0.05062) -- (0.56667,0.04694) -- (0.58333,0.04340) -- (0.60000,0.04000) -- (0.61667,0.03674) -- (0.63333,0.03361) -- (0.65000,0.03062) -- (0.66667,0.02778) -- (0.68333,0.02507) -- (0.70000,0.02250) -- (0.71667,0.02007) -- (0.73333,0.01778) -- (0.75000,0.01562) -- (0.76667,0.01361) -- (0.78333,0.01174) -- (0.80000,0.01000) -- (0.81667,0.00840) -- (0.83333,0.00694) -- (0.85000,0.00563) -- (0.86667,0.00444) -- (0.88333,0.00340) -- (0.90000,0.00250) -- (0.91667,0.00174) -- (0.93333,0.00111) -- (0.95000,0.00063) -- (0.96667,0.00028) -- (0.98333,0.00007) -- (1.00000,0.00000) -- (0,0.2500000) -- cycle;
    \draw[annOrange!80!black,thick] (0,0.2500000) -- (1,0.0000000);
    \draw[annBlue,thick] (0.00000,0.25000) -- (0.01667,0.24174) -- (0.03333,0.23361) -- (0.05000,0.22562) -- (0.06667,0.21778) -- (0.08333,0.21007) -- (0.10000,0.20250) -- (0.11667,0.19507) -- (0.13333,0.18778) -- (0.15000,0.18062) -- (0.16667,0.17361) -- (0.18333,0.16674) -- (0.20000,0.16000) -- (0.21667,0.15340) -- (0.23333,0.14694) -- (0.25000,0.14062) -- (0.26667,0.13444) -- (0.28333,0.12840) -- (0.30000,0.12250) -- (0.31667,0.11674) -- (0.33333,0.11111) -- (0.35000,0.10563) -- (0.36667,0.10028) -- (0.38333,0.09507) -- (0.40000,0.09000) -- (0.41667,0.08507) -- (0.43333,0.08028) -- (0.45000,0.07563) -- (0.46667,0.07111) -- (0.48333,0.06674) -- (0.50000,0.06250) -- (0.51667,0.05840) -- (0.53333,0.05444) -- (0.55000,0.05062) -- (0.56667,0.04694) -- (0.58333,0.04340) -- (0.60000,0.04000) -- (0.61667,0.03674) -- (0.63333,0.03361) -- (0.65000,0.03062) -- (0.66667,0.02778) -- (0.68333,0.02507) -- (0.70000,0.02250) -- (0.71667,0.02007) -- (0.73333,0.01778) -- (0.75000,0.01562) -- (0.76667,0.01361) -- (0.78333,0.01174) -- (0.80000,0.01000) -- (0.81667,0.00840) -- (0.83333,0.00694) -- (0.85000,0.00563) -- (0.86667,0.00444) -- (0.88333,0.00340) -- (0.90000,0.00250) -- (0.91667,0.00174) -- (0.93333,0.00111) -- (0.95000,0.00063) -- (0.96667,0.00028) -- (0.98333,0.00007) -- (1.00000,0.00000);
    \draw[densely dashed,gray] (1,0) -- (1,0.0000000);
    \draw[->] (-.05,0) -- (1.13,0) node[right] {$y$};
    \draw[->] (0,0) -- (0,.30);
    \node[below] at (0,0) {$j$};
    \node[below] at (1,0) {$j+1$};
    \node[annBlue] at (.26,.040) {$\int_j^{j+1}g$};
    \node[annBlue,right] at (.78,.012) {$g$};
    \node[annOrange!80!black,above] at (.38,.16) {chord};
    \node[annOrange!80!black] at (.62,.060) {$E_j$};
    \node[font=\footnotesize] at (.53,-.085) {(a) Orange area: trapezoidal error};
  \end{scope}
  \begin{scope}[shift={(7.15cm,0)},x=4.6cm,y=22cm]
    \draw[->] (-.04,0) -- (1.13,0) node[right] {$s$};
    \draw[->] (0,0) -- (0,.17) node[above] {$K\left(s\right)$};
    \draw[annPurple,thick] (0.00000,0.00000) -- (0.01562,0.00769) -- (0.03125,0.01514) -- (0.04688,0.02234) -- (0.06250,0.02930) -- (0.07812,0.03601) -- (0.09375,0.04248) -- (0.10938,0.04871) -- (0.12500,0.05469) -- (0.14062,0.06042) -- (0.15625,0.06592) -- (0.17188,0.07117) -- (0.18750,0.07617) -- (0.20312,0.08093) -- (0.21875,0.08545) -- (0.23438,0.08972) -- (0.25000,0.09375) -- (0.26562,0.09753) -- (0.28125,0.10107) -- (0.29688,0.10437) -- (0.31250,0.10742) -- (0.32812,0.11023) -- (0.34375,0.11279) -- (0.35938,0.11511) -- (0.37500,0.11719) -- (0.39062,0.11902) -- (0.40625,0.12061) -- (0.42188,0.12195) -- (0.43750,0.12305) -- (0.45312,0.12390) -- (0.46875,0.12451) -- (0.48438,0.12488) -- (0.50000,0.12500) -- (0.51562,0.12488) -- (0.53125,0.12451) -- (0.54688,0.12390) -- (0.56250,0.12305) -- (0.57812,0.12195) -- (0.59375,0.12061) -- (0.60938,0.11902) -- (0.62500,0.11719) -- (0.64062,0.11511) -- (0.65625,0.11279) -- (0.67188,0.11023) -- (0.68750,0.10742) -- (0.70312,0.10437) -- (0.71875,0.10107) -- (0.73438,0.09753) -- (0.75000,0.09375) -- (0.76562,0.08972) -- (0.78125,0.08545) -- (0.79688,0.08093) -- (0.81250,0.07617) -- (0.82812,0.07117) -- (0.84375,0.06592) -- (0.85938,0.06042) -- (0.87500,0.05469) -- (0.89062,0.04871) -- (0.90625,0.04248) -- (0.92188,0.03601) -- (0.93750,0.02930) -- (0.95312,0.02234) -- (0.96875,0.01514) -- (0.98438,0.00769) -- (1.00000,0.00000);
    \draw[dotted,gray] (0,.125) -- (.5,.125) -- (.5,0);
    \node[left] at (0,.125) {$1/8$};
    \node[below] at (0,0) {$0$};
    \node[below] at (.5,0) {$1/2$};
    \node[below] at (1,0) {$1$};
    \node[font=\footnotesize] at (.53,-.052) {(b) $K\left(s\right)=s\left(1-s\right)/2\le1/8$};
  \end{scope}
\end{tikzpicture}
\caption{
(a) Orange trapezoidal error $E_j$ between the illustrative convex profile $g\left(j+s\right)=\left(1-s\right)^2/4$, $0\le s\le1$, and its endpoint chord. (b) The exact remainder kernel has maximum $1/8$; multiplying by the total increase of $g'$ gives the error bound.}
\label{fig:trapezoid}
\end{figure}
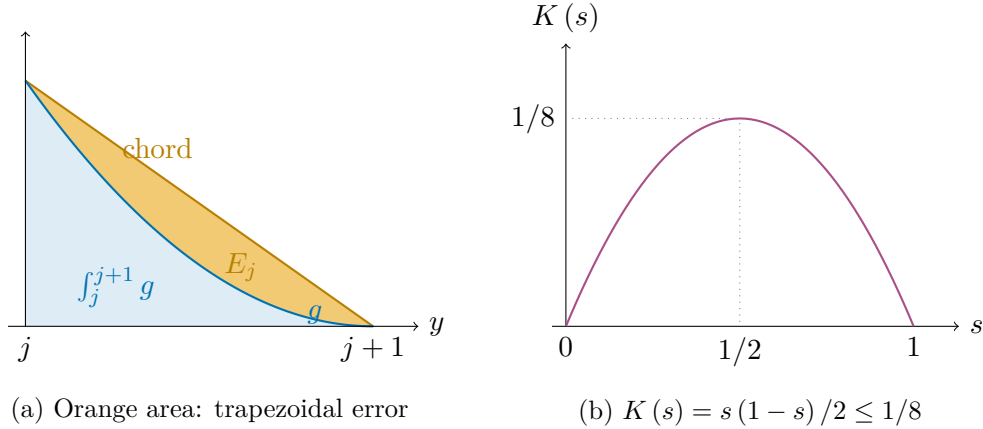

Write $S=\sum_{j=0}^{\left\lfloor b\right\rfloor}g\left(j\right)$. The trapezoidal
estimate just proved yields
\[
 \int_0^b g\left(y\right)\,\dd y\ge S-\frac12g\left(0\right)-\frac1{16}.
\]
Since $n=\left\lfloor b\right\rfloor+1$, we have $b\ge n-1$; also
$g\left(0\right)=G_U\left(a\right)\le G_U\left(t\right)$. Combining the estimates gives
\[
 \begin{aligned}
 \sum_{j=0}^{\left\lfloor b\right\rfloor}\left\lfloor g\left(j\right)+\frac34\right\rfloor
 &\ge S+\frac b4-\frac12g\left(0\right)-\frac38w_0-\frac1{16}\\
 &\ge S+\frac n4-\frac12G_U\left(t\right)-\frac38U^{1/3}-\frac5{16}.
 \end{aligned}
\]
Here $5/16=1/4+1/16$: the first term accounts for converting
the interval length to the number of grid points, and the second
is the trapezoidal error. Besides the stated unit-grid assumptions,
only $U\ge2$ and $0\le t<U$ were used.
In particular, this lemma imposes no upper bound on the relative
width of an annulus.
\end{proof}

\begin{proposition}\label{prop:allratioestimate}
Normalize the inner radius to one, set $R=1+\delta>1$, and let
$U=Rt$. For $t\ge2$, the Neumann counting function satisfies
\begin{align}
 N_A\left(t^2\right)-\frac{R^2-1}{4}t^2
 \ge{}&\mathcal L_\delta t-\sqrt{\frac{\pi Rt}{\delta}}
       -2I_t-4\eta_t\left(t-1\right)-\frac34U^{1/3}-\frac{39}{8},\label{eq:allratioestimate}
\end{align}
where
\[
 \mathcal L_\delta=1+\frac\delta2-\frac{2\delta}\pi
 +\frac{\sqrt{R^2-1}-\arccos\left(1/R\right)}\pi.
\]
Moreover, with $a_*=1/2-1/\pi$ and $b_*=2/\pi$,
\[
 \mathcal L_\delta>a_*R+b_*,\,\,\,
 a_*>9/50,\quad b_*>7/11.
\]
\end{proposition}
\begin{proof}
Keep the variable-phase comparison on $0\le v\le t-1$ from
Proposition~\ref{prop:explicitcount}. The core argument uses only
the phase bounds and Lemma~\ref{lem:actionquadrature}, which hold
for every $R>1$; its validity does not require $\delta\le1/4$. Apply
Lemma~\ref{lem:outerfloors} to each angular grid in $\left(t,Rt\right]$.
Here $N_c$ and $N_o$ are the numbers of angular orders, with
multiplicities, in $\left[0,t\right]$ and $\left(t,U\right]$. The two outer contributions
are at least their action sum plus
$N_o/4-f_t\left(t\right)-\left(3/4\right)U^{1/3}-5/8$.
Indeed, apply the preceding lemma to the two identical positive-order
grids; $G_U\left(t\right)=f_t\left(t\right)$ and each grid contributes the term
$-G_U\left(t\right)/2$. Direct integer counting gives
\[
 \frac{N_c}{2}+\frac{N_o}{4}
 =\frac{\left\lfloor t\right\rfloor+\left\lfloor U\right\rfloor+1}{2}
 \ge\frac{t+U-1}{2}.
\]
With $T=t-1$, the estimates used in
Proposition~\ref{prop:explicitcount} give, in the core,
\[
 \sum_{v\le T}\eta_t\left(v\right)\le2I_t+2\eta_t\left(T\right),\,\,\,
 \left|\sum_{v\le T}\psi\left(f_t\left(v\right)-\eta_t\left(v\right)\right)\right|
 \le\sqrt{\frac{\pi Rt}{\delta}}
       +2\left(f_t\left(0\right)-f_t\left(t\right)\right)+2\eta_t\left(T\right)+\frac32.
\]
The sums are over the angular multiset $\mathcal V$. The strip
$\left(T,t\right]$ contains two angular modes and costs at most two relative to the core contribution
$f_t+1/2$. By Lemma~\ref{lem:actionquadrature}, replacing the full
action sum by its integral costs less than $1/4$. Consequently
the constant loss is bounded by
$1/2+5/8+3/2+2+1/4=39/8$.
The linear term is
\[
 \frac{R+1}{2}t-2\left(f_t\left(0\right)-f_t\left(t\right)\right)-f_t\left(t\right)
 =\mathcal L_\delta t,
\]
since $f_t\left(0\right)=\delta t/\pi$ and
$f_t\left(t\right)=t\left(\sqrt{R^2-1}-\arccos\left(1/R\right)\right)/\pi$.
This proves
\eqref{eq:allratioestimate} without a restriction on $\delta$.

For the affine lower bound observe that
\[
 \mathcal L_\delta-a_*R-b_*
 =\frac12-\frac{R-\sqrt{R^2-1}+\arccos\left(1/R\right)}\pi>0.
\]
The expression $R-\sqrt{R^2-1}+\arccos\left(1/R\right)$ has derivative
$1-\sqrt{R^2-1}/R>0$ and tends to $\pi/2$ as
$R\to+\infty$. The rational bounds follow from
$157/50<\pi<22/7$.
\end{proof}

\begin{theorem}\label{thm:outereighty}
Every annulus satisfies the ordinary Neumann counting bound
\[
 N_A\left(\Lambda\right)>\frac{R^2-r^2}{4}\Lambda
 \,\,\,\text{whenever}\,\,\,R\sqrt\Lambda\ge80.
\]
\end{theorem}
\begin{proof}
Theorem~\ref{thm:thinall} already covers $\delta\le1/4$.
For the other annuli normalize $r=1$, write $U=Rt\ge80$, and
put $\rho=1/R\le4/5$. First suppose $t\ge2$.
For all such $t$, the exact correction formulas in
Proposition~\ref{prop:allratioestimate} give
\[
 2I_t<\frac16\sqrt t+\frac1{12},\,\,\,
 4\eta_t\left(t-1\right)<\frac13\sqrt t.
\]
For the first inequality use
$\left(t-1\right)/\sqrt{2t-1}\le\sqrt{t/2}$ and
$\arcsin\left(1-1/t\right)\le\pi/2$ in the expression for $I_t$.
Then $7/\left(12\pi\sqrt2\right)<1/6$. For the second bound use
\[
 4\eta_t\left(t-1\right)=
 \frac{7+4/t-2/t^2}{6\pi\left(2-1/t\right)^{3/2}}\sqrt t,
\]
whose numerator is at most nine and denominator greater than 27.

If $\rho\le1/2$, then the curvature term is less than
$\left(251/100\right)\sqrt t$, since
$\pi/\left(1-\rho\right)\le2\pi<\left(251/100\right)^2$.
The total square-root loss is therefore less than
$\left(301/100\right)\sqrt t$. The inequality
\[
 \frac7{11}t-\frac{301}{100}\sqrt t
 \ge-\frac{\left(301/100\right)^2}{4\left(7/11\right)}
 =-\frac{996611}{280000}
\]
is obtained by completing a square in $\sqrt t$.
Using it in \eqref{eq:allratioestimate} yields
\[
 N_A\left(t^2\right)-\frac{R^2-1}{4}t^2
 >\frac9{50}U-\frac34U^{1/3}-\frac{119}{24}
       -\frac{996611}{280000}.
\]
At $U=80$ this is greater than $315881/120000>0$, using
$80^{1/3}<13/3$; the right-hand side increases afterwards.

If $1/2\le\rho\le4/5$, then $t\ge U/2\ge40$, and the
curvature term is less than $4\sqrt t$, since
$\pi/\left(1-\rho\right)\le5\pi<16$.
The function $\left(7/11\right)t-\left(9/2\right)\sqrt t$ increases on $t\ge40$.
Consequently the excess is greater than
\[
 \left(\frac9{50}+\frac7{22}\right)U
 -\frac92\sqrt{U/2}-\frac34U^{1/3}-\frac{119}{24}.
\]
At $U=80$, use $\sqrt{40}<19/3$ and $80^{1/3}<13/3$
to bound this from below by $4153/1320>0$.
After division by $U$, every subtracted term decreases with $U$.

It remains to consider $t<2$. Write
$B_s=\left\{x\in\mathbb R^2:\left|x\right|<s\right\}$, so $B=B_1$ is the unit disk
and $A=B_R\setminus\overline{B_1}$. Let $N_{B_s}\left(\Lambda\right)$
denote its Neumann count at energy $\Lambda$, and put
$\mathcal N_B\left(Z\right)=N_{B_1}\left(Z^2\right)$. Both conventions include zero.

Restricting a function in $H^1\left(B_R\right)$ to the two pieces gives an
inclusion into $H^1\left(B_1\right)\oplus H^1\left(A\right)$. The direct sum consists
of pairs of functions on the two pieces; unlike restrictions of
a single $H^1\left(B_R\right)$ function, their boundary values (traces) at the cutting circle
need not agree. Thus cutting and imposing independent Neumann
conditions enlarges the form domain, while the norm and energy
remain the sums of their values on the two pieces. The min--max
principle \cite{Davies} lowers the eigenvalues of the separated
problem. Its spectrum is the union of the two piece spectra,
with multiplicities, so
\[
 N_{B_R}\left(t^2\right)\le N_{B_1}\left(t^2\right)+N_A\left(t^2\right).
\]
The direction of the count inequality is opposite to that of the
eigenvalue inequality: lower eigenvalues give more modes below
the same threshold. Finally the substitution $x=Ry$ gives
$\mu_j\left(B_R\right)=R^{-2}\mu_j\left(B_1\right)$, and therefore
\[
 N_{B_R}\left(t^2\right)=\mathcal N_B\left(Rt\right)=\mathcal N_B\left(U\right),
 \,\,\, N_{B_1}\left(t^2\right)=\mathcal N_B\left(t\right).
\]
Rearranging the count comparison proves
\begin{equation}\label{eq:diskbracketinginner}
 N_A\left(t^2\right)\ge\mathcal N_B\left(U\right)-\mathcal N_B\left(t\right).
\end{equation}
No extra zero-mode correction is needed: the counts on both
pieces and the original disk all include their zero modes.
The inner count is at most three: angular orders at least two are
absent, and radial Dirichlet--Neumann interlacing leaves at most one
level in each of orders zero and one. Here $j_{m,1}$ denotes the
first positive zero of $J_m$; $j_{0,1}>2$, as follows from the
alternating series for $J_0$, and these zeros increase with order
\cite[Section~10.21]{NIST}.
Theorem~2.3 of \cite{FLPSballs} bounds the disk spectral count
below by its weighted lattice count. Proposition~6.5 of that paper
then bounds the latter count. Together they imply, for $U\ge2$,
\begin{equation}\label{eq:diskFLPSremainder}
 \mathcal N_B\left(U\right)>
 \frac{U^2}{4}+\frac14\left[
  3G_U^{-1}\left(\frac14\right)
  -\left(1+\frac4\pi\right)U-3\right].
\end{equation}
Here $G_U$ is the disk action in Section~2.4, and
$G_U^{-1}\left(1/4\right)$ is its unique angular argument at height $1/4$.
The quarter-height tail estimate in Lemma~\ref{lem:outerfloors} gives
\[
 U-G_U^{-1}\left(\frac14\right)
 \le\left(\frac{9\pi^2}{128}\right)^{1/3}U^{1/3}<U^{1/3}.
\]
Substitution in \eqref{eq:diskFLPSremainder}, followed by $\pi>3$, yields
\begin{align*}
 \mathcal N_B\left(U\right)
 &>\frac{U^2}{4}+\left(\frac12-\frac1\pi\right)U
       -\frac34U^{1/3}-\frac34\\
 &>\frac{U^2}{4}+\frac U6-\frac34U^{1/3}-\frac34.
\end{align*}
Thus the annular excess is greater than
$F\left(U\right)=U/6-\left(3/4\right)U^{1/3}-15/4$. Since $80^{1/3}<5$,
$F\left(80\right)>35/6$, and
\[
 F'\left(U\right)=\frac16-\frac{1}{4U^{2/3}}>0\,\,\,\left(U\ge80\right).
\]
This proves the required positivity throughout the remaining range.
Scaling proves the theorem.
\end{proof}

The frequency $80$ is a common threshold at which the explicit
lower bounds in these three cases are positive.

\subsection{Sharper analytic phase tools}

Smaller frequency thresholds shorten the angular-layer bridges
needed below. We obtain them by refining the sawtooth integral
and the sampling at an integer cut, and by improving the
phase-error cap for integer angular orders.

\begin{lemma}[A one-sided continuous sawtooth bound]\label{lem:analyticonesided}
Let $q\in C^2([0,T])$, with $q'(0)=0$ and $q''\le-\kappa<0$.
Set $\psi(y)=\{y\}-1/2$, taking $\psi(j)=-1/2$ at integers.
Then
\[
 \int_0^T\psi(q(x))\,\mathrm dx\le\frac1{3\sqrt\kappa}.
\]
The constant is independent of the endpoint values of $q$.
\end{lemma}
\begin{proof}
Put $u=q(0)-q(x)$, $g(x)=-q'(x)$, and $B=q(0)-q(T)$.
For $x>0$, $g(x)>0$, so this is a strictly increasing substitution.
Its inverse density is $w(u)=1/g(x(u))$, which is decreasing.
Since $du/dx=g$,
\[
 \frac{\mathrm d}{\mathrm du}g(x(u))^2=2g'(x(u))\ge2\kappa,
 \qquad w(u)\le\frac1{\sqrt{2\kappa u}}.
\]
The singularity at $u=0$ is integrable, so the substitution gives
\begin{equation}\label{eq:analytic-sawtooth-substitution}
 \int_0^T\psi(q(x))\,\mathrm dx
       =\int_0^B\psi(q(0)-u)w(u)\,\mathrm du.
\end{equation}
Figure~\ref{fig:analytic-sawtooth-integrands} shows both integrands
for $q(x)=11/4-x^2/2$, when $\kappa=1$ and $w(u)=(2u)^{-1/2}$.
The composed sawtooth decreases with slope $-1$ between its upward
unit jumps. Pair any full negative half-wave immediately followed
by a full positive half-wave. If the negative half-wave starts at
$c$, their integrals sum to
\[
 -\int_0^{1/2}s\,w(c+s)\,\mathrm ds
 +\int_0^{1/2}s\,w(c+1-s)\,\mathrm ds\le0,
\]
because $c+s\le c+1-s$ and $w$ is decreasing. A partial following
positive half-wave cannot compensate more than the full positive
half-wave. Ending in the negative half-wave only lowers the integral.
Thus no subsequent prefix exceeds the integral accumulated through
the first positive half-wave.

If the initial right-hand value of the sawtooth is positive, its
height is some $p\le1/2$ and this first positive contribution is
\[
 \int_0^p(p-u)w(u)\,\mathrm du
 \le\frac1{\sqrt{2\kappa}}\int_0^p\frac{p-u}{\sqrt u}\,\mathrm du
 =\frac4{3\sqrt{2\kappa}}p^{3/2}\le\frac1{3\sqrt\kappa}.
\]
If the initial value is nonpositive, discard the initial negative
segment. Moving the next positive half-wave to $u=0$ increases its
weighted integral; it has height $1/2$ and satisfies the same bound.
If $q(0)$ is an integer, the actual value at $u=0$ is irrelevant
to the integral; its right-hand value is $1/2$. Every possible
endpoint $B$ is covered by this prefix argument.
\end{proof}

\begin{figure}[htbp]
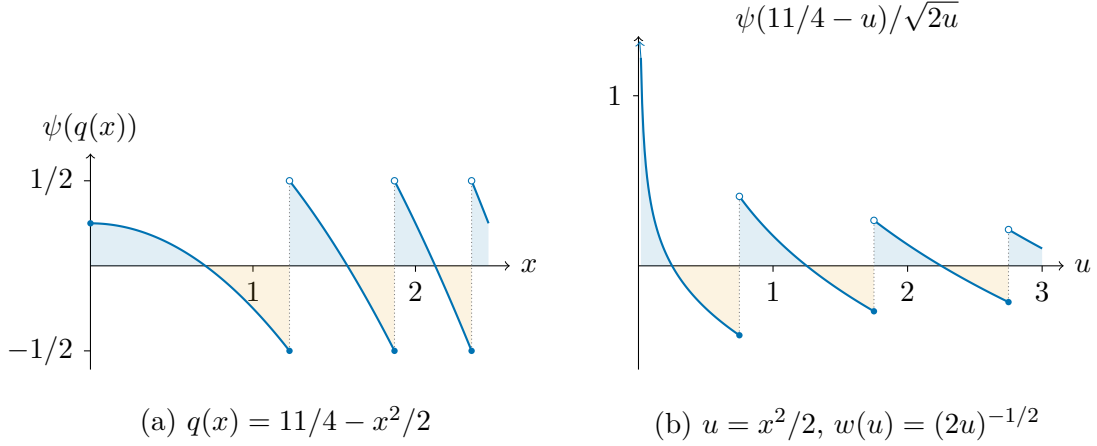

\centering

\caption{The two integrands in Lemma~\ref{lem:analyticonesided} for
$q(x)=11/4-x^2/2$ and $0\le x\le\sqrt6$.
Blue and orange shading mark positive and negative contributions.
In the transformed integral, the decreasing weight makes each full
negative half-wave dominate the following positive half-wave.
Filled lower and open upper dots specify the left-continuous convention
at a downward integer crossing of $q$.
The curves illustrate the substitution and pairing used in the proof.}
\label{fig:analytic-sawtooth-integrands}
\end{figure}

\begin{lemma}[Combined sampling with an integer cut]\label{lem:analyticintegersampling}
Let $M\ge1$ be an integer, let $f,a\in C^2([0,M])$,
and suppose that $f''\le0$ and $q=f-a$ satisfies
$q'(0)=0$, $q''\le-\kappa<0$.
Use the multiset $\mathcal V=\{0,1,1,2,2,\ldots\}$,
and sample strictly below $M$.
Then, with $I=\int_0^M a(x)\,\mathrm dx$,
\[
 \sum_{v\in\mathcal V,\,v<M}
     \bigl(a(v)+\psi(q(v))\bigr)
 \le 2I+\frac2{3\sqrt\kappa}
       +f(0)-f(M)-a(0)-\psi(q(0)).
\]
In particular, if $a(0)\ge0$, the final two terms may be replaced by $1/2$.
\end{lemma}
\begin{proof}
Put $h=a+\psi\circ q=f-\lfloor q\rfloor-1/2$.
Because $q'(x)\le-\kappa x<0$ for $x>0$, its integer crossings
form a finite set $\mathcal C\subset[0,M)$.
The function $h$ is left-continuous, has derivative $f'$ between
crossings, and has a right jump of size one at each $s\in\mathcal C$.
A crossing at zero is included when $q(0)$ is an integer;
a right jump after $M$ is excluded when $q(M)$ is an integer.
Consequently
\begin{equation}\label{eq:analytic-jump-count}
 J:=\#\mathcal C=\lfloor q(0)\rfloor-\lfloor q(M)\rfloor
  =q(0)-q(M)+\psi(q(M))-\psi(q(0)).
\end{equation}

We derive the sampling identity with these endpoint conventions.
Write $\nu_M=\delta_0+2\sum_{k=1}^{M-1}\delta_k$, where
$\delta_k$ denotes unit mass at $k$, and set
\[
 D(s)=2\lfloor s\rfloor+1-2s=1-2\{s\},\qquad 0\le s<M.
\]
The strictly truncated grid has mass $2M-1$, so its discrepancy
from twice the interval length at $M$ is $-1$.
For every $0\le s<M$, direct counting of the grid points strictly
to the right of $s$ gives
\begin{equation}\label{eq:analytic-sampling-kernel}
 \nu_M((s,M))-2(M-s)=-1-D(s).
\end{equation}
In particular, $D(k)=1$ at integer $k$, which handles a crossing
that coincides with a grid point.
The left-continuous representation is
\[
 h(x)=h(0)+\int_0^x f'(s)\,\mathrm ds
                 +\sum_{s\in\mathcal C,\ s<x}1.
\]
Sample and integrate this equality, using
\eqref{eq:analytic-sampling-kernel} for every continuous increment
and every jump. This yields
\begin{align*}
 \int h\,\mathrm d\nu_M-2\int_0^M h(x)\,\mathrm dx
 &=-h(0)-\int_0^M(1+D(s))f'(s)\,\mathrm ds
                  -\sum_{s\in\mathcal C}(1+D(s))\\
 &=-h(M)-\int_0^M D(s)f'(s)\,\mathrm ds
                  -\sum_{s\in\mathcal C}D(s).
\end{align*}
Thus the exact identity is
\begin{equation}\label{eq:analytic-sampling-identity}
 \sum_{v<M}h(v)-2\int_0^M h(x)\,\mathrm dx
 =-h(M)-\int_0^M D(s)f'(s)\,\mathrm ds
                  -\sum_{s\in\mathcal C}D(s).
\end{equation}
The angular multiset supplies the multiplicities in the sum on the
left. Since $D\ge-1$, its jump term is at most $J$.
For the continuous term, the periodic function
$Q(s)=\{s\}(1-\{s\})$ is continuous, has $Q'=D$ almost everywhere,
and satisfies $Q\ge0$, $Q(0)=Q(M)=0$. Hence
\[
 \int_0^M Df'=-\int_0^M Qf''\ge0.
\]
Insert \eqref{eq:analytic-jump-count} and
$h(M)=a(M)+\psi(q(M))$ in
\eqref{eq:analytic-sampling-identity}.
The terms at $M$ cancel and $q=f-a$ gives
\[
 \sum_{v<M}h(v)
 \le2\int_0^M h(x)\,\mathrm dx
        +f(0)-f(M)-a(0)-\psi(q(0)).
\]
Lemma~\ref{lem:analyticonesided} bounds the remaining integral of
$\psi\circ q$ by $1/(3\sqrt\kappa)$ and proves the assertion.
If $a(0)\ge0$, then $-a(0)-\psi(q(0))\le1/2$.
\end{proof}

\begin{lemma}[A smaller cap for integer angular orders]\label{lem:analyticintegercap}
For every integer $m\ge1$ and $x\ge m$,
\[
 \frac{\varphi_m(x)-A_m(x)-\pi/4}\pi<\frac{13}{50}.
\]
For $m=0$ and $x\ge2$, the same inequality holds.
\end{lemma}
\begin{proof}
The clipped enclosure from \cite[Theorem~1.10]{Phase}, used in
Lemma~\ref{lem:globalphaseerror}, bounds the numerator by
$z_m^*-\pi/4$.
The envelope minima decrease with the real order by
\cite[Lemma~1.8]{Phase}; hence
$z_m^*\le z_1^*\le B_1(8/5)$ for every integer $m\ge1$.
Here $B_1=A_1+\pi/4+E_1$ is the upper auxiliary function in
\eqref{eq:phaseinput}.
The alternating cosine series gives
\[
 \cos(89/100)\ge
 1-\frac{(89/100)^2}{2}+\frac{(89/100)^4}{24}
       -\frac{(89/100)^6}{720}>\frac58,
\]
and hence $\arccos(5/8)>89/100$. As
$31/5<\sqrt{39}<25/4$, one obtains
\[
 A_1(8/5)=\frac{\sqrt{39}}5-\arccos(5/8)<\frac9{25},
 \qquad
 E_1(8/5)=\frac{1315}{468\sqrt{39}}<\frac{6575}{14508}.
\]
All remaining comparisons are between explicit rationals:
\[
 \frac9{25}+\frac{6575}{14508}
       =\frac{294947}{362700}
       <\frac{2041}{2500}<\frac{13\pi}{50},
\]
where the last inequality uses $\pi>157/50$.
At a turning endpoint, strictness follows by
comparing the phase at $m$ with its value at the envelope minimizer;
it does not follow merely from continuity of a strict inequality.
For $m=0$, the explicit raw error is $E_0(x)=3/(8x)$,
and $3/(8\pi x)\le3/(16\pi)<13/50$ for $x\ge2$.
\end{proof}

\begin{proposition}[An improved analytic annular phase estimate]\label{prop:analyticphase}
Normalize the inner radius to one, set $R=1+\delta>1$, $U=Rt$, and
assume $t\ge2$. Choose any integer $1\le M<t$, and put
\[
 \eta_t(x)=\frac{9t^2-2x^2}{24\pi(t^2-x^2)^{3/2}},\qquad
 I_{t,M}=\frac{7M/\sqrt{t^2-M^2}+2\arcsin(M/t)}{24\pi},
\]
\[
 \kappa=\frac\delta{\pi U}+\frac{23}{24\pi t^3},\qquad
 B=2\bigl(\lfloor t\rfloor-M+1\bigr),\qquad
 c=\frac{13}{50},
 \quad a_* =\frac12-\frac1\pi.
\]
Let $f_t=G_U-G_t$ be the annular action and
$\tau_\delta=\arccos(1/R)/(2\pi)<1/4$.
Then the ordinary Neumann count satisfies
\begin{align*}
 N_A(t^2)-\frac{R^2-1}{4}t^2
 \ge{}&a_*U+(1-a_*)t+f_t(M)-f_t(t)
       -\frac2{3\sqrt\kappa}-2I_{t,M}
       -\left(\frac12+c\right)B\\
 &-\frac34U^{1/3}-\frac{13}8-\tau_\delta.
\end{align*}
The assumption $t\ge2$ ensures $U\ge2$ for the outer floor-sum
estimate and ensures that the zero-order phase bound used in the core applies.
\end{proposition}
\begin{proof}
The exact count in Lemma~\ref{lem:exactphasecount}, the enclosures
\eqref{eq:phaseinput}, and Lemma~\ref{lem:analyticintegercap} give
\[
 n_v(t^2)\ge\lfloor f_t(v)-\eta_t(v)\rfloor+1\quad(v<M),
 \qquad n_v(t^2)\ge f_t(v)-c\quad(M\le v\le t).
\]
The band consists of positive integer orders, with multiplicity two;
its weighted count is exactly $B$. The core baseline is $f_t+1/2$,
so the band loses at most $(1/2+c)B$.
The function $q=f_t-\eta_t$ has $q'(0)=0$.
Lemma~\ref{lem:actionquadrature} gives
$-f_t''\ge\delta/(\pi U)$ and $f_t''<0$ on $[0,M]$.
For the error, direct differentiation gives
\[
 \eta_t''(x)=
 \frac{23t^4+86t^2x^2-4x^4}{24\pi(t^2-x^2)^{7/2}}
 \ge\frac{23}{24\pi t^3},\qquad 0\le x<t;
\]
the numerator is at least $23t^4$ and the denominator is at most
$24\pi t^7$. Thus
\[
 -q''=-f_t''+\eta_t''\ge
 \frac\delta{\pi U}+\frac{23}{24\pi t^3}=\kappa.
\]
Apply Lemma~\ref{lem:analyticintegersampling} with $f=f_t$ and
$a=\eta_t$. Since $\eta_t(0)\ge0$, the loss from the orders $v<M$
relative to their baseline $f_t(v)+1/2$ is at most
\[
 2I_{t,M}+\frac2{3\sqrt\kappa}
       +f_t(0)-f_t(M)+\frac12.
\]

On $(t,U]$, the outer convex floor-sum argument in
Lemma~\ref{lem:outerfloors} gives an action gain of
$N_o/4-f_t(t)-(3/4)U^{1/3}-5/8$.
The core has gain $N_c/2$ before the band loss, and
\[
 \frac{N_c}2+\frac{N_o}4
       \ge\frac{t+U-1}2.
\]
By Lemma~\ref{lem:actionquadrature}, the full action sum is at least
the Weyl term minus $\tau_\delta$. Collect the losses. The constants are
$1/2+1/2+5/8=13/8$, and
\[
 \frac{t+U}2-f_t(0)
 =\left(\frac12-\frac1\pi\right)U
       +\left(\frac12+\frac1\pi\right)t.
\]
The positive term $f_t(M)-f_t(t)$ is retained. All endpoint grid
values and all angular multiplicities have been accounted for.
\end{proof}

\begin{lemma}[Two smaller uniform high-frequency regions]\label{lem:analytic23_28}
Put $\rho=r/R$ and $U=R\sqrt\Lambda$.
The strict ordinary Neumann P\'olya counting inequality holds in
either of the regions
\[
 \frac12\le\rho\le\frac45,\quad U\ge23,
 \qquad\text{or}\qquad
 \frac3{10}\le\rho\le\frac45,\quad U\ge28.
\]
These are sufficient frequency regions.
\end{lemma}
\begin{proof}
Apply Proposition~\ref{prop:analyticphase} with
$c=13/50$ and $M=\lfloor t-1\rfloor$, where $t=\rho U$.
Then $M\ge1$, $M<t$, and the band contains exactly two positive
integer orders, so $B=4$.
Since $M\le t-1$, the integral formula gives
\[
 2I_{t,M}\le2I_{t,t-1}<\frac16\sqrt t+\frac1{12}.
\]
Discard the nonnegative term $f_t(M)-f_t(t)$ and the extra positive
term $23/(24\pi t^3)$ in $\kappa$.

The outer floor-sum loss can be sharpened from
$(3/4)U^{1/3}$ to $(2/3)U^{1/3}$ without changing its proof.
Indeed, its quarter-height width satisfies
\[
 w_0\le\left(\frac{9\pi^2}{128}\right)^{1/3}U^{1/3}
       <\frac89 U^{1/3}.
\]
For the final comparison, $\pi<22/7$ gives
$9\pi^2/128<1089/1568<512/729=(8/9)^3$.
The two outer grids together have loss $(3/4)w_0$.

Using $a_*>9/50$, $1-a_*>9/11$, $\pi<22/7$, and
$\tau_\delta<1/4$, we have the explicit lower excess
\begin{equation}\label{eq:analytic-F}
 \begin{aligned}
 F(\rho,U)={}&\left(\frac9{50}+\frac{9\rho}{11}\right)U
       -\frac23\sqrt{\frac{22\rho U}{7(1-\rho)}}
       -\frac16\sqrt{\rho U}
       -\frac23U^{1/3}-\frac{2999}{600}.
 \end{aligned}
\end{equation}
The constant is
$2999/600=13/8+1/4+1/12+4(1/2+13/50)$.
The spectral excess is strictly greater than $F$.
Figure~\ref{fig:analytic-F} illustrates its two threshold sections.

For fixed $U$, this function is concave on $[3/10,4/5]$.
To check this directly, put $h(\rho)=\sqrt{\rho/(1-\rho)}$.
Then
\[
 h''(\rho)=\frac{4\rho-1}{4\rho^{3/2}(1-\rho)^{5/2}},
\]
and therefore
\[
 F_{\rho\rho}
 =\frac{\sqrt U}{4\rho^{3/2}}
 \left[-\frac23\sqrt{\frac{22}{7}}
       \frac{4\rho-1}{(1-\rho)^{5/2}}+\frac16\right]<0.
\]
For $\rho\ge3/10$,
$(4\rho-1)/(1-\rho)^{5/2}>1/5$ and $\sqrt{22/7}>5/4$.
The bracket is therefore strictly smaller than
$-(2/3)(5/4)(1/5)+1/6=0$.
For fixed $\rho$, $F(\rho,U)/U$ increases with $U$, because all
subtracted terms have strictly negative powers of $U$ after division.
It remains to check two endpoints in each of the two intervals.
These four explicit rational comparisons establish the claimed
continuous high-frequency regions.

At $U=23$, use
\[
 \frac23\sqrt{506/7}<\frac{851}{150},\quad
 \frac16\sqrt{23/2}<\frac{17}{30},\quad
 \frac23\,23^{1/3}<\frac{40}{21}
\]
for $\rho=1/2$. This yields
\[
 F(1/2,23)>
 \frac{3726}{275}
       -\left(\frac{851}{150}+\frac{17}{30}
                 +\frac{40}{21}+\frac{2999}{600}\right)
 =\frac{18757}{46200}>0.
\]
For $\rho=4/5$ use
\[
 \frac23\sqrt{2024/7}<\frac{1701}{150},\qquad
 \frac16\sqrt{92/5}<\frac{43}{60}.
\]
Then
\[
 F(4/5,23)>
 \frac{10557}{550}
       -\left(\frac{1701}{150}+\frac{43}{60}
                 +\frac{40}{21}+\frac{2999}{600}\right)
 =\frac{10847}{46200}>0.
\]

At $U=28$ and $\rho=3/10$, use
\[
 \frac23\sqrt{264/7}<\frac{41}{10},\quad
 \frac16\sqrt{42/5}<\frac{29}{60},\quad
 \frac23\,28^{1/3}<\frac{31}{15}.
\]
This gives
\[
 F(3/10,28)>
 \frac{3276}{275}
       -\left(\frac{41}{10}+\frac{29}{60}
                 +\frac{31}{15}+\frac{2999}{600}\right)
 =\frac{349}{1320}>0.
\]
At $U=28$ and $\rho=4/5$, use
\[
 \frac23\sqrt{352}<\frac{188}{15},\qquad
 \frac16\sqrt{112/5}<\frac45.
\]
Consequently
\[
 F(4/5,28)>
 \frac{6426}{275}
       -\left(\frac{188}{15}+\frac45
                 +\frac{31}{15}+\frac{2999}{600}\right)
 =\frac{3919}{1320}>0.
\]
Every displayed square-root comparison follows by squaring positive
rational numbers, and each cube-root comparison by cubing them.
Concavity and monotonicity finish both claimed regions.
\end{proof}

\begin{figure}[htbp]
\centering
\begin{tikzpicture}[x=18cm,y=.93cm,font=\small]
  \draw[->] (.285,0)--(.825,0) node[right] {$\rho$};
  \draw[->] (.3,-.12)--(.3,4.36) node[above] {$F(\rho,U)$};
  \foreach \xx/\lab in {.3/{3/10},.4/{2/5},.5/{1/2},.6/{3/5},.7/{7/10},.8/{4/5}} {\draw (\xx,0)--(\xx,-.07) node[below] {$\lab$};}
  \foreach \yy in {1,2,3,4} {\draw (.3,\yy)--(.296,\yy) node[left] {$\yy$};}
\draw[annBlue,thick] (0.5000000,0.4215850) -- (0.5037500,0.4473666) -- (0.5075000,0.4728299) -- (0.5112500,0.4979674) -- (0.5150000,0.5227716) -- (0.5187500,0.5472344) -- (0.5225000,0.5713479) -- (0.5262500,0.5951038) -- (0.5300000,0.6184937) -- (0.5337500,0.6415089) -- (0.5375000,0.6641405) -- (0.5412500,0.6863795) -- (0.5450000,0.7082166) -- (0.5487500,0.7296421) -- (0.5525000,0.7506463) -- (0.5562500,0.7712190) -- (0.5600000,0.7913499) -- (0.5637500,0.8110283) -- (0.5675000,0.8302433) -- (0.5712500,0.8489837) -- (0.5750000,0.8672378) -- (0.5787500,0.8849938) -- (0.5825000,0.9022394) -- (0.5862500,0.9189620) -- (0.5900000,0.9351486) -- (0.5937500,0.9507858) -- (0.5975000,0.9658598) -- (0.6012500,0.9803563) -- (0.6050000,0.9942607) -- (0.6087500,1.0075579) -- (0.6125000,1.0202321) -- (0.6162500,1.0322673) -- (0.6200000,1.0436468) -- (0.6237500,1.0543534) -- (0.6275000,1.0643691) -- (0.6312500,1.0736757) -- (0.6350000,1.0822540) -- (0.6387500,1.0900844) -- (0.6425000,1.0971463) -- (0.6462500,1.1034187) -- (0.6500000,1.1088797) -- (0.6537500,1.1135064) -- (0.6575000,1.1172753) -- (0.6612500,1.1201621) -- (0.6650000,1.1221412) -- (0.6687500,1.1231863) -- (0.6725000,1.1232701) -- (0.6762500,1.1223640) -- (0.6800000,1.1204384) -- (0.6837500,1.1174626) -- (0.6875000,1.1134044) -- (0.6912500,1.1082304) -- (0.6950000,1.1019059) -- (0.6987500,1.0943945) -- (0.7025000,1.0856583) -- (0.7062500,1.0756579) -- (0.7100000,1.0643519) -- (0.7137500,1.0516972) -- (0.7175000,1.0376486) -- (0.7212500,1.0221591) -- (0.7250000,1.0051792) -- (0.7287500,0.9866571) -- (0.7325000,0.9665386) -- (0.7362500,0.9447668) -- (0.7400000,0.9212819) -- (0.7437500,0.8960211) -- (0.7475000,0.8689186) -- (0.7512500,0.8399049) -- (0.7550000,0.8089070) -- (0.7587500,0.7758478) -- (0.7625000,0.7406462) -- (0.7662500,0.7032166) -- (0.7700000,0.6634684) -- (0.7737500,0.6213059) -- (0.7775000,0.5766280) -- (0.7812500,0.5293273) -- (0.7850000,0.4792902) -- (0.7887500,0.4263959) -- (0.7925000,0.3705162) -- (0.7962500,0.3115148) -- (0.8000000,0.2492463);
\draw[annGreen,thick] (0.3000000,0.3128249) -- (0.3062500,0.3899822) -- (0.3125000,0.4669881) -- (0.3187500,0.5438203) -- (0.3250000,0.6204567) -- (0.3312500,0.6968754) -- (0.3375000,0.7730547) -- (0.3437500,0.8489727) -- (0.3500000,0.9246077) -- (0.3562500,0.9999380) -- (0.3625000,1.0749416) -- (0.3687500,1.1495966) -- (0.3750000,1.2238808) -- (0.3812500,1.2977719) -- (0.3875000,1.3712474) -- (0.3937500,1.4442842) -- (0.4000000,1.5168591) -- (0.4062500,1.5889487) -- (0.4125000,1.6605288) -- (0.4187500,1.7315749) -- (0.4250000,1.8020622) -- (0.4312500,1.8719651) -- (0.4375000,1.9412574) -- (0.4437500,2.0099125) -- (0.4500000,2.0779027) -- (0.4562500,2.1451999) -- (0.4625000,2.2117751) -- (0.4687500,2.2775984) -- (0.4750000,2.3426390) -- (0.4812500,2.4068651) -- (0.4875000,2.4702439) -- (0.4937500,2.5327413) -- (0.5000000,2.5943222) -- (0.5062500,2.6549502) -- (0.5125000,2.7145875) -- (0.5187500,2.7731948) -- (0.5250000,2.8307314) -- (0.5312500,2.8871547) -- (0.5375000,2.9424206) -- (0.5437500,2.9964830) -- (0.5500000,3.0492938) -- (0.5562500,3.1008029) -- (0.5625000,3.1509577) -- (0.5687500,3.1997035) -- (0.5750000,3.2469826) -- (0.5812500,3.2927348) -- (0.5875000,3.3368969) -- (0.5937500,3.3794024) -- (0.6000000,3.4201817) -- (0.6062500,3.4591610) -- (0.6125000,3.4962632) -- (0.6187500,3.5314065) -- (0.6250000,3.5645049) -- (0.6312500,3.5954672) -- (0.6375000,3.6241973) -- (0.6437500,3.6505931) -- (0.6500000,3.6745465) -- (0.6562500,3.6959428) -- (0.6625000,3.7146601) -- (0.6687500,3.7305686) -- (0.6750000,3.7435302) -- (0.6812500,3.7533973) -- (0.6875000,3.7600125) -- (0.6937500,3.7632074) -- (0.7000000,3.7628016) -- (0.7062500,3.7586016) -- (0.7125000,3.7503993) -- (0.7187500,3.7379709) -- (0.7250000,3.7210754) -- (0.7312500,3.6994522) -- (0.7375000,3.6728196) -- (0.7437500,3.6408722) -- (0.7500000,3.6032784) -- (0.7562500,3.5596773) -- (0.7625000,3.5096754) -- (0.7687500,3.4528426) -- (0.7750000,3.3887075) -- (0.7812500,3.3167525) -- (0.7875000,3.2364075) -- (0.7937500,3.1470427) -- (0.8000000,3.0479607);
\fill[annBlue] (0.5,0.4215850) circle (1.5pt);
\fill[annBlue] (0.8,0.2492463) circle (1.5pt);
\fill[annGreen] (0.3,0.3128249) circle (1.5pt);
\fill[annGreen] (0.8,3.0479607) circle (1.5pt);
  \node[annBlue,above] at (.63,1.07) {$U=23$};
  \node[annGreen,above] at (.65,3.66) {$U=28$};
\end{tikzpicture}
\caption{The lower-excess function $F$ in
\eqref{eq:analytic-F}, on the two radius intervals used in
Lemma~\ref{lem:analytic23_28}.
Concavity in $\rho$ reduces positivity to the indicated endpoints;
the increase of $F(\rho,U)/U$ extends it to every larger $U$.
The endpoint signs follow from the rational comparisons in the proof.}
\label{fig:analytic-F}
\end{figure}
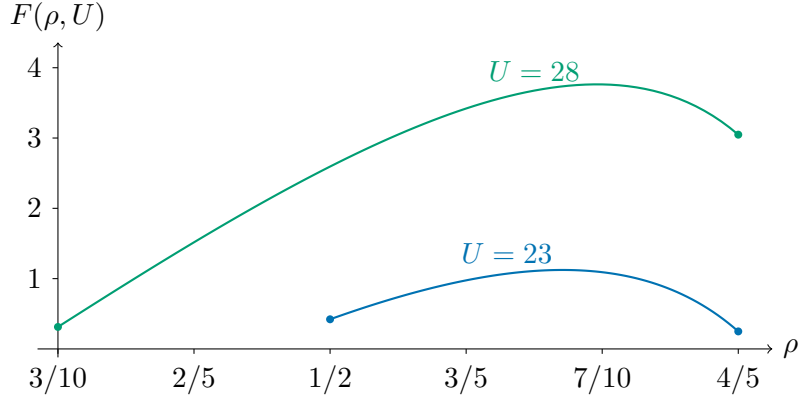

\begin{lemma}[A simpler integral-error bound]\label{lem:analyticintegralerror}
For every $t\ge2$ and $0\le M\le t-1$,
\[
 2I_{t,M}<\frac{\sqrt t}{7}.
\]
\end{lemma}
\begin{proof}
Monotonicity of the integral in its upper endpoint reduces this to
$M=t-1$. Put $y=1/t\le1/2$.
Put $\phi(y)=(1-y)(1-y/2)^{-1/2}$. Its derivatives satisfy
\[
 \phi(0)=1,\qquad \phi'(0)=-\frac34,\qquad
 \phi''(y)=\frac{y-5}{16(1-y/2)^{5/2}}<0.
\]
Its tangent at zero therefore gives
$\phi(y)\le1-3y/4$ on $[0,1/2]$;
see Figure~\ref{fig:analytic-integral-tangent}.
Furthermore,
\[
 \arccos(1-y)=\int_0^y\frac{\mathrm ds}{\sqrt{2s-s^2}}
       \ge\int_0^y\frac{\mathrm ds}{\sqrt{2s}}=\sqrt{2y},
\]
so $\arcsin(1-y)\le\pi/2-\sqrt{2y}$.
The exact formula in Proposition~\ref{prop:analyticphase} now gives
\begin{equation}\label{eq:analytic-integral-square}
 \begin{aligned}
 \frac{2I_{t,t-1}}{\sqrt t}
 &=\frac1{12\pi}\left[
       \frac7{\sqrt2}\phi(y)+2\sqrt y\arcsin(1-y)\right]\\
 &\le\frac1{12\pi}\left[
       \frac7{\sqrt2}+\pi\sqrt y-\frac{37}{4\sqrt2}y\right].
 \end{aligned}
\end{equation}
Completing the square, with $s=\sqrt y$, gives
\[
 \pi s-\frac{37}{4\sqrt2}s^2
 =\frac{\pi^2\sqrt2}{37}
   -\frac{37}{4\sqrt2}\left(s-\frac{2\pi\sqrt2}{37}\right)^2
 \le\frac{\pi^2\sqrt2}{37}.
\]
Consequently
\[
 \frac{2I_{t,t-1}}{\sqrt t}
 \le\frac7{12\pi\sqrt2}+\frac{\pi\sqrt2}{444}
  <\frac{125}{942}+\frac{55}{5439}<\frac17.
\]
The rational bounds use $157/50<\pi<22/7$ and
$7/5<\sqrt2<10/7$.
For the final comparison,
$125/942+55/5439=243895/1707846$ and
$7\cdot243895<1707846$.
\end{proof}

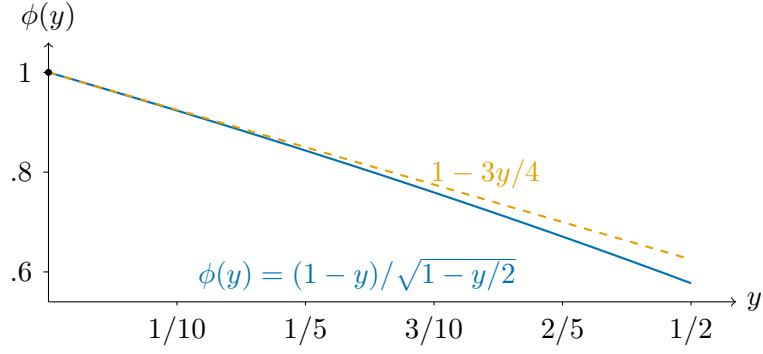
\begin{figure}[htbp]
\centering
\begin{tikzpicture}[x=17cm,y=6.6cm,font=\small]
  \draw[->] (0,.54)--(.535,.54) node[right] {$y$};
  \draw[->] (0,.54)--(0,1.06) node[above] {$\phi(y)$};
  \foreach \xx/\lab in {.1/{1/10},.2/{1/5},.3/{3/10},.4/{2/5},.5/{1/2}} {\draw (\xx,.54)--(\xx,.533) node[below] {$\lab$};}
  \foreach \yy in {.6,.8,1} {\draw (0,\yy)--(-.004,\yy) node[left] {$\yy$};}
\draw[annBlue,thick] (0.0000000,1.0000000) -- (0.0062500,0.9953064) -- (0.0125000,0.9906005) -- (0.0187500,0.9858822) -- (0.0250000,0.9811515) -- (0.0312500,0.9764082) -- (0.0375000,0.9716523) -- (0.0437500,0.9668838) -- (0.0500000,0.9621024) -- (0.0562500,0.9573082) -- (0.0625000,0.9525010) -- (0.0687500,0.9476807) -- (0.0750000,0.9428473) -- (0.0812500,0.9380007) -- (0.0875000,0.9331407) -- (0.0937500,0.9282674) -- (0.1000000,0.9233805) -- (0.1062500,0.9184801) -- (0.1125000,0.9135659) -- (0.1187500,0.9086379) -- (0.1250000,0.9036961) -- (0.1312500,0.8987403) -- (0.1375000,0.8937704) -- (0.1437500,0.8887863) -- (0.1500000,0.8837879) -- (0.1562500,0.8787751) -- (0.1625000,0.8737479) -- (0.1687500,0.8687060) -- (0.1750000,0.8636495) -- (0.1812500,0.8585781) -- (0.1875000,0.8534918) -- (0.1937500,0.8483905) -- (0.2000000,0.8432740) -- (0.2062500,0.8381423) -- (0.2125000,0.8329953) -- (0.2187500,0.8278327) -- (0.2250000,0.8226546) -- (0.2312500,0.8174607) -- (0.2375000,0.8122511) -- (0.2437500,0.8070254) -- (0.2500000,0.8017837) -- (0.2562500,0.7965258) -- (0.2625000,0.7912516) -- (0.2687500,0.7859610) -- (0.2750000,0.7806538) -- (0.2812500,0.7753298) -- (0.2875000,0.7699891) -- (0.2937500,0.7646314) -- (0.3000000,0.7592566) -- (0.3062500,0.7538646) -- (0.3125000,0.7484552) -- (0.3187500,0.7430283) -- (0.3250000,0.7375838) -- (0.3312500,0.7321215) -- (0.3375000,0.7266413) -- (0.3437500,0.7211430) -- (0.3500000,0.7156264) -- (0.3562500,0.7100916) -- (0.3625000,0.7045382) -- (0.3687500,0.6989661) -- (0.3750000,0.6933752) -- (0.3812500,0.6877654) -- (0.3875000,0.6821364) -- (0.3937500,0.6764881) -- (0.4000000,0.6708204) -- (0.4062500,0.6651330) -- (0.4125000,0.6594259) -- (0.4187500,0.6536988) -- (0.4250000,0.6479516) -- (0.4312500,0.6421841) -- (0.4375000,0.6363961) -- (0.4437500,0.6305875) -- (0.4500000,0.6247580) -- (0.4562500,0.6189076) -- (0.4625000,0.6130359) -- (0.4687500,0.6071429) -- (0.4750000,0.6012283) -- (0.4812500,0.5952919) -- (0.4875000,0.5893336) -- (0.4937500,0.5833531) -- (0.5000000,0.5773503);
\draw[annOrange,thick,dashed] (0.0000000,1.0000000) -- (0.0062500,0.9953125) -- (0.0125000,0.9906250) -- (0.0187500,0.9859375) -- (0.0250000,0.9812500) -- (0.0312500,0.9765625) -- (0.0375000,0.9718750) -- (0.0437500,0.9671875) -- (0.0500000,0.9625000) -- (0.0562500,0.9578125) -- (0.0625000,0.9531250) -- (0.0687500,0.9484375) -- (0.0750000,0.9437500) -- (0.0812500,0.9390625) -- (0.0875000,0.9343750) -- (0.0937500,0.9296875) -- (0.1000000,0.9250000) -- (0.1062500,0.9203125) -- (0.1125000,0.9156250) -- (0.1187500,0.9109375) -- (0.1250000,0.9062500) -- (0.1312500,0.9015625) -- (0.1375000,0.8968750) -- (0.1437500,0.8921875) -- (0.1500000,0.8875000) -- (0.1562500,0.8828125) -- (0.1625000,0.8781250) -- (0.1687500,0.8734375) -- (0.1750000,0.8687500) -- (0.1812500,0.8640625) -- (0.1875000,0.8593750) -- (0.1937500,0.8546875) -- (0.2000000,0.8500000) -- (0.2062500,0.8453125) -- (0.2125000,0.8406250) -- (0.2187500,0.8359375) -- (0.2250000,0.8312500) -- (0.2312500,0.8265625) -- (0.2375000,0.8218750) -- (0.2437500,0.8171875) -- (0.2500000,0.8125000) -- (0.2562500,0.8078125) -- (0.2625000,0.8031250) -- (0.2687500,0.7984375) -- (0.2750000,0.7937500) -- (0.2812500,0.7890625) -- (0.2875000,0.7843750) -- (0.2937500,0.7796875) -- (0.3000000,0.7750000) -- (0.3062500,0.7703125) -- (0.3125000,0.7656250) -- (0.3187500,0.7609375) -- (0.3250000,0.7562500) -- (0.3312500,0.7515625) -- (0.3375000,0.7468750) -- (0.3437500,0.7421875) -- (0.3500000,0.7375000) -- (0.3562500,0.7328125) -- (0.3625000,0.7281250) -- (0.3687500,0.7234375) -- (0.3750000,0.7187500) -- (0.3812500,0.7140625) -- (0.3875000,0.7093750) -- (0.3937500,0.7046875) -- (0.4000000,0.7000000) -- (0.4062500,0.6953125) -- (0.4125000,0.6906250) -- (0.4187500,0.6859375) -- (0.4250000,0.6812500) -- (0.4312500,0.6765625) -- (0.4375000,0.6718750) -- (0.4437500,0.6671875) -- (0.4500000,0.6625000) -- (0.4562500,0.6578125) -- (0.4625000,0.6531250) -- (0.4687500,0.6484375) -- (0.4750000,0.6437500) -- (0.4812500,0.6390625) -- (0.4875000,0.6343750) -- (0.4937500,0.6296875) -- (0.5000000,0.6250000);
  \fill[black] (0,1) circle (1.3pt);
  \node[annOrange,above] at (.34,.745) {$1-3y/4$};
  \node[annBlue,below] at (.24,.65) {$\phi(y)=(1-y)/\sqrt{1-y/2}$};
\end{tikzpicture}
\caption{The concave factor $\phi(y)$ in the proof of
Lemma~\ref{lem:analyticintegralerror} lies below its tangent at zero.
This tangent bound and $\arccos(1-y)\ge\sqrt{2y}$ give
\eqref{eq:analytic-integral-square}; completing the square then
bounds the integral error uniformly for $0\le y\le1/2$.}
\label{fig:analytic-integral-tangent}
\end{figure}

\begin{lemma}[A uniform threshold of $40$ on the middle radius range]\label{lem:analytic40}
For $1/16\le\rho\le4/5$ and $U\ge40$, the strict ordinary
Neumann P\'olya counting inequality holds.
\end{lemma}
\begin{proof}
Apply Proposition~\ref{prop:analyticphase} with
$M=\lfloor t-1\rfloor$ and $c=13/50$.
Here $t=\rho U\ge5/2$, so its hypotheses hold and the band has
$B=4$. Use Lemma~\ref{lem:analyticintegralerror} for $2I_{t,M}$
and the outer-width improvement from the proof of
Lemma~\ref{lem:analytic23_28}. Retain
$\kappa=\delta/(\pi U)+23/(24\pi t^3)$.
Put $a_0=57/314$ and $b_0=257/314=1-a_0$.
Since $a_*=1/2-1/\pi>a_0$, the linear term satisfies
\[
 a_*U+(1-a_*)t\ge(a_0+b_0\rho)U.
\]
The preceding bounds show that the spectral excess is strictly
greater than
\begin{equation}\label{eq:analytic-H}
 \begin{aligned}
 H(\rho,U)={}&(a_0+b_0\rho)U
 -\frac23\sqrt{\frac{(22/7)\rho U}
                  {1-\rho+23/(24\rho^2U^2)}}
 -\frac{\sqrt{\rho U}}7-\frac23U^{1/3}-\frac{983}{200}.
 \end{aligned}
\end{equation}
Here $983/200=13/8+1/4+4(1/2+13/50)$.
Dropping the positive curvature correction gives a smaller function
\begin{equation}\label{eq:analytic-Hzero}
 H_0(\rho,U)=(a_0+b_0\rho)U
       -\frac23\sqrt{\frac{(22/7)\rho U}{1-\rho}}
       -\frac{\sqrt{\rho U}}7-\frac23U^{1/3}-\frac{983}{200}.
\end{equation}
Figure~\ref{fig:analytic-H} shows both functions at $U=40$.
We prove positivity using $H_0$ on $[1/10,4/5]$ and $H$ on
$[1/16,1/10]$.

\begin{center}\begin{minipage}{\linewidth}
\centering
\begin{tikzpicture}[font=\small]
\begin{scope}[x=7.0cm,y=.29cm]
  \draw[->] (.05,0)--(.835,0) node[right] {$\rho$};
  \draw[->] (.0625,-.5)--(.0625,12.0) node[above] {$H,H_0$};
  \foreach \xx/\lab in {.0625/{1/16},.3/{3/10},.5/{1/2},.8/{4/5}} {\draw (\xx,0)--(\xx,-.27) node[below] {$\lab$};}
  \foreach \yy in {4,8,12} {\draw (.0625,\yy)--(.055,\yy) node[left] {$\yy$};}
\draw[annBlue,thick] (0.0625000,0.0972677) -- (0.0717187,0.2137416) -- (0.0809375,0.3434554) -- (0.0901563,0.4833994) -- (0.0993750,0.6314409) -- (0.1085938,0.7860205) -- (0.1178125,0.9459658) -- (0.1270313,1.1103716) -- (0.1362500,1.2785235) -- (0.1454688,1.4498454) -- (0.1546875,1.6238648) -- (0.1639063,1.8001873) -- (0.1731250,1.9784793) -- (0.1823437,2.1584543) -- (0.1915625,2.3398635) -- (0.2007812,2.5224883) -- (0.2100000,2.7061347) -- (0.2192188,2.8906285) -- (0.2284375,3.0758122) -- (0.2376563,3.2615418) -- (0.2468750,3.4476848) -- (0.2560938,3.6341180) -- (0.2653125,3.8207262) -- (0.2745313,4.0074007) -- (0.2837500,4.1940384) -- (0.2929688,4.3805406) -- (0.3021875,4.5668123) -- (0.3114062,4.7527616) -- (0.3206250,4.9382988) -- (0.3298438,5.1233359) -- (0.3390625,5.3077864) -- (0.3482812,5.4915642) -- (0.3575000,5.6745840) -- (0.3667188,5.8567598) -- (0.3759375,6.0380057) -- (0.3851562,6.2182346) -- (0.3943750,6.3973582) -- (0.4035938,6.5752866) -- (0.4128125,6.7519280) -- (0.4220313,6.9271881) -- (0.4312500,7.1009699) -- (0.4404688,7.2731734) -- (0.4496875,7.4436949) -- (0.4589063,7.6124267) -- (0.4681250,7.7792567) -- (0.4773438,7.9440679) -- (0.4865625,8.1067378) -- (0.4957813,8.2671379) -- (0.5050000,8.4251330) -- (0.5142188,8.5805806) -- (0.5234375,8.7333303) -- (0.5326563,8.8832228) -- (0.5418750,9.0300891) -- (0.5510938,9.1737495) -- (0.5603125,9.3140126) -- (0.5695313,9.4506741) -- (0.5787500,9.5835155) -- (0.5879688,9.7123024) -- (0.5971875,9.8367833) -- (0.6064063,9.9566872) -- (0.6156250,10.0717219) -- (0.6248438,10.1815714) -- (0.6340625,10.2858932) -- (0.6432813,10.3843153) -- (0.6525000,10.4764328) -- (0.6617188,10.5618035) -- (0.6709375,10.6399437) -- (0.6801563,10.7103222) -- (0.6893750,10.7723549) -- (0.6985938,10.8253966) -- (0.7078125,10.8687332) -- (0.7170313,10.9015713) -- (0.7262500,10.9230265) -- (0.7354688,10.9321092) -- (0.7446875,10.9277077) -- (0.7539062,10.9085681) -- (0.7631250,10.8732695) -- (0.7723437,10.8201943) -- (0.7815625,10.7474912) -- (0.7907812,10.6530298) -- (0.8000000,10.5343435);
\draw[annOrange,thick,dashed] (0.0625000,-0.0435117) -- (0.0717187,0.0945305) -- (0.0809375,0.2407220) -- (0.0901563,0.3935377) -- (0.0993750,0.5518341) -- (0.1085938,0.7147261) -- (0.1178125,0.8815114) -- (0.1270313,1.0516204) -- (0.1362500,1.2245827) -- (0.1454688,1.4000037) -- (0.1546875,1.5775476) -- (0.1639063,1.7569249) -- (0.1731250,1.9378834) -- (0.1823437,2.1202008) -- (0.1915625,2.3036795) -- (0.2007812,2.4881420) -- (0.2100000,2.6734276) -- (0.2192188,2.8593896) -- (0.2284375,3.0458932) -- (0.2376563,3.2328132) -- (0.2468750,3.4200328) -- (0.2560938,3.6074423) -- (0.2653125,3.7949377) -- (0.2745313,3.9824201) -- (0.2837500,4.1697946) -- (0.2929688,4.3569696) -- (0.3021875,4.5438563) -- (0.3114062,4.7303681) -- (0.3206250,4.9164199) -- (0.3298438,5.1019279) -- (0.3390625,5.2868091) -- (0.3482812,5.4709806) -- (0.3575000,5.6543595) -- (0.3667188,5.8368627) -- (0.3759375,6.0184061) -- (0.3851562,6.1989046) -- (0.3943750,6.3782714) -- (0.4035938,6.5564182) -- (0.4128125,6.7332543) -- (0.4220313,6.9086867) -- (0.4312500,7.0826192) -- (0.4404688,7.2549527) -- (0.4496875,7.4255841) -- (0.4589063,7.5944063) -- (0.4681250,7.7613077) -- (0.4773438,7.9261717) -- (0.4865625,8.0888760) -- (0.4957813,8.2492922) -- (0.5050000,8.4072853) -- (0.5142188,8.5627129) -- (0.5234375,8.7154243) -- (0.5326563,8.8652600) -- (0.5418750,9.0120508) -- (0.5510938,9.1556165) -- (0.5603125,9.2957654) -- (0.5695313,9.4322924) -- (0.5787500,9.5649784) -- (0.5879688,9.6935881) -- (0.5971875,9.8178690) -- (0.6064063,9.9375489) -- (0.6156250,10.0523343) -- (0.6248438,10.1619078) -- (0.6340625,10.2659251) -- (0.6432813,10.3640124) -- (0.6525000,10.4557624) -- (0.6617188,10.5407306) -- (0.6709375,10.6184301) -- (0.6801563,10.6883268) -- (0.6893750,10.7498327) -- (0.6985938,10.8022983) -- (0.7078125,10.8450048) -- (0.7170313,10.8771531) -- (0.7262500,10.8978521) -- (0.7354688,10.9061049) -- (0.7446875,10.9007909) -- (0.7539062,10.8806459) -- (0.7631250,10.8442370) -- (0.7723437,10.7899322) -- (0.7815625,10.7158631) -- (0.7907812,10.6198790) -- (0.8000000,10.4994884);
  \node[below] at (.43,-1.55) {(a) $U=40$, $1/16\le\rho\le4/5$};
\end{scope}
\begin{scope}[shift={(-1.075cm,0)},x=134cm,y=4.0cm]
  \draw[->] (.061,0)--(.103,0) node[right] {$\rho$};
  \draw[->] (.0625,-.09)--(.0625,.81) node[above] {$H,H_0$};
  \foreach \xx/\lab in {.0625/{1/16},.08/{2/25},.1/{1/10}} {\draw (\xx,0)--(\xx,-.015) node[below] {$\lab$};}
  \foreach \yy in {.2,.4,.6} {\draw (.0625,\yy)--(.0621,\yy) node[left] {$\yy$};}
\draw[annBlue,thick] (0.0625000,0.0972677) -- (0.0629688,0.1028092) -- (0.0634375,0.1083951) -- (0.0639062,0.1140248) -- (0.0643750,0.1196976) -- (0.0648438,0.1254129) -- (0.0653125,0.1311700) -- (0.0657812,0.1369683) -- (0.0662500,0.1428072) -- (0.0667187,0.1486862) -- (0.0671875,0.1546047) -- (0.0676563,0.1605620) -- (0.0681250,0.1665577) -- (0.0685937,0.1725912) -- (0.0690625,0.1786619) -- (0.0695313,0.1847693) -- (0.0700000,0.1909130) -- (0.0704687,0.1970924) -- (0.0709375,0.2033070) -- (0.0714063,0.2095563) -- (0.0718750,0.2158399) -- (0.0723437,0.2221573) -- (0.0728125,0.2285080) -- (0.0732813,0.2348916) -- (0.0737500,0.2413077) -- (0.0742188,0.2477557) -- (0.0746875,0.2542353) -- (0.0751563,0.2607461) -- (0.0756250,0.2672877) -- (0.0760938,0.2738596) -- (0.0765625,0.2804614) -- (0.0770312,0.2870929) -- (0.0775000,0.2937535) -- (0.0779688,0.3004429) -- (0.0784375,0.3071608) -- (0.0789063,0.3139068) -- (0.0793750,0.3206805) -- (0.0798438,0.3274816) -- (0.0803125,0.3343097) -- (0.0807812,0.3411646) -- (0.0812500,0.3480458) -- (0.0817188,0.3549531) -- (0.0821875,0.3618860) -- (0.0826563,0.3688445) -- (0.0831250,0.3758280) -- (0.0835937,0.3828363) -- (0.0840625,0.3898692) -- (0.0845313,0.3969262) -- (0.0850000,0.4040072) -- (0.0854688,0.4111118) -- (0.0859375,0.4182398) -- (0.0864063,0.4253909) -- (0.0868750,0.4325648) -- (0.0873437,0.4397613) -- (0.0878125,0.4469801) -- (0.0882813,0.4542209) -- (0.0887500,0.4614834) -- (0.0892187,0.4687676) -- (0.0896875,0.4760730) -- (0.0901563,0.4833994) -- (0.0906250,0.4907467) -- (0.0910938,0.4981146) -- (0.0915625,0.5055028) -- (0.0920313,0.5129112) -- (0.0925000,0.5203394) -- (0.0929688,0.5277873) -- (0.0934375,0.5352547) -- (0.0939062,0.5427414) -- (0.0943750,0.5502471) -- (0.0948438,0.5577716) -- (0.0953125,0.5653148) -- (0.0957813,0.5728765) -- (0.0962500,0.5804563) -- (0.0967188,0.5880542) -- (0.0971875,0.5956700) -- (0.0976562,0.6033034) -- (0.0981250,0.6109543) -- (0.0985938,0.6186225) -- (0.0990625,0.6263079) -- (0.0995313,0.6340102) -- (0.1000000,0.6417292);
\draw[annOrange,thick,dashed] (0.0625000,-0.0435117) -- (0.0629688,-0.0367193) -- (0.0634375,-0.0299008) -- (0.0639062,-0.0230566) -- (0.0643750,-0.0161871) -- (0.0648438,-0.0092925) -- (0.0653125,-0.0023730) -- (0.0657812,0.0045709) -- (0.0662500,0.0115390) -- (0.0667187,0.0185310) -- (0.0671875,0.0255467) -- (0.0676563,0.0325857) -- (0.0681250,0.0396478) -- (0.0685937,0.0467328) -- (0.0690625,0.0538403) -- (0.0695313,0.0609701) -- (0.0700000,0.0681220) -- (0.0704687,0.0752958) -- (0.0709375,0.0824911) -- (0.0714063,0.0897077) -- (0.0718750,0.0969454) -- (0.0723437,0.1042041) -- (0.0728125,0.1114833) -- (0.0732813,0.1187830) -- (0.0737500,0.1261029) -- (0.0742188,0.1334427) -- (0.0746875,0.1408023) -- (0.0751563,0.1481815) -- (0.0756250,0.1555801) -- (0.0760938,0.1629978) -- (0.0765625,0.1704344) -- (0.0770312,0.1778898) -- (0.0775000,0.1853637) -- (0.0779688,0.1928561) -- (0.0784375,0.2003665) -- (0.0789063,0.2078950) -- (0.0793750,0.2154413) -- (0.0798438,0.2230052) -- (0.0803125,0.2305866) -- (0.0807812,0.2381853) -- (0.0812500,0.2458010) -- (0.0817188,0.2534337) -- (0.0821875,0.2610832) -- (0.0826563,0.2687492) -- (0.0831250,0.2764317) -- (0.0835937,0.2841305) -- (0.0840625,0.2918454) -- (0.0845313,0.2995763) -- (0.0850000,0.3073229) -- (0.0854688,0.3150853) -- (0.0859375,0.3228631) -- (0.0864063,0.3306563) -- (0.0868750,0.3384647) -- (0.0873437,0.3462882) -- (0.0878125,0.3541266) -- (0.0882813,0.3619798) -- (0.0887500,0.3698476) -- (0.0892187,0.3777300) -- (0.0896875,0.3856267) -- (0.0901563,0.3935377) -- (0.0906250,0.4014628) -- (0.0910938,0.4094019) -- (0.0915625,0.4173549) -- (0.0920313,0.4253215) -- (0.0925000,0.4333018) -- (0.0929688,0.4412956) -- (0.0934375,0.4493027) -- (0.0939062,0.4573231) -- (0.0943750,0.4653566) -- (0.0948438,0.4734031) -- (0.0953125,0.4814625) -- (0.0957813,0.4895347) -- (0.0962500,0.4976196) -- (0.0967188,0.5057169) -- (0.0971875,0.5138268) -- (0.0976562,0.5219490) -- (0.0981250,0.5300834) -- (0.0985938,0.5382299) -- (0.0990625,0.5463884) -- (0.0995313,0.5545589) -- (0.1000000,0.5627411);
  \node[annBlue,above] at (.083,.48) {$H(\rho,40)$};
  \node[annOrange,below] at (.086,.30) {$H_0(\rho,40)$};
  \node[below] at (.0825,-.115) {(b) Small-radius detail};
\end{scope}
\begin{scope}[shift={(-19.6cm,-5.2cm)},x=47cm,y=3.3cm]
  \draw[->] (.496,0)--(.61,0) node[right] {$\rho$};
  \draw[->] (.5,0)--(.5,.97) node[above] {$H_0(\rho,22)$};
  \foreach \xx/\lab in {.5/{1/2},.55/{11/20},.6/{3/5}} {\draw (\xx,0)--(\xx,-.018) node[below] {$\lab$};}
  \foreach \yy in {.2,.4,.6,.8} {\draw (.5,\yy)--(.4985,\yy) node[left] {$\yy$};}
\draw[annGreen,thick] (0.5000000,0.1965066) -- (0.5012500,0.2045466) -- (0.5025000,0.2125524) -- (0.5037500,0.2205238) -- (0.5050000,0.2284606) -- (0.5062500,0.2363623) -- (0.5075000,0.2442288) -- (0.5087500,0.2520597) -- (0.5100000,0.2598549) -- (0.5112500,0.2676140) -- (0.5125000,0.2753368) -- (0.5137500,0.2830229) -- (0.5150000,0.2906721) -- (0.5162500,0.2982841) -- (0.5175000,0.3058586) -- (0.5187500,0.3133953) -- (0.5200000,0.3208939) -- (0.5212500,0.3283542) -- (0.5225000,0.3357758) -- (0.5237500,0.3431584) -- (0.5250000,0.3505018) -- (0.5262500,0.3578055) -- (0.5275000,0.3650694) -- (0.5287500,0.3722931) -- (0.5300000,0.3794763) -- (0.5312500,0.3866186) -- (0.5325000,0.3937198) -- (0.5337500,0.4007795) -- (0.5350000,0.4077975) -- (0.5362500,0.4147733) -- (0.5375000,0.4217067) -- (0.5387500,0.4285973) -- (0.5400000,0.4354448) -- (0.5412500,0.4422488) -- (0.5425000,0.4490091) -- (0.5437500,0.4557252) -- (0.5450000,0.4623969) -- (0.5462500,0.4690237) -- (0.5475000,0.4756053) -- (0.5487500,0.4821414) -- (0.5500000,0.4886316) -- (0.5512500,0.4950756) -- (0.5525000,0.5014729) -- (0.5537500,0.5078233) -- (0.5550000,0.5141262) -- (0.5562500,0.5203815) -- (0.5575000,0.5265886) -- (0.5587500,0.5327473) -- (0.5600000,0.5388570) -- (0.5612500,0.5449175) -- (0.5625000,0.5509283) -- (0.5637500,0.5568891) -- (0.5650000,0.5627994) -- (0.5662500,0.5686588) -- (0.5675000,0.5744670) -- (0.5687500,0.5802235) -- (0.5700000,0.5859279) -- (0.5712500,0.5915798) -- (0.5725000,0.5971788) -- (0.5737500,0.6027243) -- (0.5750000,0.6082161) -- (0.5762500,0.6136537) -- (0.5775000,0.6190366) -- (0.5787500,0.6243644) -- (0.5800000,0.6296366) -- (0.5812500,0.6348528) -- (0.5825000,0.6400126) -- (0.5837500,0.6451154) -- (0.5850000,0.6501608) -- (0.5862500,0.6551484) -- (0.5875000,0.6600776) -- (0.5887500,0.6649480) -- (0.5900000,0.6697590) -- (0.5912500,0.6745103) -- (0.5925000,0.6792013) -- (0.5937500,0.6838315) -- (0.5950000,0.6884004) -- (0.5962500,0.6929075) -- (0.5975000,0.6973523) -- (0.5987500,0.7017343) -- (0.6000000,0.7060528);
\fill[annGreen] (0.5,0.1965066) circle (1.4pt);
\fill[annGreen] (0.6,0.7060528) circle (1.4pt);
  \node[below] at (.55,-.13) {(c) The radius interval $[1/2,3/5]$ at $U=22$};
\end{scope}
\end{tikzpicture}

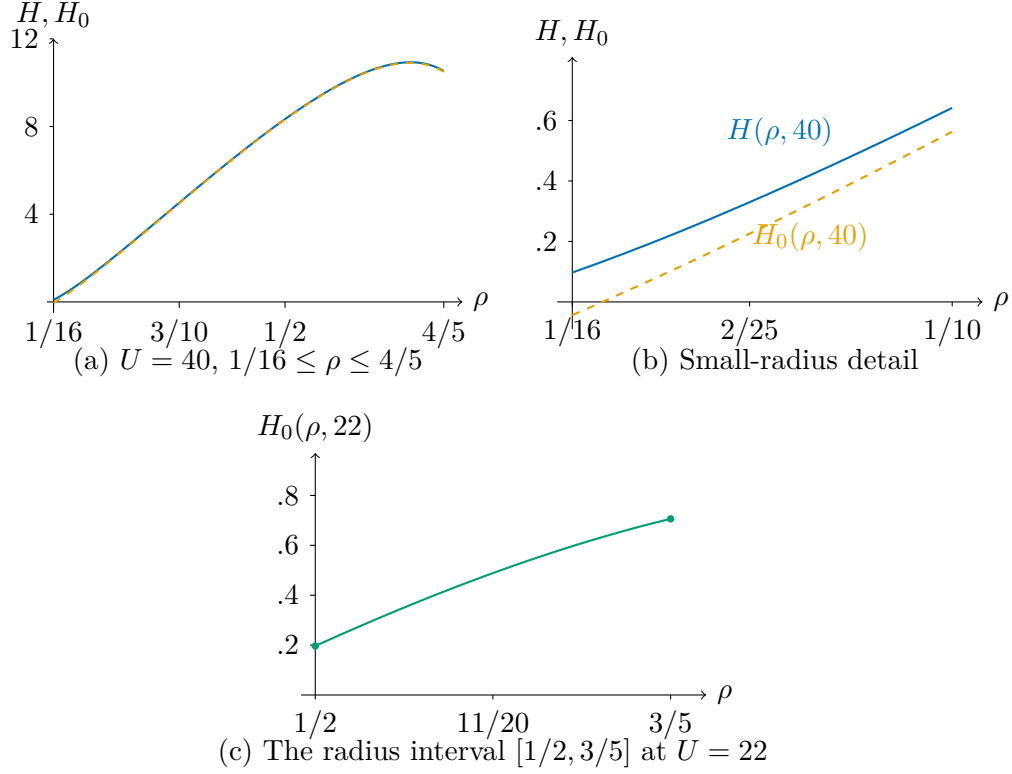
\captionof{figure}{The lower-excess functions in \eqref{eq:analytic-H} and
\eqref{eq:analytic-Hzero}.
Panels (a) and (b) use $U=40$; the second displays the positive
curvature correction used on
$[1/16,1/10]$: $H_0$ is insufficient at $\rho=1/16$, whereas
the proof gives $H(1/16,40)>1/20$.
Panel (c) shows the concave section $H_0(\rho,22)$ used on
$[1/2,3/5]$.
The figures illustrate the analytic positivity and monotonicity arguments.}
\label{fig:analytic-H}
\end{minipage}\end{center}

First consider $1/10\le\rho\le4/5$.
The function $H_0$ is increasing in $\rho$ on $[1/10,3/10]$
and concave on $[3/10,4/5]$.
For the first statement, $h(\rho)=\sqrt{\rho/(1-\rho)}$ has
$h'(\rho)<9/4$ on the first interval: its derivative decreases
up to $1/4$ and then increases, and both endpoint values are
smaller than $9/4$. Hence
\[
 (H_0)_\rho
 >b_0U-\left(\frac{27}{10}+\frac5{21}\right)\sqrt U
 >b_0U-3\sqrt U>0\qquad(U\ge40).
\]
We used $\sqrt{22/7}<9/5$, $\sqrt\rho>3/10$, and $b_0>4/5$.
The second-derivative argument in Lemma~\ref{lem:analytic23_28},
with $1/7$ replacing $1/6$ in the bracket, proves concavity.
Also $H_0/U$ increases in $U$.
It therefore suffices to test $\rho=1/10,4/5$ at $U=40$.
The bounds
\[
 \frac23\sqrt{880/63}<\frac52,\qquad
 \frac23\sqrt{3520/7}<15,\qquad
 \frac23\,40^{1/3}<\frac{23}{10},\qquad
 \frac{\sqrt{32}}7<\frac67
\]
give
\begin{align*}
 H_0(1/10,40)
 &>\frac{1654}{157}
       -\left(\frac52+\frac27+\frac{23}{10}+\frac{983}{200}\right)>0,\\
 H_0(4/5,40)
 &>\frac{5252}{157}
       -\left(15+\frac67+\frac{23}{10}+\frac{983}{200}\right)>0.
\end{align*}

It remains to consider $1/16\le\rho\le1/10$.
For fixed $U$, the strengthened curvature term has derivative
at most $\sqrt2$ times the derivative of the curvature term in $H_0$.
To verify this, write $c_U=23/(24U^2)$,
$x=c_U/[\rho^2(1-\rho)]$, and
\[
 h_{c_U}(\rho)=\sqrt{\frac{\rho}{1-\rho+c_U/\rho^2}}.
\]
Direct differentiation gives
\[
 \frac{h_{c_U}'(\rho)}{h'(\rho)}
       =\frac{1+3(1-\rho)x}{(1+x)^{3/2}}
       \le\frac{1+3x}{(1+x)^{3/2}}\le\sqrt2.
\]
The last function attains its maximum at $x=1$.
Also $h'$ decreases on this interval and
$h'(1/16)=2(16/15)^{3/2}<9/4$.
Using $\sqrt2<10/7$ gives
\[
 H_\rho>b_0U-\frac{29}{7}\sqrt U>0\qquad(U\ge40).
\]
Thus its radius minimum is at $\rho=1/16$.

The normalized function $H/U$ increases in $U\ge40$.
For its only nontrivial term, put
$C_\rho=23/[24\rho^2(1-\rho)]$; after division by $U$, the
curvature loss is a positive constant times
$\sqrt U/\sqrt{U^2+C_\rho}$, which decreases when $U^2>C_\rho$.
Throughout $[1/16,4/5]$,
$C_\rho\le3680/3<40^2$, using $\rho\ge1/16$ and $1-\rho\ge1/5$.
Every other normalized loss also decreases.
At $(\rho,U)=(1/16,40)$, one has $t=5/2$, $\delta=15$, and
\[
 \kappa=\frac1\pi\left(\frac38+\frac{23}{375}\right)
       =\frac{1309}{3000\pi}.
\]
Consequently
\[
 \frac2{3\sqrt\kappa}<\frac95,\qquad
 \frac{\sqrt{5/2}}7<\frac8{35},
\]
and
\[
 H(1/16,40)>
 \frac{5845}{628}
       -\left(\frac95+\frac8{35}+\frac{23}{10}+\frac{983}{200}\right)
       >\frac1{20}>0.
\]
All endpoints and all frequencies in the claimed middle range
are now covered by monotonicity and concavity.
\end{proof}

\section{Proof of Theorem~\ref{thm:allannuli}}\label{sec:allproof}

We now combine the estimates of Section~\ref{sec:uniformphase} with
variational and angular-layer comparisons. The small-hole estimate
and the thin-annulus theorem cover the ends of the radius interval;
the remaining arguments cover the intervening radii. Complete
multi-page tables are collected in Appendix~\ref{sec:analytic-table-data}.
The definitions, reduction to rational inequalities, and representative
calculations are included in the proofs below.

\subsection{Small holes: the 1/16 range}

The outer radius is normalized to one. Write
$N_\rho(\Lambda)=N_{A_{\rho,1}}(\Lambda)$ for the inclusive annular
count, and let $\mathcal N_B(X)$ denote the inclusive Neumann count
on the unit disk at energy $X^2$. Write $Y=\rho X$ for the frequency
on the removed disk.

\begin{lemma}[A disk upper staircase]\label{lem:sixteenth-diskupper}
The inclusive disk count satisfies
\[
 \begin{aligned}
 \mathcal N_B(Y)&=1&&(0\le Y\le9/5),&\quad \mathcal N_B(Y)&\le3&&(Y\le3),\\
 \mathcal N_B(Y)&\le6&&(Y\le4),&\quad \mathcal N_B(Y)&\le10&&(Y\le5).
\end{aligned}
\]
\end{lemma}
\begin{proof}
We use standard Bessel-zero identities to obtain the four counting
bounds. For real $\nu>0$, the Rayleigh sum in
\cite[(21)]{ChungLeePark} gives
\[
 \sum_{k\ge1}(j'_{\nu,k})^{-4}
 =\frac{\nu^2+8\nu+8}{16\nu^2(\nu+1)^2(\nu+2)}.
\]
The sum in that reference includes both positive and negative zeros;
the displayed positive-zero sum is half of it. We use this identity
only for integer orders $m\ge1$, and treat order zero separately.
The first summand is strictly smaller than the full sum, so
\[
 (j'_{1,1})^4>\frac{192}{17}>(9/5)^4>9,\qquad
 (j'_{2,1})^4>\frac{576}{7}>81,
\]
\[
 (j'_{3,1})^4>\frac{11520}{41}>256,\qquad
 (j'_{4,1})^4>\frac{4800}{7}>625.
\]
Thus the first roots in orders $1,2,3,4$ are respectively
greater than $9/5$, $3$, $4$, and $5$.
Also $j'_{m,1}>m$ for every $m\ge1$
\cite[(17)]{ChungLeePark}; thus orders $m\ge5$ are absent through
frequency five.

Consecutive positive derivative roots in a fixed order $m\ge1$
are more than $\pi$ apart. Indeed,
\cite[Theorem~7.2, (7.6), with $\alpha=1$]{LorchMuldoonSzego1972}
states that the differences between consecutive derivative zeros
strictly decrease once the zeros exceed the absolute value of the
order. Here this includes every positive root, since $j'_{m,1}>m$.
The McMahon expansion \cite[(10.21.20)]{NIST} gives their limit:
\[
 d_{m,k}:=j'_{m,k+1}-j'_{m,k}\searrow\pi.
\]
Strict decrease to this limit implies $d_{m,k}>\pi$ for every $k$.

In angular order zero the positive derivative roots are the
positive roots $j_{1,k}$ of $J_1$. Its ordinary Bessel product \cite[(10.21.15)]{NIST}
gives $\sum_{k\ge1}j_{1,k}^{-4}=1/192$; hence
$j_{1,1}>192^{1/4}>3$.
The same derivative-zero theorem applies in order zero to the
strictly positive roots of $J_0'=-J_1$, with the root at the origin
excluded. Their consecutive gaps are therefore also greater than
$\pi$.

Through $9/5$ there is only the constant mode. Through three,
only the constant and the first order-one pair can occur.
Through four there are at most the constant, one positive
zero-order mode, and one pair each in orders one and two.
Through five, order one can have at most two pairs: its third
root exceeds $9/5+2\pi>5$. The second roots in orders two
and three exceed $3+\pi>5$ and $4+\pi>5$; the second
zero-order positive root exceeds $3+\pi>5$.
All orders at least four are absent through five.
Thus the respective maximum counts are $1,3,6,10$.
All endpoints are safe because the root lower bounds are strict.
\end{proof}

\begin{lemma}[An explicit disk surplus from frequency twenty-seven]
\label{lem:sixteenth-disksurplus}
For $X\ge27$,
\[
 \mathcal N_B(X)>\frac{X^2}{4}
             +\frac{27}{250}X-\frac34+\frac{63}{880}.
\]
\end{lemma}
\begin{proof}
The adjustable remainder following \cite[Theorem~2.7]{GMWZ},
also used in Appendix~\ref{sec:diskanalyt}, states
\[
 \mathcal N_B(X)>\frac{X^2}{4}
 +(A\cos\sigma-C-1/\pi)X-\frac34
 +\frac{\cos\sigma}{4\pi}
\]
provided
\[
 X\ge\frac{\pi}{4(\sin\sigma-\sigma\cos\sigma)},
\]
where
\[
 A=\frac{24\pi-7\sqrt2}{8(4\pi-\sqrt2)},\qquad
 C=\frac{8\pi-3\sqrt2}{8(4\pi-\sqrt2)},\qquad A=C+\frac12.
\]
Choose $\sigma=9/20$. The identity
\[
 \sin\sigma-\sigma\cos\sigma=\int_0^\sigma s\sin s\,\mathrm ds
\]
and $\sin s\ge s-s^3/6$ give
\[
 \sin\sigma-\sigma\cos\sigma
 \ge\frac{\sigma^3}{3}-\frac{\sigma^5}{30}
 =\frac{2856951}{96000000}>\frac{297}{10000}.
\]
Since $\pi<22/7$, the admissible threshold is less than
$55000/2079<27$. The alternating cosine series gives
\[
 \cos(9/20)>1-\frac{\sigma^2}{2}
                   +\frac{\sigma^4}{24}-\frac{\sigma^6}{720}>
 \frac9{10}.
\]
Moreover $C<15/64$: after cross multiplication this is
equivalent to $\pi<9\sqrt2/4$, which follows from
$\pi<22/7$ and $\sqrt2>7/5$.
Using $A=C+1/2$ avoids applying an upper bound for $A$
as a lower bound. With $\pi>157/50$ we obtain
\begin{align*}
 A\cos\sigma-C-\frac1\pi
 &=\frac12\cos\sigma-C(1-\cos\sigma)-\frac1\pi\\
 &>\frac9{20}-\frac{15}{640}-\frac{50}{157}\\
 &=\frac{273}{640}-\frac{50}{157}
 =\frac{27}{250}+\frac{229}{2512000}
 >\frac{27}{250}.
\end{align*}
Finally $\cos\sigma/(4\pi)>63/880$. Substitution proves
the stated disk surplus.
\end{proof}

\begin{theorem}[All frequencies for holes of radius at most $1/16$]
\label{thm:sixteenth-smallhole}
For every $0<\rho\le1/16$ and every $X\ge0$,
\[
 N_\rho(X^2)>(1-\rho^2)X^2/4.
\]
\end{theorem}
\begin{proof}
When $X^2\le32$, Proposition~\ref{prop:analyticfirsteight} applies, since $(1-\rho^2)X^2\le32$.
When $X\ge80$, apply Theorem~\ref{thm:outereighty}.

It remains to consider $\sqrt{32}\le X<80$. Then
$Y=\rho X<5$. Neumann splitting of the unit disk along the
inner circle gives
\[
 N_\rho(X^2)\ge\mathcal N_B(X)-\mathcal N_B(Y).
\]
If $Y\le\sqrt3$, Lemma~\ref{lem:sixteenth-diskupper} gives
$\mathcal N_B(Y)=1$. Lemma~\ref{lem:diskanalyt_surplus} gives
$\mathcal N_B(X)-1\ge X^2/4$. Consequently,
\[
 N_\rho(X^2)\ge X^2/4>(X^2-Y^2)/4.
\]

If $\sqrt3<Y\le3$, then $X\ge16Y>16\sqrt3>27$,
and the upper staircase gives $\mathcal N_B(Y)\le3$.
The stronger disk surplus and $Y^2/4>3/4$ yield
\[
 N_\rho(X^2)-\frac{X^2-Y^2}{4}
 >\frac{27}{250}X-3+\frac{63}{880}.
\]
Since $\sqrt3>173/100$, the right side is greater than
\[
 \frac{27}{250}\frac{692}{25}-3+\frac{63}{880}
 =\frac{33567}{550000}>0.
\]

If $3<Y\le4$, then $X>48$, $Y^2/4>9/4$, and
$\mathcal N_B(Y)\le6$. Thus
\[
 N_\rho(X^2)-\frac{X^2-Y^2}{4}
 >\frac{27}{250}X-\frac92+\frac{63}{880}
 >\frac{171}{250}>0.
\]
If $4<Y<5$, then $X>64$, $Y^2/4>4$, and
$\mathcal N_B(Y)\le10$. Hence
\[
 N_\rho(X^2)-\frac{X^2-Y^2}{4}
 >\frac{27}{250}X-\frac{27}{4}+\frac{63}{880}
 >\frac{81}{500}>0.
\]
These cases exhaust the remaining range and prove the theorem.
\end{proof}

\subsection{Small holes: extension to 1/10}

We normalize the outer radius to one, put $Y=\rho X$, and include
the zero eigenvalue in all counting functions. The argument uses the disk phase comparison in Appendix~\ref{sec:diskanalyt}.

\begin{lemma}[Two explicit disk margins]\label{lem:tenth-diskmargins}
The inclusive Neumann disk count satisfies
\[
 \mathcal N_B(X)-X^2/4\ge9/4\quad(X\ge18),\qquad
 \mathcal N_B(X)-X^2/4\ge15/4\quad(X\ge30).
\]
\end{lemma}
\begin{proof}
First recall the disk remainder obtained from
\cite[Theorem~2.7]{GMWZ}:
\[
 \mathcal N_B(X)-X^2/4\ge
 (A+\tfrac1{2(2\pi X-1)})M_0-CX
 -\lfloor X/\pi+3/4\rfloor,
\]
where
\[
 A=\frac{24\pi-7\sqrt2}{8(4\pi-\sqrt2)},\quad C=A-\tfrac12,
 \quad M_0=\lfloor G_X^{-1}(1/4)\rfloor+1.
\]
The positive reciprocal term discarded below makes each subsequent
estimate strict.
For $0<w\le X$, $\arccos(1-z)>\sqrt{2z}$ gives
\[
 G_X(X-w)>\frac{2\sqrt2}{3\pi\sqrt X}w^{3/2}.
\]
Thus $G_X^{-1}(1/4)>X-c_wX^{1/3}$, with
$c_w=(9\pi^2/128)^{1/3}$. The bounds
\[
 C<15/64,\quad A<47/64,\quad
 \tfrac12-\tfrac1\pi>9/50,\quad c_w<887/1000
\]
follow from $157/50<\pi<22/7$ and $\sqrt2>7/5$.
Discarding the positive reciprocal term gives
\[
 \mathcal N_B(X)-X^2/4>F(X):=
 \frac9{50}X-\frac{41689}{64000}X^{1/3}-\frac34.
\]
The function $F$ increases for $X\ge28$: $X^{2/3}>9$ there,
so $F'(X)>9/50-41689/1728000>0$.

For $24\le X\le28$ put $c=11/4$ and
$j=\lfloor X-c\rfloor+1\in\{22,\ldots,26\}$.
Since
\[
 c_w^3X<\frac{1089}{1568}28<(11/4)^3,
\]
we have $G_X^{-1}(1/4)>X-c\ge j-1$, hence $M_0\ge j$.
Also $X<j+c$, so the source remainder is greater than
\[
 \frac j2-\frac{15c}{64}
 -\left\lfloor\frac{50(j+c)}{157}+\frac34\right\rfloor.
\]
The last floor for $j=22,23,24,25$ equals $8,8,9,9$,
respectively, and is at most $j/2-3$. For $j=26$ the same
bound follows directly from
$50(j+11/4)/157+3/4\le j/2-3$.
The remainder exceeds $3-165/256=603/256>9/4$.
For $X\ge28$, $28^{1/3}<31/10$ gives
\[
 F(28)-9/4>
 \frac{126}{25}-3-\frac{1292359}{640000}
 =\frac{13241}{640000}>0.
\]
Thus the first margin is established for $X\ge24$.

For $35\le X\le38$, take $c=3$ and
$j=\lfloor X-3\rfloor+1\in\{33,34,35,36\}$.
Here $c_w^3X<1089\cdot38/1568<27$, so again $M_0\ge j$.
The four floors $\lfloor50(j+3)/157+3/4\rfloor$ are
$12,12,12,13$. Consequently the remainder exceeds
$9/2-45/64=243/64>15/4$.
For $X\ge38$, monotonicity of $F$, $38^{1/3}<27/8$, and
\[
 F(38)-15/4>
 \frac{171}{25}-\frac92-\frac{1125603}{512000}>0
\]
establish the second margin.

It remains to bridge $18\le X\le24$ and $30\le X\le35$.
Write $H(t)=\sqrt{1-t^2}+t\arcsin t$.
As in Appendix~\ref{sec:diskanalyt}, $H'(t)=\arcsin t\ge t+t^3/6$ implies
$H(t)\ge1+t^2/2+t^4/24$. Therefore, with
\[
 P(k,m)=384k^2+48km^2+m^4,
\]
the disk action obeys
\[
 G_{2\sqrt k}(m)+\frac34>
 \frac{7P(k,m)}{4224k\sqrt k}-\frac m2+\frac34.
\]
Choose a rational upper bound $a/b\ge\sqrt k$ from the following table. The comparison $G_{2\sqrt k}(m)+3/4>n$ follows from
the single integer inequality
\[
 D(k,m,n;a,b):=7bP(k,m)-1056ka(4n+2m-3)>0.
\]
All $4n+2m-3$ used below are positive. The rational square-root
bounds are checked by $a^2-b^2k\ge0$.
The vector $(M_1,\ldots,M_s)$ retains angular orders through
$M_n$ at floor level $n$. Since $G_X$ decreases in its angular
argument, the cited disk phase comparison supplies
$c=\sum_{n=1}^s(2M_n+1)$ modes at $X=2\sqrt k$.

\begin{center}\small
\begin{tabular}{r|r|l|r}
$k$&$a/b$&$(M_1,\ldots,M_s)$&$c$\\\hline
81&900/100&15,11,8,5,3,0&90\\
87&933/100&16,11,8,5,3,1&94\\
91&954/100&16,12,9,6,3,1&100\\
97&985/100&16,12,9,6,4,2,0&105\\
102&1010/100&17,13,9,7,4,2,0&111\\
108&1040/100&17,13,10,7,5,2,0&115\\
112&1059/100&18,14,10,7,5,3,0&121\\
118&1087/100&18,14,11,8,5,3,1&127\\
124&1114/100&19,15,11,8,6,3,1&133\\
130&1141/100&19,15,12,9,6,4,2,0&142\\
139&1179/100&20,16,12,9,7,4,2,0&148\\
225&1500/100&26,21,18,15,12,9,7,4,2,0&238\\
234&1530/100&27,22,18,15,12,10,7,5,3,0&248\\
244&15621/1000&27,23,19,16,13,10,8,5,3,1&260\\
256&1600/100&28,23,19,16,13,11,8,6,4,1&268\\
264&1625/100&28,24,20,17,14,11,8,6,4,2,0&279\\
275&16584/1000&29,24,20,17,14,12,9,7,4,2,0&287\\
283&1683/100&29,25,21,18,15,12,9,7,5,2,0&297\\
293&1712/100&30,25,21,18,15,12,10,7,5,3,1&305\\
301&1735/100&30,26,22,18,15,13,10,8,5,3,1&313\\
\end{tabular}
\end{center}

Table~\ref{tab:smallhole-actions} in Appendix~\ref{sec:analytic-table-data} lists the positive values $D(k,M_n,n;a,b)$ in increasing order of $n$, together with $a^2-b^2k$.

For the first eleven rows, the next $k$ equals $c-3$.
Each row covers $4k\le X^2\le4(c-9/4)$, so these closed
intervals overlap; the first begins at $18^2$ and the last
ends beyond $24^2$. They prove the first margin in the
remaining range. For the last nine rows, the next $k$ equals
$c-4$. The corresponding intervals
$4k\le X^2\le4(c-15/4)$ overlap, begin at $30^2$, and end
beyond $35^2$. They prove the second margin. This finishes
both disk comparisons using only the displayed integer inequalities.
\end{proof}

We use the disk upper staircase in Lemma~\ref{lem:sixteenth-diskupper}.

\begin{proposition}[All frequencies for holes of radius at most $1/10$]\label{prop:analytic10}
For every $0<\rho\le1/10$ and every $X\ge0$,
\[
 N_\rho(X^2)>(1-\rho^2)X^2/4
 \qquad(0<\rho\le1/10,\ X\ge0).
\]
\end{proposition}
\begin{proof}
Theorem~\ref{thm:sixteenth-smallhole} covers $\rho\le1/16$.
Lemma~\ref{lem:analytic40} covers $X\ge40$ in the remaining radius range. For $X^2\le32$ use
Proposition~\ref{prop:analyticfirsteight}. In the remaining
range $\sqrt{32}\le X<40$ one has $Y=\rho X<4$.
Neumann splitting of the unit disk along the inner circle gives
\[
 N_\rho(X^2)\ge\mathcal N_B(X)-\mathcal N_B(Y).
\]
If $Y\le9/5$, Lemma~\ref{lem:sixteenth-diskupper} gives
$\mathcal N_B(Y)=1$, and Lemma~\ref{lem:diskanalyt_surplus} gives
$\mathcal N_B(X)-1\ge X^2/4$. These estimates imply the strict
annular inequality, since $Y>0$.
If $9/5<Y\le3$, then $X>18$ and the first disk margin gives
\[
 N_\rho(X^2)-\frac{X^2-Y^2}{4}
 \ge\frac94-3+\frac{Y^2}{4}
 >\frac94-3+\frac{81}{100}=\frac3{50}>0.
\]
If $3<Y<4$, then $X>30$, and the second margin gives
\[
 N_\rho(X^2)-\frac{X^2-Y^2}{4}
 \ge\frac{15}{4}-6+\frac{Y^2}{4}>0.
\]
The three cases cover the whole remaining region.
\end{proof}

\subsection{Analytic low energies}

We next control the low-energy range uniformly in the radius ratio.
We begin with bounds for the first eight eigenvalues. Disk bracketing,
polynomial radial comparisons and angular ground modes then extend
the count to the starting energy of the discrete bridges.

\begin{proposition}[An analytic low-frequency bound]
\label{prop:analyticlowfrequency}
Normalize the outer radius to one, put $D=1-\rho^2$, and set
$Q=DX^2$. For every $0<\rho<1$ and $0\le Q\le100$,
\[
 N_\rho(X^2)>\frac Q4.
\]

\end{proposition}
\begin{proof}
Proposition~\ref{prop:analyticfirsteight} already proves the assertion
for $Q\le32$. It remains to consider $32\le Q\le100$.
Write $\lambda_{m,n}(\rho)$ for the radial eigenvalue in angular
order $m\ge0$, with radial indices $n=0,1,\ldots$.
An angular order has multiplicity one for $m=0$ and two for $m\ge1$.

\paragraph{Small inner radii.}
Lemma~\ref{lem:sixteenth-diskupper} gives
$\mathcal N_B(Y)=1$ for $0\le Y\le9/5$.

If $\rho\le7/40$ and $Q\le100$, then
\[
 (\rho X)^2=\frac{\rho^2Q}{1-\rho^2}
 \le\frac{4900}{1551}<\frac{81}{25}.
\]
For $Q\ge32$ one has $X\ge\sqrt{32}$. Neumann bracketing at
the inner circle, followed by the disk surplus
\eqref{eq:diskanalyt_surplus}, therefore gives
\[
 N_\rho(X^2)\ge\mathcal N_B(X)-\mathcal N_B(\rho X)
 =\mathcal N_B(X)-1\ge\frac{X^2}{4}>\frac Q4.
\]

\paragraph{Polynomial comparisons for the middle radius range.}
Suppose $7/40\le\rho\le1/2$ and put $z=\rho^2$.
For $m\ge0$, use the $d$ radial monomials $s^{m+2i}$,
$0\le i\le d-1$. Their mass and energy matrices are
\[
 M_{ij}=\int_\rho^1s^{2m+2i+2j+1}\,\mathrm ds,
 \qquad
 K_{ij}=\bigl((m+2i)(m+2j)+m^2\bigr)
         \int_\rho^1s^{2m+2i+2j-1}\,\mathrm ds,
\]
with $K_{00}=0$ when $m=0$. The mass matrix is positive definite.
These functions belong to the Neumann form domain; derivative
conditions at the endpoints are not required for trial functions.
Writing $S_j(z)=\sum_{p=0}^jz^p$, the matrix of $DK-QM$, multiplied
by $2/D>0$, is the explicit polynomial matrix
\[
 H^{(m,Q,d)}_{ij}(z)=
 \frac{(m+2i)(m+2j)+m^2}{m+i+j}
       (1-z)S_{m+i+j-1}(z)
 -\frac{Q}{m+i+j+1}S_{m+i+j}(z).
\]
Its first term is defined to be zero when $m+i+j=0$.
By the min--max principle, $c$ negative directions of this matrix
give $c$ radial eigenvalues strictly below $Q/D$.

Direct determinant expansion gives the following five identities:
\[
\begin{array}{c|c|c}
 m&Q&\det H^{(m,Q,3)}(z)\\\hline
 0&36&\dfrac45(1-z)^4P_0(z)\\
 1&36&\dfrac1{75}(1-z)^4P_1(z)\\
 2&48&\dfrac{16}{2625}(1-z)^4P_2(z)\\
 3&70&\dfrac1{7560}(1-z)^4P_3(z)\\
 4&96&\dfrac{32}{15435}(1-z)^4P_4(z)
\end{array}
\]
Here
\[
\begin{aligned}
 P_0(z)&=41-106z-175z^2,\\
 P_1(z)&=73+1613z-z^2B_1(z),\\
 P_2(z)&=2+264z+4781z^2-z^3B_2(z),\\
 P_3(z)&=29+6101z+37250z^2+351408z^3-z^4B_3(z),\\
 P_4(z)&=1+224z+1386z^2+4928z^3+31710z^4-z^5B_4(z),
\end{aligned}
\]
where
\[
\begin{aligned}
 B_1(z)={}&4462+5642z+4787z^2+1195z^3,\\
 B_2(z)={}&6852+12680z+11932z^2+9269z^3+5336z^4+978z^5,\\
 B_3(z)={}&203943+622449z+665901z^2+580827z^3\\
          &+518568z^4+458650z^5+140851z^6+25999z^7,\\
 B_4(z)={}&4088+40159z+47560z^2+43393z^3+38528z^4\\
          &+40110z^5+46312z^6+18354z^7+5656z^8+1049z^9.
\end{aligned}
\]
All five $P_j$ are positive on $0\le z\le1/4$.
Indeed, $P_0$ decreases and $P_0(1/4)=57/16>0$.
The other tails have positive coefficients and satisfy
\[
 B_1(1/4)=\frac{396183}{64}<6200,\qquad
 B_2(1/4)=\frac{5598401}{512}<11000,
\]
\[
 B_3(1/4)=\frac{6762643707}{16384}<420000,\qquad
 B_4(1/4)=\frac{4713467801}{262144}<18000.
\]
Using $z^{j+1}B_j(z)\le z^jB_j(1/4)/4$ gives
\[
\begin{aligned}
 P_1(z)&\ge73+63z>0,\\
 P_2(z)&\ge2+264z+2031z^2>0,\\
 P_3(z)&\ge29+6101z+37250z^2+246408z^3>0,\\
 P_4(z)&\ge1+224z+1386z^2+4928z^3+27210z^4>0.
\end{aligned}
\]
Also $H_{00}<0$ in every case: for $m>0$, its first term is
$2m(1-z^m)\le2m$, whereas its subtracted term is at least
$Q/(m+1)>2m$; for $m=0$, it is $-Q$.
A nonsingular symmetric $3\times3$ matrix with positive determinant
has zero or two negative eigenvalues, and $H_{00}<0$ excludes zero.
Consequently
\[
 D\lambda_{0,1}<36,\quad D\lambda_{1,1}<36,\quad
 D\lambda_{2,1}<48,\quad D\lambda_{3,1}<70,\quad
 D\lambda_{4,1}<96\qquad(0<\rho\le1/2).
\]

One more order-zero radial level is available for $\rho\le1/4$.
The corresponding four-dimensional identity is
\[
 \det H^{(0,74,4)}(z)
 =-\frac{37}{378000}(1-z)^9P_5(z),\qquad
 P_5(z)=51827-861561z+930201z^2+685933z^3.
\]
On $0\le z\le1/16$,
\[
 P_5'(z)\le-861561+\frac{1860402}{16}
                   +\frac{2057799}{256}
 =-\frac{188735385}{256}<0,
 \qquad P_5(1/16)=\frac{7292925}{4096}>0.
\]
Thus its determinant is negative. The preceding order-zero
three-dimensional comparison at $Q=36$ supplies two negative
directions, which persist at $Q=74$ because the mass form is
positive definite. A nonsingular symmetric $4\times4$ matrix
with negative determinant has an odd number of negative eigenvalues;
there are therefore three. We have proved
\[
 D\lambda_{0,2}<74\qquad(0<\rho\le1/4).
\]

For the angular ground levels, Lemma~\ref{lem:angularfirstmonotone}
and its disk-restriction argument bound $\lambda_{m,0}(\rho)$
strictly by the disk quotient of any regular trial function.
Use
\[
 f_{m,p}(s)=s^m\left(1-\frac{m}{m+2p}s^2\right)^p,
 \qquad c_j=\binom pj\left(-\frac{m}{m+2p}\right)^j.
\]
Their exact disk mass and energy are
\[
 M=\sum_{i,j=0}^p\frac{c_ic_j}{2m+2i+2j+2},\qquad
 K=\sum_{i,j=0}^p
 \frac{((m+2i)(m+2j)+m^2)c_ic_j}{2m+2i+2j}.
\]
The needed values are
\[
\begin{array}{c|c|c|c|c}
 m&p&M&K&K/M\\\hline
 4&1&13/630&16/27&1120/39\\
 5&1&31/2352&190/343&9120/217\\
 6&1&1/112&33/64&231/4\\
 7&2&26899/10541520&1371944/7247295&21951104/295889\\
 8&2&191/115830&2072/13365&53872/573\\
 8&3&1344991/1514142630&48836664/588833245&879059952/9414937
\end{array}
\]
The last three quotients are less than $371/5$, $941/10$, and
$187/2$, respectively. On $7/40\le\rho\le1/2$,
$D\le1551/1600$. Hence the ground modes through $m=4,5,6,7,8$
lie strictly below the normalized energies $36,44,56,72,92$,
respectively. For the last three the sufficient comparisons are
\[
 \frac{1551}{1600}\frac{231}{4}=56-\frac{119}{6400}<56,
 \qquad
 \frac{1551}{1600}\frac{371}{5}=72-\frac{579}{8000}<72,
\]
\[
 \frac{1551}{1600}\frac{941}{10}=92-\frac{12509}{16000}<92.
\]
If $\rho\ge1/4$, the last trial in the table gives instead
\[
 D\lambda_{8,0}<\frac{15}{16}\frac{187}{2}
 =88-\frac{11}{32}<88.
\]
The ground eigenvalues increase with the angular potential $m^2/s^2$,
so a bound for order $m$ includes every smaller order.

The following table now lists modes strictly below $4k/D$.
Ground modes through order $g$ contribute $1+2g$ modes;
first positive radial levels through order $h$ contribute
$1+2h$ additional modes. The order-zero second positive radial
level contributes one more when listed.
\[
\begin{array}{c|c|c|c|c}
 k&\text{radius condition}&g&\text{additional radial modes}&\text{count}\\\hline
 9\text{--}10&7/40\le\rho\le1/2&4&n=1,\ 0\le m\le1&12\\
 11\text{--}13&7/40\le\rho\le1/2&5&n=1,\ 0\le m\le1&14\\
 14\text{--}17&7/40\le\rho\le1/2&6&n=1,\ 0\le m\le2&18\\
 18\text{--}21&7/40\le\rho\le1/2&7&n=1,\ 0\le m\le3&22\\
 22&7/40\le\rho\le1/4&7&n=1,\ 0\le m\le3;\ (m,n)=(0,2)&23\\
 22&1/4\le\rho\le1/2&8&n=1,\ 0\le m\le3&24\\
 23&7/40\le\rho\le1/2&8&n=1,\ 0\le m\le3&24\\
 24\text{--}25&7/40\le\rho\le1/2&8&n=1,\ 0\le m\le4&26
\end{array}
\]
For example, the first row has $g=4$ and contains precisely the
following guaranteed modes:
\[
 \begin{array}{c|c|c}
 \text{radial index}&\text{angular orders}&\text{multiplicity}\\\hline
 n=0&m=0,1,2,3,4&1+2+2+2+2=9\\
 n=1&m=0,1&1+2=3
 \end{array}
\]
Here $n=0,m=0$ is the constant eigenfunction. Each positive angular
order contributes its cosine and sine factors; these are independent
eigenfunctions. The $n=1$ entries are different radial eigenvalues
from the $n=0$ entries in the same angular order. The preceding
estimates put all twelve modes strictly below $36/D$, and therefore
below $4k/D$ for both $k=9$ and $k=10$. Thus the last column is
$9+3=12$, with no double counting.

Each count is at least $k+1$ throughout the stated index group.
Together with Proposition~\ref{prop:analyticfirsteight}, this proves
$\nu_k<4k/D$ for $1\le k\le25$ on $7/40\le\rho\le1/2$.

\paragraph{Two radial rows for larger inner radii.}
A radially constant trial function gives
\[
 D\lambda_{m,0}\le2m^2\log(1/\rho).
\]
Lemma~\ref{lem:radial}, which is valid at arbitrary relative width,
and $\pi^2<10$ give
\[
 \lambda_{m,1}\le\frac{\pi^2}{(1-\rho)^2}
                     +\frac{m^2+3/4}{\rho^2}.
\]
Put
\[
 B_Q(\rho)=\frac{Q\rho^2}{1-\rho^2}
           -\frac{10\rho^2}{(1-\rho)^2}-\frac34.
\]
Then $m^2\le B_Q(\rho)$ implies $D\lambda_{m,1}<Q$.
Moreover,
\[
 B_Q'(\rho)=\frac{2\rho}{(1-\rho^2)^2}
 \left(Q-10\frac{(1+\rho)^2}{1-\rho}\right).
\]
The quotient in parentheses increases strictly, so $B_Q$ has
at most one interior maximum and its minimum on any closed
radius interval is attained at an endpoint.

Use $I_1=[1/2,3/5]$, $I_2=[3/5,7/10]$, and $I_3=[7/10,3/4]$.
The required logarithmic bounds at their left endpoints are
\[
 2\log2<\gamma_1=\frac{25}{18},\qquad
 2\log(5/3)<\gamma_2=\frac{46}{45},\qquad
 2\log(10/7)<\gamma_3=\frac{51}{70}.
\]
For the first two, with $x=(z-1)/(z+1)$,
\[
 2\log z=4\sum_{j\ge0}\frac{x^{2j+1}}{2j+1}
 <4\left(x+\frac{x^3}{3(1-x^2)}\right),
\]
and substitution of $x=1/3,1/4$ gives the fractions above.
The third follows from $2\log z<z-z^{-1}$ for $z>1$;
the difference has derivative $(z-1)^2/z^2$ and vanishes at one.
At the four radius endpoints,
\[
\begin{aligned}
 B_Q(1/2)&=Q/3-43/4,& B_Q(3/5)&=9Q/16-93/4,\\
 B_Q(7/10)&=49Q/51-1987/36,& B_Q(3/4)&=9Q/7-363/4.
\end{aligned}
\]
For $\rho\in I_j$ and $Q=4k$, keep ground modes through $g$
and, when listed, first positive radial modes through $h$.
The following margins are exactly
$4k_- -\gamma_jg^2$ and
$\min\{B_{4k_-}(\inf I_j),B_{4k_-}(\sup I_j)\}-h^2$.
They are positive and increase with $k$.
\[
\begin{array}{c|c|r|r|r|c|c}
 j&[k_-,k_+]&g&h&\text{count}&\text{ground margin}&\text{higher-row margin}\\\hline
 1&[9,10]&5&--&11&23/18&--\\
 1&[11,12]&5&1&14&167/18&1/2\\
 1&[13,14]&6&2&18&2&2\\
 1&[15,17]&6&3&20&10&1/4\\
 1&[18,20]&7&3&22&71/18&17/4\\
 1&[21,22]&7&4&24&287/18&5/4\\
 1&[23,25]&8&4&26&28/9&47/12\\\hline
 2&[9,10]&5&--&11&94/9&--\\
 2&[11,12]&6&--&13&36/5&--\\
 2&[13,14]&7&--&15&86/45&--\\
 2&[15,16]&7&1&18&446/45&889/612\\
 2&[17,20]&8&3&24&116/45&41/36\\
 2&[21,25]&9&4&28&6/5&8\\\hline
 3&[9,13]&7&--&15&3/10&--\\
 3&[14,16]&8&--&17&328/35&--\\
 3&[17,18]&9&--&19&629/70&--\\
 3&[19,25]&10&2&26&22/7&83/28
\end{array}
\]
The count is $1+2g$ when the higher row is omitted and
$2+2g+2h$ otherwise; it is always at least $k_++1$.
Thus $\nu_k<4k/D$ for $9\le k\le25$ on $1/2\le\rho\le3/4$.

For $3/4\le\rho<1$, the ground row alone suffices throughout
$32\le Q\le100$. Indeed, $2\log(1/\rho)<7/12$, whence
\[
 N_\rho(X^2)\ge1+2\left\lfloor
 \sqrt{\frac{Q}{2\log(1/\rho)}}\right\rfloor
 >2\sqrt{\frac{12Q}{7}}-1.
\]
After subtracting $Q/4$, the right side is concave in $Q$.
It is positive at $Q=32$ and $Q=100$, as
$384/7>(9/2)^2$ and $1200/7>13^2$.

Finally, on $7/40\le\rho\le3/4$, take
$k=\lfloor Q/4\rfloor\le25$. The eigenvalue bounds above and
Proposition~\ref{prop:analyticfirsteight} give
$N_\rho(X^2)\ge k+1>Q/4$, including $Q=100$.
Together with the small-radius and largest-radius arguments,
this proves the assertion for every $0<\rho<1$.
\end{proof}
\subsection{Fixed-area discrete bridges}
The remaining parameter ranges require individual angular counts.
We keep $Q=(1-\rho^2)U^2$ fixed because the comparison term $Q/4$
then stays fixed while the annular shape varies. With $t=\rho U$,
the action $G_U(m)-G_t(m)$ measures the available radial phase in
angular order $m$. The next lemma reduces its lower bound on an
entire radius interval to two endpoint values. This reduction
retains the integer radial layers that a summed phase estimate
would lose.

\begin{lemma}[Fixed-area endpoint principle]\label{lem:analyticareaendpoints}
For fixed $Q>0$ and integer $m\ge0$, set
\[
 U=\sqrt{t^2+Q},\qquad F_m(t)=G_U(m)-G_t(m).
\]
On any closed interval of $t\ge0$, the minimum of $F_m$ is attained
at an endpoint. Equivalently, at fixed $Q=(1-\rho^2)U^2$, the
minimum over a radius-ratio interval is attained at one of its two
endpoints. At each fixed radius ratio $F_m$ is nondecreasing in $Q$.
\end{lemma}
\begin{proof}
For $m=0$, $F_0=(\sqrt{t^2+Q}-t)/\pi$ is decreasing.
For $m>0$ and $t<m$, $G_t(m)=0$ and $G_U(m)$ is nondecreasing.
For $t>m$, differentiation gives
\[
 \pi F_m'(t)=\frac tU\sqrt{1-\frac{m^2}{U^2}}
                   -\sqrt{1-\frac{m^2}{t^2}}.
\]
Both terms on the right are nonnegative. Their squared difference is
\[
 \frac Q{U^2}
       \left(m^2\left(\frac1{t^2}+\frac1{U^2}\right)-1\right).
\]
The expression in parentheses strictly decreases, is positive at
$t=m$, and tends to $-1$. Hence $F_m'$ changes sign at most once,
from positive to negative. There are no interior strict minima.
Also $t=\rho\sqrt{Q/(1-\rho^2)}$ increases with $\rho$.
For fixed $\rho$, differentiate with respect to $U$. In the core
the derivative is
\[
 \frac1\pi\left(\sqrt{1-\frac{m^2}{U^2}}
              -\rho\sqrt{1-\frac{m^2}{\rho^2U^2}}\right)>0;
\]
in an outer branch only the first nonnegative term remains.
The continuity at turning points completes the proof.
\end{proof}

\begin{figure}[htbp]
\centering
\begin{tikzpicture}[x=70cm,y=11cm,font=\small]
 \draw[->] (.497,0)--(.605,0) node[right] {$\rho$};
 \draw[->] (.5,-.008)--(.5,.44) node[above] {$F_8(168,\rho)$};
 \foreach \xx/\lab in {.5/{1/2},.525/{},.55/{11/20},.575/{},.6/{3/5}}
 {\draw (\xx,0)--(\xx,-.005) node[below] {$\lab$};}
 \foreach \yy/\lab in {0/{1.2},.1/{1.3},.2/{1.4},.3/{1.5},.4/{1.6}}
 {\draw (.5,\yy)--(.499,\yy) node[left] {$\lab$};}
 \draw[densely dashed,gray] (.5,.06)--(.6,.06)
 node[right] {$1+c=63/50$};
 \draw[blue!75!black,thick] (0.5000000,0.2624051) -- (0.5012500,0.2657694) -- (0.5025000,0.2691516) -- (0.5037500,0.2725520) -- (0.5050000,0.2759706) -- (0.5062500,0.2794076) -- (0.5075000,0.2828630) -- (0.5087500,0.2863370) -- (0.5100000,0.2898297) -- (0.5112500,0.2933412) -- (0.5125000,0.2968717) -- (0.5137500,0.3004212) -- (0.5150000,0.3039899) -- (0.5162500,0.3075779) -- (0.5175000,0.3111854) -- (0.5187500,0.3148124) -- (0.5200000,0.3184591) -- (0.5212500,0.3221256) -- (0.5225000,0.3258121) -- (0.5237500,0.3295187) -- (0.5250000,0.3332454) -- (0.5262500,0.3366568) -- (0.5275000,0.3396488) -- (0.5287500,0.3424034) -- (0.5300000,0.3449733) -- (0.5312500,0.3473878) -- (0.5325000,0.3496661) -- (0.5337500,0.3518222) -- (0.5350000,0.3538668) -- (0.5362500,0.3558084) -- (0.5375000,0.3576540) -- (0.5387500,0.3594096) -- (0.5400000,0.3610803) -- (0.5412500,0.3626704) -- (0.5425000,0.3641839) -- (0.5437500,0.3656241) -- (0.5450000,0.3669941) -- (0.5462500,0.3682967) -- (0.5475000,0.3695345) -- (0.5487500,0.3707097) -- (0.5500000,0.3718244) -- (0.5512500,0.3728807) -- (0.5525000,0.3738803) -- (0.5537500,0.3748248) -- (0.5550000,0.3757160) -- (0.5562500,0.3765551) -- (0.5575000,0.3773436) -- (0.5587500,0.3780828) -- (0.5600000,0.3787739) -- (0.5612500,0.3794180) -- (0.5625000,0.3800162) -- (0.5637500,0.3805696) -- (0.5650000,0.3810790) -- (0.5662500,0.3815456) -- (0.5675000,0.3819700) -- (0.5687500,0.3823532) -- (0.5700000,0.3826960) -- (0.5712500,0.3829992) -- (0.5725000,0.3832634) -- (0.5737500,0.3834894) -- (0.5750000,0.3836778) -- (0.5762500,0.3838294) -- (0.5775000,0.3839447) -- (0.5787500,0.3840243) -- (0.5800000,0.3840688) -- (0.5812500,0.3840787) -- (0.5825000,0.3840547) -- (0.5837500,0.3839971) -- (0.5850000,0.3839066) -- (0.5862500,0.3837836) -- (0.5875000,0.3836285) -- (0.5887500,0.3834418) -- (0.5900000,0.3832239) -- (0.5912500,0.3829754) -- (0.5925000,0.3826964) -- (0.5937500,0.3823876) -- (0.5950000,0.3820492) -- (0.5962500,0.3816816) -- (0.5975000,0.3812852) -- (0.5987500,0.3808604) -- (0.6000000,0.3804074);
 \fill[blue!75!black] (.5,0.2624051) circle (1.8pt);
 \fill[blue!75!black] (.6,0.3804074) circle (1.8pt);
 \draw[densely dotted] (.5809883,0)--(.5809883,.3841000);
 \node[above] at (.563,.384) {single interior maximum};
 \fill[orange!80!black] (.5,.1666697) circle (1.7pt);
 \draw[orange!80!black,->] (.5,.171)--(.5,.255);
 \node[orange!80!black,right] at (.501,.16) {$T_4(L_1,8)$};
 \fill[orange!80!black] (.6,.2575617) circle (1.7pt);
 \draw[orange!80!black,->] (.6,.263)--(.6,.373);
 \node[orange!80!black,left] at (.599,.245) {$P_2(3/5,L_2,8)$};
\end{tikzpicture}
\caption{The fixed-area action for the second layer in the
$Q_0=168$ row of Proposition~\ref{prop:analytichalfbridge}.
Its minimum on $[1/2,3/5]$ is attained at an endpoint, although
the function need not be monotone. The two orange points are the
rational lower bounds used for the endpoint actions; both exceed
$1+c$. The exact comparisons are given in the worked row.}
\label{fig:analytic-fixedarea-action}
\end{figure}
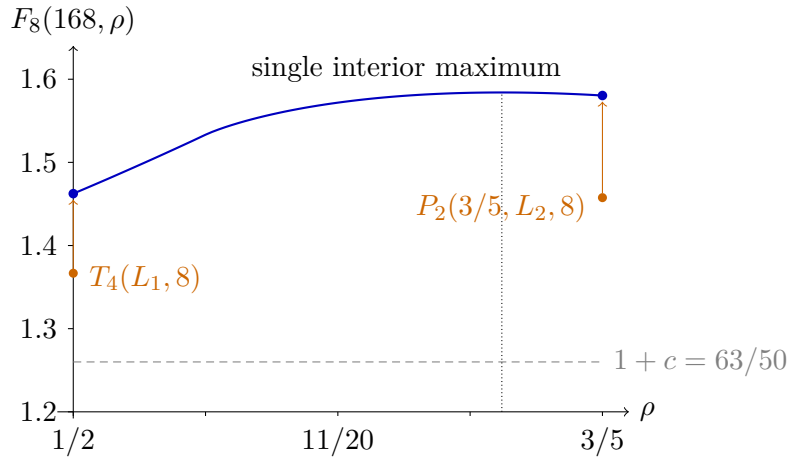

\paragraph{A uniform discrete consequence.}
To display both geometric variables, write the same action as
\[
 F_m(Q,\rho):=G_{\sqrt{Q/(1-\rho^2)}}(m)
             -G_{\rho\sqrt{Q/(1-\rho^2)}}(m).
\]
For positive integer orders the cap $c=13/50$ gives in the core
$n_m\ge\max\{0,\lfloor F_m(Q,\rho)-c\rfloor+1\}$.
In the outer region the existing bound is
$n_m\ge\lfloor G_U(m)+3/4\rfloor$.
Since $c+3/4=101/100>1$, the former lower bound holds in the outer
region as well. For $m=0$ it also holds whenever $t\ge1$, since
$3/(8\pi t)<13/50$.
For an interval $\rho\in[\rho_-,\rho_+]$ and $Q\ge Q_0$, therefore,
\[
 n_m\ge\max\left\{0,
  \left\lfloor\min\{F_m(Q_0,\rho_-),F_m(Q_0,\rho_+)\}-c\right\rfloor+1
                    \right\}.
\]
If $m<t_-=\rho_-\sqrt{Q_0/(1-\rho_-^2)}$, the cap $c$ can be
replaced by $\min\{c,\eta_{t_-}(m)\}$: the raw error decreases
with $t$ at fixed $m$. Indeed,
\[
 \frac{\partial}{\partial t}\eta_t(m)
     =-\frac{t(9t^2+12m^2)}{24\pi(t^2-m^2)^{5/2}}<0.
\]
This preserves the discrete initial phase information rather than
averaging it into a global sawtooth loss.

\begin{lemma}[A simple ground-mode family]\label{lem:analyticgroundfamily}
With outer radius one, the lowest radial eigenvalue in angular
order $m$ is strictly smaller, for every $0<\rho<1$, than
\[
 b_m=\frac{2(3m+2)(m+1)(m+3)}{5m+14}.
\]
Consequently $U^2>b_p$ guarantees at least one radial mode for
every order $0\le m\le p$.
\end{lemma}
\begin{proof}
In the unit disk use $v(s)=s^m(1-as^2)$ with
$a=(m+2)/(m+4)$. Its radial energy and mass are
\[
 \int_0^1\left(sv'^2+\frac{m^2}{s}v^2\right)\,\mathrm ds
       =\frac{2(3m+2)}{(m+4)^2},\qquad
 \int_0^1sv^2\,\mathrm ds
       =\frac{5m+14}{(m+1)(m+4)^2(m+3)}.
\]
The quotient is $b_m$.
For completeness, the radial Neumann ground eigenvalue
$\beta_m(\rho)$ decreases strictly with the inner radius for $m>0$.
Normalize its positive ground state $v_\rho$ by
$\int_\rho^1sv_\rho^2\,\mathrm ds=1$.
Differentiating the simple variational eigenvalue and using the
Neumann condition at the moving endpoint gives
\[
 \beta_m'(\rho)=\rho\left(\beta_m(\rho)-\frac{m^2}{\rho^2}\right)
                         v_\rho(\rho)^2<0.
\]
The strict sign follows already from the constant trial function:
$\beta_m\le2m^2\log(1/\rho)/(1-\rho^2)<m^2/\rho^2$.
Restriction of the disk trial to $[\varepsilon,1]$ shows
$\limsup_{\varepsilon\downarrow0}\beta_m(\varepsilon)\le b_m$;
strict monotonicity then implies $\beta_m(\rho)<b_m$.
For $m=0$, $\beta_0=0<b_0$ directly.
Finally the variational forms increase with $m^2$, so their
ground eigenvalues are nondecreasing with $m$.
\end{proof}

\paragraph{Action minorants for the remaining finite cases.}
For $0\le z\le1$,
\[
 \sqrt{1-z}\ge(1-z)\left(1+\frac z2\right),\qquad
 \sqrt{1-z}\ge(1-z)\left(1+\frac z2+\frac{3z^2}{8}\right).
\]
The first squared inequality follows from
$(1-z)(1+z/2)^2=1-3z^2/4-z^3/4\le1$.
The second follows from
\[
 (1-z)\left(1+\frac z2+\frac{3z^2}{8}\right)^2
       =1-\frac{5z^3}{8}-\frac{15z^4}{64}-\frac{9z^5}{64}\le1.
\]
Thus, for $m\le t<U$,
\[
 F_m\ge\frac1\pi\left((U-t)
   -\frac{m^2}{2}\left(\frac1t-\frac1U\right)
   -\frac{m^4}{6}\left(\frac1{t^3}-\frac1{U^3}\right)\right),
\]
and, more sharply,
\[
 F_m\ge\frac1\pi\left((U-t)
   -\frac{m^2}{2}\left(\frac1t-\frac1U\right)
   -\frac{m^4}{24}\left(\frac1{t^3}-\frac1{U^3}\right)
   -\frac{3m^6}{40}\left(\frac1{t^5}-\frac1{U^5}\right)\right).
\]
The discarded terms have been controlled by pointwise polynomial
inequalities, rather than by numerical sampling.

\begin{proposition}[A complete analytic bridge on a radius interval]\label{prop:analytichalfbridge}
For $1/2\le\rho\le3/5$ and $Q=(1-\rho^2)U^2\ge100$,
the Neumann annular counting function satisfies $N>Q/4$.

\end{proposition}
\begin{proof}
We use two complementary lower bounds. The fixed-area actions
$F_m$ give integer counts in each angular order. Summing those
counts as in Lemma~\ref{lem:analytic40} gives a simpler function
$H_0$ bounding the total excess $N-Q/4$. Thus $H_0$ is an
aggregate count bound, rather than a bound for any one $F_m$.
Its positive term is linear in $U$, while its losses have orders
$U^{1/2}$, $U^{1/3}$ and $1$. This makes it useful above an
explicit frequency threshold; the individual $F_m$ estimates then
cover the remaining finite range.

More precisely, the cut $M=\lfloor t-1\rfloor$ stays at least one
unit below the turning order $t$ and gives the integral error
$2I_{t,t-1}<\sqrt t/7$. Use $c=13/50$, the four-mode band bound,
and the weaker curvature $\kappa=(1-\rho)/(\pi\rho U)$ in the preceding
summed estimate. Since $0<t<U$ and $\pi>157/50$, the two linear terms satisfy
\[
 \frac{U+t}{2}-\frac{U-t}{\pi}
 >\frac{57}{314}U+\frac{257}{314}t.
\]
Together with $\pi<22/7$, this gives
$983/200=13/8+1/4+4(1/2+13/50)$ as the constant loss. These choices
produce exactly the following function from
Lemma~\ref{lem:analytic40}:
\[
 N-\frac Q4\ge H_0(\rho,U),
\]
where
\[
 H_0=\left(\frac{57}{314}+\frac{257}{314}\rho\right)U
       -\frac23\sqrt{\frac{22\rho U}{7(1-\rho)}}
       -\frac{\sqrt{\rho U}}7-\frac23U^{1/3}-\frac{983}{200}.
\]
This is a slightly weaker expression than the estimate retaining
the positive $23/(24\pi t^3)$ curvature correction.
The quotient $H_0/U$ increases with $U$.
Also $H_0$ is concave in $\rho\in[1/2,3/5]$.
Indeed, for $h(\rho)=\sqrt{\rho/(1-\rho)}$,
\[
 h''(\rho)=\frac{4\rho-1}{4\rho^{3/2}(1-\rho)^{5/2}}
       \ge\frac1{4\rho^{3/2}},
\]
so the second derivative of the curvature loss dominates the
positive second derivative $\sqrt U/(28\rho^{3/2})$ of the other
square-root loss.
At $U=22$, use $\sqrt7>13/5$, $\sqrt{11}<10/3$,
$\sqrt{66/5}<11/3$, $\sqrt{726/7}<51/5$ and
$\sqrt[3]{22}<17/6$. The two endpoint bounds are
\[
 H_0(1/2,22)>
 \frac{4081}{314}-\frac{220}{39}-\frac{10}{21}
       -\frac{17}{9}-\frac{983}{200}
       =\frac{1947011}{25716600}>0,
\]
\[
 H_0(3/5,22)>
 \frac{11616}{785}-\frac{34}{5}-\frac{11}{21}
       -\frac{17}{9}-\frac{983}{200}
       =\frac{1324907}{1978200}>0.
\]
Thus all $U\ge22$ in the claimed radius interval are covered.
In particular $Q\ge364$ implies
$U^2\ge4Q/3>22^2$.

The threshold $U=22$ was chosen so that these two rational endpoint
lower bounds are positive. Since $H_0/U$ increases with $U$ and
$H_0$ is concave in $\rho$, these endpoint inequalities control
every $U\ge22$ on the radius interval. The discrete action
estimates are needed only below this threshold. The endpoint
reduction is illustrated in Figure~\ref{fig:analytic-H}(c).

It remains to cover $100\le Q<364$. The following finite list
specifies nested sets of guaranteed radial layers. A value $p_j$
means that every order $0\le m\le p_j$ contributes at least $j$
radial modes; an omitted $p_3$ means that no third layer is used.
The total lower count in a row is
\[
 C=(2p_1+1)+(2p_2+1)+(2p_3+1),
\]
with the last summand omitted when appropriate.
\[
\begin{array}{c|rrr|r|c}
 Q_0&p_1&p_2&p_3&C&\text{covered }Q\text{ interval}\\\hline
 100&9&5&-&30&[100,120)\\
 120&10&6&-&34&[120,136)\\
 136&11&7&-&38&[136,152)\\
 152&12&8&-&42&[152,168)\\
 168&12&8&2&47&[168,188)\\
 188&13&9&4&55&[188,220)\\
 220&14&10&6&63&[220,252)\\
 252&15&11&8&71&[252,284)\\
 284&16&12&9&77&[284,308)\\
 308&17&13&9&81&[308,324)\\
 324&18&13&10&85&[324,340)\\
 340&18&14&10&87&[340,348)\\
 348&19&14&11&91&[348,364)
\end{array}
\]
Every row covers its indicated interval because $Q<4C$ there.
Endpoints of the intervals are included in the next row, so no
strict-count convention is lost at a join.

The following checks verify this table. For $\rho_1=1/2$, $\rho_2=3/5$, put
$U_i=\sqrt{Q_0/(1-\rho_i^2)}$.
The numbers $L_i$ in the next table are rational lower bounds:
each is an integer divided by $1000$, and one checks by squaring
$L_1^2<4Q_0/3$ and $L_2^2\le25Q_0/16$.
Define $P_n(\rho,L,m)$ to be the action minorant obtained by
integrating
\[
 (1-z)\sum_{k=0}^{n-1}\frac{\binom{2k}{k}}{4^k}z^k
       \le\sqrt{1-z},\qquad z=\frac{m^2}{s^2},
\]
from $\rho L$ to $L$, then multiplying by $7/22$ instead of
$1/\pi$. Only $n=2,3,4$ are needed. In particular, the first polynomial
minorant is explicitly
\[
 P_2(\rho,L,m)=\frac7{22}\left[(1-\rho)L
 -\frac{m^2}{2L}(\rho^{-1}-1)
 -\frac{m^4}{6L^3}(\rho^{-3}-1)\right].
\]
Thus a table entry $P_2$ means substitution into this rational
expression, with no transcendental function remaining.
The $n=4$ pointwise inequality follows, after squaring, from
\[
 (1-z)\left(1+\frac z2+\frac{3z^2}{8}+\frac{5z^3}{16}\right)^2
 =1-\frac{35z^4}{64}-\frac{7z^5}{32}
       -\frac{35z^6}{256}-\frac{25z^7}{256}\le1.
\]
All these action minorants increase with $L$, since their
integrated forms are a positive multiple of $L$ minus a sum of
positive multiples of negative odd powers of $L$.

For an outer endpoint use a convex-action tangent. Write
\[
 T_4(L,m)=\frac{4949}{22000}L-\frac m4,\quad
 T_3(L,m)=\frac{3031}{11000}L-\frac m3,\quad
 T_{16}(L,m)=\frac{5817}{22000}L-\frac{5m}{16}.
\]
These are lower bounds for $G_U(m)$ when $U\ge L$.
They follow from the tangent angles $\pi/4$, $\pi/3$, $5\pi/16$,
respectively, and
\[
 \sqrt2>\frac{707}{500},\qquad
 \sqrt3>\frac{433}{250},\qquad
 \sin\left(\frac{5\pi}{16}\right)>\frac{831}{1000},\qquad
 \pi<\frac{22}{7}.
\]
For the last sine bound, use $\sqrt2<283/200$ to obtain
$\sin(\pi/8)>191/500$ by squaring, and then use the half-angle
identity
$\sin(5\pi/16)=\sqrt{2+2\sin(\pi/8)}/2>831/1000$.

The two symbols in each profile column specify the lower bounds
used at $(\rho_1,\rho_2)$. In the $p_3$ column, $a_3$ is a
uniform upper bound for the derivative-phase error divided by
$\pi$. The cap $c$ is used for every $p_2$ check.
\[
\begin{array}{c|rr|cc|c}
 Q_0&1000L_1&1000L_2&p_2\text{ profiles}&p_3\text{ profiles}&a_3\\\hline
100&11547&12500&(P_2,P_2)&-&-\\
120&12649&13693&(P_2,P_2)&-&-\\
136&13466&14577&(T_4,P_2)&-&-\\
152&14236&15411&(T_{16},P_2)&-&-\\
168&14966&16201&(T_4,P_2)&(P_2,P_2)&1/50\\
188&15832&17139&(T_4,P_2)&(P_2,P_2)&1/40\\
220&17126&18540&(T_4,P_2)&(P_2,P_2)&1/25\\
252&18330&19843&(T_4,P_2)&(P_2,P_3)&1/10\\
284&19459&21065&(T_4,P_2)&(P_3,P_3)&37/200\\
308&20264&21937&(T_4,P_2)&(P_2,P_2)&11/100\\
324&20784&22500&(T_4,P_2)&(P_4,P_3)&13/50\\
340&21291&23048&(T_4,T_4)&(P_2,P_2)&9/40\\
348&21540&23318&(T_4,T_4)&(T_3,P_3)&13/50
\end{array}
\]
Substitution of these displayed rational inputs gives, in every
row, an excess greater than $1/250$ above $1+c$ for both $p_2$
endpoints. In every row with a third layer, it gives an excess
greater than $1/125$ above $2+a_3$ for both $p_3$ endpoints.
These are rational inequalities between explicitly displayed
polynomials and rational tangent lines; no evaluation of Bessel
functions, numerical inverse phase, or parameter grid is involved.

For clarity, the raw-error checks for $a_3<c$ reduce to the
following additional rational inputs. Put $D=Q_0/3-p_3^2$.
The lower bounds $s<\sqrt D$ shown below imply
\[
 \eta_{\sqrt{Q_0/3}}(p_3)
       <\frac{3Q_0-2p_3^2}{24(157/50)Ds}<a_3,
\]
using $\pi>157/50$. Each $s^2<D$ and the last strict inequality
follow by multiplication of integers.
\[
\begin{array}{c|rrrrrrr}
 Q_0&168&188&220&252&284&308&340\\\hline
 s&36/5&34/5&61/10&447/100&369/100&93/20&73/20
\end{array}
\]
The cap suffices for the remaining third-layer rows.
Monotonicity of $F_m$ in $m$ and of $\eta_t(m)$ in $m$ extends
each largest-order check to every smaller order in that layer.
The fixed-$Q$ endpoint principle and monotonicity in $Q$ then
extend the count through the full claimed rectangle.

\paragraph{A worked row combining endpoint actions and raw error.}
Consider $Q_0=168$, whose row has $(p_1,p_2,p_3)=(12,8,2)$.
Figure~\ref{fig:analytic-fixedarea-action} illustrates the endpoint
principle for its second layer.
Here
\[
 U_1^2=224,\qquad U_2^2=\frac{525}{2},\qquad
 L_1=\frac{14966}{1000},\quad L_2=\frac{16201}{1000}.
\]
The lower-frequency assertions are exact, since
\[
 U_1^2-L_1^2=\frac{4711}{250000}>0,\qquad
 U_2^2-L_2^2=\frac{27599}{1000000}>0.
\]
For the second layer, $m=8$ is an outer order at $\rho=1/2$,
because $t_1^2=56<64$. The tangent and the polynomial endpoint
bounds give
\[
 \begin{aligned}
 T_4(L_1,8)-\frac43&=\frac{1100101}{33000000}>0,\\
 P_2(3/5,L_2,8)&>P_2(3/5,16,8)
 =\frac{12649}{8910}>\frac43>1+\frac{13}{50}.
 \end{aligned}
\]
The polynomial is applicable at the second endpoint because
$8<(3/5)16$. This proves at least two radial modes through
angular order eight. For the third layer, use $m=2$ and the
smaller lower frequencies $14<L_1$, $81/5<L_2$ to obtain
\[
 \begin{aligned}
 P_2(1/2,L_1,2)&>\frac{1007}{462}>\frac{203}{100},\\
 P_2(3/5,L_2,2)&>
 \frac{24088908557}{11837848275}>\frac{203}{100}.
 \end{aligned}
\]
For the last comparison the exact excess is
$46461247/9470278620>0$. The raw phase error has the uniform bound
\[
 \eta_{\sqrt{56}}(2)
 <\frac{3\cdot168-2\cdot2^2}
       {24(157/50)\,52\,(36/5)}
 =\frac{3875}{220428}<\frac1{50},
\]
where $(36/5)^2<52$ and $\pi>157/50$.
Consequently $F_2-\eta_t(2)>203/100-1/50>2$, so the third
radial layer is present through order two. The first layer follows
from $b_{12}=7410/37<224\le U^2$. The angular multiplicities now
give the complete count
\[
 C=(1+2\cdot12)+(1+2\cdot8)+(1+2\cdot2)
   =25+17+5=47.
\]
The endpoint principle extends each assertion to every
$1/2\le\rho\le3/5$, and monotonicity in $Q$ extends it to
$168\le Q<188=4C$. The next row includes $Q=188$.

Finally, for the first-layer entries up to $Q_0=340$,
$b_{p_1}<4Q_0/3\le U^2$, directly from the displayed rational
formula for $b_m$. At $Q_0=348$ the first layer $p_1=19$ is instead
obtained from the outer action: at $Q_0=348$, throughout this radius interval
$t<19$, and the $\pi/6$ tangent gives
\[
 G_U(19)\ge\frac7{44}L_1-\frac{19}{6}
       =\frac{13}{50}+\frac1{6600}>c,
       \qquad L_1=\frac{21540}{1000}.
\]
This last row illustrates a different first-layer argument:
the phase cap gives at least one mode through order nineteen,
without a ground-trial comparison at that order. With the second
and third layers through orders fourteen and eleven, respectively,
its total is $39+29+23=91$. Hence it covers $348\le Q<364$;
the endpoint $Q=364$ is covered by the $H_0$ estimate.
Thus all first, second and third layer assertions in the table
are accounted for, completing this radius interval.
\end{proof}

\begin{proposition}[The left middle-radius bridge]\label{prop:analyticleftbridge}
Let $Q=(1-\rho^2)U^2$. For $1/10\le\rho\le1/5$ and $Q\ge100$,
the inclusive annular Neumann count satisfies $N_\rho(U^2)>Q/4$.
\end{proposition}
\begin{proof}
We first lower the analytic high-frequency threshold on this interval
to $U=38$. The preceding integer-cut estimate gives the lower excess
\[
 H_0(\rho,U)=\left(\frac{57}{314}+\frac{257}{314}\rho\right)U
 -\frac23\sqrt{\frac{22\rho U}{7(1-\rho)}}
 -\frac{\sqrt{\rho U}}7-\frac23 U^{1/3}-\frac{983}{200}.
\]
On $[1/10,1/5]$, $h(\rho)=\sqrt{\rho/(1-\rho)}$ satisfies
$h'(\rho)<9/4$ and $\sqrt\rho>3/10$. With
$\sqrt{22/7}<9/5$ and $257/314>4/5$, therefore,
\[
 (H_0)_\rho>\frac{257}{314}U-3\sqrt U>0\qquad(U\ge38).
\]
Also $H_0/U$ increases with $U$. At $(\rho,U)=(1/10,38)$,
\[
 \frac23\sqrt{836/63}<49/20,\qquad \sqrt{19/5}<2,
 \qquad 38^{1/3}<27/8,
\]
and consequently
\[
 H_0(1/10,38)>
 \frac{15713}{1570}-\frac{49}{20}-\frac27-\frac94-\frac{983}{200}>0.
\]
All $U\ge38$ in the claimed radius interval are covered.

For the remaining range we use the fixed-area endpoint principle
(Lemma~\ref{lem:analyticareaendpoints}) and explicitly count radial
layers. Put $c=13/50$ and $F_m=G_U(m)-G_{\rho U}(m)$.
For positive integer orders the proved phase cap gives
$n_m\ge\max\{0,\lfloor F_m-c\rfloor+1\}$.
The same inequality holds for $m=0$ in the present tables even if
$\rho U<2$: every inner argument is at least one, and the raw
order-zero phase error satisfies
\[
 \eta_t(0)=\frac{3}{8\pi t}<\frac{3}{8(157/50)}<13/50
 \qquad(t\ge1).
\]
The raw phase bound applies for every $t>0$ in order zero.
At fixed $Q_0$, the endpoint principle replaces $F_m$ by the minimum
of its values at the two displayed radius endpoints. Increasing $Q$
only increases each action.

Here are entirely rational lower bounds for those endpoint actions.
For $L\le\sqrt{Q_0/(1-r^2)}$ and $m\le rL$, put
\begin{align*}
 P(r,L,m)=\frac7{22}\bigg[&(1-r)L
 -\frac{m^2}{2L}(r^{-1}-1)-\frac{m^4}{24L^3}(r^{-3}-1)\\
 &-\frac{m^6}{80L^5}(r^{-5}-1)
 -\frac{5m^8}{112L^7}(r^{-7}-1)\bigg].
\end{align*}
This is a lower bound for $F_m(Q_0,r)$: integrate the pointwise
minorant
\[
 \sqrt{1-z}\ge(1-z)(1+z/2+3z^2/8+5z^3/16)
\]
with $z=m^2/s^2$ from $rL$ to $L$ and use $1/\pi>7/22$.
The minorant follows by squaring positive quantities, since
\[
 (1-z)(1+z/2+3z^2/8+5z^3/16)^2
 =1-35z^4/64-7z^5/32-35z^6/256-25z^7/256\le1.
\]
The result increases with $L$, so a rational lower frequency is valid.

In an outer branch $m\ge r\sqrt{Q_0/(1-r^2)}$, use instead one
of the following supporting lines for the convex disk action:
\[
\begin{array}{c|c|c}
\text{code}&\text{line}&\text{angle}\\\hline
A&7L/44-m/6&\pi/6\\
B&4949L/22000-m/4&\pi/4\\
C&5817L/22000-5m/16&5\pi/16\\
D&3031L/11000-m/3&\pi/3\\
E&6461L/22000-3m/8&3\pi/8\\
F&343L/1100-7m/16&7\pi/16\\
G&7L/22-m/2&\pi/2\\
H&1337L/11000-m/8&\pi/8
\end{array}
\]
Indeed a supporting line at angle $\sigma$ is
$L\sin\sigma/\pi-m\sigma/\pi$.
The rational sine bounds used here are
\[
 \sqrt2>707/500,\quad\sqrt3>433/250,\quad
 \sin(\pi/8)>191/500,\quad\sin(3\pi/8)>923/1000,
\]
\[
 \sin(5\pi/16)>831/1000,\qquad\sin(7\pi/16)>49/50.
\]
They follow by squaring and half-angle identities. For example,
$\sqrt2<283/200$ gives $\sin(\pi/8)>191/500$;
$\sqrt2>707/500$ gives $\sin(3\pi/8)>923/1000$.
The last two follow from
$\sin(5\pi/16)=\sqrt{2+2\sin(\pi/8)}/2$ and
$\sin(7\pi/16)=\sqrt{2+2\sin(3\pi/8)}/2$.

In the counting table, trial $B$ guarantees the ground row through $g$
by Lemma~\ref{lem:analyticgroundfamily} and
\[
 (1-a^2)b_g<Q_0,\qquad
 b_g=\frac{2(3g+2)(g+1)(g+3)}{5g+14}.
\]
Trial $P$ guarantees that row by the same phase inequalities used
for the higher layers, with $j=1$ and $m=g$.
Every entry $p_j$ guarantees $j$ radial modes in orders $0,\ldots,p_j$.
The radial layers are nested, and their distinct contributions give
\[
 C=2g+1+\sum_{j\ge2}(2p_j+1).
\]
Thus a row proves $N_\rho(U^2)\ge C>Q/4$ whenever
$Q_0\le Q<4C$. Consecutive rows have $Q_0=4C$ from the preceding
row, so their half-open intervals meet without a gap.

The complete data are given in Table~\ref{tab:left-counts} in Appendix~\ref{sec:analytic-table-data}.

Table~\ref{tab:left-checks} in Appendix~\ref{sec:analytic-table-data} gives the data for the endpoint inequalities.
$1000L_a$ and $1000L_b$ are integers; $L_r^2\le Q_0/(1-r^2)$
is checked by cross multiplication. A profile pair gives the codes
used at $a,b$, respectively; $P$ means the polynomial just defined.
The vectors correspond in order to $j=2,3,\ldots$ for ground trial
$B$, and to $j=1,2,\ldots$ for ground trial $P$.
An asterisk means $\alpha_j=c$ and no raw-error denominator is needed.
An ordinary entry is $(\alpha_j,s_j)$, where, with
$d_j=a^2Q_0/(1-a^2)-p_j^2$ (and $p_1=g$),
\[
 s_j^2\le d_j,\qquad
 \frac{9a^2Q_0/(1-a^2)-2p_j^2}
 {24(157/50)d_js_j}<\alpha_j<c.
\]
Both are rational inequalities. They guarantee the required raw-error
upper bound uniformly throughout the radius and energy range.
Every polynomial entry satisfies $p_j\le rL_r$; every supporting-line
entry satisfies $p_j^2\ge r^2Q_0/(1-r^2)$. Thus every profile is used
in its proved region.

Substitution into these rational formulas gives, at every
endpoint and in every layer,
\[
 \text{selected profile}>j-1+\alpha_j+1/5000.
\]
Each assertion is a finite rational inequality whose denominator is
positive. For example, the first row uses
\[
 E(10050/1000,4)>1+13/50+1/5000,\qquad
 E(10114/1000,4)>1+13/50+1/5000.
\]
The action decreases with angular order, and the raw error increases
with that order, so the checked last angular order guarantees every
earlier order in its layer. The fixed-area endpoint principle then
guarantees every intervening radius and every $Q\ge Q_0$.
All zero-order rows are legitimate by the raw bound explained above.

The first chain ends at $Q<1452$ and the second at $Q<1416$.
Their respective unhandled values satisfy
\[
 Q\ge1452>38^2(1-(1/10)^2),\qquad
 Q\ge1416>38^2(1-(3/20)^2).
\]
They therefore have $U\ge38$, already proved analytically.
This completes the strict counting inequality on the whole interval.
\end{proof}
\begin{proposition}[The middle-radius analytic bridge]\label{prop:analyticmiddlebridge}
For $1/5\le\rho\le1/2$ and $Q=(1-\rho^2)U^2\ge100$,
the inclusive Neumann annular counting function satisfies $N>Q/4$.
\end{proposition}
\begin{proof}
We first prove three analytic upper-frequency cutoffs. The
general estimate from Proposition~\ref{prop:analyticphase} and Lemma~\ref{lem:analyticintegralerror}, valid for $t=\rho U\ge2$, gives $N-Q/4>H_0(\rho,U)$,
where
\[
 H_0=\left(\frac{57}{314}+\frac{257\rho}{314}\right)U
 -\frac23\sqrt{\frac{22\rho U}{7(1-\rho)}}
 -\frac{\sqrt{\rho U}}7-\frac23 U^{1/3}-\frac{983}{200}.
\]
Its normalized value $H_0/U$ increases with $U$.
For $1/5\le\rho\le3/10$ and $U\ge32$, the derivative bound
from the same lemma gives $(H_0)_\rho>(257/314)U-3\sqrt U>0$;
the last inequality follows from $257/314>4/5$ and $U\ge32>225/16$.
On $[3/10,1/2]$, $H_0$ is concave in $\rho$.
Indeed, writing $h(\rho)=\sqrt{\rho/(1-\rho)}$, one has
\[
 h''(\rho)=\frac{4\rho-1}{4\rho^{3/2}(1-\rho)^{5/2}}
 \ge\frac1{20\rho^{3/2}},
\]
while $\sqrt{22/7}>3/2$; the negative curvature contribution
is therefore stronger than $-\sqrt U/(20\rho^{3/2})$, and the
positive contribution from $-\sqrt{\rho U}/7$ is only
$\sqrt U/(28\rho^{3/2})$.

The following rational inputs prove the endpoint positivity.
Here $A^2>22\rho U/[7(1-\rho)]$, $B^2>\rho U$, and $D^3>U$.
Each comparison is checked by squaring or cubing rational numbers.
\[
\begin{array}{c|r|c|c|c|c}
 \rho&U&(57/314+257\rho/314)U&A&B&D\\\hline
 1/5&32&8672/785&51/10&51/20&16/5\\
 3/10&28&9387/785&31/5&29/10&61/20\\
 2/5&28&11186/785&77/10&67/20&61/20\\
 2/5&24&9588/785&71/10&31/10&29/10\\
 1/2&24&2226/157&87/10&7/2&29/10
\end{array}
\]
In each row the displayed linear term exceeds
$2A/3+B/7+2D/3+983/200$ (the smallest excess exceeds $18/100$).
Consequently the intervals $[1/5,3/10]$, $[3/10,2/5]$ and
$[2/5,1/2]$ are covered respectively for $U\ge32,28,24$.
For their complements it suffices to treat respectively
\[
 100\le Q<\frac{24576}{25},\qquad
 100\le Q<\frac{17836}{25},\qquad
 100\le Q<\frac{12096}{25}.
\]

Tables~\ref{tab:middle-a-counts}--\ref{tab:middle-c-checks} in Appendix~\ref{sec:analytic-table-data} give the remaining inputs.
The cap is $c=13/50$. By Lemma~\ref{lem:analyticareaendpoints},
the minimum of the action $F_m=G_U(m)-G_{\rho U}(m)$ on a
fixed-$Q$ radius interval is at an endpoint, and it increases with
$Q$. It decreases with $m$.
If $F_p>j-1+a$ at both endpoints, with a uniform phase-error
upper bound $a\le c$ for $0\le m\le p$, then every such order
contributes at least $j$ radial modes. For $a=c$ this holds
uniformly in the core and outer branches. For $a<c$, we use
$p<t_-$ and $a>\eta_{t_-}(p)$, since $\eta_t(m)$ increases in
$m$ and decreases in $t$.
Here $t_-\ge(1/5)\sqrt{100/(1-1/25)}>2$, so the same
integer-phase bound also applies to the zero order.

The symbol $P$ denotes the following polynomial action minorant:
\[
 \begin{aligned}
P(r,L,m)=\frac7{22}\bigg[&(1-r)L
 -\frac{m^2(r^{-1}-1)}{2L}-\frac{m^4(r^{-3}-1)}{24L^3}\\
 &-\frac{m^6(r^{-5}-1)}{80L^5}-\frac{5m^8(r^{-7}-1)}{112L^7}\bigg].
\end{aligned}
\]
It is valid when $m\le rL$ and $U\ge L$, by integrating
$(1-z)(1+z/2+3z^2/8+5z^3/16)\le\sqrt{1-z}$,
where $z=m^2/s^2$, over $rL\le s\le L$.
For completeness its square reduces to
\[
 (1-z)(1+z/2+3z^2/8+5z^3/16)^2
 =1-35z^4/64-7z^5/32-35z^6/256-25z^7/256\le1.
\]
The function $P$ increases in $L$.
At an outer endpoint, the labels $6,4,5,3,8,7$ mean respectively
the convex-action tangents
\[
\begin{array}{c|l|c}
\text{label}&\text{minorant at }(L,m)&\text{tangent angle}\\\hline
6&7L/44-m/6&\pi/6\\
4&4949L/22000-m/4&\pi/4\\
5&5817L/22000-5m/16&5\pi/16\\
3&3031L/11000-m/3&\pi/3\\
8&6461L/22000-3m/8&3\pi/8\\
7&343L/1100-7m/16&7\pi/16
\end{array}
\]
The bounds for labels $4,5,3$ are those already established
in the action minorants. The additional labels $8,7$
use $\sin(3\pi/8)>923/1000$ and $\sin(7\pi/16)>49/50$.
Indeed $\sqrt2>707/500$ gives
$\sin^2(3\pi/8)=(2+\sqrt2)/4>1707/2000>(923/1000)^2$,
and the half-angle identity then gives
$\sin^2(7\pi/16)=(1+\sin(3\pi/8))/2>1923/2000>(49/50)^2$.
All tangent coefficients also use $\pi<22/7$.

In each count table $(p_2,\ldots,p_k)$ lists the successive
radial layers, and
\[
 C=(2p_1+1)+\sum_{j=2}^{k}(2p_j+1).
\]
Every row covers $Q_0\le Q<4C$; the next row starts exactly at
$4C$. The first layer comes from the ground-mode family of
Lemma~\ref{lem:analyticgroundfamily}: the label $B$ means
$(1-r_-^2)b_{p_1}<Q_0$. When a first-layer label is a pair of
numbers, it instead denotes the two tangent checks $F_{p_1}>c$.
All $L_\pm$ are the displayed integers divided by $1000$.
Their squared lower-bound checks are
$L_\pm^2\le Q_0/(1-r_\pm^2)$.

In the check tables the ordered pairs give the two endpoint
minorants for layers $2,\ldots,k$. The default error is $c$.
An entry $(j,e,s)$ in the last column replaces it in layer $j$
by $e/1000$ and gives the lower square-root input $s/1000$.
With
\[
 T=\frac{r_-^2Q_0}{1-r_-^2},\qquad D_j=T-p_j^2,
\]
each exceptional entry satisfies the entirely rational checks
\[
 D_j>0,\quad (s/1000)^2\le D_j,\qquad
 \frac{9T-2p_j^2}{24(157/50)D_j(s/1000)}<\frac e{1000},
\]
which imply $\eta_{\sqrt T}(p_j)<e/1000$, using $\pi>157/50$.
For each profile $P$ one checks $p_j\le r_\pm L_\pm$; for a
tangent one checks $p_j>r_\pm\sqrt{Q_0/(1-r_\pm^2)}$.
Direct substitution of the displayed rational inputs proves
that every endpoint profile exceeds its required $j-1+a_j$.
The excess in every check is greater than $1/100000$.
Thus the checks are polynomial or rational inequalities between
rational numbers; they follow directly from the displayed polynomial and rational formulas.

\paragraph{Radius interval $[1/5,3/10]$.}

Tables~\ref{tab:middle-a-counts} and~\ref{tab:middle-a-checks} in Appendix~\ref{sec:analytic-table-data} give the layer counts and endpoint checks.

The final counting row reaches $4C=992>24576/25$, the stated area cutoff.

\paragraph{Radius interval $[3/10,2/5]$.}

Tables~\ref{tab:middle-b-counts} and~\ref{tab:middle-b-checks} in Appendix~\ref{sec:analytic-table-data} give the layer counts and endpoint checks.

The final counting row reaches $4C=736>17836/25$, the stated area cutoff.

\paragraph{Radius interval $[2/5,1/2]$.}

Tables~\ref{tab:middle-c-counts} and~\ref{tab:middle-c-checks} in Appendix~\ref{sec:analytic-table-data} give the layer counts and endpoint checks.

The final counting row reaches $4C=500>12096/25$, the stated area cutoff.

The nested-layer count therefore gives $N\ge C>Q/4$ in every
remaining interval. Together with the three upper-frequency
cutoffs, this proves the proposition.
\end{proof}

\begin{proposition}[The wide-radius bridge]\label{prop:analyticwidebridge}
Put $Q=(1-\rho^2)U^2$. For $3/5\le\rho\le4/5$ and $Q\ge100$,
the inclusive Neumann annular count satisfies $N_\rho(U^2)>Q/4$.
\end{proposition}
\begin{proof}
The analytic phase estimate for $1/2\le\rho\le4/5$, $U\ge23$,
already covers the upper-frequency part. We connect the remaining frequencies to that estimate as follows.

Let $c=13/50$ and $F_m=G_U(m)-G_{\rho U}(m)$. The integer-order
phase estimate gives $n_m\ge\lfloor F_m-c\rfloor+1$ when this
right side is positive. In the outer part this also follows from
$n_m\ge\lfloor G_U(m)+3/4\rfloor$, since $3/4>1-c$.
For a radius interval $[a,b]$ and lower area parameter $Q_0$, the
fixed-area endpoint principle gives
\[
 F_m(Q,\rho)\ge\min\{F_m(Q_0,a),F_m(Q_0,b)\},
 \qquad Q\ge Q_0,\quad a\le\rho\le b.
\]
For any integer $p<t_-=a\sqrt{Q_0/(1-a^2)}$, the same estimate
uses the smaller error $\eta_{t_-}(p)$ uniformly in all orders
$0\le m\le p$. Indeed $\eta_t(m)$ increases with $m$ and decreases
with $t$. All zero-order inner arguments here exceed two.

\paragraph{Profile estimates.}
For a rational $L\le\sqrt{Q_0/(1-r^2)}$ define
\begin{align*}
 P_4(r,L,m)=\frac7{22}\bigg[&(1-r)L
 -\frac{m^2}{2L}(r^{-1}-1)
 -\frac{m^4}{24L^3}(r^{-3}-1)\\
 &-\frac{m^6}{80L^5}(r^{-5}-1)
 -\frac{5m^8}{112L^7}(r^{-7}-1)\bigg].
\end{align*}
When $m\le rL$, this is a lower bound for $F_m(Q_0,r)$.
To prove it, for $0\le z\le1$ use
\[
 (1-z)\left(1+\frac z2+\frac{3z^2}{8}+\frac{5z^3}{16}\right)
 \le\sqrt{1-z}.
\]
Both sides are nonnegative; squaring reduces the assertion to
\[
 (1-z)\left(1+\frac z2+\frac{3z^2}{8}+\frac{5z^3}{16}\right)^2
 =1-\frac{35z^4}{64}-\frac{7z^5}{32}
       -\frac{35z^6}{256}-\frac{25z^7}{256}\le1.
\]
Integrate with $z=m^2/s^2$ from $rL$ to $L$ and use $1/\pi>7/22$.
The resulting expression increases with $L$: it is a positive
multiple of $L$ minus positive multiples of negative odd powers of $L$.

For an outer endpoint $m\ge r\sqrt{Q_0/(1-r^2)}$, use instead
\[
 T_4(L,m)=\frac{4949}{22000}L-\frac m4,\qquad
 T_{16}(L,m)=\frac{5817}{22000}L-\frac{5m}{16}.
\]
These are supporting lines for the convex disk profile, obtained
at angles $\pi/4$ and $5\pi/16$. They use
$\sqrt2>707/500$, $\sin(5\pi/16)>831/1000$, and $\pi<22/7$.
The latter sine inequality follows by successive half-angle
identities: $\sqrt2<283/200$ implies $\sin(\pi/8)>191/500$;
then $\sin(5\pi/16)=\sqrt{2+2\sin(\pi/8)}/2>831/1000$.
All these comparisons follow by squaring positive rational numbers.

\paragraph{First layers and counts.}
Let $B$ denote the ground-mode bound
\[
 b_m=\frac{2(3m+2)(m+1)(m+3)}{5m+14},\qquad
 (1-a^2)b_g<Q_0.
\]
The previously proved disk trial and strict annular ground monotonicity
then put orders $0,\ldots,g$ below $Q_0/(1-\rho^2)$.
Let $G$ denote the constant radial trial, which instead uses
$\gamma_a g^2<Q_0$, where
\[
 \gamma_{7/10}=\frac{849}{1190}>2\log(10/7).
\]
This logarithmic bound follows from
$2\log z<4[x+x^3/(3(1-x^2))]$, $x=(z-1)/(z+1)$.

An entry $p_j$ guarantees at least $j$ radial modes in orders
$0,\ldots,p_j$. The total count is
$C=(2g+1)+(2p_2+1)+(2p_3+1)$, with the last term omitted
when $p_3$ is absent. Each row proves the count throughout
$Q_0\le Q<4C$. The consecutive $Q_0$ values equal the preceding
$4C$, so their half-open intervals meet without a gap.
\[
\begin{array}{c|r|rc|rr|r}
 [a,b]&Q_0&g&\text{trial}&p_2&p_3&C\\\hline
 [3/5,13/20]&100&10&B&6&-&34\\
 &136&12&B&8&-&42\\
 &168&13&B&9&-&46\\
 &184&14&B&10&-&50\\
 &200&15&B&10&3&59\\
 &236&16&B&12&7&73\\
 &292&18&B&14&10&87\\\hline
 [13/20,7/10]&100&11&B&6&-&36\\
 &144&13&B&9&-&46\\
 &184&15&B&10&-&52\\
 &208&16&B&11&-&56\\
 &224&17&B&12&-&60\\
 &240&17&B&12&4&69\\
 &276&19&B&14&8&85\\\hline
 [7/10,4/5]&100&11&G&4&-&32\\
 &128&13&G&8&-&44\\
 &176&16&B&11&-&56\\
 &224&18&B&13&-&64\\
 &256&19&B&14&-&68
\end{array}
\]

\paragraph{Complete rational comparisons for the higher layers.}
In the next table $1000L_a$ and $1000L_b$ are integers specifying
lower bounds for the two outer arguments. In each row
$L_r^2\le Q_0/(1-r^2)$ for $r=a,b$, by multiplying positive integers.
All profile choices are $(P_4,P_4)$, except the two explicitly
indicated second-layer entries. The numbers $a_j$ are uniform error
upper bounds. A dash in an $s_j$ column means the cap $a_j=c$ is used.
Otherwise put $d_j=a^2Q_0/(1-a^2)-p_j^2$; direct squaring gives
$s_j^2\le d_j$, and
\[
 \eta_{t_-}(p_j)
 <\frac{9a^2Q_0/(1-a^2)-2p_j^2}{24(157/50)d_js_j}<a_j.
\]
Thus every comparison is rational. The last inequality holds strictly
also when $s_j^2=d_j$, since $\pi>157/50$.
\begin{center}\small\setlength{\tabcolsep}{3pt}
\begin{tabular}{r|rr|cc|cc}
 $Q_0$&$1000L_a$&$1000L_b$&$a_2$&$s_2$&$a_3$&$s_3$\\\hline
 \multicolumn{7}{c}{$[a,b]=[3/5,13/20]$}\\
 100&12500&13159&8/125&9/2&-&-\\
 136&14577&15345&169/1000&707/200&-&-\\
 168&16201&17056&37/200&1837/500&-&-\\
 184&16955&17849&13/50&-&-&-\\
 200&17677&18609&61/250&707/200&13/1000&10173/1000\\
 236&19202&20215&13/50&-&19/1000&9151/1000\\
 292&21360&22486&13/50&-&33/1000&1603/200\\\hline
 \multicolumn{7}{c}{$[a,b]=[13/20,7/10]$}\\
 100&13159&14002&7/200&1219/200&-&-\\
 144&15790&16803&87/1000&2467/500&-&-\\
 184&17849&18994&33/500&5883/1000&-&-\\
 208&18978&20195&43/500&5583/1000&-&-\\
 224&19694&20957&89/500&2229/500&-&-\\
 240&20385&21693&97/1000&281/50&11/1000&1579/125\\
 276&21861&23263&13/50&-&7/500&1468/125\\\hline
 \multicolumn{7}{c}{$[a,b]=[7/10,4/5]$}\\
 100&14002&16666&2/125&2237/250&-&-\\
 128&15842&18856&29/1000&7679/1000&-&-\\
 176&18576&22110&51/1000&1387/200&-&-\\
 224&20957&24944&17/250&3399/500&-&-\\
 256&22404&26666&69/1000&1767/250&-&-
\end{tabular}\end{center}
For $[a,b]=[3/5,13/20]$, the $Q_0=236$ second-layer profile
pair is $(T_{16},P_4)$ and the $Q_0=292$ pair is $(T_4,P_4)$.
They are valid outer endpoints, since respectively
$12^2>(3/5)^2\,236/(1-(3/5)^2)$ and
$14^2>(3/5)^2\,292/(1-(3/5)^2)$.
All other entries satisfy $p_j\le rL_r$ and use $P_4$.
Substitution into the displayed rational formulas gives, for every
second-layer endpoint, a margin greater than $1/250$ above $1+a_2$;
for every third-layer endpoint, it gives a margin greater than
$1/125$ above $2+a_3$. These inequalities consist only of finite
additions, multiplications and divisions of the rationals.
For example, in the first row the two comparisons are
\[
 P_4(3/5,25/2,6)>1+8/125+1/250,
 \quad P_4(13/20,13159/1000,6)>1+8/125+1/250.
\]
Monotonicity in $m$ then proves every lower angular order in a layer;
the endpoint principle proves every intervening radius.

The three radius-interval blocks of the count table reach respectively $Q<348$, $Q<340$, and $Q<272$.
For the remaining $Q$ values, respectively,
\[
 Q\ge348>23^2(1-(3/5)^2),\quad
 Q\ge340>23^2(1-(13/20)^2),\quad
 Q\ge272>23^2(1-(7/10)^2).
\]
Hence $U\ge23$ and the analytic high-frequency estimate applies.
This proves the strict counting inequality on the full claimed range.
\end{proof}
\subsection{Strict counting and completion}\label{sec:allassembly}

The parameter ranges are now covered. We first check that the
comparisons also apply to strict counts, which yields the strict
eigenvalue inequalities. We then assemble the annular proof and
transfer it to products and congruent sectors.

\begin{proposition}[Strict spectral inequalities]\label{prop:strictspectral}
Put $D=R^2-r^2$ and use the strict counting function defined in
the introduction. The preceding radial and phase comparisons establish
\[
 \mathcal M_A(\Lambda)>\frac{D\Lambda}{4}\qquad(\Lambda>0).
\]
Consequently $\nu_k(A_{r,R})<4k/D$ for every $k\ge1$.
\end{proposition}
\begin{proof}
For $t>0$, the annular radial strict count is
\[
 n_m^{<}(t^2)=
 \left\lceil\frac{\varphi_m(Rt)-\varphi_m(t)}\pi\right\rceil.
\]
This is the strict-count identity \eqref{eq:strictphasecount}.
Whenever an actual phase difference $d$ strictly exceeds a lower
bound $\ell$, one has
$\lceil d\rceil\ge\lfloor\ell\rfloor+1$.
Thus every lower count deduced from a strict phase enclosure remains
valid for $n_m^{<}$. This includes the core, band, outer-region,
integer-cut, and finite angular-layer estimates.

Lemma~\ref{lem:radial} and the strict constant radial trial comparisons
show that the radial-row and thin-annulus estimates also count strictly
subthreshold modes at positive energy. The negative polynomial-form comparisons are already strict.

The disk inputs used in the small-hole arguments admit the same
upgrade. With the derivative-phase convention of
Lemma~\ref{lem:exactphasecount}, the disk
radial strict count at positive frequency $X$ is
\[
 n_{B,m}^{<}(X)=\left\lceil\frac{\varphi_m(X)}\pi-\frac12\right\rceil.
\]
For $m=0$ this includes the constant mode. For $X>m$, the strict
lower phase enclosure gives
\[
 n_{B,m}^{<}(X)\ge\left\lfloor G_X(m)+\frac34\right\rfloor.
\]
Orders $m\ge X$ have zero on the right. Hence the disk lattice
comparison, the disk surplus, and the two additional disk
margins all hold for the strict disk count. The extra polynomial
modes in the starred disk row have quotient strictly below its
energy threshold. Neumann bracketing holds for strict counts too,
and subtracting the inclusive count of the inner disk is only a
weaker upper subtraction. The small-hole proofs are unchanged.

These observations apply to every parameter range covered above and yield the stated
strict count for every $\Lambda>0$. At $\Lambda=4k/D$ the strict
count is at least $k+1$, which proves the eigenvalue assertion.
\end{proof}

\begin{proof}[Proof of Theorem~\ref{thm:allannuli}]
Scale the outer radius to one and put $\rho=r/R$, $X=R\sqrt\Lambda$,
and $Q=(1-\rho^2)X^2$. If $\rho\ge4/5$, use the thin-annulus theorem.
If $\rho\le1/10$, use Proposition~\ref{prop:analytic10}.
For $1/10\le\rho\le4/5$ and $Q\le100$, use
Proposition~\ref{prop:analyticlowfrequency}.
For the remaining $Q\ge100$, the analytic bridges cover respectively
$[1/10,1/5]$ (Proposition~\ref{prop:analyticleftbridge}),
$[1/5,1/2]$ (Proposition~\ref{prop:analyticmiddlebridge}),
$[1/2,3/5]$ (Proposition~\ref{prop:analytichalfbridge}), and
$[3/5,4/5]$ (Proposition~\ref{prop:analyticwidebridge}).
Each bridge includes its lower endpoint and overlaps its analytic
frequency threshold. These ranges cover every radius ratio and every nonnegative energy.
This proves $N_A(\Lambda)>(R^2-r^2)\Lambda/4$.
Proposition~\ref{prop:strictspectral} gives
$\mathcal M_A(4k/(R^2-r^2))>k$ for $k\ge1$.
There are therefore at least $k+1$ eigenvalues strictly below this
energy, including zero. This proves
$\nu_k<4k/(R^2-r^2)$ and completes Theorem~\ref{thm:allannuli}.
\end{proof}

Theorem~\ref{thm:allannuli} now gives the two geometric consequences
stated in the introduction. The strict counting estimate on the annulus passes to products by a
Riesz-mean inequality, and to congruent sectors by Neumann bracketing.
These are the established transfer arguments of Laptev \cite{Laptev}
and Filonov, Levitin, Polterovich, and Sher \cite{FLPSballs}.

\begin{proof}[Proof of Corollary~\ref{cor:annularproducts}]
Write $A=A_{r,R}$ and let $\mu_j\left(\Omega\right)$, $j\ge0$, be the Neumann
eigenvalues of $\Omega$, including all zero eigenvalues if $\Omega$
has more than one component. The spectrum of a Cartesian product
consists of sums of the spectra of its factors. Thus, for $E>0$,
\[
 \mathcal M_{A\times\Omega}\left(E\right)
 =\sum_{\mu_j\left(\Omega\right)<E}\mathcal M_A\left(E-\mu_j\left(\Omega\right)\right)
 >\frac{\left|A\right|}{4\pi}\sum_{j\ge0}\left(E-\mu_j\left(\Omega\right)\right)_+,
\]
The sum is the Riesz mean of order one defined in the introduction. The strict inequality follows term by term
from Proposition~\ref{prop:strictspectral}; the sum is nonempty because of the
zero mode. Laptev's Neumann trace inequality
\cite[Theorem~3.1]{Laptev} gives
\[
 \sum_{j\ge0}\left(E-\mu_j\left(\Omega\right)\right)_+
 \ge \frac{\left|\Omega\right|}{\left(4\pi\right)^{d/2}\Gamma\left(2+d/2\right)}E^{1+d/2}.
\]
Since
$\omega_{d+2}=\pi^{1+d/2}/\Gamma\left(2+d/2\right)$, the product of the
two coefficients is
$\omega_{d+2}\left|A\right|\,\left|\Omega\right|/\left(2\pi\right)^{d+2}$, the Weyl coefficient
in dimension $d+2$. At
\[
 E_k=4\pi^2\left(\frac{k}
 {\omega_{d+2}\left|A\right|\,\left|\Omega\right|}\right)^{2/\left(d+2\right)},
\]
the Weyl term equals $k$, so the strict counting estimate gives
at least $k+1$ eigenvalues strictly below $E_k$, including zero.
This yields the corollary.
\end{proof}

\begin{proof}[Proof of Corollary~\ref{cor:annularsectors}]
The $q$ rotated copies of $S_q$ partition the annulus up to their
boundary seams. Allowing independent traces on the seams enlarges
the Neumann form domain. Strict-count Neumann bracketing and
Proposition~\ref{prop:strictspectral} therefore give, for $E>0$,
\[
 q\mathcal M_{S_q}(E)\ge\mathcal M_A(E)
       >\frac{q|S_q|}{4\pi}E.
\]
The inclusive count is at least the strict count, so the same
strict lower bound holds for $N_{S_q}(E)$ at positive energy.
At $E=0$ its constant eigenfunction gives $N_{S_q}(0)=1>0$.
Taking $E=4qk/(R^2-r^2)$ gives at least $k+1$ eigenvalues strictly
below $E$, proving $\mu_k(S_q)<4qk/(R^2-r^2)$.
This is the Neumann tiling transfer of
\cite[Theorem~1.8 and Remark~1.9]{FLPSballs}, with strict counting
used at positive energy.
\end{proof}

\section{Positive corrections to the P\'olya bounds}\label{sec:positivecorrections}

The estimates used in the preceding proof retain positive terms
beyond the Weyl area contribution. We first obtain a uniform explicit
surplus at high frequency, then quantify the first eight eigenvalues,
and finally join the high-index estimate to the strict bounds at the
remaining indices by continuity.

\subsection{An explicit high-frequency surplus}

Retaining the positive terms in the high-frequency estimates gives
a quantitative surplus above the Weyl term. We compare the
remainders in the different radius regimes to obtain one explicit
bound, and then invert it to sharpen the eigenvalue inequality.

\begin{theorem}[A positive high-frequency counting surplus]
\label{thm:quantitativeallannuli}
Let $A=\{x\in\mathbb R^2:r<|x|<R\}$, where $0<r<R$.
Write
\[
 D=R^2-r^2,\qquad \rho=r/R,\qquad
 U=R\sqrt\Lambda,\qquad t=r\sqrt\Lambda,
 \qquad Q=D\Lambda=(1-\rho^2)U^2.
\]
The counting function uses the inclusive convention fixed in the introduction.
For every $U\ge80$,
\begin{equation}\label{eq:quantitativeallannuligap}
 N_A(\Lambda)>
 \frac D4\Lambda+\frac U6-\frac34U^{1/3}-6.
\end{equation}
In particular, uniformly in the radius ratio,
\begin{equation}\label{eq:quantitativeallannulilinear}
 N_A(\Lambda)>\frac D4\Lambda+\frac U{20}
 \qquad(U\ge80).
\end{equation}
\end{theorem}

\begin{proof}
We use the general phase estimate with the improved integer cut,
the cap $13/50$, and the simpler integral-error bound. As established
in the proof of Lemma~\ref{lem:analytic40}, it gives, whenever
$t\ge2$,
\begin{align*}
 N_A(\Lambda)-Q/4>H_0(\rho,U)
 ={}&a_0U+b_0t
 -\frac23\sqrt{\frac{22t}{7(1-\rho)}}
 -\frac{\sqrt t}{7}-\frac23U^{1/3}-\frac{983}{200},\\
 &a_0=\frac{57}{314}>\frac16,\qquad
 b_0=\frac{257}{314}>\frac45.
\end{align*}
The hypothesis here is $t\ge2$; the later frequency threshold
$U\ge40$ in that lemma was a sufficient positivity threshold on
its specified radius interval, not a restriction on this phase
estimate. We treat the possible $t<2$ separately below.

First suppose $\rho\le4/5$ and $t\ge2$.
If $\rho\le1/2$, the two square-root losses involving $t$ are
less than $K\sqrt t$, where
\[
 K=\frac23\sqrt{\frac{44}{7}}+\frac17<\frac{11}{6}.
\]
Indeed $\sqrt{44/7}<71/28$, since $4928<5041$, and
$2(71/28)/3+1/7=11/6$.
Completing a square gives
\[
 b_0t-K\sqrt t>
 \frac45t-\frac{11}{6}\sqrt t\ge-\frac{605}{576}.
\]
Since
\[
 \frac{983}{200}+\frac{605}{576}
   =\frac{85901}{14400}<6,
\]
the count remainder is greater than
$a_0U-(2/3)U^{1/3}-6$.
If instead $1/2\le\rho\le4/5$, then $t\ge U/2$ and the
same two square-root losses are less than
$(59/21)\sqrt U$: use $\rho/(1-\rho)\le4$,
$\sqrt{88/7}<4$ and $\sqrt t\le\sqrt U$.
For $U\ge80>64$,
\[
 \frac{59}{21}\sqrt U<\frac{59}{168}U<\frac25U<b_0t.
\]
Thus this case has the same lower remainder. In both cases it
strictly dominates the remainder in
\eqref{eq:quantitativeallannuligap}, because $a_0>1/6$ and
$2/3<3/4$.

Next suppose $\rho>4/5$ and put
$\delta=(R-r)/r\in(0,1/4)$. Now $U=(1+\delta)t$ and
$t\ge4U/5\ge64$, so all radial and phase estimates used here
are within their proved ranges. If
$t\le1/(2\delta^2)$, Proposition~\ref{prop:mixedlow} and
\eqref{eq:mixedB} imply
\begin{align*}
 N_A(\Lambda)-Q/4
 &>-\frac{\delta^2t^2}{4}+\frac37t-1-\frac{3\delta}{8}\\
 &\ge\frac{17}{56}t-\frac{35}{32}
 \ge\frac{17}{70}U-\frac{35}{32}.
\end{align*}
This exceeds $U/6-(3/4)U^{1/3}-6$, since $17/70>1/6$
and $35/32<6$.
If $t\ge1/(2\delta^2)$, then $\delta t\ge4$.
To check this uniformly, $\delta\le1/8$ gives
$\delta t\ge1/(2\delta)\ge4$, whereas $\delta\ge1/8$
gives $\delta t\ge64/8=8$.
Using $1-\rho=\delta/(1+\delta)$ and
$(22/7)(1+\delta)\le55/14<4$, we obtain
\[
 \frac23\sqrt{\frac{22t}{7(1-\rho)}}<\frac23t.
\]
Also $a_0U+b_0t=(1+a_0\delta)t>t$.
The phase estimate therefore gives
\[
 N_A(\Lambda)-Q/4>
 \frac4{15}U-\frac{\sqrt U}{7}
       -\frac23U^{1/3}-\frac{983}{200}.
\]
Its excess above the required remainder is
\[
 \frac U{10}-\frac{\sqrt U}{7}
       +\frac1{12}U^{1/3}+\frac{217}{200}>0,
\]
because completing a square gives
$U/10-\sqrt U/7\ge-5/98$ and $217/200>5/98$.
This proves the thin-radius case without a separate limiting
argument as $\rho$ tends to one.

Finally suppose $t<2$. Neumann splitting of the disk gives
\[
 N_A(\Lambda)\ge\mathcal N_B(U)-\mathcal N_B(t),
\]
where $\mathcal N_B(Z)$ is the inclusive Neumann count of the
unit disk at energy $Z^2$.
The same disk tail estimate used in the all-annuli
high-frequency proof gives
\[
 \mathcal N_B(U)>
 \frac{U^2}{4}+\frac U6-\frac34U^{1/3}-\frac34,
 \qquad \mathcal N_B(t)\le3\quad(t<2).
\]
These are the disk phase estimate of
\cite[Proposition 6.5]{FLPSballs} with the quarter-height
tail bound of Lemma~\ref{lem:outerfloors}, and the
first-root exclusion already used in that proof.
Consequently
\[
 N_A(\Lambda)-Q/4>
 \frac U6-\frac34U^{1/3}-\frac{15}{4}+\frac{t^2}{4},
\]
which also implies \eqref{eq:quantitativeallannuligap}.

To obtain \eqref{eq:quantitativeallannulilinear}, divide the
remainder by $U$. The result
$1/6-(3/4)U^{-2/3}-6/U$ increases with $U$.
At $U=80$, the bound $80^{1/3}<13/3$ gives
\[
 \frac{80}{6}-\frac34\,80^{1/3}-6>
 \frac{49}{12}>\frac{80}{20}.
\]
This proves the stated uniform linear surplus and its positivity.
\end{proof}

\begin{corollary}[An explicit subtraction from the eigenvalue bound]
\label{cor:quantitativeallannulieigen}
With the notation and inclusive eigenvalue indexing above, every
integer $k\ge1604$ satisfies
\begin{align}
 \nu_k(A)
 &\le\frac4D\left(\sqrt{k+\frac1{400}}-\frac1{20}\right)^2
 \label{eq:quantitativeallannulieigen}\\
 &=\frac{4k}{D}
 -\frac{2}{5D}\left(\sqrt{k+\frac1{400}}-\frac1{20}\right)
 <\frac{4k}{D}.\notag
\end{align}
Thus the quantity subtracted from the ordinary P\'olya upper
bound is explicit, positive, and of order $\sqrt{k}/D$.
\end{corollary}
\begin{proof}
Set $z=2(\sqrt{k+1/400}-1/20)$ and $\Lambda=z^2/D$.
Then
\[
 z^2/4+z/20=k,
 \qquad k\ge1604=80^2/4+80/20\quad\Longrightarrow\quad z\ge80.
\]
Since $D<R^2$, one has $U=Rz/\sqrt D>z\ge80$.
The count bound gives
\[
 N_A(\Lambda)>D\Lambda/4+U/20
 \ge z^2/4+z/20=k.
\]
Because the count is inclusive and integer-valued, it contains at
least $k+1$ eigenvalues, hence $\nu_k(A)\le\Lambda$.
Expansion of the square gives the displayed positive subtraction.
The inclusive convention yields the non-strict first inequality;
positivity of the subtraction yields the final strict comparison
with $4k/D$.
\end{proof}

\begin{remark}[A stronger geometry-dependent form]
The same proof retains the full $U/20$ surplus if one puts
$\beta=R/(20\sqrt D)$. It gives
\[
 \nu_k(A)\le\frac4D(\sqrt{k+\beta^2}-\beta)^2
 \quad\text{whenever}\quad
 k\ge1600\frac D{R^2}+4.
\]
The threshold is the value of $Q/4+U/20$ at $U=80$.
This version improves both the subtraction and the applicable
index threshold when the annulus is thin.
\end{remark}

\begin{proposition}[Stronger high-frequency coefficients]
\label{prop:largerhighcoefficient}
Under the notation of Theorem~\ref{thm:quantitativeallannuli},
\begin{equation}\label{eq:largerhighcoefficient}
 N_A(\Lambda)>\frac{D\Lambda}{4}+\frac{49}{960}U
 \qquad(U\ge80).
\end{equation}
More generally, for every $0<b<1/6$ there is an explicit $T_b$
such that $N_A(\Lambda)>D\Lambda/4+bU$ for every annulus and
every $U\ge T_b$. One possible choice is
\[
 T_b=\max\left\{80,
       \left(\frac{3}{2(1/6-b)}\right)^{3/2},
       \frac{12}{1/6-b}\right\}.
\]
\end{proposition}
\begin{proof}
The stronger estimate \eqref{eq:quantitativeallannuligap} gives
\[
 \frac{N_A(\Lambda)-D\Lambda/4}{U}>
 F(U):=\frac16-\frac34U^{-2/3}-\frac6U.
\]
The function $F$ increases with $U$. Since
$80^{-2/3}=80^{1/3}/80$ and $80^{1/3}<13/3$, we have
$(3/4)80^{-2/3}<13/320$. Consequently,
\[
 F(80)>\frac16-\frac{13}{320}-\frac3{40}
       =\frac{49}{960}=\frac1{20}+\frac1{960}>\frac1{20}.
\]
This proves \eqref{eq:largerhighcoefficient} without changing
the frequency threshold. For the general assertion put
$\varepsilon=1/6-b>0$. The stated threshold makes each of
$(3/4)U^{-2/3}$ and $6/U$ at most $\varepsilon/2$.
\end{proof}

\begin{remark}[A coefficient--threshold tradeoff]
For example, $F(125)=133/1500>1/12$, so
$N_A(\Lambda)>D\Lambda/4+U/12$ for $U\ge125$.
Inversion yields the explicit improvement
\[
 \nu_k(A)\le\frac4D\left(\sqrt{k+1/144}-1/12\right)^2
 \qquad(k\ge3917).
\]
Indeed $3917$ is the least integer not smaller than
$125^2/4+125/12=11750/3$. Its leading subtraction is
$(2/3)\sqrt k/D$, compared with $(2/5)\sqrt k/D$ in
Corollary~\ref{cor:quantitativeallannulieigen}.
These thresholds are sufficient for the stated coefficients.
\end{remark}

\subsection{Low eigenvalues and the full spectrum}

The explicit tail bound leaves a finite range of indices. We obtain
sharp constants for the first pair and explicit bounds for the
first eight indices, then use continuity and the two geometric
limits to prove positive gaps at each remaining index. These gaps
join the tail estimate in a bound over the entire spectrum.

\begin{proposition}[Sharp bounds for the first positive pair]
\label{prop:optimalfirstpair}
Let $a=j'_{1,1}$ be the first positive zero of $J_1'$ and put
\[
 \delta_*:=1-\frac{a^2}{4}.
\]
For every annulus,
\begin{equation}\label{eq:optimalfirstpair}
 (R^2-r^2)\nu_1(A_{r,R})
 =(R^2-r^2)\nu_2(A_{r,R})<a^2.
\end{equation}
The number $a^2$ is the least possible radius-independent
upper bound in \eqref{eq:optimalfirstpair}. \end{proposition}
\begin{proof}
Normalize $R=1$ and write $\rho=r/R$.
The first-eigenvalue upper bound is also a consequence of the
classical Szeg\H{o}--Weinberger inequality \cite{Szego,Weinberger}.
The separated argument below identifies its twofold occurrence
and the limiting sharpness within the annular family.
The ground level $\beta_1(\rho)$ in angular order one is less
than $112/33<4$ by Lemma~\ref{lem:angularfirstmonotone}.
Every angular order $m\ge2$ has ground level at least $m^2\ge4$.
For radial index $n\ge1$, the form inequality gives
$\lambda_{m,n}\ge\lambda_{0,n}+m^2$.
The lower bound \eqref{eq:radial-first-lower} from
Lemma~\ref{lem:radial} gives
\[
 \lambda_{0,1}>\frac{\pi^2}{(1-\rho)^2}>4.
\]
Thus the first two positive eigenvalues are exactly the
cosine and sine pair belonging to $\beta_1(\rho)$.

Lemma~\ref{lem:angularfirstmonotone} proves that
$\beta_1(\rho)$ strictly decreases with $\rho$.
For completeness, its limit at a vanishing hole is $a^2$.
For sufficiently small $\rho$, the separated characteristic equation, with
$z_\rho=\sqrt{\beta_1(\rho)}\in[1,\sqrt{112/33}]$, is
\[
 J_1'(z_\rho)-
 \frac{J_1'(\rho z_\rho)}{Y_1'(\rho z_\rho)}Y_1'(z_\rho)=0.
\]
The small-argument expansions of the Bessel functions
\cite[Section~10.7]{NIST}, together with their derivative
recurrences, give
\[
 J_1'(x)=\frac12+O(x^2),\qquad
 Y_1'(x)=\frac{2}{\pi x^2}+O(|\log x|),
\]
so the inner ratio is $O(\rho^2)$ uniformly on this compact
$z$ interval. Every limit point of $z_\rho$ is therefore a
zero of $J_1'$. The root exclusions and spacing used in
Lemma~\ref{lem:sixteenth-diskupper} show that $a$ is its only
zero in $[1,\sqrt{112/33}]$. Hence
\[
 \lim_{\rho\downarrow0}\beta_1(\rho)=a^2,
 \qquad \beta_1(\rho)<a^2.
\]
Multiplication by $1-\rho^2$ gives the first assertion and
\[
 \sup_{0<\rho<1}(1-\rho^2)\nu_1(A_{\rho,1})=a^2,
\]
which proves optimality of the common upper constant.

The Rayleigh-sum bound used in Lemma~\ref{lem:sixteenth-diskupper} gives $a^4>192/17>9$, while the disk trial
$s(1-s^2/3)$ has quotient $112/33$ and is not an eigenfunction.
Thus $3<a^2<112/33<4$. \end{proof}

\begin{proposition}[Corrections at the first eight indices and a sharp low-energy surplus]\label{prop:firsteightpositive}
Let $a=j'_{1,1}$. For every annulus and $1\le k\le8$,
\[
 \nu_k(A_{r,R})<\frac{4k-\varepsilon_k}{D},
 \qquad
 (\varepsilon_1,\ldots,\varepsilon_8)=
 \left(4-a^2,8-a^2,\frac{21}{8},\frac{53}{8},
       \frac{68}{31},\frac{192}{31},\frac{14}{13},\frac{128}{39}\right).
\]
Consequently
\begin{equation}\label{eq:optimallowcount}
 N_A(\Lambda)>\frac{D\Lambda}{4}+1-\frac{a^2}{4}
 \qquad(0\le D\Lambda\le32).
\end{equation}
The same corrected inequality holds for $\mathcal M_A$ when
$0<D\Lambda\le32$. The first two eigenvalue subtractions and
the uniform additive counting constant are optimal.
\end{proposition}
\begin{proof}
Scale the outer radius to one and put $\rho=r/R$.
Proposition~\ref{prop:optimalfirstpair} proves the first two
subtractions and their sharpness. The angular trial bounds in
Lemma~\ref{lem:angularfirstmonotone} give
\[
 \nu_4<75/8,\qquad \nu_6<552/31,\qquad \nu_8<1120/39.
\]
Since $D=1-\rho^2<1$, these prove the displayed subtractions
at indices $3,4,5,6,8$.

For the seventh index use the radial polynomial from Proposition~\ref{prop:analyticfirsteight}
$p(s)=1-3s^2+3s^4/2$, subtracting its weighted annular mean.
Its full-disk mass is $1/10$ and energy is $3/2$.
Moreover, $\int_0^1sp(s)\,\mathrm ds=0$ and $|p(s)|\le1$ on $[0,1]$.
Writing $\bar p_\rho$ for its annular mean and $D=1-\rho^2$, we obtain
\[
 \begin{aligned}
 \int_\rho^1s(p-\bar p_\rho)^2\,\mathrm ds
 &=\frac1{10}-\int_0^\rho sp^2\,\mathrm ds
       -\frac{2}{D}\left(\int_0^\rho sp\,\mathrm ds\right)^2\\
 &\ge\frac1{10}-\frac{\rho^2}{2}-\frac{\rho^4}{2D}
  =\frac1{10}-\frac{\rho^2}{2D}\ge\frac1{15}
       \qquad(\rho\le1/4).
 \end{aligned}
\]
Its annular energy is strictly smaller than $3/2$, so its quotient
is less than $45/2$. Together with the constant and the first three
nonzero angular ground sectors, this gives $\nu_7<45/2$.
For $\rho\ge1/4$, the fourth ground sector gives instead
\[
 D\nu_7<D\frac{1120}{39}
 \le\frac{15}{16}\frac{1120}{39}=\frac{350}{13}.
\]
As $45/2<350/13$, the uniform bound is
$D\nu_7<350/13=28-14/13$.

Since $3<a^2<4$, the smallest of the eight subtractions is
$4-a^2\in(0,1)$.
If $n=N_A(\Lambda)\le8$, the next eigenvalue satisfies
$\Lambda<\nu_n<(4n-\varepsilon_n)/D$, giving
$n-D\Lambda/4>\varepsilon_n/4\ge1-a^2/4$.
If $n\ge9$ and $D\Lambda\le32$, its excess is at least one.
For the strict count use $\Lambda\le\nu_n$ in the first case;
the strict upper bound on $\nu_n$ still gives the same strict excess.
To prove sharpness, take $\rho_j\downarrow0$ and energies
$\Lambda_j<\nu_1(A_{\rho_j,1})$ tending to $a^2$.
Proposition~\ref{prop:optimalfirstpair} gives $N_{A_{\rho_j,1}}(\Lambda_j)=1$ and
\[
 N_{A_{\rho_j,1}}(\Lambda_j)
 -\frac{1-\rho_j^2}{4}\Lambda_j\longrightarrow1-\frac{a^2}{4}.
\]
Any larger uniform additive constant therefore fails on these annuli.
\end{proof}

\begin{proposition}[A uniform positive correction at each fixed index]\label{prop:fixedindexpositive}
For each fixed integer $k\ge1$ there exists $c_k>0$, independent
of the two annular radii, such that
\[
 \nu_k(A_{r,R})\le\frac{4k-c_k}{R^2-r^2}.
\]
This is an existence statement; it does not provide an explicit
formula for $c_k$ for arbitrary $k$.
\end{proposition}
\begin{proof}
With outer radius one, put
$L_k(\rho)=(1-\rho^2)\nu_k(A_{\rho,1})$.
On every compact subinterval of $0<\rho<1$, mapping the radial
interval linearly onto $[0,1]$ gives forms with uniformly positive,
continuous coefficients. The min--max principle implies continuity
of each ordered eigenvalue, including at angular-sector crossings.
Theorem~\ref{thm:allannuli} gives $L_k(\rho)<4k$ there.

As $\rho\uparrow1$, the constant radial trials in angular orders
through $m=\lceil k/2\rceil$ give
\[
 L_k(\rho)\le2m^2\log(1/\rho)\longrightarrow0.
\]
As $\rho\downarrow0$, restrict the first $k+1$ disk eigenfunctions
to the annulus. Their mass and energy matrices converge to their
disk matrices, since the deleted integrals tend to zero.
The min--max principle gives
\[
 \limsup_{\rho\downarrow0}L_k(\rho)\le\mu_k(B)<4k.
\]
The last strict disk inequality follows at the first eight indices
from the explicit trials, and at the remaining indices from the
strict disk lattice comparison and the surplus in Appendix~\ref{sec:diskanalyt}. Thus both endpoints are separated from $4k$.
Continuity and compactness on the intervening radius interval give
\[
 c_k:=4k-\sup_{0<\rho<1}L_k(\rho)>0.
\]
Scaling proves the assertion. Weyl's law at any fixed radius ratio
also implies $c_k/k\to0$; a fixed relative improvement of the Weyl
coefficient cannot hold at all indices.
\end{proof}

\begin{theorem}[A positive correction over the entire spectrum]
\label{thm:fullspectrumsqrt}
There is a dimensionless constant $\gamma>0$, independent of the annular
radii and the spectral index, such that every annulus satisfies
\begin{equation}\label{eq:fullspectrumsqrteigen}
 \nu_k(A_{r,R})\le\frac{4k-\gamma\sqrt k}{D},\qquad k\ge1,
 \quad D=R^2-r^2.
\end{equation}
Its inclusive counting function consequently satisfies
\begin{equation}\label{eq:fullspectrumsqrtcount}
 N_A(\Lambda)>\frac{D\Lambda}{4}
       +\frac \gamma8+\frac \gamma{16}\sqrt{D\Lambda},\qquad\Lambda\ge0.
\end{equation}
The proof establishes existence of $\gamma$, rather than an explicit
numerical value. The high-index estimate in
Corollary~\ref{cor:quantitativeallannulieigen} is fully explicit.
\end{theorem}
\begin{proof}
Expansion of \eqref{eq:quantitativeallannulieigen} gives, for $k\ge1604$,
\[
 D\nu_k\le4k-\left(\frac25\sqrt{k+\frac1{400}}-\frac1{50}\right)
       \le4k-\frac38\sqrt k.
\]
Indeed the difference of the two subtractions is at least
$\sqrt k/40-1/50>0$ for $k\ge1$.
Proposition~\ref{prop:fixedindexpositive} supplies $c_k>0$ for each of
the remaining finitely many indices. Thus
\[
 \gamma=\min\left\{\frac38,
       \min_{1\le k\le1603}\frac{c_k}{\sqrt k}\right\}>0
\]
proves \eqref{eq:fullspectrumsqrteigen} for all $k\ge1$.
The minimum involves finitely many unspecified spectral gaps; it is
not a proposed numerical value of $\gamma$.

For the count let $Q=D\Lambda$ and $n=N_A(\Lambda)\ge1$.
The next eigenvalue has $\nu_n>\Lambda$, by the inclusive convention.
Equation~\eqref{eq:fullspectrumsqrteigen} therefore gives
\[
 n-\frac Q4>\frac \gamma4\sqrt n.
\]
In particular $n>Q/4$, so the right side is both at least $\gamma/4$
and greater than $(\gamma/8)\sqrt Q$. Hence
\[
 n-\frac Q4>\max\left\{\frac \gamma4,\frac \gamma8\sqrt Q\right\}
       \ge\frac \gamma8+\frac \gamma{16}\sqrt Q,
\]
which proves the stated count bound, including $\Lambda=0$.
\end{proof}

\subsection{Sharpness and counterexamples}\label{sec:correctionoptimality}

We now compare the quantitative bounds with the two-term Weyl law
and the small-hole limit. These comparisons establish the sharp
square-root order, bound the possible universal coefficients, and
identify stronger corrections that fail at high or low energy.

\paragraph{Asymptotic order.}
For each fixed annulus, the two-term Neumann Weyl law
\cite[Theorem~1.1(2)]{GJWY} gives
\begin{equation}\label{eq:countoptimalityasymptotic}
 N_A(\Lambda)=\frac{D\Lambda}{4}
     +\frac{R+r}{2}\sqrt\Lambda+O_{r,R}(\Lambda^{1/3}).
\end{equation}
Taking the left and right limits at an eigenvalue bounds
the corresponding multiplicity by the remainder; inversion gives
\begin{equation}\label{eq:eigenoptimalityasymptotic}
 D\nu_k=4k-4\frac{R+r}{\sqrt D}\sqrt k+O_{r,R}(k^{1/3}).
\end{equation}
Consequently the orders $\sqrt\Lambda$ in the counting surplus
and $\sqrt k$ in the eigenvalue subtraction cannot be replaced
by larger powers with a fixed positive coefficient, even at a
single fixed annulus. A fixed decrease in the leading Weyl
coefficient is likewise impossible.

\paragraph{Universal coefficients.}
These order statements do not make the displayed coefficients
optimal. A uniform high-frequency bound
$N_A(\Lambda)\ge D\Lambda/4+bU$ must have $b\le1/2$:
for a fixed ratio its limiting surplus divided by $U$ is
$(1+\rho)/2$, and one can then let $\rho\downarrow0$.
Similarly a universal tail bound
$D\nu_k\le4k-c\sqrt k$ must have $c\le4$.
The fixed-geometry asymptotics alone do not provide a threshold
uniform as the radius ratio tends to zero or one.

For a correction valid at every index the first-pair result is
more restrictive. Define the best possible eigenvalue coefficient by
\[
 \Gamma_*:=\inf_{\substack{0<\rho<1\\k\ge1}}
 \frac{4k-(1-\rho^2)\nu_k(A_{\rho,1})}{\sqrt k}.
\]
Theorem~\ref{thm:fullspectrumsqrt} proves $\Gamma_*>0$, whereas
Proposition~\ref{prop:optimalfirstpair} gives
\[
 0<\Gamma_*\le4-a^2.
\]
Its value is not determined by these arguments. The existence
constant $\gamma$ in Theorem~\ref{thm:fullspectrumsqrt} is not
asserted to equal $\Gamma_*$. In the thin limit, at each fixed
index the constant radial trials give
$(1-\rho^2)\nu_k\to0$; hence that limit contributes
$4\sqrt k$ to this quotient and causes no vanishing gap.

\paragraph{Examples.}
For illustration, $a^2\approx3.38995771667$, so
$4-a^2\approx0.61004228333$ and
$\delta_*\approx0.15251057083$. These numbers are illustrations
of the exact Bessel-zero constants, not premises of the proof.
For example, on the fixed annulus $A_{1,4}$,
\eqref{eq:countoptimalityasymptotic} shows that
$N_A(\Lambda)\ge D\Lambda/4+(3/4)U$ fails at every sufficiently
large energy, since its true limiting coefficient is $5/8$.
On $A_{1,2}$, a subtraction $7\sqrt k/D$ fails at every
sufficiently large index, because the coefficient in
\eqref{eq:eigenoptimalityasymptotic} is $4\sqrt3<7$.
Any all-index coefficient larger than $4-a^2$ fails at the
first eigenvalue of sufficiently small-hole annuli.

\begin{remark}[A concrete low-energy exclusion]
On $A_{1,2}$ one has $N_A(1/4)=1$: nonzero angular modes have
eigenvalues strictly above $1/R^2=1/4$, and every positive
zero-order radial mode is above $\pi^2/(R-r)^2$.
Since $D\Lambda/4=3/16$, this gives a completely explicit
counterexample to the stronger all-energy bound
$N_A(\Lambda)\ge D\Lambda/4+1$.
It also excludes adding both the full zero-mode constant and
an additional positive square-root term at every energy.
\end{remark}

\appendix
\section{The first eight positive eigenvalues}
\label{sec:analyticfirsteight}

\indent We give uniform bounds for the first eight positive eigenvalues using
explicit trial functions. These bounds supply the low-energy endpoint
in the small-hole argument of Theorem~\ref{thm:sixteenth-smallhole}
and the corrections in Proposition~\ref{prop:firsteightpositive}.

\begin{lemma}\label{lem:angularfirstmonotone}
Normalize the outer radius to one. For an integer $m\ge1$, let
$\beta_m\left(\rho\right)$ be the lowest eigenvalue of
\[
 -u''-s^{-1}u'+m^2s^{-2}u=\beta u,
 \,\,\, \rho<s<1,\,\,\, u'\left(\rho\right)=u'\left(1\right)=0.
\]
Then $\beta_m$ is strictly decreasing on $\left(0,1\right)$. In particular,
\[
 \beta_1\left(\rho\right)<\frac{112}{33},\quad
 \beta_2\left(\rho\right)<\frac{75}{8},\quad
 \beta_3\left(\rho\right)<\frac{552}{31},\quad
 \beta_4\left(\rho\right)<\frac{1120}{39}.
\]
\end{lemma}
\begin{proof}
Fix $m\ge1$. The variational characterization of the lowest
eigenvalue for this regular separated boundary problem is
\cite{Davies}
\[
 \beta_m\left(\rho\right)
 =\min_{0\ne v\in H^1\left(\rho,1\right)}
 \frac{\displaystyle\int_\rho^1
       \left(s\left|v'\right|^2+\frac{m^2}{s}\left|v\right|^2\right)\dd s}
      {\displaystyle\int_\rho^1s\left|v\right|^2\dd s}.
\]
A minimizer may be chosen nonnegative, since replacing it by its
absolute value does not change this quotient. The differential
equation then makes it strictly positive on the closed interval:
an interior zero of a nonnegative solution would also have zero
derivative, and a zero at either endpoint is accompanied by the
Neumann condition. In both cases uniqueness for the initial-value
problem would force the solution to vanish identically. Moreover,
every solution satisfying the outer Neumann condition is determined
by its value at $s=1$, so the eigenspace is one-dimensional.
Choose the positive eigenfunction $u\left(s,\rho\right)$ normalized by
$u\left(1,\rho\right)=1$.

We first justify its differentiation with respect to $\rho$.
Let $U\left(s,\lambda\right)$ solve
\[
 -\left(sU_s\right)_s+\frac{m^2}{s}U=\lambda sU,
 \,\,\, U\left(1,\lambda\right)=1,\,\,\,
 U_s\left(1,\lambda\right)=0.
\]
Subscripts denote partial derivatives. On any compact interval with
$s>0$, this initial-value solution depends smoothly on $s$ and
$\lambda$. The eigenvalue equation is
$F\left(\rho,\lambda\right):=U_s\left(\rho,\lambda\right)=0$.
At an eigenvalue put $Z=\partial_\lambda U$. Differentiating the
equation for $U$ gives
\[
 -\left(sZ_s\right)_s+\frac{m^2}{s}Z
 =\lambda sZ+sU,\,\,\, Z\left(1\right)=Z_s\left(1\right)=0.
\]
Multiplying this equation by $U$, subtracting the equation for $U$
multiplied by $Z$, and integrating yields
\[
 \int_\rho^1sU^2\dd s
 =\left[s\left(ZU_s-UZ_s\right)\right]_\rho^1
 =\rho U\left(\rho,\lambda\right)Z_s\left(\rho,\lambda\right).
\]
Thus, at $\lambda=\beta_m\left(\rho\right)$,
\[
 \partial_\lambda F\left(\rho,\beta_m\left(\rho\right)\right)
 =\frac{\displaystyle\int_\rho^1sU^2\dd s}
        {\rho U\left(\rho,\beta_m\left(\rho\right)\right)}>0.
\]
The implicit-function theorem therefore gives a smooth local
eigenvalue branch and a smooth eigenfunction
$u\left(s,\rho\right)=U\left(s,\beta_m\left(\rho\right)\right)$.
This branch is the lowest eigenvalue: under the linear change
$s=\rho+\left(1-\rho\right)y$, the Rayleigh quotient is defined on
the fixed space $H^1\left(0,1\right)$, with positive coefficients
depending continuously on $\rho$. On compact subsets of
$0<\rho<1$, their relative changes are uniformly small, so the
minimum is continuous in $\rho$. It must consequently agree locally
with the isolated root just constructed. Since $\rho$ was arbitrary,
$\beta_m$ is differentiable throughout $\left(0,1\right)$.

For the derivative calculation, write $\beta=\beta_m\left(\rho\right)$,
$u=u\left(s,\rho\right)$, and
$\dot u=\partial_\rho u\left(s,\rho\right)$, where $s$ is held fixed.
Define the energy and mass
\[
 E\left(\rho\right)=\int_\rho^1
       \left(su_s^2+\frac{m^2}{s}u^2\right)\dd s,
 \,\,\,
 M\left(\rho\right)=\int_\rho^1su^2\dd s>0.
\]
The weak eigenvalue equation gives $E=\beta M$. Leibniz's rule,
including the moving lower endpoint and the dependence of $u$ on
$\rho$, gives
\begin{align*}
 E'\left(\rho\right)
 &=-\left(\rho u_s\left(\rho,\rho\right)^2
              +\frac{m^2}{\rho}u\left(\rho,\rho\right)^2\right)
   +2\int_\rho^1\left(su_s\dot u_s+\frac{m^2}{s}u\dot u\right)\dd s,
 \\
 M'\left(\rho\right)
 &=-\rho u\left(\rho,\rho\right)^2
   +2\int_\rho^1su\dot u\dd s.
\end{align*}
Here $u_s\left(\rho,\rho\right)=0$. The cross term in $E'$ is
computed by integration by parts:
\begin{align*}
 \int_\rho^1\left(su_s\dot u_s+\frac{m^2}{s}u\dot u\right)\dd s
 &=\left[su_s\dot u\right]_\rho^1
   +\int_\rho^1\left(-\left(su_s\right)_s+\frac{m^2}{s}u\right)\dot u\dd s
 \\
 &=\beta\int_\rho^1su\dot u\dd s.
\end{align*}
The boundary term vanishes because the original eigenfunction has
$u_s=0$ at both endpoints. No boundary condition on $\dot u$ is
needed: it is an admissible test function for the Neumann weak
equation. In fact, differentiating the moving boundary condition
$u_s\left(\rho,\rho\right)=0$ gives
$\dot u_s\left(\rho,\rho\right)+u_{ss}\left(\rho,\rho\right)=0$,
so $\dot u$ generally does not satisfy a homogeneous Neumann
condition at the inner endpoint.

Differentiate $E=\beta M$ and substitute the preceding formulas.
The two integrals containing $u\dot u$ cancel exactly, leaving
\[
 \beta_m'\left(\rho\right)M\left(\rho\right)
 =E'\left(\rho\right)-\beta M'\left(\rho\right)
 =\rho\left(\beta-\frac{m^2}{\rho^2}\right)
       u\left(\rho,\rho\right)^2.
\]
Therefore
\[
 \beta_m'\left(\rho\right)
 =\frac{\rho\left(\beta_m\left(\rho\right)-m^2/\rho^2\right)
             u\left(\rho,\rho\right)^2}
        {\displaystyle\int_\rho^1su\left(s,\rho\right)^2\dd s}.
\]

To determine its sign, use the constant trial function $v=1$.
Its energy is $m^2\log\left(1/\rho\right)$ and its mass is
$\left(1-\rho^2\right)/2$. Hence
\[
 \beta_m\left(\rho\right)
 \le\frac{2m^2\log\left(1/\rho\right)}{1-\rho^2}
 =\frac{m^2\log\left(\rho^{-2}\right)}{1-\rho^2}
 <\frac{m^2\left(\rho^{-2}-1\right)}{1-\rho^2}
 =\frac{m^2}{\rho^2}.
\]
The strict inequality uses $\log z<z-1$ for $z>1$, with
$z=\rho^{-2}$. Since $\rho>0$, $u\left(\rho,\rho\right)>0$, and
$M\left(\rho\right)>0$, the derivative formula now gives
$\beta_m'\left(\rho\right)<0$ at every $0<\rho<1$.
Thus $\beta_m$ is strictly decreasing on the whole interval.
The assumption $m\ge1$ is essential for strictness; in order zero,
the constant eigenfunction gives $\beta_0\equiv0$.

For the explicit constants use
\[
 f_m\left(s\right)=s^m\left(1-a_ms^2\right),\,\,\,
 \left(a_1,a_2,a_3,a_4\right)=\left(1/3,1/2,2/3,2/3\right).
\]
The factor $s^m$ ensures regularity at the disk centre. The quadratic
correction reduces the energy of the pure monomial. These rational coefficients give bounds sufficient for
the first-eight-eigenvalue estimate below. Direct integration on the
unit disk, with the common angular factor suppressed, gives
\[
 K_m=\int_0^1\left(f_m'^2+m^2s^{-2}f_m^2\right)s\dd s,
 \,\,\, M_m=\int_0^1f_m^2s\dd s,
\]
Since
$f_m'\left(s\right)=ms^{m-1}-a_m\left(m+2\right)s^{m+1}$, expansion and integration
of the monomials give
\[
\begin{aligned}
 K_m
 &=\int_0^1\left(2m^2s^{2m-1}
     -4a_mm\left(m+1\right)s^{2m+1}
     +a_m^2\left(m^2+\left(m+2\right)^2\right)s^{2m+3}\right)\dd s\\
 &=m-2a_mm+\frac{a_m^2\left(m^2+\left(m+2\right)^2\right)}{2m+4}
   =m\left(1-a_m\right)^2+\frac{2a_m^2}{m+2},\\
 M_m
 &=\int_0^1\left(s^{2m+1}-2a_ms^{2m+3}+a_m^2s^{2m+5}\right)\dd s\\
 &=\frac1{2m+2}-\frac{a_m}{m+2}+\frac{a_m^2}{2m+6}.
\end{aligned}
\]
Substituting the four rational coefficients gives the exact values
\[
 \begin{array}{c|c|c|c|c}
 m&a_m&K_m&M_m&K_m/M_m\\\hline
 1&1/3&14/27&11/72&112/33\\
 2&1/2&5/8&1/15&75/8\\
 3&2/3&23/45&31/1080&552/31\\
 4&2/3&16/27&13/630&1120/39
 \end{array}
\]
Fix $\rho>0$. Restricting $f_m$ to $\left[\varepsilon,1\right]$
gives a trial quotient tending to $K_m/M_m$ as
$\varepsilon\downarrow0$. For $0<\varepsilon<\rho/2$, monotonicity
and the variational bound give
$\beta_m\left(\rho/2\right)\le\beta_m\left(\varepsilon\right)
\le E_\varepsilon\left(f_m\right)/M_\varepsilon\left(f_m\right)$,
where $E_\varepsilon$ and $M_\varepsilon$ are the energy and mass
integrals over $\left[\varepsilon,1\right]$. Taking the limit and then
using strict monotonicity yields
\[
 \beta_m\left(\rho\right)<\beta_m\left(\rho/2\right)\le K_m/M_m.
\]
\end{proof}

\begin{proposition}\label{prop:analyticfirsteight}
For every $0<r<R$, with $D=R^2-r^2$ and $\nu_0=0$,
\[
 \nu_k\left(A_{r,R}\right)<\frac{4k}{D}\,\,\,\left(1\le k\le8\right).
\]
In particular, the counting assertion needed here is
\[
 N_A\left(\Lambda\right)>\frac{D\Lambda}{4}
 \,\,\,\text{if }0\le D\Lambda\le32.
\]
\end{proposition}
\begin{proof}
Scale to $R=1$, put $\rho=r$ and $D=1-\rho^2$.
Each nonzero angular order supplies two real modes, and distinct
orders are orthogonal for both the energy and mass forms.
Lemma~\ref{lem:angularfirstmonotone}, together with the zero mode, gives
\[
 \nu_2<\frac{112}{33}<4,\,\,\,
 \nu_4<\frac{75}{8}<12,\,\,\,
 \nu_6<\frac{552}{31}<20,\,\,\,
 \nu_8<\frac{1120}{39}<32.
\]
Since $D<1$, these prove the required bounds for $k=1,2,3,4,5,6,8$.

Only the seventh index needs an additional argument: the fourth
angular estimate alone is slightly larger than its Weyl threshold $28$.
The split at $\rho=1/5$ makes both the area reduction and
the radial trial-function estimate below sufficient. First suppose $\rho\ge1/5$. The fourth angular order gives
\[
 D\nu_7<D\frac{1120}{39}
 \le\frac{24}{25}\frac{1120}{39}<28.
\]
If $0<\rho\le1/5$, use the radial polynomial
$p\left(s\right)=1-3s^2+(3/2)s^4$, subtracting its weighted mean on
$\left[\rho,1\right]$. Its full-disk weighted mean is zero, its mass is $1/10$,
and its energy is $3/2$. To quantify the mean subtraction, put
\[
 W_\rho=\int_\rho^1s\,\dd s=\frac{1-\rho^2}{2},\,\,\,
 b_\rho=\int_\rho^1p\left(s\right)s\,\dd s
        =-\int_0^\rho p\left(s\right)s\,\dd s.
\]
The mean-zero trial function is $p-b_\rho/W_\rho$, with squared
weighted norm $\int_\rho^1p^2s\,\dd s-b_\rho^2/W_\rho$.
Since $\left|p\right|\le1$ on $\left[0,\rho\right]$, the lost disk mass is at most
$\rho^2/2$, and $\left|b_\rho\right|\le\rho^2/2$. Its norm is therefore
at least
\[
 \frac1{10}-\frac{\rho^2}{2\left(1-\rho^2\right)}
 \ge\frac1{10}-\frac1{48}=\frac{19}{240}.
\]
Its Rayleigh quotient is therefore at most $360/19<20$.
Together with the constant and the first three nonzero angular
orders this gives an eight-dimensional trial space whose maximal
quotient is less than $20$. Hence $D\nu_7<20<28$ also in this case.

For the counting statement set $q=D\Lambda$. If $0\le q<4$, the
zero mode suffices. Otherwise $k=\left\lfloor q/4\right\rfloor$ lies between
one and eight, and $\nu_k<4k/D\le\Lambda$. Thus
$N_A\left(\Lambda\right)\ge k+1>q/4$, including $q=32$.
\end{proof}

The trial functions above give the specific quantitative bounds needed
in the small-hole argument. Their use is an application of the standard
min--max principle \cite{Davies}.

\section{A disk counting surplus}
\label{sec:diskanalyt}

\indent Let $B=\left\{x\in\mathbb R^2:\left|x\right|<1\right\}$ be the unit disk and let
$\mathcal N_B\left(X\right)$ denote its Neumann count, including zero, at energy
$X^2$. We need a surplus of one mode over the Weyl term.
The following bound supplies the extra disk mode needed at the small-hole
endpoint. Its strict-count form also provides the small-hole endpoint
estimate in Proposition~\ref{prop:fixedindexpositive}.
The proof combines a disk remainder estimate with
94 explicit integer comparisons arranged in 22 rows below. The phase-counting input
and the adjustable high-frequency remainder are established results
of \cite{FLPSballs,GMWZ}.

The lower endpoint $X^2=32=4\cdot8$ connects the disk surplus with Proposition~\ref{prop:analyticfirsteight}. With the outer radius normalized to one and $\rho=r/R$, the inequality $X^2\le32$ implies $\left(1-\rho^2\right)X^2\le32$, so that proposition already covers the lower-energy part of the small-hole argument. The present estimate starts at an endpoint shared with that bound.
\begin{lemma}\label{lem:diskanalyt_surplus}
For every $X\ge\sqrt{32}$,
\begin{equation}\label{eq:diskanalyt_surplus}
 \mathcal N_B\left(X\right)-1\ge\frac{X^2}{4}.
\end{equation}
\end{lemma}

\begin{proof}
Use the disk action $G_X$ from \eqref{eq:diskaction}.
The sectorwise derivative-zero bound is
\cite[Proposition~3.1, equation~(3.3)]{FLPSballs}.
Here $j'_{m,k}$ is the $k$th derivative zero in angular order $m$,
with the convention $j'_{0,1}=0$ of Section~2 of that paper.
For $0\le m\le X$ it gives
\[
 \#\left\{k\ge1:j'_{m,k}\le X\right\}
 \ge\left\lfloor G_X\left(m\right)+\frac34\right\rfloor.
\]
In that paper the frequency variable is $\lambda=X$, the spectral
threshold is $\lambda^2$, and the multiplicity weights are
$\kappa_{2,0}=1$ and $\kappa_{2,m}=2$ for $m\ge1$.
Thus its $\mathcal N_{\mathbb D}^{\mathrm N}\left(\lambda\right)$ equals
our $\mathcal N_B\left(X\right)$. Summing the sector inequalities gives
\cite[Theorem~2.3 and equations~(2.8)--(2.9)]{FLPSballs}:
\begin{equation*}
 \mathcal N_B\left(X\right)\ge\mathcal P^{\mathrm N}_2\left(X\right):=
 \left\lfloor G_X\left(0\right)+\frac34\right\rfloor
 +2\sum_{m=1}^{\left\lfloor X\right\rfloor}
       \left\lfloor G_X\left(m\right)+\frac34\right\rfloor.
\end{equation*}
The $3/4$ shift and these weights are exactly the Neumann
lattice convention of the cited paper. The spectral count includes
zero, but the lattice lower bound need not count it at very small
frequency: at $X=0$ the two sides are one and zero, respectively.
No additional zero mode is added to the displayed lower bound.
The multiplicity is two for each positive angular order and one
for order zero. We first prove the desired surplus for $X\ge20$,
then give an explicit finite bridge to $X=\sqrt{32}$.

\paragraph{The analytic range $X\ge20$.}
For $X\ge2$, \cite[Theorem~2.7, equations~(2.3)--(2.4), and
Section~2.2]{GMWZ} imply
\[
 \mathcal N_B\left(X\right)\ge\frac{X^2}{4}+\mathcal R\left(X\right),
\]
where, whenever $0<\sigma<\pi/2$ and
\begin{equation}\label{eq:diskanalyt_sigma}
 X\ge\frac{\pi}{4\left(\sin\sigma-\sigma\cos\sigma\right)},
\end{equation}
one has
\begin{equation*}
 \mathcal R\left(X\right)>
 \left(A\cos\sigma-C-\frac1\pi\right)X
 -\frac34+\frac{\cos\sigma}{4\pi},
\end{equation*}
with
\[
 A=\frac{24\pi-7\sqrt2}{8\left(4\pi-\sqrt2\right)},\,\,\,
 C=\frac{8\pi-3\sqrt2}{8\left(4\pi-\sqrt2\right)}.
\]
Notice that $A=C+1/2$ and $0<C<1/4$.
Choose $\sigma=1/2$. This value balances a positive linear remainder
with a simple Taylor bound on its admissible frequency threshold. This choice suffices to start the analytic range at $X=20$. Since
\[
 \sin\left(1/2\right)-\tfrac12\cos\left(1/2\right)
 =\int_0^{1/2}t\sin t\,\dd t
 \ge\int_0^{1/2}\left(t^2-t^4/6\right)\,\dd t=\frac{13}{320},
\]
the bound $\pi<22/7$ implies that the threshold in
\eqref{eq:diskanalyt_sigma} is less than $1760/91<20$.
Moreover, $\cos\left(1/2\right)\ge7/8$ and $\pi>25/8$ give
\begin{align*}
 A\cos\left(1/2\right)-C-\frac1\pi
 &=\frac12\cos\left(1/2\right)-C\left(1-\cos\left(1/2\right)\right)-\frac1\pi\\
 &>\frac7{16}-\frac1{32}-\frac8{25}=\frac{69}{800}.
\end{align*}
Also $\cos\left(1/2\right)/\left(4\pi\right)>49/704$. Thus, for $X\ge20$,
\begin{equation}\label{eq:diskanalyt_high}
 \mathcal R\left(X\right)>
 \frac{69}{800}X-\frac34+\frac{49}{704}
 \ge\frac{3677}{3520}>1.
\end{equation}
This proves \eqref{eq:diskanalyt_surplus}, strictly, on that range.

\paragraph{A polynomial lower bound for the action.}
Set $H\left(t\right)=\sqrt{1-t^2}+t\arcsin t$ for $0\le t\le1$.
Differentiation cancels the two terms $t/\sqrt{1-t^2}$ and gives
$H'\left(t\right)=\arcsin t$. For $0\le u<1$, convexity of
$v\mapsto\left(1-v\right)^{-1/2}$ gives $\left(1-u^2\right)^{-1/2}\ge1+u^2/2$.
Hence
\[
 H'\left(t\right)=\int_0^t\frac{\dd u}{\sqrt{1-u^2}}
 \ge t+\frac{t^3}{6}.
\]
Since $H\left(0\right)=1$, integration from zero to $t$ proves
\begin{equation*}
 H\left(t\right)\ge1+\frac{t^2}{2}+\frac{t^4}{24}.
\end{equation*}
The constant bound follows from the positive integral
\[
 \frac{22}{7}-\pi
 =\int_0^1\frac{t^4\left(1-t\right)^4}{1+t^2}\dd t>0.
\]
\paragraph{The finite bridge.}
The finite bridge is constructed in layers: a row with lower count $c$
covers the energy interval $\left[4k,4\left(c-1\right)\right]$, so choosing the next $k=c-1$
ensures closed overlap. The angular endpoints are chosen to make the polynomial
sign tests positive while keeping the table short.
In the following table, $\left(M_1,\ldots,M_s\right)$ lists the largest
angular order retained at floor level $n=1,\ldots,s$. The next
column is a vector of strict rational lower bounds
$\ell_{k,n}=d_{k,n}/1000$ for $G_{2\sqrt{k}}\left(M_n\right)+3/4$.
Thus the $n$th entry is strictly greater than $n$. These decimals
denote exact fractions, not rounded values of the action.

The identity $G_X\left(m\right)=XH\left(m/X\right)/\pi-m/2$
converts the polynomial lower bound into action estimates.
To check each bound, put $m=M_n$ and
$P=384k^2+48km^2+m^4$. The bound for $H$ above and $\pi<22/7$ give
\[
 G_{2\sqrt{k}}\left(m\right)+\frac34
 >\frac{7P}{4224k\sqrt{k}}-\frac m2+\frac34.
\]
For $d=d_{k,n}$ the right side exceeds $d/1000$ exactly when
$7000P>4224k\sqrt{k}\left(d-750+500m\right)$. Both sides are positive,
so squaring reduces this to the explicit integer test
\[
 T\left(k,m,d\right):=\left(7000P\right)^2-k\left(4224k\left(d-750+500m\right)\right)^2>0.
\]
All $94$ tests for the $22$ rows have positive integer values.
For example, $k=8$, $m=4$, $d=1018$ give $P=30976$ and
\[
 T\left(8,4,1018\right)=216832000^2-8\cdot76640256^2
            =26285506035712>0.
\]
This proves $G_{2\sqrt8}\left(4\right)+3/4>1.018>1$.
Among the $94$ integer tests the smallest value is
$T\left(8,1,2078\right)=3506047446592>0$; every other row can be checked
by the same displayed integer formula.
Because $G_X$ decreases in its angular argument, the $n$th
comparison supplies $2M_n+1$ modes at floor level $n$.
Summing these layers gives the count $c$, except for the starred
row, whose two additional modes are justified below.
\begin{center}
\small\setlength{\tabcolsep}{4pt}
\begin{longtable}{r|l|l|r}
\hline
$k$ & $(M_1,\ldots,M_s)$ & $(\ell_{k,1},\ldots,\ell_{k,s})$ & $c$\\
\hline
\endfirsthead

\hline
$k$ & $(M_1,\ldots,M_s)$ & $(\ell_{k,1},\ldots,\ell_{k,s})$ & $c$\\
\hline
\endhead

\hline
\endfoot

8&$4,1$&$1.018,2.078$&12\\
11&$4,1$&$1.255,2.384$&$14^{*}$\\
13&$5,2,0$&$1.118,2.133,3.044$&17\\
16&$6,2,0$&$1.044,2.375,3.295$&19\\
18&$6,3,0$&$1.152,2.120,3.449$&21\\
20&$7,3,1$&$1.011,2.257,3.113$&25\\
24&$7,4,1$&$1.196,2.130,3.383$&27\\
26&$8,4,2$&$1.044,2.247,3.057$&31\\
30&$8,5,2,0$&$1.206,2.104,3.293,4.235$&34\\
33&$9,5,3,0$&$1.084,2.257,3.030,4.405$&38\\
37&$10,6,3,1$&$1.002,2.101,3.239,4.133$&44\\
43&$10,7,4,1$&$1.194,2.031,3.118,4.435$&48\\
47&$11,7,4,2,0$&$1.091,2.193,3.299,4.159,5.112$&53\\
52&$12,8,5,2,0$&$1.018,2.062,3.117,4.383,5.338$&59\\
58&$12,8,5,3,1$&$1.178,2.280,3.359,4.190,5.106$&63\\
62&$13,9,6,3,1$&$1.064,2.101,3.128,4.351,5.270$&69\\
68&$14,9,6,4,2$&$1.001,2.298,3.348,4.152,5.036$&75\\
74&$14,10,7,4,2,0$&$1.136,2.174,3.183,4.372,5.261,6.224$&80\\
79&$15,11,7,5,2,0$&$1.039,2.023,3.350,4.131,5.441,6.406$&86\\
85&$15,11,8,5,3,1$&$1.165,2.191,3.177,4.334,5.194,6.125$&92\\
91&$16,12,9,6,3,1$&$1.080,2.060,3.008,4.123,5.395,6.328$&100\\
99&$17,12,9,6,4,2,0$&$1.032,2.267,3.240,4.371,5.210,6.113,7.081$&107\\
\hline
\end{longtable}
\end{center}

The extra modes in the starred row are needed to advance the finite bridge. Its lattice layers alone give $c=12$ at $X^2=44$, which covers only the interval $\left[44,44\right]$ and does not reach the next row at $X^2=52$. Raising the count to $14$ extends the covered interval to $\left[44,52\right]$, preserving the closed chain of overlaps. This is the purpose of the following variational comparison.

For the starred row $k=11$, the phase layers supply 12 modes
and use no mode of angular order five. In that sector, take
\[
 f\left(r\right)=r^5\left(1-3r^2/4\right).
\]
Suppressing the common angular factor, direct integration gives
\[
 \int_0^1\left(f'\left(r\right)^2+25r^{-2}f\left(r\right)^2\right)r\,\dd r
 =\frac{53}{112},\,\,\,
 \int_0^1f\left(r\right)^2r\,\dd r=\frac{61}{5376}.
\]
The quotient is $2544/61<44$. The cosine and sine angular
factors yield two independent modes, orthogonal for both forms
to the sectors used in the phase lower count. Hence
$\mathcal N_B\left(\sqrt{44}\right)\ge14$, as claimed.

Each row yields $\mathcal N_B\left(X\right)\ge c$ throughout
$4k\le X^2\le4\left(c-1\right)$, by monotonicity in the energy. Therefore
$\mathcal N_B\left(X\right)-1\ge X^2/4$ on that closed interval. Each next value of
$k$ is exactly the preceding $c-1$, so these intervals cover
$32\le X^2\le424$. This overlaps the range $X\ge20$ proved in
\eqref{eq:diskanalyt_high}, completing the proof.
\end{proof}

The surplus in \eqref{eq:diskanalyt_surplus}, rather than the disk
P\'olya inequality alone, is what allows one to subtract the zero mode
of the removed disk in the small-hole proof.

\section{Complete angular-layer tables}\label{sec:analytic-table-data}

The tables collect the rational data used in the proof of
Theorem~\ref{thm:allannuli}. Their verification follows the formulas
and worked examples in Section~\ref{sec:allproof}. In each count
table, a layer ending at angular order $p_j$ contributes $2p_j+1$
eigenfunctions: one in order zero and two in every positive order.
The endpoint profiles and error bounds are listed separately.
\begin{center}\scriptsize\setlength{\tabcolsep}{3pt}
\begin{longtable}{r|p{4.55in}|r}
\caption{Integer margins for the disk-action comparisons in Lemma~\ref{lem:tenth-diskmargins}.}\label{tab:smallhole-actions}\\
$k$&$D(k,M_n,n;a,b)$&$a^2-b^2k$\\\hline
\endfirsthead
\multicolumn{3}{l}{\textit{Table \thetable\ (continued)}}\\
$k$&$D(k,M_n,n;a,b)$&$a^2-b^2k$\\\hline
\endhead

$81$&$(24939900,\allowbreak 24634300,\allowbreak 16086400,\allowbreak 61479100,\allowbreak 17552700,\allowbreak 146966400)$&$0$\\
$87$&$(114592,\allowbreak 84155548,\allowbreak 81584800,\allowbreak 136583452,\allowbreak 89431452,\allowbreak 65989852)$&$489$\\
$91$&$(29259328,\allowbreak 22150464,\allowbreak 2950332,\allowbreak 45024000,\allowbreak 144969468,\allowbreak 120452668)$&$116$\\
$97$&$(79817440,\allowbreak 87009120,\allowbreak 73548060,\allowbreak 124989600,\allowbreak 59077600,\allowbreak 19799200,\allowbreak 6751200)$&$225$\\
$102$&$(37901500,\allowbreak 23321980,\allowbreak 141484860,\allowbreak 28902460,\allowbreak 131881600,\allowbreak 90587200,\allowbreak 76867200)$&$100$\\
$108$&$(91123900,\allowbreak 91635580,\allowbreak 65475520,\allowbreak 112307260,\allowbreak 23972860,\allowbreak 184561600,\allowbreak 170035200)$&$1600$\\
$112$&$(30335424,\allowbreak 3054016,\allowbreak 122895808,\allowbreak 176153404,\allowbreak 84593404,\allowbreak 24001404,\allowbreak 240576000)$&$1481$\\
$118$&$(89240448,\allowbreak 76949632,\allowbreak 33844924,\allowbreak 71367616,\allowbreak 185208508,\allowbreak 121390908,\allowbreak 89616508)$&$1569$\\
\end{longtable}\end{center}

\begin{center}\scriptsize
\begin{longtable}{r|rc|l|r}
\caption{Radial layers and their multiplicity counts for Proposition~\ref{prop:analyticleftbridge}.}\label{tab:left-counts}\\
$Q_0$&$g$&\text{trial}&$(p_2,p_3,\ldots)$&$C$\\\hline
\endfirsthead
\multicolumn{5}{l}{\textit{Table \thetable\ (continued)}}\\
$Q_0$&$g$&\text{trial}&$(p_2,p_3,\ldots)$&$C$\\\hline
\endhead

\multicolumn{5}{c}{$[a,b]=[1/10,3/20]$}\\
100&8&B&4,1&29\\
116&8&B&5,2&33\\
132&9&B&5,3,0&38\\
152&10&B&6,3,1&44\\
176&11&B&7,4,2&52\\
208&12&B&8,5,2&58\\
232&13&B&8,5,3,0&63\\
\end{longtable}\end{center}

\begin{center}\scriptsize
\setlength{\tabcolsep}{3pt}\begin{longtable}{r|rr|p{1.65in}|p{2.5in}}
\caption{Endpoint action profiles and phase-error bounds for Table~\ref{tab:left-counts}.}\label{tab:left-checks}\\
$Q_0$&$1000L_a$&$1000L_b$&$\text{profiles}$&$(\alpha_j,\allowbreak s_j)$\\
\hline
\endfirsthead
\multicolumn{5}{l}{\textit{Table \thetable\ (continued)}}\\
$Q_0$&$1000L_a$&$1000L_b$&$\text{profiles}$&$(\alpha_j,\allowbreak s_j)$\\
\hline
\endhead

\multicolumn{5}{c}{$[a,b]=[1/10,3/20]$}\\
$100$&$10050$&$10114$&$E/E,\allowbreak P/P$&$*,\allowbreak *$\\
$116$&$10824$&$10893$&$D/D,\allowbreak F/F$&$*,\allowbreak *$\\
$132$&$11547$&$11620$&$E/E,\allowbreak F/F,\allowbreak P/P$&$*,\allowbreak *,\allowbreak (13/125,\allowbreak 577/500)$\\
$152$&$12390$&$12469$&$D/D,\allowbreak F/F,\allowbreak P/P$&$*,\allowbreak *,\allowbreak *$\\
$176$&$13333$&$13418$&$D/D,\allowbreak E/E,\allowbreak F/P$&$*,\allowbreak *,\allowbreak *$\\
$208$&$14494$&$14587$&$C/C,\allowbreak E/E,\allowbreak F/P$&$*,\allowbreak *,\allowbreak *$\\
$232$&$15308$&$15405$&$D/D,\allowbreak E/E,\allowbreak F/F,\allowbreak P/P$&$*,\allowbreak *,\allowbreak *,\allowbreak (79/1000,\allowbreak 153/100)$\\
\end{longtable}\end{center}

\begin{center}\scriptsize\setlength{\tabcolsep}{3pt}\begin{longtable}{r|r|p{1.35in}|r|rr|c}
\caption{Layers on $[1/5,3/10]$ in Proposition~\ref{prop:analyticmiddlebridge}.}\label{tab:middle-a-counts}\\
$Q_0$&$p_1$&$(p_2,\allowbreak \ldots,\allowbreak p_k)$&$C$&$1000L_-$&$1000L_+$&$p_1\text{ proof}$\\
\hline
\endfirsthead
\multicolumn{7}{l}{\textit{Table \thetable\ (continued)}}\\
$Q_0$&$p_1$&$(p_2,\allowbreak \ldots,\allowbreak p_k)$&$C$&$1000L_-$&$1000L_+$&$p_1\text{ proof}$\\
\hline
\endhead

$100$&$8$&$(4,\allowbreak 1)$&$29$&$10206$&$10482$&$B$\\
$116$&$9$&$(5,\allowbreak 2)$&$35$&$10992$&$11290$&$B$\\
$140$&$10$&$(6,\allowbreak 3)$&$41$&$12076$&$12403$&$B$\\
$164$&$11$&$(7,\allowbreak 4)$&$47$&$13070$&$13424$&$B$\\
$188$&$11$&$(7,\allowbreak 4,\allowbreak 1)$&$50$&$13994$&$14373$&$B$\\
$200$&$12$&$(8,\allowbreak 5,\allowbreak 2)$&$58$&$14433$&$14824$&$B$\\
$232$&$13$&$(9,\allowbreak 6,\allowbreak 3)$&$66$&$15545$&$15966$&$B$\\
$264$&$14$&$(9,\allowbreak 6,\allowbreak 4)$&$70$&$16583$&$17032$&$B$\\
\end{longtable}\end{center}

\begin{center}\scriptsize\setlength{\tabcolsep}{3pt}\begin{longtable}{r|p{2.15in}|p{2.6in}}
\caption{Endpoint profiles and raw-error bounds for Table~\ref{tab:middle-a-counts}.}\label{tab:middle-a-checks}\\
$Q_0$&$\text{profiles}$&$\text{raw errors }(j,\allowbreak e,\allowbreak s)$\\
\hline
\endfirsthead
\multicolumn{3}{l}{\textit{Table \thetable\ (continued)}}\\
$Q_0$&$\text{profiles}$&$\text{raw errors }(j,\allowbreak e,\allowbreak s)$\\
\hline
\endhead

$100$&$(8,\allowbreak 8),\allowbreak  (P,\allowbreak P)$&$(3,\allowbreak 84,\allowbreak 1779)$\\
$116$&$(3,\allowbreak 3),\allowbreak  (P,\allowbreak P)$&$-$\\
$140$&$(3,\allowbreak 3),\allowbreak  (7,\allowbreak P)$&$-$\\
$164$&$(5,\allowbreak 3),\allowbreak  (8,\allowbreak P)$&$-$\\
$188$&$(3,\allowbreak 3),\allowbreak  (7,\allowbreak P),\allowbreak  (P,\allowbreak P)$&$(4,\allowbreak 51,\allowbreak 2614)$\\
$200$&$(5,\allowbreak 5),\allowbreak  (8,\allowbreak 8),\allowbreak  (P,\allowbreak P)$&$(4,\allowbreak 99,\allowbreak 2081)$\\
$232$&$(5,\allowbreak 5),\allowbreak  (8,\allowbreak 8),\allowbreak  (P,\allowbreak P)$&$-$\\
$264$&$(5,\allowbreak 3),\allowbreak  (8,\allowbreak 8),\allowbreak  (7,\allowbreak P)$&$-$\\
\end{longtable}\end{center}

\begin{center}\scriptsize\setlength{\tabcolsep}{3pt}\begin{longtable}{r|r|p{1.35in}|r|rr|c}
\caption{Layers on $[3/10,2/5]$ in Proposition~\ref{prop:analyticmiddlebridge}.}\label{tab:middle-b-counts}\\
$Q_0$&$p_1$&$(p_2,\allowbreak \ldots,\allowbreak p_k)$&$C$&$1000L_-$&$1000L_+$&$p_1\text{ proof}$\\
\hline
\endfirsthead
\multicolumn{7}{l}{\textit{Table \thetable\ (continued)}}\\
$Q_0$&$p_1$&$(p_2,\allowbreak \ldots,\allowbreak p_k)$&$C$&$1000L_-$&$1000L_+$&$p_1\text{ proof}$\\
\hline
\endhead

$100$&$8$&$(4,\allowbreak 1)$&$29$&$10482$&$10910$&$B$\\
$116$&$9$&$(5,\allowbreak 2)$&$35$&$11290$&$11751$&$B$\\
$140$&$10$&$(6,\allowbreak 3)$&$41$&$12403$&$12909$&$B$\\
$164$&$11$&$(7,\allowbreak 4)$&$47$&$13424$&$13972$&$B$\\
$188$&$12$&$(8,\allowbreak 5)$&$53$&$14373$&$14960$&$B$\\
$212$&$13$&$(8,\allowbreak 5,\allowbreak 0)$&$56$&$15263$&$15886$&$B$\\
$224$&$13$&$(9,\allowbreak 6,\allowbreak 2)$&$64$&$15689$&$16329$&$B$\\
$256$&$14$&$(10,\allowbreak 7,\allowbreak 4)$&$74$&$16772$&$17457$&$B$\\
\end{longtable}\end{center}

\begin{center}\scriptsize\setlength{\tabcolsep}{3pt}\begin{longtable}{r|p{2.15in}|p{2.6in}}
\caption{Endpoint profiles and raw-error bounds for Table~\ref{tab:middle-b-counts}.}\label{tab:middle-b-checks}\\
$Q_0$&$\text{profiles}$&$\text{raw errors }(j,\allowbreak e,\allowbreak s)$\\
\hline
\endfirsthead
\multicolumn{3}{l}{\textit{Table \thetable\ (continued)}}\\
$Q_0$&$\text{profiles}$&$\text{raw errors }(j,\allowbreak e,\allowbreak s)$\\
\hline
\endhead

$100$&$(8,\allowbreak P),\allowbreak  (P,\allowbreak P)$&$(3,\allowbreak 44,\allowbreak 2981)$\\
$116$&$(3,\allowbreak 8),\allowbreak  (P,\allowbreak P)$&$(3,\allowbreak 62,\allowbreak 2733)$\\
$140$&$(3,\allowbreak 3),\allowbreak  (P,\allowbreak P)$&$(3,\allowbreak 133,\allowbreak 2201)$\\
$164$&$(3,\allowbreak 3),\allowbreak  (P,\allowbreak P)$&$-$\\
$188$&$(5,\allowbreak 5),\allowbreak  (8,\allowbreak P)$&$-$\\
$212$&$(3,\allowbreak 3),\allowbreak  (8,\allowbreak P),\allowbreak  (P,\allowbreak P)$&$(4,\allowbreak 27,\allowbreak 4578)$\\
$224$&$(5,\allowbreak 5),\allowbreak  (8,\allowbreak P),\allowbreak  (P,\allowbreak P)$&$(4,\allowbreak 33,\allowbreak 4260)$\\
$256$&$(5,\allowbreak 5),\allowbreak  (8,\allowbreak 8),\allowbreak  (P,\allowbreak P)$&$(4,\allowbreak 92,\allowbreak 3052)$\\
\end{longtable}\end{center}

\begin{center}\scriptsize\setlength{\tabcolsep}{3pt}\begin{longtable}{r|r|p{1.35in}|r|rr|c}
\caption{Layers on $[2/5,1/2]$ in Proposition~\ref{prop:analyticmiddlebridge}.}\label{tab:middle-c-counts}\\
$Q_0$&$p_1$&$(p_2,\allowbreak \ldots,\allowbreak p_k)$&$C$&$1000L_-$&$1000L_+$&$p_1\text{ proof}$\\
\hline
\endfirsthead
\multicolumn{7}{l}{\textit{Table \thetable\ (continued)}}\\
$Q_0$&$p_1$&$(p_2,\allowbreak \ldots,\allowbreak p_k)$&$C$&$1000L_-$&$1000L_+$&$p_1\text{ proof}$\\
\hline
\endhead

$100$&$9$&$(5)$&$30$&$10910$&$11547$&$B$\\
$120$&$9$&$(6)$&$32$&$11952$&$12649$&$B$\\
$128$&$10$&$(6,\allowbreak 1)$&$37$&$12344$&$13063$&$B$\\
$148$&$11$&$(7,\allowbreak 3)$&$45$&$13273$&$14047$&$B$\\
$180$&$12$&$(8,\allowbreak 5)$&$53$&$14638$&$15491$&$B$\\
$212$&$13$&$(9,\allowbreak 6)$&$59$&$15886$&$16812$&$B$\\
$236$&$14$&$(10,\allowbreak 7)$&$65$&$16761$&$17738$&$B$\\
$260$&$15$&$(10,\allowbreak 7)$&$67$&$17593$&$18618$&$B$\\
$268$&$15$&$(11,\allowbreak 7)$&$69$&$17861$&$18903$&$B$\\
$276$&$15$&$(11,\allowbreak 8,\allowbreak 1)$&$74$&$18126$&$19183$&$B$\\
$296$&$16$&$(11,\allowbreak 8,\allowbreak 4)$&$82$&$18771$&$19866$&$B$\\
$328$&$17$&$(12,\allowbreak 9,\allowbreak 5)$&$90$&$19760$&$20912$&$B$\\
$360$&$18$&$(13,\allowbreak 10,\allowbreak 7)$&$100$&$20701$&$21908$&$(6,\allowbreak 6)$\\
$400$&$19$&$(14,\allowbreak 11,\allowbreak 8)$&$108$&$21821$&$23094$&$(6,\allowbreak 6)$\\
$432$&$20$&$(15,\allowbreak 11,\allowbreak 8)$&$112$&$22677$&$24000$&$(6,\allowbreak 6)$\\
$448$&$20$&$(15,\allowbreak 12,\allowbreak 9)$&$116$&$23094$&$24440$&$B$\\
$464$&$20$&$(16,\allowbreak 12,\allowbreak 9)$&$118$&$23502$&$24873$&$B$\\
$472$&$21$&$(16,\allowbreak 12,\allowbreak 9)$&$120$&$23704$&$25086$&$(6,\allowbreak 6)$\\
$480$&$21$&$(16,\allowbreak 12,\allowbreak 10,\allowbreak 1)$&$125$&$23904$&$25298$&$(6,\allowbreak 6)$
\end{longtable}\end{center}

\begin{center}\scriptsize\setlength{\tabcolsep}{3pt}\begin{longtable}{r|p{2.15in}|p{2.6in}}
\caption{Endpoint profiles and raw-error bounds for Table~\ref{tab:middle-c-counts}.}\label{tab:middle-c-checks}\\
$Q_0$&$\text{profiles}$&$\text{raw errors }(j,\allowbreak e,\allowbreak s)$\\
\hline
\endfirsthead
\multicolumn{3}{l}{\textit{Table \thetable\ (continued)}}\\
$Q_0$&$\text{profiles}$&$\text{raw errors }(j,\allowbreak e,\allowbreak s)$\\
\hline
\endhead

$100$&$(3,\allowbreak P)$&$-$\\
$120$&$(3,\allowbreak P)$&$-$\\
$128$&$(3,\allowbreak P),\allowbreak  (P,\allowbreak P)$&$(3,\allowbreak 26,\allowbreak 4835)$\\
$148$&$(3,\allowbreak P),\allowbreak  (P,\allowbreak P)$&$(3,\allowbreak 38,\allowbreak 4380)$\\
$180$&$(5,\allowbreak 3),\allowbreak  (P,\allowbreak P)$&$(3,\allowbreak 122,\allowbreak 3047)$\\
$212$&$(5,\allowbreak 5),\allowbreak  (P,\allowbreak P)$&$-$\\
$236$&$(5,\allowbreak 5),\allowbreak  (8,\allowbreak P)$&$-$\\
$260$&$(5,\allowbreak 5),\allowbreak  (P,\allowbreak P)$&$-$\\
\end{longtable}\end{center}

\par\medskip\noindent
\textbf{Use of AI tools.}
OpenAI Codex assisted with mathematical exposition and explicit
calculations. The authors take full responsibility for the
mathematical content of this article.
\par\medskip\noindent
\textbf{Data availability.}
All quantities used in the proofs are specified in this article.
\par\medskip\noindent
\textbf{Conflict of interest.}
The authors declare that they have no conflicts of interest.


\begin{thebibliography}{99}
\small\setlength{\itemsep}{3pt}

\bibitem{Berezin}
F.~A. Berezin,
Covariant and contravariant symbols of operators,
\emph{Math. USSR-Izv.} \textbf{6} (1972), no.~5, 1117--1151.

\bibitem{BucurHenrot}
D.~Bucur and A.~Henrot,
Maximization of the second non-trivial Neumann eigenvalue,
\emph{Acta Math.} \textbf{222} (2019), no.~2, 337--361.

\bibitem{Davies}
E.~B. Davies,
\emph{Spectral theory and differential operators},
Cambridge Studies in Advanced Mathematics, vol.~42,
Cambridge University Press, Cambridge, 1995.

\bibitem{FilonovConvex}
N.~Filonov,
On the P\'olya conjecture for the Neumann problem in planar convex domains,
\emph{Comm. Pure Appl. Math.} \textbf{78} (2025), no.~3, 537--544.

\bibitem{FLPSballs}
N.~Filonov, M.~Levitin, I.~Polterovich, and D.~A. Sher,
P\'olya's conjecture for Euclidean balls,
\emph{Invent. Math.} \textbf{234} (2023), 129--169.

\bibitem{LorchMuldoonSzego1972}
\begingroup\hbadness=10000
L. Lorch, M. E. Muldoon, and P. Szeg\H{o},
Higher monotonicity properties of certain Sturm--Liouville functions.~IV,
\emph{Canad. J. Math.} \textbf{24} (1972), no.~2, 349--368.
\href{https://doi.org/10.4153/CJM-1972-029-9}
{doi:10.4153/CJM-1972-029-9}.
\par\endgroup

\bibitem{ChungLeePark}
S.-Y. Chung, S. Lee, and Y. W. Park,
Complex zeros of Bessel function derivatives and associated
orthogonal polynomials, 2024,
\href{https://arxiv.org/abs/2406.09746}{arXiv:2406.09746}.

\bibitem{Phase}
N.~Filonov, M.~Levitin, I.~Polterovich, and D.~A. Sher,
Uniform enclosures for the phase and zeros of Bessel functions
and their derivatives,
\emph{SIAM J. Math. Anal.} \textbf{56} (2024), no.~6, 7644--7682.
\href{https://arxiv.org/abs/2402.06956v3}{Author version: arXiv:2402.06956v3}.

\bibitem{FLPS}
N.~Filonov, M.~Levitin, I.~Polterovich, and D.~A. Sher,
P\'olya's conjecture for Dirichlet eigenvalues of annuli,
\emph{J. London Math. Soc.} \textbf{113} (2026), no.~2, e70425.

\bibitem{FLPSneumann}
N.~Filonov, M.~Levitin, I.~Polterovich, and D.~A. Sher,
P\'olya's conjecture for higher-dimensional Neumann balls,
preprint, 2026,
\href{https://arxiv.org/abs/2607.29305}{arXiv:2607.29305}.

\bibitem{FrankLarson}
R.~L. Frank and S.~Larson,
Semiclassical inequalities for Dirichlet and Neumann Laplacians
on convex domains,
\emph{Comm. Pure Appl. Math.} \textbf{79} (2026), no.~3, 762--822.

\bibitem{FreitasSalavessa}
P.~Freitas and I.~Salavessa,
Families of non-tiling domains satisfying P\'olya's conjecture,
\emph{J. Math. Phys.} \textbf{64} (2023), no.~12, 121503.

\bibitem{GanJiangLin}
Z.~Gan, R.~Jiang, and F.-H. Lin,
The improved Berezin--Li--Yau inequality and Kr\"oger inequality
and consequences,
\emph{Sci. China Math.} (2026).

\bibitem{GMWZ}
J.~Guo, C.~Miao, W.~Wang, and G.~Zhan,
Improvement of P\'olya's conjecture for balls and cylinders,
to appear in \emph{Sci. China Math.},
\href{https://arxiv.org/abs/2511.17050}{arXiv:2511.17050}.

\bibitem{GJWY}
J.~Guo, T.~Jiang, Z.~Wang, and X.~Yang,
Bessel functions and Weyl's law for balls and spherical shells,
\emph{Adv. Math.} \textbf{484} (2026), 110715;
\href{https://arxiv.org/abs/2412.14059}{arXiv:2412.14059}.

\bibitem{HW}
X.~He and Z.~Wang,
P\'olya's conjecture for thin products,
\emph{Int. Math. Res. Not. IMRN} (2026), no.~16, rnag182.

\bibitem{Ivrii}
V.~Ya. Ivrii,
Second term of the spectral asymptotic expansion of the
Laplace--Beltrami operator on manifolds with boundary,
\emph{Funct. Anal. Appl.} \textbf{14} (1980), 98--106.

\bibitem{JiangLin}
R.~Jiang and F.-H. Lin,
P\'olya's conjecture up to $\epsilon$-loss and quantitative
estimates for the remainder of Weyl's law,
preprint, 2025,
\href{https://arxiv.org/abs/2507.04307}{arXiv:2507.04307}.

\bibitem{Kellner}
R.~Kellner,
On a theorem of Polya,
\emph{Amer. Math. Monthly} \textbf{73} (1966), no.~8, 856--858.

\bibitem{Kroger}
P.~Kr\"oger,
Upper bounds for the Neumann eigenvalues on a bounded domain
in Euclidean space,
\emph{J. Funct. Anal.} \textbf{106} (1992), no.~2, 353--357.

\bibitem{Laptev}
A.~Laptev,
Dirichlet and Neumann eigenvalue problems on domains in
Euclidean spaces,
\emph{J. Funct. Anal.} \textbf{151} (1997), no.~2, 531--545.

\bibitem{LiNeumann}
Y.~Li,
P\'olya's conjecture for the Neumann eigenvalues on Euclidean balls,
preprint, 2026,
\href{https://arxiv.org/abs/2607.25958}{arXiv:2607.25958}.

\bibitem{LiYau}
P.~Li and S.-T. Yau,
On the Schr\"odinger equation and the eigenvalue problem,
\emph{Comm. Math. Phys.} \textbf{88} (1983), no.~3, 309--318.

\bibitem{Melas}
A.~D. Melas,
A lower bound for sums of eigenvalues of the Laplacian,
\emph{Proc. Amer. Math. Soc.} \textbf{131} (2003), no.~2, 631--636.

\bibitem{NIST}
F.~W.~J. Olver, A.~B. Olde Daalhuis, D.~W. Lozier, B.~I. Schneider,
R.~F. Boisvert, C.~W. Clark, B.~R. Miller, B.~V. Saunders,
H.~S. Cohl, and M.~A. McClain (eds.),
\emph{NIST Digital Library of Mathematical Functions},
Chapter~10, especially Sections~10.5 and~10.21.

\bibitem{Polya1954}
G.~P\'olya,
\emph{Mathematics and plausible reasoning},
vols.~I--II, Princeton University Press, Princeton, NJ, 1954.

\bibitem{Polya1961}
G.~P\'olya,
On the eigenvalues of vibrating membranes,
\emph{Proc. London Math. Soc.} (3) \textbf{11} (1961), 419--433.

\bibitem{SZ}
Z.~Song and W.~Zou,
Some inequalities on Laplacian eigenvalues with Neumann type
boundary data for planar disk and annulus,
\emph{Partial Differ. Equ. Appl.} \textbf{7} (2026), article~15.

\bibitem{Szego}
G.~Szeg\H{o},
Inequalities for certain eigenvalues of a membrane of given area,
\emph{J. Rational Mech. Anal.} \textbf{3} (1954), 343--356.

\bibitem{Urakawa}
H.~Urakawa,
Lower bounds for the eigenvalues of the fixed vibrating membrane problems,
\emph{Tohoku Math. J.} (2) \textbf{36} (1984), no.~2, 185--189.

\bibitem{Weinberger}
H.~F. Weinberger,
An isoperimetric inequality for the $N$-dimensional free
membrane problem,
\emph{J. Rational Mech. Anal.} \textbf{5} (1956), 633--636.

\bibitem{Weyl}
H.~Weyl,
\"Uber die asymptotische verteilung der eigenwerte,
\emph{Nachr. Ges. Wiss. G\"ottingen, Math.-Phys. Kl.}
(1911), 110--117.

\end{thebibliography}
\end{document}